\documentclass[oneside,10pt,reqno]{amsart}
\usepackage{amssymb,amsmath,amsthm,bbm,enumerate,mdwlist,url,multirow,hyperref,amsthm,amsaddr}
\usepackage[pdftex]{graphicx}
\usepackage[shortlabels]{enumitem}
\usepackage[T1]{fontenc}
\allowdisplaybreaks
\theoremstyle{definition}
\newtheorem{definition}{Definition}
\theoremstyle{theorem}
\newtheorem{proposition}[definition]{Proposition}
\newtheorem{lemma}[definition]{Lemma}
\newtheorem{theorem}[definition]{Theorem}
\newtheorem{corollary}[definition]{Corollary}

\numberwithin{equation}{section}
\numberwithin{definition}{section}
\theoremstyle{remark}
\newtheorem{remark}[definition]{Remark}

\newtheorem{question}[definition]{Question}
\newtheorem{example}[definition]{Example}

\DeclareMathOperator*{\esssup}{ess\,sup}
\DeclareMathOperator*{\essinf}{ess\,inf}

\def\PP{\mathbb{P}}

\def\AA{\mathcal{A}}

\def\AAA{\mathsf{A}}
\def\HH{\mathcal{H}}
\def\LL{\mathcal{L}}

\def\Cl{\mathrm{Cl}}

\def\L2{\mathrm{L}^2}
\def\pr{\mathrm{pr}}
\def\QQ{\mathbb{Q}}
\def\BB{\mathcal{B}}

\def\AA{\mathcal{A}}
\def\II{\mathcal{I}}

\def\EE{\mathcal{E}}
\def\id{\mathrm{id}}
\def\GG{\mathcal{G}}
\def\FF{\mathcal{F}}

\def\TT{\mathcal{T}}
\def\UU{\mathcal{U}}

\def\Ber{\mathrm{Ber}}
\def\Aut{\mathrm{Aut}}
\def\dd{\mathrm{d}}

\def\KK{\mathcal{K}}

\def\Bin{\mathrm{Bin}}
\def\PPP{\mathbb{P}}

\def\at{\mathrm{at}}

 \def\NNN{\mathbf{N}}
 \def\stable{{\FF_{\mathrm{stb}}}}
 \def\L2{\mathrm{L}^2}
 \def\Lattice{\hat{\PP}}

\def\id{\mathrm{id}}
\def\dd{\mathrm{d}}

\def\rr{\tilde{r}}

\def\rr{\mathsf{r}}
\def\KK{\mathsf{K}}
\def\inff{\mathbf{I}}
\def\var{\mathrm{var}}
\newcounter{count1}
\makeatletter
\def\@tocline#1#2#3#4#5#6#7{\relax
  \ifnum #1>\c@tocdepth 
  \else
    \par \addpenalty\@secpenalty\addvspace{#2}%
    \begingroup \hyphenpenalty\@M
    \@ifempty{#4}{%
      \@tempdima\csname r@tocindent\number#1\endcsname\relax
    }{%
      \@tempdima#4\relax
    }%
    \parindent\z@ \leftskip#3\relax \advance\leftskip\@tempdima\relax
    \rightskip\@pnumwidth plus4em \parfillskip-\@pnumwidth
    #5\leavevmode\hskip-\@tempdima
      \ifcase #1
       \or\or \hskip 1em \or \hskip 2em \else \hskip 3em \fi%
      #6\nobreak\relax
    \hfill\hbox to\@pnumwidth{\@tocpagenum{#7}}\par
    \nobreak
    \endgroup
  \fi}
\makeatother

\makeatletter
\@namedef{subjclassname@2020}{%
  \textup{2020} Mathematics Subject Classification}
\makeatother
\begin{document}
\title{A potpourri of results in the general theory of stochastic noises}
\author{Matija Vidmar}
\address{Department of Mathematics, Faculty of Mathematics and Physics, University of Ljubljana}
\email{matija.vidmar@fmf.uni-lj.si}

\begin{abstract}
The objects under inspection, on a given probability space, are noise(-type) Boolean algebras -- distributive non-empty sublattices of the lattice of all complete sub-$\sigma$-fields, whose every element admits an independent complement. Special attention is given to the spectral decompositions of the algebras of operators   generated by the conditional expectations of their members (acting on $\L2$).

 Atoms of the spectra are identified in  explicit terms. For a reverse filtration admitting an innovating sequence  of equiprobable random signs, a discreteness property of the  spectral measure of the associated noise Boolean algebra is shown to imply product-typeness. 
 
 Noise projections on the spectral space are introduced, which correspond to restricting a noise to a part of its domain space. They appear to play a natural (albeit technical) role in the general analysis. In particular through them manifests itself the tensor structure of the spectral decomposition. 
 
 The spectrum is precisely delineated in the classical case (i.e. when the noise Boolean algebra is complete), a kind-of standard ``symmetric Fock space'' form thereof is procured. The latter result leads  to a new characterization of classicality and blackness involving ``spectral independence''. 
\end{abstract}


\thanks{Financial support from the Slovenian Research Agency is acknowledged (programme No. P1-0402). Parts of this text were completed while the author was in receipt of a Fernandes Fellowship at the University of Warwick; he is grateful for the former as well as the hospitality of the latter.} 

\date{\today}

\maketitle
\vspace{-1cm}
\tableofcontents 
\normalsize
\section{Introduction}

In  \cite{tsirelson}  Tsirelson distilled stochastic noises into a definitive fundamental class of stochastic objects: ``a noise is a kind of homomorphism from a Boolean algebra of domains to the lattice of $\sigma$-fields'' \cite[Abstract]{tsirelson}; the appropriate abstraction of the image of such a homomorphism was christened a noise-type Boolean algebra, and appears already to capture many of the salient features thereof (e.g. classicality, blackness, chaos expansion). Noise-type Boolean algebras become the basic object of study.

At this level of abstraction, the hallmark property of a stochastic noise $N$, namely that $N(A_1\cup A_2)$ is generated by the independent $\sigma$-fields $N(A_1)$ and $N(A_2)$ whenever $A_1$ and $A_2$ are disjoint domains, is reflected in the two defining properties of a noise-type Boolean algebra $B$: first, that each member $x$ of $B$ admits a unique $x'\in B$ that is independent of $x$ and generates the whole together with $x$ (an independent complement); second, that  $y=(y\cap x)\lor (y\cap x')$ for $\{x,y\}\subset B$ (which is equivalent to the $\lor$-$\cap$ distributivity of $B$, given the first). What is in some sense lost in this  framework is the ``translational invariance'' of a noise, or at least one does not a priori insist on it (but see Item~\ref{noise-boolean:factorization} of Subsection~\ref{subsection:noise-boolean}). 

In this text we will be interested in questions concerning the general theory of stochastic noises only,  by which  we mean roughly those which concern any noise Boolean algebra without reference to an ex ante given indexation by (factorization over) domains of some particular ambient (perhaps topological, or measurable) space. Though, factorizations will make an appearence -- but only in examples, or else a posteriori, as derivative objects of the noise Boolean algebra (and its spectrum), not as given a priori. 

Contributions fall into two categories. --- In the first -- (relatively: some care is needed) straightforward extensions to the general setting of  results known in more or less special cases of noises (usually one-dimensional, continuous factorizations over the intervals of $\mathbb{R}$). Though only formally new it seems worthwhile putting them on record in their (apparent) full scale of relevance. Under this category we list  Subsection~\ref{subsection:subnoises-parts-of-spectrum} and Section~\ref{section:separate-classsical}. --- In the second -- results which may be considered (as far as the author knows) new or at at any rate non-trivial modifications of more particular settings. We would put in this basket the remainder of the text with the exception of the preliminary Section~\ref{sec:setting}. The highlights of this category have been put on record in the abstract. 

Apart from what has been noted in the preceding, at this point the table of contents is as good a guide to what this paper delivers (aims to deliver) as any; we would avoid repetition here, deferring further ``introductory'' comments instead to the beginnings of the individual sections. The reader is invited to cast a glance at these should he/she desire a more detailed ``landscape'' view before reading, or in deciding whether or not to read the remainder of this article [in (any) detail].

It appears that \cite{tsirelson} is to date the only (other)  paper that studies noise-type Boolean algebras without ``dragging along'' an indexation by (factorization over) an algebra of domains. The preprints \cite{tsirelson-arxiv-1,tsirelson-arxiv-2,tsirelson2011noise} should be mentioned alongside \cite{tsirelson} as they contain some results which the published version \cite{tsirelson} does not.   For a general historical account and a description of earlier/other perspectives in the area of stochastic noises it seems difficult to better \cite[Introduction, Section~1.6]{tsirelson}; see also the older survey \cite{tsirelson-nonclassical}. In any event, a wide selection  of related literature  will appear along the way. 

In some sense what follows simply contains answers to a myriad of questions which the author found himself puzzling over while reading \cite{tsirelson} and the older pieces of literature on stochastic noises. Understanding some of them opened still new ones. Many of these are noted explicitly, they are scattered throughout the text. 

\section{Basic notions, setting, notation, preliminaries}\label{sec:setting}
The groundwork is laid; dry, perhaps, but necessary to ensure ``smooth sailing'' thereafter. All assertions made are meant to be well-known, evident (after a ``minute's thought''), or are suitably referenced/their validity is argued. Besides fixing the setting and notation, the purpose of this section is to summarize the relevant  state of the art -- to the extent that will be needed in the continuation of the text. Though not strictly speaking necessary, the reader had probably better be familiar with \cite{tsirelson}  before proceding further (not in the sense that the present article is ``deeper'' --- which it is not --- but in the sense that \cite{tsirelson} is its ``immediate predecessor''). 

\subsection{Generalities}
\begin{enumerate}[wide, labelwidth=!, labelindent=0pt,label=\textbf{(\alph*)}]
\item\label{generalities:a} $\sigma(\cdots)$ will  mean the sigma-field generated by whatever stands in lieu of $\cdots$ \emph{and} by the negligible sets of the relevant measure (the measure itself may or may not  be complete, the measurable negligible sets of the domain of the measure are meant), according to the context (the measure is not made explicit in this piece of notation, which should not cause confusion). 
For $\sigma$-fields $\GG$ and $\HH$ we use $\GG/\HH$ to denote the collection of all $\GG/\HH$-measurable maps. The Borel $\sigma$-field of a topological space $X$ is denoted $\mathcal{B}_X$ (the topology being understood from context). $b\FF$ are the real-valued bounded $\FF$-measurable maps.

\item Except for one place, where it will be explicit (Remark~\ref{rmk:complex-to-real}), all L-spaces and all Hilbert spaces considered in this text are real. Their complex versions could be used instead, but the intervention of complex numbers in what follows is at best tangential, therefore the real spaces seem altogether more appropriate to the context.

\item For a set $A$, ${A\choose n}$ is the set of the $n$-element subsets of $A$, $n\in \mathbb{N}_0$, and then $(2^A)_{\mathrm{fin}}:=\cup_{n\in \mathbb{N}_0}{A\choose n}$ (resp. $(2^A)_{\mathrm{inf}}$) is the collection of its finite (resp. infinite) subsets (while $2^A$ will be the whole power set of $A$, as usual).  We routinely identify ${T\choose 1}$ with $T$ without special mention. By convention $0^0=1$. $c_A$ denotes the counting measure on $(A,2^A)$. Sometimes we write $A_{<b}:=\{a\in A:a<b\}$ for $A\subset [-\infty,\infty]$ and $b\in [-\infty,\infty]$, similarly $A_{\geq b}$ etc. are defined.

\item\label{generalities:lebesgue-rohlin} A standard measurable space is one that is endowed with the Borel $\sigma$-field of a complete separable metric. As such it is always Borel isomorphic (isomorphic as a measurable space) to one of $\mathbb{R}$, $\mathbb{N}$ or $[n]:=\{1,\ldots,n\}$ ($[0]=\emptyset$; $[n]_0:=[n]\cup \{0\}$) for some $n\in \mathbb{N}_0$ with their standard $\sigma$-fields \cite[Corollary~6.8.8]{bogachev}. Often useful is  the Blackwell property \cite[Exercise~14.16]{kechris}: if $(A_n)_{n\in \mathbb{N}}$ is a sequence of Borel sets in a standard measurable space, then a Borel set is $\sigma(A_n:n\in \mathbb{N})$-measurable iff it is saturated for the equivalence relation $\sim$ given by $(x\sim y)$ $\Leftrightarrow$ ($\mathbbm{1}_{A_n}(x)=\mathbbm{1}_{A_n}(y)$ for all $n\in \mathbb{N}$).

By a  standard (Lebesgue–Rokhlin)  measure space $(S,\Sigma,\mu)$ we will undertand one that results from the completion of a $\sigma$-finite measure on a standard measurable space, possibly adjoined with an extra arbitrary null set (and completed again). Any standard measure space is mod-$0$ isomorphic to an interval of $\mathbb{R}$ with Lebesgue measure, a countable family of point atoms, or a combination --- disjoint union ---  of both \cite[Theorem~9.4.7]{bogachev} \cite[Theorem~1.4.6]{petersen1983ergodic}. Here, a mod-$0$ isomorphism between two measure spaces is a map which carries bimeasurably, measure-preservingly (i.e. pushes forward the measure on the first space to the measure on the second space) and bijectively a conegligible set onto a conegligible set; if only the last property holds we speak of a mod-$0$ bijection.  

A measurable and measure-preserving map between two standard measure spaces that is injective off a null set is automatically a mod-$0$ isomorphism \cite[p.~22, Section~2.5, Theorem on isomorphisms]{rohlin}. 
The product (completion implicit) of two standard measure spaces is a standard measure space \cite[Section~3.4]{rohlin}. The Blackwell property is translated as follows  into the setting of a standard measure space: if $(A_n)_{n\in \mathbb{N}}$ is a sequence of measurable sets in a standard measure space, then a set is $\sigma(A_n:n\in \mathbb{N})$-measurable iff it differs by a null set from a measurable set that is saturated for the equivalence relation $\sim$ given by $(x\sim y)$ iff ($\mathbbm{1}_{A_n}(x)=\mathbbm{1}_{A_n}(y)$ for all $n\in \mathbb{N}$).

In case $(S,\Sigma,\mu)$  is a standard measure space and $\mu$ is a probability we speak of a standard probability (space) \cite[Section~2.4]{ito-standard}. If one has a measure-preserving map of a standard probability space into a countably separated complete probability space, then the latter is also standard \cite[Theorems~3-2~\&~4-3]{de-la-Rue}. 

\item Throughout $(\Omega,\FF,\PP)$ is an \emph{essentially separable} probability space, maybe complete, maybe not, it does not matter.  Essential separability means, by definition, existence of a countable $\AA\in 2^\FF$ such that $\sigma(\AA)=\FF$. It is equivalent to the  Hilbert space $\L2(\PP)$ being separable, and  it is also the same as demanding that the topology of the probability-of-symmetric-difference pseudometric $d_\PP:\FF\times\FF\to [0,1]$ associated to $\PP$ is separable.  Note that every standard probability space is complete, by definition, but also essentially separable. The converse fails in general.  Indeed, under our assumptions, $(\Omega,\FF,\PP)$ need not even be essentially countably separated (which, together with essential separability still would not be equivalent to standardness -- e.g. due to the ``image-measure catastrophe'' \cite[p.~30]{schwartz}), i.e. there need not be any co-negligible set, whose points are separated by a countable family of measurable sets. Simple (counter)examples also show that one can have a complete $\PP$ be essentially separable without there being a denumerable $\AA$ such that $\overline{\PP\vert_\AA}=\PP$, while the converse is clearly true.

\item An isomorphism between two  probability spaces $(Q,\mathcal{Q},\QQ)$ and $(O,\mathcal{O},\mathbb{O})$ (automorphism when $\QQ=\mathbb{O}$) is a bijective map between $\mathcal{Q}/_\QQ$ and $\mathcal{O}/_\mathbb{O}$ that preserves complements, countable unions and probabilities (that is to say, we have a $\sigma$-isomorphism of measure algebras). Such a map has \cite{dellacherie} a natural unique extension to a bijective map $\mathrm{L}^0(\QQ)\to \mathrm{L}^0(\mathbb{O})$, which preserves the laws  and real-valued Borelian applications on finite sub-collections of $\mathrm{L}^0(\QQ)$ \cite[Definition~5.1]{laurent-main} (more generally this is true not only of random variables, but random elements valued in any standard measurable space \cite[Lemma~5.5]{laurent-main}). In turn, a bijective map $\mathrm{L}^0(\QQ)\to \mathrm{L}^0(\mathbb{O})$, which preserves the laws  and real-valued Borelian applications, results in this way from a unique isomorphism between $\QQ$ and $\mathbb{O}$; one is not necessarily careful about distinguising between the two points of view. If in the definition of an isomorphism we forego the bijective property we get an embedding of $\QQ$ into $\mathbb{O}$ \cite[Definition~1.1]{watanabe}; it is automatically injective and therefore an isomorphism onto its image. 

Of course a bimeasurable measure-preserving bijection, and more generally every mod-$0$ isomorphism, induces an isomorphism of probabilities, but not every isomorphism of probabilities is got in this way  (not even when they are complete and essentially separable). However, if one has a $\sigma$-isomorphism of the measure algebras attached to two standard measure spaces (not necessarily probability), then there is a mod-$0$ isomorphism between the two \cite[Theorem~1.4.7]{petersen1983ergodic}.

Up to isomorphism $\PP$ is determined uniquely by specifying the nonincreasing sequence of the sizes of its atoms and $\PP$ is isomorphic to a standard probability (just because essential separability is also equivalent to the existence of an essentially generating random variable, whose diffuse part we may assume is uniform). (Of course it does not mean  that $\PP$ is itself Lebesgue-Rokhlin. Rather, a complete $\PP$ is standard iff there is a negligible set off of which there is a random variable $X$ that is injective, generates $\FF$ mod-$0$ (the $\PP\vert_{X^{-1}(\mathcal{B}_\mathbb{R})}$-completion is meant) and such that the image of $X$ has inner measure one w.r.t. $X_\star\PP$.  We will not need this  fact.)

The group of all automorphisms  of $\PP$ under composition is denoted $\Aut(\PP)$. It is naturally embedded as a closed subspace into the Polish group of all unitary operators on $\L2(\PP)$ under the strong operator topology \cite{uribe2014}, so is itself a Polish group. 


\item Expectations are written as $\PP[f]:=\mathbb{E}_\PP[f]$, $f\cdot \PP:=(\FF\ni F \mapsto \PP[f;F])$ is the indefinite integral and $f_\star\PP:=\PP\circ f^{-1}$ is the law of $f$ under $\PP$ (more generally this notation is used for integrals against, and push-forwards of measures); conditional expectations as $\PP[f\vert x]$ (where $x$ is the conditioning $\sigma$-field). $\langle\cdot,\cdot\rangle$ and $\Vert\cdot\Vert$ are the usual inner product and norm on $\L2(\PP)$. $\L2(\PP)_0:=\{f\in \L2(\PP):\PP[f]=0\}$ is the zero-mean subspace of $\L2(\PP)$.



\item\label{generalities:lattice}  $\Lattice$ denotes the  bounded lattice of complete sub-$\sigma$-fields of $\FF$: bottom element $0_\PP:=\PP^{-1}(\{0,1\})$, top element $1_\PP:=\FF$, meet $\land:=\cap$ and join $\lor:=\text{complete generated $\sigma$-field}$. We allow $0_\PP=1_\PP$. In general $ \Lattice$ is not distributive \cite[Example~1.1]{vidmar_2019}.  For $x\in \Lattice$ and $A\subset \Lattice$ we put $A_x:=A\cap 2^x$, ${}_xA:=\{a\in A:x\subset a\}$; for further $y\in \Lattice$, $A_{x,y}:=\{u\lor v:(u,v)\in A_x\times A_y\}$.  We say $y\in\Lattice$ is an independent complement of $x\in \Lattice$ if $x$ is independent of $y$ ($\therefore$ $x\land y=0_\PP$) and $x\lor y=1_\PP$. Concerning the general arithmetic of (independent) $\sigma$-fields the reader may consult \cite{vidmar_2019}; the distributivity \cite[Proposition~3.4]{vidmar_2019} and the ``tensorisation'' of independent conditioning \cite[Lemma~2.2]{vidmar_2019} are especially useful.

\item For $x\in \Lattice$, $\PP_x:\L2(\PP)\to \L2(\PP\vert_x)$ is the $\L2$-conditional expectation operator w.r.t. $x$; $\PP_x(X)=\PP[X\vert x]$ for $X\in \L2(\PP)$. By Jensen, for $\{x,y\}\subset \Lattice$, $\PP_x\leq \PP_y$ iff $x\subset y$. Closed linear subspace of $\L2(\PP)$ of the form $\L2(\PP\vert_\GG)$ for  a $\GG\in \Lattice$ are called $\L2$-subspaces; they are precisely the ones which are closed under $\vert\cdot\vert$ and which contain the one-dimensional space of constants, see \cite[Theorem~3]{Sid57}, where they are called ``measurable subspaces'' instead (our terminology is rather from \cite[Fact~2.1]{tsirelson}, which uses the term ``type $\L2$ (sub)spaces'', but  the ``type'' seems superfluous).

\item For independent $x$ and $y$ from $\Lattice$, $\L2(\PP\vert_{x\lor y})=\L2(\PP\vert_x)\otimes \L2(\PP\vert_y)$ up to the natural unitary isomorphism that uniquely extends $(\L2(\PP\vert_x)\otimes \L2(\PP\vert_y)\ni f_x\otimes f_y\mapsto f_xf_y\in \L2(\PP\vert_{x\lor y}))$. We will sometimes avoid the qualification ``up to the natural unitary isomorphism'' (or an alias thereof) in the preceding, its presence being understood implicitly. In general, if $C\subset \L2(\PP_x)$ and $D\subset\L2(\PP\vert_{x'})$ for some $x\in B$, then by $C\otimes D$ we mean the closure of the linear span (in $\L2(\PP)$) of $\{cd:(c,d)\in C\times D\}$.
 
 \item\label{generalities:topology} $\Lattice$ is endowed with the strong operator topology of the associated conditional expectations, i.e. the topology generated by the evaluation maps $(\Lattice\ni x\mapsto \PP_x(f)\in \L2(\PP))$, $f\in \L2(\PP)$. No other topology on $\Lattice$ will be used (though, generally speaking, many others are of interest/have attracted attention \cite[Figure~1]{VIDMAR-conv}).  We proceed to make explicit a series of relevant facts concerning this topology of $\Lattice$.

\begin{enumerate}[(i)]
\item\label{metric} Fix a total countable family $(\rr_k)_{k\in \KK}$ in $\L2(\PP)$, where $\KK$ is an initial segment of $\mathbb{N}$ or $\KK=\mathbb{N}$.  The topology of $\Lattice$ is metrizable; a metric $d_\PP$ compatible with it is given by $$d_{\Lattice}(x,y):=\sum_{k\in \KK}2^{-k}\frac{\Vert \PP[\rr_k\vert x]-\PP[\rr_k\vert y]\Vert}{1+\Vert \PP[\rr_k\vert x]-\PP[\rr_k\vert y]\Vert},\quad (x,y)\in \Lattice\times \Lattice.$$ 
\begin{proof} Linearity and density argument. \end{proof}
\item\label{rmk:general:ii} Suppose $z=(z_n)_{n\in \mathbb{N}}$ is a sequence in $\Lattice$ and $z_\infty\in \Lattice$. Then $\lim_{n\to\infty}z_n=z_\infty$ 
 iff $\lim_{n\to\infty}\PP_{z_n}(\rr_k)=\PP_{z_\infty}(\rr_k)$ in $\L2(\PP)$ for all $k\in \KK$ iff  $\lim_{n\to\infty}\PP[\PP[f\vert z_n]^2]=\PP[\PP[f\vert z_\infty]^2]$ for all $f\in \L2(\PP)$. 
 \begin{proof}
 The first equivalence follows by linearity and density. For the second, the condition (being clearly necessary) implies by polarization $\lim_{n\to\infty}\PP[\PP[f\vert z_n]g]=\PP[\PP[f\vert z_\infty]g]$  for $\{f,g\}\subset \L2(\PP)$, therefore $\PP[(\PP[f\vert z_n]-\PP[f\vert z_\infty])^2]=\PP[\PP[f\vert z_n]^2]+\PP[\PP[f\vert z_\infty]^2]-2\PP[\PP[f\vert z_n]\PP[f\vert z_\infty]]\to \PP[(\PP[f\vert z_\infty])^2]+\PP[(\PP[f\vert z_\infty])^2]-2\PP[\PP[f\vert z_\infty]^2]=0$ as $n\to\infty$ for all $f\in\L2(\PP)$.
 \end{proof}
\item Still the family $(\rr_k)_{k\in \KK}$ is as in Item~\ref{metric}. Let $(z_n)_{n\in \mathbb{N}}$ be Cauchy in the metric $d_\PP$. For each $k\in \KK$, the sequence $(\PP[\rr_k\vert z_n])_{n\in \mathbb{N}}$ is Cauchy in $\L2(\PP)$, hence convergent, denote the limit by $\rr_{k}^\infty$. Put $z_\infty:=\sigma(\rr_k^\infty:k\in \KK)$. Then $\lim_{n\to\infty}z_n=z_\infty$. 
\begin{proof}
It suffices to check that, for each $k\in \KK$, $\PP[\rr_k\vert z_\infty]=\rr_k^\infty$ a.s.-$\PP$. Let $N\in \KK$ and $\{g_1,\ldots,g_N\}\subset C_b(\mathbb{R})$. Then $\PP[\rr_kg_1(\rr_1^\infty)\cdots g_N(\rr_N^\infty)]=\lim_{n\to \infty}\PP[\rr_kg_1(\PP[\rr_1\vert z_n])\cdots g_N(\PP[\rr_N\vert z_n])]=\lim_{n\to \infty}\PP[\PP[\rr_k\vert z_n]g_1(\PP[\rr_1\vert z_n])\cdots g_N(\PP[\rr_N\vert z_n])]=\PP[\rr_k^\infty g_1(\rr_1^\infty)\cdots g_N(\rr_N^\infty)]$. A routine monotone class argument allows to conclude. 
\end{proof}
Thus $(\Lattice,d_{\Lattice})$ is a complete metric space.
\item\label{rmk:general:Polish} All the elements of $\Lattice$ are essentially separable (just because separability is inherited by the subspaces $\L2(\PP\vert_x)$, $x\in \Lattice$). $\Lattice$ itself ($\therefore$ by metrizability \ref{metric} every subset of $\Lattice$) is also separable. Indeed, because $\FF$ is essentially separable, it admits a countable essentially generating algebra $\AA$; then the family of the $\sigma(\BB)$, $\BB\in (2^\AA)_{\mathrm{fin}}$, is dense in $\Lattice$. 
\begin{proof}
Let $x\in \Lattice$; we show that a sequence from the family converges to it. The $\sigma$-field $x$ admits an  essentially generating countable set. By increasing martingale convergence we may hence assume that $x$ is essentially finitely generated. Approximating elements of the atoms of $x$ by elements of the algebra $\AA$ we get that $x$ is as close to a $\sigma(\BB)$, $\BB\in (2^\AA)_{\mathrm{fin}}$, as we like, even in the Hausdorff metric \cite[p.~87]{VIDMAR-conv}. 
\end{proof}
In conjunction with the previous item it means that $\Lattice$ is Polish. 
\item\label{rmk:general:vii} If $x$ and $y$ are independent members of $\Lattice$, then the maps $$\Lattice_x\times \Lattice_y\ni (u,v)\mapsto u\lor v\in \Lattice_{x,y}$$ $$\Lattice_{x,y}\ni z\mapsto (z\land x)\lor (z\land y)\in \Lattice_x\times \Lattice_y$$ are mutually inverse topological homeomorphisms and lattice isomorphisms \cite[Theorem~3.8]{tsirelson}.
\item\label{inclusion-closed} The set $\subset_{\Lattice}:=\{(x,y)\in \Lattice\times \Lattice:x\subset y\}$ is closed in $\Lattice\times \Lattice$. 
\begin{proof}Let $(x_n)_{n\in \mathbb{N}}$ and $(y_n)_{n\in \mathbb{N}}$ in $\Lattice$ be convergent sequences in $\Lattice$ with limits $x_\infty$ and $y_\infty$ respectively, satisfying $x_n\subset  y_n$ for all $n\in \mathbb{N}$. For $f\in \L2(\PP)$ we have $\PP[\PP[f\vert x_n]^2]\leq  \PP[\PP[f\vert y_n]^2]$ for each $n\in \mathbb{N}$ (because $\PP_{x_n}\leq \PP_{y_n}$), therefore $\PP[\PP[f\vert x_\infty]^2]\leq  \PP[\PP[f\vert y_\infty]^2]$, which means $\PP_{x_\infty}\leq \PP_{y_\infty}$, i.e. $x_\infty\subset y_\infty$. 
\end{proof}
In particular, if  $\lim_{n\to\infty}z_n=z_\infty$, then $\liminf_{n\to\infty}z_n\subset z_\infty\subset\limsup_{n\to\infty} z_n$. Conversely, if $\liminf_{n\to\infty}z_n=\limsup_{n\to\infty}z_n$ (``set-theoretic'' convergence), then $\lim_{n\to\infty}z_n$ exists (it is well-known, see e.g. \cite{VIDMAR-conv} for  a reference).  
\item\label{lemma:sandwich}  If two sequences $(x_n)_{n\in \mathbb{N}}$ and $(y_n)_{n\in \mathbb{N}}$ in $\Lattice$ have the same limit $l\in \Lattice$ and if $(z_n)_{n\in \mathbb{N}}$ is another sequence in $\Lattice$ with $x_n\subset z_n\subset y_n$ for all $n\in \mathbb{N}$, then $\lim_{n\to\infty}z_n=l$. 
\begin{proof}
Let $f\in \L2(\PP)$. Then $\PP[\PP[f\vert x_n]^2]\leq \PP[\PP[f\vert z_n]^2]\leq \PP[\PP[f\vert y_n]^2]$ for each $n\in \mathbb{N}$. Passing to the limit one gets $\lim_{n\to \infty}\PP[\PP[f\vert z_n]^2]=\PP[\PP[f\vert l]^2]$. Apply Item~\ref{rmk:general:ii}.
\end{proof}
\end{enumerate}

\setcounter{count1}{\value{enumi}}

\end{enumerate}

\subsection{Noise Boolean algebras}\label{subsection:noise-boolean}

\begin{enumerate}[wide, labelwidth=!, labelindent=0pt,label=\textbf{(\alph*)}]
\setcounter{enumi}{\value{count1}}

\item Throughout, either implicitly by assumption in absence of assertion, or by explicit assertion (in examples), $B$ is a noise Boolean algebra on $(\Omega,\FF,\PP)$, i.e. a non-empty distributive sublattice of $\Lattice$, whose every element admits an independent complement \cite[Definition~1.1]{tsirelson} (in \cite{tsirelson} it is called a ``noise type Boolean algebra'', but the ``type'' seems superfluous). 

In $B$ independence is the same as pairwise independence, and for $\{x,y\}\subset B$, $x$ is independent of $y$ iff $x\land y=0_\PP$ (which is certainly not the case in general in $\Lattice$).  The (it is automatically unique) independent complement of an $x\in B$ in $B$ is denoted $x'$. With this operation of complementation and the meet $\land$, join $\lor$, least element $0_\PP$ and greatest element $1_\PP$ of Item~\ref{generalities:lattice} $B$ is indeed a Boolean algebra; the degenerate Boolean algebra having $0_\PP=1_\PP$ is allowed but is not interesting.  For $A\subset B$, $A':=\{a':a\in A\}$. Filters, ideals, ultrafilters and such like of $B$ refer to, respectively, filters, ideals, ultrafilters and such like of the Boolean algebra $B$.

The intersection of an arbitrary non-empty family of noise Boolean algebras under $\PP$ is again a noise Boolean algebra under $\PP$. The corresponding statement for unions of linearly ordered (w.r.t. inclusion) families is also true.

\item\label{noise-boolean:factorization} The noise Boolean algebra $B$ may, but need not be introduced as the range of a noise factorization $N$, that is to say of a homomorphism of bounded lattices $N:\mathcal{B}\to\Lattice$ such that $N(b)$ is independent of $N(b')$ for  all $b$ belonging to a Boolean algebra $\mathcal{B}$. $B$ can always be enhanced artificially to a noise factorization by indexing it with itself.
\begin{example}[Discrete classical noise Boolean algebra]\label{example:classical}
 Suppose  that $A\subset \Lattice$ is an independency [$\therefore$ countable],  $\lor A=1_\PP$ and $0_\PP\notin A$ (the latter just for ``non-redundancy''); we would say that $A$ is a generating independency. Then $B:=\{\lor C:C\in 2^A \}$ is a noise Booean algebra -- the  classical discrete noise Boolean algebra associated to $A$. One can start with indexed independences (of non-degenerate random elements) in lieu of $A$ in the obvious way.
\end{example}
Any finite noise Boolean algebra is of the preceding form for a finite independency $A$.

\begin{example}[Simplest nonclassical noise Boolean algebra]\label{example:simplest-nonclassical}
Let $\mathcal{B}$ be the finite-cofinite algebra on $\mathbb{N}$ and suppose $(\xi_n)_{n\in \mathbb{N}}$ is a sequence of independent equiprobable random signs, which generate $1_\PP$. For $b\in \mathcal{B}$ finite let $N_b:=\sigma(\xi_i\xi_{i+1}:i\in b)$; for $b\in \mathcal{B}$ cofinite let $N_b:=\sigma(\xi_i\xi_{i+1}:i\in b_{<\max(\mathbb{N}\backslash b)})\lor \sigma(\xi_i:i\in b_{> \max(\mathbb{N}\backslash b)})$ ($\max\emptyset:=0$). Then $N:=(\BB\ni b\mapsto N_b)$ is a noise factorization and $B:=\{N_b:b\in \BB\}$ is the simplest nonclassical noise Boolean algebra of  \cite[Section~1.2]{tsirelson}. $B$ can be written as the $\uparrow$ (it means, nondecreasing) union of classical finite noise Boolean algebras $B_n$,  $n\in \mathbb{N}_0$, where, for  $n\in \mathbb{N}_0$, $B_n$ is attached to the independency $\xi_1\xi_2,\ldots,\xi_n\xi_{n+1},\xi\vert_{\mathbb{N}_{>n}}$. 
\end{example}

To incorporate ``translational invariance'' into the above framework, a system of homogeneities of a noise factorization $N:\BB\to \Lattice$ may be introduced as a triple $(H,T,\Theta)$, where $H$ is a group and where $T:H\to \mathrm{Aut}(\BB)$ and $\Theta:H\to \mathrm{Aut}(\PP)$ are group homomorphisms, such that for all $b\in \BB$ and $h\in H$, $\Theta_h$ sends $N(b)$ to $N(T_h(b))$ (thus making the noise factorization $N$ into a ``(homogeneous) noise'').  

\begin{example}\label{example:cts-1d-noise}
A one-dimensional noise is  a noise factorization  $N:\AAA_\mathbb{R}\to \Lattice$, where $\AAA_\mathbb{R}$ is the Boolean algebra of finite unions of intervals of $\mathbb{R}$, and for which, in the context of the preceding paragraph, $H$ is $\mathbb{R}$ under addition, $T_h$ is translation by $h$ (to the right) and the automorphism $\Theta_h$ sends $N(b)$ to $N(b+h)$ for $b\in \AAA_\mathbb{R}$, $h\in \mathbb{R}$. In various parts of the literature on noises sometimes continuity of the group action $\Theta$ is also assumed, e.g. in \cite[Definition 3.27]{picard2004lectures}\footnote{In general, if $H$ is a topological group then a homomorphism $\Psi:H\to \mathrm{Aut}(\PP)$ is continuous iff it is continuous at the neutral element $0$ of $H$, i.e. iff $\Theta_h(f)\to f$ [in $\L2(\PP)$] as $h\to 0$, which in turn is equivalent to $\lim_{h\to 0}\PP(\Theta_h(A)\triangle A)=0$ for all $A\in \FF$; if $H$ is a Polish group, then  continuity of $\Psi$ is even equivalent to Borel measurability \cite[Theorem~2.2]{rosendal}.}; another continuity condition that usually appears is that of the sequential continuity of $N$ itself (wherein convergence of a sequence in $\AAA_\mathbb{R}$ means convergence of the endpoints of the constituent intervals).
\end{example}
 
We will not really be concerned with ``homogeneities'' of noise factorizations at all, we mention them only to indicate the connection to what has more traditionally been considered a stochastic noise. To avoid sending the wrong message, though, let us note at once that one can clearly work with homogeneities also without reference to a factorization. In fact the collection $\mathrm{Hmg}(B):=\{\Theta\in \mathrm{Aut}(\PP):\Theta(B)=B\}$ is a subgroup of $\mathrm{Aut}(\PP)$, which might be called the group of homogeneities of $B$; each member thereof automatically carries $B$ onto $B$ as a homeomorphism and as a Boolean algebra isomorphism (because isomorphisms of probabilities ``preserve everything in sight'').

\item  $B$ is said to be complete (without further qualification) if it is closed under arbitrary joins (= generated $\sigma$-fields) and meets (= intersections). The collection $\mathrm{Cl}(B):=\{\liminf x:x\text{ sequence in }B\}$ is the   sequential monotone \cite[Theorem~1.6]{tsirelson} (or topological, it is the same \cite[Proposition~3.3]{tsirelson}) closure of $B$. $B$ is said to be noise complete if it is equal to  $\overline{B}=\{x\in \Cl(B):x\text{ has an independent complement that belongs to }\Cl(B)\}$, its noise completion, i.e.  the largest noise Boolean algebra contained in $\Cl(B)$ and containing $B$ \cite[Theorem~1.7]{tsirelson}. The projections associated to the members of $\Cl(B)$ commute \cite[Eq.~(4.7)]{tsirelson} (we say that $\Cl(B)$ is commutative) and (hence) $\PP_x\PP_y=\PP_y\PP_x=\PP_{x\land y}$ for $\{x,y\}\subset \Cl(B)$. We recall further  that $(\Cl(B)\times\Cl(B)\ni (x,y)\mapsto x\land y\in  \Lattice)$ is continuous  \cite[Eq.~(4.11)]{tsirelson} ($\therefore$ $\Cl(B)$ is closed for $\land$), this being true even of any commutative subset of $\Lattice$ in lieu of $\Cl(B)$ \cite[Proposition~3.4]{tsirelson}. The parallel statement for $\lor$ in lieu of $\land$ fails in general, even just on $B$. Nevertheless, $(\Cl(B)\ni y\mapsto x\lor y\in  \Lattice)$ is continuous for each $x\in \overline{B}$ \cite[Lemma~4.4]{tsirelson} ($\therefore$ $x\lor y\in\Cl(B)$ for $x\in \overline{B}$, $y\in \Cl(B)$), i.e. one has some degree of separate continuity for $\lor$. Whether or not $\Cl(B)$ is closed for $\lor$ in general is not known \cite[Question~1.11]{tsirelson} (and this text will not provide an answer). 

\begin{example}\label{example:closure-for-simplest-nonclassical}
The noise Boolean algebra of Example~\ref{example:classical} is complete, that of Example~\ref{example:simplest-nonclassical} merely noise complete, indeed in the latter case $\Cl(B)\backslash B=\{\sigma(\xi_i\xi_{i+1}:i\in I):I\in (2^\mathbb{N})_{\mathrm{inf}}\}$. 
Note that here $\Cl(B)$ is closed for $\lor$ even though $B$ is not classical, which answers the first query of \cite[Question~1.11]{tsirelson}. 
\end{example}

\item The collection of all finite noise Boolean subalgebras  (under $\PP$) of $B$ is denoted $\mathfrak{F}_B$. $\mathfrak{F}_B$ is naturally partially ordered by inclusion and we intend this order when speaking of nets indexed by $\mathfrak{F}_B$, limits along $\mathfrak{F}_B$ etc. A partition of $B$ is a finite independency $P\subset B\backslash \{0_\PP\}$ such that $\lor P=1_\PP$ or an enumeration $(x_1,\ldots,x_n)$, $n=\vert P\vert\in \mathbb{N}_0$, of such a $P$, whichever is the more convenient; in the trivial case when $0_\PP=1_\PP$, hence $B=\{1_\PP\}$ there is only one partition of unity, namely $\emptyset$, while if $0_\PP\ne 1_\PP$ then necessarily $n\in \mathbb{N}$. Each element of $b\in \mathfrak{F}_B$ is generated (as a Boolean algebra) by a unique partition of unity of $B$, denoted $\at(b)$, whose elements are called the atoms of $b$. Conversely, each finite subset, in particular each partition of unity, $P$ of $B$ generates an element of $\mathfrak{F}_B$, the smallest noise Boolean algebra that contains $P$, and denoted $b(P)$; especially, for each $x\in B$, $b_x:=b({\{x\}})=b(\{x,x'\})=\{0_\PP,x,x',1_\PP\}\in \mathfrak{F}_B$. One often approximates $B$ with a $\uparrow$  sequence $(b_n)_{n\in \mathbb{N}}$ in  $\mathfrak{F}_B$, whose union $\cup_{n\in \mathbb{N}}b_n$ --- a noise Boolean algebra in turn ---  is dense in $B$; by \ref{generalities:topology}\ref{rmk:general:Polish} such a sequence always exists because each finite subset $A$ of $B$ is contained in some element of $\mathfrak{F}_B$, namely $b(A)$.

\item\label{generalitites:subalgebra} For $x\in B$: $B_x$ is a noise Boolean algebra under $\PP\vert_x$;  ${}_xB$ is a Boolean algebra under the obvious operations that it inherits from $B$ (or from $B_{x'}$ to which it is isomorphic via the map $(B_{x'}\ni u\mapsto x\lor u\in {}_xB)$). 

Further, it is checked easily that for $\GG\in \Lattice$, $B\vert_\GG:=\{x\land \GG:x\in B\}$ is a noise Boolean algebra under $\PP\vert_\GG$ if $\GG$ distributes over $B$, i.e. if $(x\land \GG)\lor (y\land \GG)=(x\lor y)\land \GG$ for $\{x,y\}\subset B$ (the  condition is in general not necessary: for instance it can happen that $\GG\ne 0_\PP$ and $x\land \GG=0_\PP$ for all $x\in B\backslash \{1_\PP,0_\PP\}\ne \emptyset$). Distributivity of $\GG$ over $B$ is equivalent to $(x\lor y)\land \GG\subset (x\land \GG)\lor (y\land \GG)$ for $\{x,y\}\subset B$ (because the inclusion $\supset$ is automatic); further one may assume that $x$ and $y$ are independent (for the reduction use distributivity in $B$ to write $x\lor y=x\lor (y\land x')$) and then that $y=x'$ (this time around use general distributivity over independent elements of $\Lattice$ for the reduction), in which case it is actually enough to find that $\GG=\GG_x\lor \GG_{x'}$ for some $\GG_x\in \Lattice_x$ and $\GG_{x'}\in \Lattice_{x'}$, because then automatically $\GG\land x=\GG_x$ and $\GG\land x'=\GG_{x'}$. Of course if a $\GG\in \Lattice$ distributes over $B$, then  $\GG=\lor_{p\in P}(\GG\land p)$ for any partition of unity $P$ of $B$. 

If $y\in \Cl(B)$ then $y$ distributes over $B$ because one can pass to the limit in $x_n\land (y_1\lor y_2)=(x_n\land y_1)\lor (x_n\land y_2)$ for $B\ni x_n\to y$ as $n\to\infty$, $\{y_1,y_2\}\subset B$, $y_1\land y_2=0_\PP$. On the other hand, if a $\GG\in \Lattice$ distributes over $B$, then it commutes with $B$ as well: for $x\in B$ and $f\in \FF/\mathcal{B}_{[0,\infty]}$ we may write $\PP[f\vert\GG]=\PP[f\vert (x\land \GG)\lor (x'\land \GG)]=F(X,X')$ a.s.-$\PP$ for  suitable $X\in (x\land \GG)/\mathcal{B}_\mathbb{R}$, $X'\in (x'\land \GG)/\mathcal{B}_\mathbb{R}$, $F\in \mathcal{B}_{\mathbb{R}^2}/\mathcal{B}_{[0,\infty]}$, whence $\PP[\PP[f\vert \GG]\vert x]=\PP[F(X,X')\vert x]=\PP[F(X,X')\vert X]=\PP[\PP[f\vert \GG]\GG\land x]=\PP[f\vert \GG\land x]$ a.s.-$\PP$, and thus $\PP_x\PP_\GG=\PP_{\GG\land x}$; similarly $\PP[\PP[f\vert x]\vert\GG]=\PP[\PP[f\vert x]\vert (x\land \GG)\lor (x'\land \GG)]]=\PP[\PP[f\vert x]\vert x\land \GG]=\PP[f\vert x\land \GG]$ a.s.-$\PP$, thus also $\PP_\GG\PP_x=\PP_{\GG\land x}$.

Incidentally, continuing in the theme of the preceding paragraph, one can also check that for independent $x$ and $\tilde{x}$ in $\Lattice$ there exists a noise Boolean algebra $C$ (under $\PP$) containing $B\cup \{x,\tilde{x}\}$ and  having $\tilde{x}=x'$, if and only if $x$ and $\tilde{x}$ distribute over $B$ and  $y=(y\land x)\lor (y\land \tilde{x})$ for all $y\in B$, in which case one may take $C=\{(x\land y)\lor (\tilde{x}\land \tilde{y}):(y,\tilde{y})\in B^2\}$. But we will not need this fact.

\item\label{generalities:o} An $f\in \mathrm{L}^1(\PP)$ is an additive integral of $B$ if $\PP[f]=0$ and $f=\PP[f\vert x]+\PP[f\vert x']$ for all $x\in B$. Automatically  $\PP[f\vert x]$ is also an additive integral for every $x\in B$  and (hence) $f=\sum_{x\in P}\PP[f\vert x]$ for every partition of unity $P$ of $B$; $(\PP[f\vert x])_{x\in B}$ is a kind-of process with $B$-independent additive increments  -- an additively decomposable process.  If an $f\in \mathrm{L}^1(\PP)$ decomposes into a sum of an element of $\mathrm{L}^1(\PP\vert_x)$ and of an element of $\mathrm{L}^1(\PP\vert_{x'})$ for each $x\in B$, then $f-\PP[f]$ is an additive integral of $B$. 

Likewise, an $f\in \mathrm{L}^1(\PP)$ with $\PP[f]=1$ and $f=\PP[f\vert x]\PP[f\vert x']$ for all $x\in B$ is called a multiplicative integral of $B$. For such an $f$, $\PP[f\vert x]$ is also a multiplicative integral for all $x\in B$ and $f=\prod_{x\in P}\PP[f\vert x]$ for all partitions of unity $P$ of $B$;  $(\PP[f\vert x])_{x\in B}$ is a kind-of process with $B$-independent multiplicative increments -- a multiplicatively decomposable process. If an $f\in \mathrm{L}^1(\PP)$ of non-zero mean decomposes into a product of an element of  $\mathrm{L}^1(\PP\vert_x)$ and of an element of $\mathrm{L}^1(\PP\vert_{x'})$ for each $x\in B$, then $f/\PP[f]$ is a multiplicative integral of $B$.

\setcounter{count1}{\value{enumi}}

\end{enumerate}

\subsection{Spectrum and chaos decomposition}\label{subsection:spectrum-and-chaos}

\begin{enumerate}[wide, labelwidth=!, labelindent=0pt,label=\textbf{(\alph*)}]
\setcounter{enumi}{\value{count1}}

\item\label{generalities:spectrum-intro} $B$ admits a (non-unique) spectral resolution/spectrum $((S,\Sigma,\mu);\Psi)$ in the following sense: $(S,\Sigma,\mu)$ is a standard measure space (the spectral space); $\Psi$  is a unitary isomorphism from $\L2(\PP)$ onto  the direct integral \cite[p.~169, Definition~II.4.5.3]{dixmier1981neumann} $\int_S^\oplus H_s\mu(\dd s)$ of a $\mu$-measurable field of a.e.-$\mu$ non-zero, separable Hilbert spaces $(H_s)_{s\in S}$, which sends the algebra $\AA:=\AA_B$ of self-adjoint operators on $\L2(\PP)$ generated by $\{\PP_x:x\in B\}$, i.e. the strong (or weak, what is the same by the von Neumann bicommutant theorem) closure of the real linear span of the projections $\PP_x$, $x\in B$, onto the algebra of diagonalizable operators \cite[A.84]{dixmier-c-star}.  We denote by $\alpha:\AA\to \mathrm{L}^\infty(\mu)$ the corresponding bijective map which satisfies $\Psi(Af)=\alpha(A)\Psi(f)$ for $f\in \L2(\PP)$, $A\in \AA$.  Strong operator convergence in $\AA$ of a norm-bounded  sequence (resp. that is $\uparrow$, that is  $\downarrow$) corresponds to convergence locally in (i.e. on every measurable set of finite) $\mu$-measure in $\mathrm{L}^\infty(\mu)$ (resp. that is $\uparrow$ a.e.-$\mu$, $\downarrow$ a.e.-$\mu$).

Up to unitary equivalence (modifying (if necessary) $\Psi$, but not $\alpha$ and $(S,\Sigma)$) we may replace $\mu$ by any equivalent $\sigma$-finite measure \cite[A.75]{dixmier-c-star}, and we may assume if we like (but we do not) that $(H_s)_{s\in S}$ is a ``standardized'' field in the sense that there is a $\Sigma$-measurable partition $(S_n)_{n\in \mathbb{N}\cup \{\infty\}}$ of $S$ (possibly some of the elements of the partition have $\mu$-measure zero) such that  $H_s=\mathbb{R}^n$ for $s\in S_n$, $n\in \mathbb{N}$, while $H_s=l^2(\mathbb{R})$ for $s\in S_\infty$, with the elements of $\int_S^\oplus H_s\mu(\dd s)$ being the square-integrable $\mu$-equivalence classes of the measurable cross-sections of $(H_s)_{s\in S}$ (meaning that they are measurable on restriction to each $S_n$, $n\in \mathbb{N}\cup \{\infty\}$) \cite[p.~166, Proposition~II.1.4.1]{dixmier1981neumann}. It is worth emphasizing that, despite the failure of the notation to make this explicit, $\int_S^\oplus H_s\mu(\dd s)$  actually depends on the choice of the underlying $\mu$-measurable vector fields \cite[p.~164, Definition~II.1.3.1]{dixmier1981neumann}, indeed it consists of the (equivalence classes of) $\mu$-measurable vector fields $(f_s)_{s\in S}\in (H_s)_{s\in S}$ for which $\int \Vert f(s)\Vert^2_s\mu(\dd s)<\infty$ ($\Vert \cdot \Vert_s$ being the norm on $H_s$, $s\in S$, but we will omit the $s$ in what follows); in case of a  ``standardized representation'' the ``standard'' measurability choice is made. 

Strictly speaking, we have cheated a little in the preceding, because we have worked with the real instead of the complex number field while applying the results concerning the spectral resolution of $\AA$ (these results are provided for the complex terrain). However, if pressed, this can be remedied by simply starting with the complex incarnations of the conditional expectation operators to get a commutative von Neumann algebra in the usual sense, and then returning to the real world. We make this precise in the following (fundamental) remark.

\begin{remark}\label{rmk:complex-to-real}
 Let $\AA^\mathbb{C}$ be the (complex) commutative von Neumann algebra generated by $\{\PP_x^\mathbb{C}:x\in B\}$, where $\PP_x^\mathbb{C}$ is the conditional expectation w.r.t. $x$ when viewed as acting on $\L2(\PP;\mathbb{C})$. Thus $\AA^\mathbb{C}$ is  the strong (equivalently, weak) closure of the complex linear span of the projections $\PP_x^\mathbb{C}$, $x\in B$. We get (this time, no cheating) a unitary isomorphism $\Psi^\mathbb{C}$ from $\L2(\PP;\mathbb{C})$ onto  the direct integral $\int_S^\oplus H_s^\mathbb{C}\mu(\dd s)$ of a $\mu$-measurable field of a.e.-$\mu$ non-zero separable complex Hilbert spaces $(H_s^\mathbb{C})_{s\in S}$, which sends the algebra $\AA^\mathbb{C}$ onto the algebra of diagonalizable operators: there is a unique bijection $\alpha^\mathbb{C}:\AA^\mathbb{C}\to \mathrm{L}^\infty(\mu;\mathbb{C})$ such that $\Psi^\mathbb{C}(Af)=\alpha^\mathbb{C}(A)\Psi^\mathbb{C}(f)$ for $f\in \L2(\PP;\mathbb{C})$, $A\in \AA^\mathbb{C}$. 

\noindent The subspace of $\AA^\mathbb{C}$ consisting of Hermitian (self-adjoint) elements is precisely the strong (equivalently, weak) closure of the real  linear span of the projections $\PP_x^\mathbb{C}$, $x\in B$ (by the Kaplansky density and the double commutant theorems); the map $\alpha^\mathbb{C}$ maps it onto $\mathrm{L}^\infty(\mu)$. Let $(e_i)_{i\in I}$ be an orthonormal basis of $\L2(\PP)$; it is also an orthonormal basis of $\L2(\PP;\mathbb{C})$. Fix versions of $\Psi^\mathbb{C}(e_i)$, $i\in I$, and let, for $s\in S$, $H_s$ be the real closed vector subspace generated by $\Psi^\mathbb{C}(e_i)(s)$, $i\in I$. Then for $A\in \Sigma$, multiplication by $\mathbbm{1}_A$ in  $\int_S^\oplus H_s^\mathbb{C}\mu(\dd s)$ corresponds to the action of the self-adjoint operator $L:=(\alpha^\mathbb{C})^{-1}(\mathbbm{1}_A)$ on $\L2(\PP;\mathbb{C})$, which means that, for $\{i,j\}\subset I$, the quantity $\langle e_i,Le_j\rangle=\int_A\langle \Psi(e_i)(s),\Psi(e_j)(s)\rangle_s^\mathbb{C}\mu(\dd s)$ is real ($\langle\cdot,\cdot\rangle^\mathbb{C}_s$ being the scalar product in $H_s^\mathbb{C}$, $s\in S$), i.e. $\langle \Psi(e_i)(s),\Psi(e_j)(s)\rangle_s^\mathbb{C}$ is real for $\mu$-a.e. $s$. Since $I$ is countable, discarding a $\mu$-negligible set, we may assume that $H_s$ is a totally real subspace of $H^\mathbb{C}_s$ for each $s\in S$ and we may and do view it as a real Hilbert space. Then we take for the $\mu$-measurable vector fields of $(H_s)_{s\in S}$ those that are $\mu$-measurable when viewed as elements of $(H^\mathbb{C}_s)_{s\in S}$. Clearly $\Psi:=\Psi\vert_{\L2(\PP)}$ is a unitary isomorphism from $\L2(\PP)$ onto $\int_S^\oplus H_s\mu(\dd s)$. 
It sends $\AA$ onto  the algebra of diagonalizable operators of $\int_S^\oplus H_s\mu(\dd s)$. 

\noindent When (by an abuse) we speak of the (abelian) von Neumann algebra $\AA$ we intend actually $\AA^\mathbb{C}$ togeher with the above circumvention.
\end{remark}
\vspace{-0.25cm}
If $x\in \Cl(B)$, then $\PP_x\in \AA$; whenever $\PP_x\in \AA$ for an $x\in \Lattice$, then the a.e.-$\mu$ uniquely determined element of $\Sigma$ that corresponds via the spectral resolution to the projection $\PP_x$ is denoted $S_x$ and called the spectral set of $x$ (if it has been made explicit that $S\subset\Lattice$ /which will happen on occasion/, then we agree that this meaning of $S_x$, not $S\cap \Lattice_x$ of \ref{generalities:lattice} prevails; when it is not given that $S\subset \Lattice$, then no confusion is possible as to which of the two meanings is intended). Thus $\Psi(\PP_x(f))=\mathbbm{1}_{S_x}\Psi(f)$ for $f\in \L2(\PP)$, i.e. $\mathbbm{1}_{S_x}=\alpha(\PP_x)$ a.e.-$\mu$. Of course $S_{x\land y}=S_x\cap S_y$ a.e.-$\mu$ for $\{x,y\}\subset \Lattice$ for which $\{\PP_x,\PP_y\}\subset \AA$. $S_{0_\PP}$ is associated to the minimal projection $\PP_{0_\PP}$ of $\AA$ onto the one-dimensional space of constants (just the expectation operator); because of standardness there is a unique $\emptyset_S\in S$ such that $S_{0_\PP}=\{\emptyset_S\}$ a.e.-$\mu$, and $H_{\emptyset_S}$ is one-dimensional (and might be taken to be $\mathbb{R}$ without changing the spectral space and the spectral sets $S_x$, $x\in B$). 

Convergence of sequences in $\Cl(B)$ corresponds to local convergence in  $\mu$-measure of the associated spectral sets: $x_n\to x$ as $n\to\infty$  iff $S_{x_n}\to S_x$ as $n\to\infty$ locally in  $\mu$-measure, also $x_n\uparrow x$ (resp. $x_n\downarrow x$) as $n\to\infty$  iff $S_{x_n}\uparrow S_x$ (resp. $S_{x_n}\downarrow S_x$) as $n\to\infty$ a.e.-$\mu$, this for any sequence $(x_n)_{n\in \mathbb{N}}$ in $\Cl(B)$ and any $x\in \Cl(B)$. If a noise Boolean algebra $B_0\subset B$ is countable and dense, then versions of $S_x$, $x\in B_0$, can be chosen such that $\{S_x:x\in B_0\}$ is a $\pi$-system, and automatically $\sigma(S_x:x\in B_0)= \Sigma$ ($\therefore$ the $S_x$, $x\in B_0$, separate a $\mu$-conegligible subset of $S$); we would say that $\{S_x:x\in B\}$ is an essentially generating ($\therefore$ essentially separating) $\pi$-system, and we can anyway even without choosing versions write without ambiguity that $\sigma(S_x:x\in B)=\Sigma$, because $\sigma$ also ``throws in'' the $\mu$-negligible sets as per our convention \ref{generalities:a}.

For $f\in \L2(\PP)$, the indefinite integral $\mu_f:=\Vert \Psi(f)\Vert^2\cdot \mu$ is  the $\mu$-spectral measure of $f$; $\mu_f(S_x)=\PP[\PP[f\vert x]^2]$ for $x\in \Cl(B)$, in particular for an $f\in\L2(\PP\vert_x)$ we have that $\mu_f$ is carried by $S_x$. Here $\Vert \Psi(f)\Vert^2$ is the map $(S\ni s\mapsto \Vert \Psi(f)_s\Vert^2)$, not the squared norm of $\Psi(f)$ in $\int^\oplus_S H_s\mu(\dd s)$, which is anyway just $\Vert f\Vert^2$; we trust such notational shenanigans can be gathered from context. All the $\mu_f$, $f\in \L2(\PP)$, are absolutely continuous w.r.t. $\mu$; there is always an $f\in \L2(\PP)$ for which $\mu_f\sim \mu$. If a finite measure $\nu$ is absolutely continuous w.r.t. $\mu$, then there is a (generically non-unique) $\hat{f}\in \int_S^\oplus H_s\mu(\dd s)$ such that $\Vert \hat f\Vert^2=\frac{\dd\nu}{\dd\mu}$ a.e.-$\mu$, and hence $\mu_{\Psi^{-1}(\hat{f})}=\nu$. For $f\in\L2(\PP)$ satisfying $\PP[f^2]=1$ the measure $\mu_f$ has mass one and it is helpful to think of $\mu_f$ as the probabilty law of a point in the spectral space; often the spectral space consists of subsets of a domain space in which case $\mu_f$ is then the law of a random set -- the random spectral set corresponding to $f$. 
 
 For $x\in B$ we may restrict $(S,\Sigma,\mu)$ to $(S_x,\Sigma\vert_{S_x},\mu\vert_{S_x})$ and $\Psi$ to (by slight abuse of notation) $\Psi\vert_{\L2(\PP\vert_x)}\vert_{S_x}$ in parallel, thus obtaining a spectrum for the noise Boolean subalgebra $B_x$. $B$ and $\overline{B}$ generate the same commutative von Neumann algebra: $\AA_B=\AA_{\overline{B}}$. Therefore $((S,\Sigma,\mu);\Psi)$ is also a spectrum for $\overline{B}$.

 To each $E\in \Sigma$ is associated canonically by the spectral decomposition the closed linear subspace $H(E):=\Psi^{-1}(\int_E^\oplus H_s\mu(\dd s))$ of $\L2(\PP)$; $\L2(\PP\vert_x)=H(S_x)$ for $x\in \Cl(B)$, of course.   We have $H(A\cap B)=H(A)\cap H(B)$ and $H(A\cup B)=H(A)+H(B)$ (orthogonal sum, if $A\cap B=\emptyset$ a.e.-$\mu$) for $\{A,B\}\subset \Sigma$ \cite[Eq.~(2.14)]{tsirelson}; moreover, for a sequence $(E_n)_{n\in \mathbb{N}}$ in $\Sigma$, $H(\cap_{n\in \mathbb{N}}E_n)=\cap_{n\in \mathbb{N}}H(E_n)$, while $H(\cup_{n\in \mathbb{N}}E_n)$ is the smallest closed linear subspace of $\L2(\PP)$ containing each $H(E_n)$, $n\in \mathbb{N}$ \cite[Eq.~(2.15)]{tsirelson}. The map $(\Sigma\ni A\mapsto \pr_{H(A)})$ is a projection-valued (a.k.a. spectral) measure acting on the Hilbert space $\L2(\PP)$.

More generally, for $x\in B$, let $\AA_x$ be the weak/strong closure of the real linear manifold generated by $\{\PP_{u\lor x'}:u\in B_x\}$ 
and put $\Sigma_x:=\sigma(S_{u\lor x'}:u\in B_x)$ (so $\AA_{1_\PP}=\AA$ and $\Sigma_{1_\PP}=\Sigma$); note that $\Sigma=\Sigma_x$ on $S_x$, that $S_{x'}$ is an atom of $\Sigma_x$ (it corresponds to them minimal projection $\PP_{x'}$ of $\AA_x$), and that $\Sigma_x\lor\Sigma_{x'}=\Sigma$ \cite[Eq. (7.18)]{tsirelson}. If a noise Boolean algebra $B_0\subset B_{x}$ (under $\PP\vert_x$) is countable and dense, then versions of $S_{u\lor x'}$, $u\in B_0$, may be chosen such that $\{S_{u\lor x'}:u\in B_0\}$ becomes a $\pi$-system and automatically $\sigma(S_{u\lor x'}:u\in B_0)=\Sigma_x$. There is a probability $\nu\sim\mu$ such that $\Sigma_x$ and $\Sigma_{x'}$ are $\nu$-independent \cite[Eq.~(7.17)]{tsirelson}; in fact, if $f\in \L2(\PP\vert_x)$ with $\PP[f^2]=1$ is such  that $\mu_f\sim \mu$ on $S_x$ and $f'\in \L2(\PP\vert_{x'})$ with $\PP[{f'}^2]=1$ is  such that $\mu_{f'}\sim \mu$ on $S_{x'}$ (they exist), then $\mu_{ff'}$ is a probability $\sim \mu$ under which $\Sigma_x$ and $\Sigma_{x'}$ are independent (it follows by following the hint for the proof of \cite[Fact~2.28]{tsirelson} because $\mu$ as constructed there from $\mu_f$ and $\mu_{f'}$ turns out to be precisely $\mu_{ff'}$). The obvious generalization of the preceding to a partition of unity of  $B$ (in lieu of $(x,x')$) is straightforward.

\begin{example}[Spectrum of finite noise Boolean algebra]\label{example:spectral-finite}
The simplest case is when $B$ is finite. For then we may take $S=B$, $\Sigma=2^S$, $\mu=$ counting measure, for $x\in B$, $H_x=\L2(\PP\vert_x)^\circ:=$  the orthogonal complement of (the closure of the linear span of) $\cup_{y\in B_x\backslash \{x\}}\L2(\PP\vert_y)$ in $\L2(\PP\vert_x)$, and  $\Psi=$ the  canonical isomorphism between $\L2(\PP)$ and $\oplus_{x\in B}H_x;$ in particular $\Psi^{-1}(e(g))=\prod_{a\in \at(B)}(1+g(a))$, where $e(g)(x):=\prod_{a\in \at(B)\cap 2^x}g(a)$, $x\in B$, is the ``exponential'' vector associated to $g=(g_a)_{a\in \at(B)}\in (H_a)_{a\in \at(B)}$. It is  the joint diagonalization (joint representation by multiplication operators) of the (of which there are finitely many) commuting  projections $\PP_x$, $x\in B$. For $x\in B$, $S_x=B_x$, of course. For $a\in \at(B)$, let $\pr_a:=(B\ni x\mapsto a\land x\in \{0_\PP,a\})$. Then, for $x\in B$, $\Sigma_x=\sigma(\pr_a:a\in \at(B)\cap B_x)$. 

\noindent In preparation of Section~\ref{subsection:spectral-independence} let us observe the following. \textbf{($\dagger$)} A probability $\sim \mu$ under which $\Sigma_x$ and $\Sigma_{x'}$ are independent for each $x\in B$ is given by $\mu_f$ for any $f$ of the form $\prod_{a\in \at(B)}g_a$ with $g_a\in \L2(\PP\vert_a)$ non-trivial and satisfying $\PP[g_a]\ne 0$, $\PP[g_a^2]=1$ for all $a\in \at(B)$ (random variables with these qualities for sure exist because each $a \in \at(B)$ contains an event $A_a\in \FF$ with $p_a:=\PP(A_a)\notin\{0,1\}$, namely one can take $g_a=(\sqrt{1-p}-\sqrt{p(1-p_a)/p_a})\mathbbm{1}_{A_a}+(\sqrt{1-p}+\sqrt{pp_a/(1-p_a)})\mathbbm{1}_{\Omega\backslash A_a}$ for a $p\in  (0,1)$ and then $1-\PP[g_a]^2=p$). Such probability $\mu_f$ can equivalently be described as the law of the random element of $B$ which one gets by including each atom of $B$ independently of the others with probability $1-\PP[g_a]^2$, cf. also \cite[Section~1]{tsirelson-arxiv-5} \cite[Eq.~(7.3)]{tsirelson}.  \textbf{($\dagger$)} (We will make observations connected to the content of Section~\ref{subsection:spectral-independence} in a couple more examples and will indicate their relevance to said subsection by placing them in-between  \textbf{($\dagger$)} symbols. These remarks concerning ``spectral independence'' could be deferred to Section~\ref{subsection:spectral-independence} but it seems actually more forgiving to state them en passant, as and when particular instances of spectra are considered.)

\noindent Here is a fun little fact that we will not need in the sequel, but may be illuminating. Suppose $\tilde{B}$ is another finite noise Boolean algebra that is finer than $B$ in the sense that $B\subset \tilde{B}$. Associate to $\tilde{B}$ the spectral resolution as above, in particular the spectral measure $\tilde{\mu}$. $\mu$ and $\tilde{\mu}$ are not directly comparable, however for a given $f\in\L2(\PP)$, $\mu_f$ and $\tilde{\mu}_f$ are, indeed $\tilde{\mu}_f\leq \mu_f$ on $2^B$. For, if $B$ has the $n+1$ atoms $a_1,\ldots,a_n,a$ and $\tilde{B}$ has the $n+2$ atoms $a_1,\ldots,a_n,b,c$ (so in particular $a=b\lor c$), $n\in \mathbb{N}_0$, then by inclusion-exclusion  $\mu_f(\{1_\PP\})-\tilde{\mu}_f(\{1_\PP\})=\tilde{\mu}_f(\{a_1\lor \cdots \lor a_n\lor b\})+\tilde{\mu}_f(\{a_1\lor \cdots \lor a_n\lor c\})\geq 0$ (this establishes $\tilde{\mu}_f\leq \mu_f$ only for a special constellation of $B$ vis-\`a-vis $\tilde{B}$ and even then only on $\{1_\PP\}$, but the general case follows at once from this). It corresponds to the fact that in \ref{generalities:chaoses} below the sets $P_b(y)$ are nonincreasing in  $b\in \mathfrak{F}_B$ , $ b\ni y$.
\end{example}

 \begin{example}[Spectrum of Wiener noise]\label{example:wiener}
 Suppose $B$ is the noise Boolean algebra associated to the increments of a Wiener process $W$ on $\mathbb{R}$: each element of $B$ is the $\sigma$-field $N_A$ of the increments of $W$ on some disjoint union of intervals $A$ of $\mathbb{R}$, and all such $N_A$ belong to $B$. Of course the two-sided Brownian motion $W$ (vanishing at zero) must generate the whole of $1_\PP$. Due to the Wiener chaos expansion one may take  $S=(2^\mathbb{R})_{\mathrm{fin}}$, for $\mu$ the ``symmetrized'' Lebesgue measure  on each ``fiber'' corresponding to sets of a given size, for $s\in S$ each $H_s$ is just $\mathbb{R}$, and the map $\Psi$ sends an element of $\L2(\PP)$ to the components of its Wiener chaos expansion: more precisely, for $f\in \L2(\mu)$, we have $$\Psi^{-1}(f)=\sum_{k\in \mathbb{N}_0}\int_{{\mathbb{R}\choose k}} f(t)\dd^k W_t=f(\emptyset)+\int_{-\infty}^\infty f(\{t\}) \dd W_t+\int_{-\infty}^\infty\int_{-\infty}^t f(\{t,u\})\dd W_u\dd W_t+\cdots, $$ where the stochastic integrals are in the sense of Wiener-It\^o. In particular, $\Psi^{-1}(e(g))=\exp\left(\int_{-\infty}^\infty g(t)\dd W_t-\frac{1}{2}\int_ {-\infty}^\infty g^2(t)\dd t\right)$, where $e(g)(S):=\prod_{t\in S}g(t)$, $S\in (2^\mathbb{R})_{\mathrm{fin}}$, is the exponential vector corresponding to  $g\in \L2(\mathbb{R})$, which is the well-known connection between Brownian motion and the symmetric (boson) Fock space of $\L2(\mathbb{R})$ \cite[Example~19.9]{parthasarathy}.  For a disjoint union of intervals $A$ of $\mathbb{R}$ the corresponding $\sigma$-field $N_A\in B$ has $\Sigma_{N_A}$ generated by the counting maps associated to subintervals of $A$; the spectral set $S_{N_A}$ is equal to $\{s\in S:s\subset A\}$ a.e.-$\mu$. 

 \noindent \textbf{($\dagger$)} A probability equivalent to $\mu$ under which $\Sigma_x$ and $\Sigma_{x'}$ are independent for each $x\in B$ is given by the law of a Poisson random measure with a finite intensity measure equivalent to Lebesgue measure. Indeed this probability is $\mu_f$ for $f=\frac{\Psi^{-1}(e(r/2))}{\Vert \Psi^{-1}(e(r/2))\Vert}=\sqrt{\exp\left(\int_{-\infty}^\infty r(s)\dd W_s-\frac{1}{2}\int_{-\infty}^\infty r^2(s)\dd s\right)}$ with $(r/2)^2$ being the density of the intensity measure. \textbf{($\dagger$)} 
 
 \noindent To turn the noise factorization $N$  canonically into a  one-dimensional noise (cf.  Example~\ref{example:cts-1d-noise}) one would take for $\Theta_h$ the unique isomorphism which sends each random variable $W_t-W_s$ to $W_{t+h}-W_{s+h}$, $s\leq t$ being real numbers, $h\in \mathbb{R}$.
 \end{example}

\item\label{generalities:chaoses} Let for each $b\in \mathfrak{F}_B$ and then, for $\mu$-a.e. (if versions of  $S_x$, $x\in b$, are fixed, satisfiying $S_{x\land y}=S_x\cap S_y$ for $\{x,y\}\subset B$ and $S_{1_\PP}=S$, then the a.e. qualifiers can be dropped) $s\in S$, 
\begin{itemize}
\item[--] 
$\underline{b}(s)$ be the least element $x\in b$ whose spectral set $S_x$ contains $s$ (so $\underline{b}(s)\subset x$ iff $s\in S_x$, this for all $x\in b$), in other words $\underline{b}=y$ a.e.-$\mu$ on $ P_b(y):=S_{y}\backslash \cup_{z\in b_y\backslash \{y\}}S_{z}$ for all  $y\in b$;

\item[and]
\item[--]  
 $K_b(s)$ be the number of atoms of $b$ making up $\underline{b}(s)$, so that $K_b=\sum_{x\in \at(b)}\mathbbm{1}_{S\backslash S_{x'}}$ a.e.-$\mu$, in other words $\{K_b=k\}=\cup_{I\in {\at(b)\choose k} }P_b(\lor I)$ a.e.-$\mu$, $k\in \{0,\ldots,\vert \at(b)\vert\}$.
 \end{itemize}
The sets $P_b(y)$, $y\in b$, are pairwise $\mu$-disjoint and their union is $S$ a.e.-$\mu$ (they are indeed the partition generated by the  sets $S_y$, $y\in b$). Put $K:=K_B:=\mu\text{-}\esssup_{b\in \mathfrak{F}_B} K_b$ \cite[Eq.~(7.21)]{tsirelson}; of course this notation is consistent when $B$ is finite with the $K_b$, $b\in \mathfrak{F}_B$, already introduced. We call $K$ the counting map for $B$.

\item\label{generalites:chaos-spaces} The chaos spaces of $B$ are introduced as $H^{(n)}:=H^{(n)}(B):=H(\{K=n\})$, and we also put $H^{\leq n}:=H^{\leq n}(B):=\oplus_{k=0}^nH^{(k)}$ for $n\in \mathbb{N}_0$, $H_{\mathrm{stb}}:=H_{\mathrm{stb}}(B):=\oplus_{n\in \mathbb{N}_0}H^{(n)}=H(\{K<\infty\})$. Of course $\{K=0\}=S_{0_\PP}$ a.e.-$\mu$ and $H^{(0)}$ is the one-dimensional space of constants. $\stable:=\stable(B):=\sigma(H^{(1)})=\sigma(H_\mathrm{stb})$ \cite[Proposition~7.9]{tsirelson} is the stable $\sigma$-field. The first chaos, $H^{(1)}=\{f\in \L2(\PP):f=\PP[f\vert x]+\PP[f\vert x']\text{ for all }x\in B\}$ \cite[Proposition~7.8]{tsirelson}, is especially important; its elements are the square-integrable  additive integrals of $B$, and $B$ is said to be classical/linearizable/white (resp. black) when these generate the whole of $1_\PP$, i.e. when $\FF_{\mathrm{stb}}=\FF$ (resp. when  $H^{(1)}=\{0\}$, i.e. $\FF_{\mathrm{stb}}=0_\PP$, but $0_\PP\ne 1_\PP$). $H^{(1)}$ is closed for the action of all the $\PP_x$, $x\in B$. If $B$ is classical, then $x=\sigma(\L2(\PP\vert_x)\cap {H^{(1)}})$ for $x\in B$ \cite[Lemma~6.2]{tsirelson}. We put $H_{\mathrm{sens}}:=H_{\mathrm{sens}}(B):=H(\{K=\infty\})$. Thus $\L2(\PP)=H_{\mathrm{stb}}\oplus H_{\mathrm{sens}}$.
 
 A major result of \cite{tsirelson} is that $B$ is classical iff there exists a complete noise Boolean algebra $\hat{B}$ containing $B$ iff $(\lor_{n\in \mathbb{N}}x_n)\lor (\land_{n\in \mathbb{N}}x_n')=1_\PP$ for all $\uparrow$ sequences $(x_n)_{n\in \mathbb{N}}$ in $B$ \cite[Theorem~1.5]{tsirelson} iff $K<\infty$ a.e.-$\mu$ \cite[Theorem~7.7]{tsirelson} iff $\overline{B}=\mathrm{Cl}(B)$ \cite[Corollary~4.7]{tsirelson}. In particular, $B$ is classical iff $H_{\mathrm{sens}}=\{0\}$, in which case $\L2(\PP)=H_{\mathrm{stb}}=\oplus_{n\in \mathbb{N}_0}H^{(n)}$; it is black iff $H_{\mathrm{sens}}=\L2(\PP)_0\ne \{0\}$. 
 
 Existence of atoms precludes blackness: if $a\in B$ is an atom of the Boolean algebra $B$ and $0_\PP\ne 1_\PP$, then $\{0\}\ne \L2(\PP\vert_a)_0\subset H^{(1)}$.  If $B_0$ is a noise Boolean algebra with the same closure as $B$ (in particular if $B_0$ is dense in $B$), then $B_0$ is black (resp. classical) iff  $B$ is black (resp. classical), indeed the $B$ and $B_0$ have the same first chaos \cite[Proposition~1.10]{tsirelson}.  

For $k\in \mathbb{N}_0$, $H^{(k)}$ is generated (as a closed linear subspace of $\L2(\PP)$) by the union, over all  independent $(x_1,\ldots,x_k)\in B^k$, of the subspaces $\otimes_{i\in [k]}(\L2(\PP\vert_{x_i})\cap H^{(1)})$ (see discussion at the end of \cite[p. 349]{tsirelson} (penultimate display) and the line of argument of \cite[proof of Proposition~7.9]{tsirelson}); it is of course the same as saying that $H^{(k)}$ is generated by $\cup_{x\in B}(\L2(\PP\vert_{x})\cap H^{(1)})\otimes (\L2(\PP\vert_{x'})\cap H^{(k-1)})$ for all $k\in \mathbb{N}_0$ ($H^{(-1)}:=H^{\leq -1}:=\{0\}$). From this it follows that, for all $k\in \mathbb{N}_0$, $H^{(k)}$ is closed under $\PP_x$ for all $x\in B$ (hence under all the elements of $\AA$); therefore the same is true for $H_{\mathrm{stb}}$ and $H_{\mathrm{sens}}$. Notice also that the first chaos space $H^{(1)}$ determines already all the chaos spaces $H^{(k)}$, $k\in \mathbb{N}_0$, hence $H_{\mathrm{stb}}$ and $H_{\mathrm{sens}}$; these spaces were introduced via the spectrum, but they do not depend on it. 

For all $x\in B$: $K_{B_x}=K$ a.e.-$\mu$ on ${S_x}$ and each chaos $H^{(k)}$ is ``local'' in the sense that $H^{(k)}(B_x)=H^{(k)}\cap \L2(\PP\vert_x)$, $k\in \mathbb{N}_0$. 
 
 \begin{example}\label{ex:finite-noise-H-spaces}
 Return to Example~\ref{example:spectral-finite}, the noise Boolean algebra of which is classical, of course. For $k\in \mathbb{N}_0$ we have $H^{(k)}$ equal to closure of the linear span of products of the form $f_1\cdots f_k$, where $f_i\in \L2(\PP\vert_{a_i})_0$ for $i\in [k]$ and where $a_1,\ldots,a_k$ are pairwise distinct atoms of $B$. Naturally $H^{(k)}=\{0\}$ for $k\in \mathbb{N}_{>\vert \at(B)\vert}$. $K$ counts the number of atoms that make up an element of $B$.
 \end{example}
 
 \begin{example}
 Return to Example~\ref{example:wiener}, the noise Boolean algebra of it too is classical. For $k\in \mathbb{N}_0$, $H^{(k)}$ is the $k$-th Wiener chaos, i.e. $H({\mathbb{R}\choose k})$, and $K$ is the counting map.
 \end{example}
 
 Two nonclassical, but not black, examples of a one-dimensional noise are Warren's  noises of splitting \cite{warren}, and of stickiness (``made by a Poisson snake'') \cite{warren-sticky}; for further examples of non-classical/black one-dimensional noises see \cite{spectra-harris,watanabe,Jan08thenoise}. The Brownian web gives rise to a two-dimensional black noise factorization over the Boolean algebra of rectangles of the plane \cite{ellis}; its one-dimensional projection onto the time-axis is a one-dimensional noise, which is also black (the noise of coalescence \cite[Section~7]{picard2004lectures}). Another example of a two-dimensional black noise (in the obvious meaning of this qualification) is given by the scaling limit of critical planar percolation (the noise of percolation) \cite{schramm}. None of these are completely straightforward to describe and it is an outstanding problem to provide examples of higher-dimensional black noises. 
 
 However, a ``zero-dimensional'' class of examples of  black $B$ can be given in relatively succinct terms.
 
 \begin{example}[Hierarhical voter models]\label{example:voter-model}
 The following is an abridged version of  \cite[Appendix A]{vershik-tsilevich} (originally it is from \cite[Section~4.a]{vershik-tsirelson}). 
 
 \noindent  Let $\{m,r\}\subset \mathbb{N}_{\geq 2}$, let $X$ be a set of size $r$ and let $\phi:X^m\to X$ be a symmetric map, such that $\vert\phi^{-1}(\{x\})\vert$ is constant in $x\in X$. The latter ensures that the push-forward of the uniform law on $X^m$ by the map $\phi$ is the uniform law on $X$. It is helpful to think of $X$ as a collection of candidates and of $\phi$ as an election rule, assigning to the votes of $m$ voters a ``winner''. We assume further that the following ``anti-additivity'' property holds (it would fail automatically for $m=1$, hence the exclusion of $m=1$).
 \begin{quote}
 Whenever $f$ and $g$ are two real functions on $X$ satisfying $f(\phi(a_1,\ldots,a_m))=g(a_1)+\cdots +g(a_m)$ for all $a_1,\ldots,a_m$ from $X$, then $f$ is constant. 
 \end{quote} 
 It is important to, and we maintain  at once that this property implies the seemingly stronger assertion:
  \begin{quote}
 Whenever $f$ and $g_1,\ldots,g_m$ are real functions on $X$ satisfying $f(\phi(a_1,\ldots,a_m))=g_1(a_1)+\cdots +g_m(a_m)$ for all $a_1,\ldots,a_m$ from $X$, then $f$ is constant. 
 \begin{proof}
By symmetry of $\phi$ we have $f(\phi(a_1,\ldots,a_m))=\frac{1}{n}(f(\phi(a_1,a_2,\ldots,a_n))+f(\phi(a_2,\ldots a_n,a_1))+\cdots+ f(\phi(a_n,a_1,\ldots,a_{n-1})))=g(a_1)+\cdots+g(a_n)$, where $g(a):=\frac{1}{n}\sum_{i=1}^ng_i(a)$, $a\in X$.
 \end{proof}
  \end{quote} 
 It is clear that the property is met by the majority vote rule when $m$ is odd and $r=2$, so such maps certainly do exist. Another case when it is met is for (*) $m=r=2$ and $\phi$ the parity function on $X=\{-1,1\}$ (so $\{\phi=1\}=\{(1,1),(-1,-1)\}$).  Ultimately it will be seen that this property will ensure the absence of non-zero additive integrals and therefore the blackness of the noise Boolean algebra that we are about to define.
 
\noindent To specify the noise Boolean algebra in question we let $T_m$ be the rooted infinite $m$-ary tree (for notational simplicity to be thought of as the set of vertices, with the graph structure implicit),  $\Omega=\{\omega\in X^{T_m}:\omega(v)=\phi(\omega\vert_{\text{ sons of $v$}})\text{ for all }v\in T_m\}$. $\PP$ is the completion of the unique probability on the $\sigma$-field generated by the coordinate map $(\mathsf{X}_v)_{v\in T_m}$ on $\Omega$ under which for each $n\in \mathbb{N}_0$ the $\mathsf{X}_v$, $v\in T_m[n]:=$ vertices from level $n$ of $T_m$, are independent and uniformly distributed. Such measure exists thanks to the push-forward property noted above (technically we get it as the inverse limit \cite[Theorem~3.2]{parthasarathy1967probability} of the uniform measures on the spaces $X^{T_m[n]}$, $n\in \mathbb{N}_0$, under the projection maps induced by $\phi$). Because $r\geq 2$, $0_\PP\ne 1_\PP$, i.e. the probability $\PP$ is not trivial.

\noindent Next, for $n\in \mathbb{N}_0$ and for $v\in T_m[n]$ let $N_v$ be the $\sigma$-field generated by the $\mathsf{X}_w$, $w\in \{\text{the descendants of $v$}\}$; then the $N_v$, $v\in T_m[n]$, are atoms of a finite classical noise Boolean algebra $B_n$. We take $B:=\cup_{n\in \mathbb{N}_0}B_n=\{N_v:v\in T_m\}$, which is  also a noise Boolean algebra. In case (*) we shall speak of the simplest black noise Boolean algebra.

\noindent For $k\in \mathbb{N}_0$ we denote $\Sigma_k:=\sigma(\mathsf{X}_v:v\in T_m[k])$, the sigma-field generated by the voters at level $k$ of the tree. Clearly $\Sigma_k\uparrow 1_\PP$ as $k\to\infty$ and hence $\cup_{k\in \mathbb{N}_0}\L2(\PP\vert_{\Sigma_k})_0$ is dense in $\L2(\PP)_0$.

\noindent To show $B$ is black it will suffice to demonstrate that for all $k\in \mathbb{N}_0$ and then for all $f\in \L2(\PP\vert_{\Sigma_k})_0$ the projection of $f$ onto $H^{(1)}(B)$ vanishes. But for $n\in\mathbb{N}_{\geq k}$ the projection of $f$ onto the first chaos of $B_n$ is equal to $\sum_{v\in T_m[n]}\PP[f\vert N_v]$. If for $v\in T_m[n]$ we denote then by $v_n=v,\ldots,v_{k+1},v_k$ the lineage of $v$ between the levels $n$ and $k$ (in this order), then $\PP[f\vert N_v]=\PP[f\vert N_{v_k}]\vert N_{v_{k+1}}\vert \cdots\vert N_{v_n}]=\PP[f\vert \mathsf{X}_{v_k}\vert \mathsf{X}_{v_{k+1}}\vert\cdots \vert \mathsf{X}_{v_n}]$. Consequently $\mathrm{pr}_{H^{(1)}(B_n)}(f)=(\pr_n\cdots \pr_{k+1})(\PP[f\vert \mathsf{X}_{v_k}])$, where for $l\in \mathbb{N}_0$, $\mathrm{pr}_{l+1}$ is the projection from $H_l:=\oplus_{v\in T_n[l]}\L2(\PP\vert_{\sigma(\mathsf{X}_v)})_0$ onto $H_{l+1}$. By the antiadditivity condition $H_0$ has a trivial intersection with $H_1$, therefore (because we are dealing with finite-dimensional spaces) $\Vert \pr_1\Vert<1$. By the tree structure $\Vert \pr_l\Vert=\Vert \pr_1\Vert$ for all $l\in \mathbb{N}_0$. The fact that $\pr_{H^{(1)}(B)}(f)=0$ now follows on taking the limit as $n\to\infty$ in $\Vert\pr_{H^{(1)}(B)}(f)\Vert\leq \Vert \pr_{H^{(1)}(B_n)}(f)\Vert\leq \Vert \pr_1\Vert^n\Vert f\Vert$. (For the simplest black noise Boolean algebra the proof of blackness is easier: all finite non-empty products of the random signs of $\mathsf{X}$ of  a given level $T_2[n]$ are total in $\L2(\PP)_0$ as $n$ ranges over $\mathbb{N}_0$ and have zero projection onto the first chaos.) 

 \end{example}
 \end{enumerate}

\section{Some Bonferroni-type inequalities}
This section contains an interesting family of inequalities \eqref{bonferroni}-\eqref{bonferroni-1} involving the conditional expectation operators of members of a noise Boolean algebra. These inequalities may fail when the $\sigma$-fields belong to no noise Boolean algebra (as we will see in Example~\ref{example:counter-for-bonferroni}).

\begin{proposition}\label{corollary:inequalities}
Let $n\in \mathbb{N}_0$ and $x=(x_1,\ldots,x_n)\in B^n$. Then (as operators on $\L2(\PP)$)
\begin{equation}\label{bonferroni}
\PP_{\lor_{[n]} x}+\sum_{k=1}^n(-1)^{k}\sum_{J\in {[n]\choose k}}\PP_{(\lor_{[n]\backslash J}x)\lor (\land_{j\in J}\lor_{J\backslash \{j\}}x)}\geq 0
\end{equation}
with equality (on the whole of $\L2(\PP)$) iff one of the $x_i$, $i\in [n]$, is contained in the join of the others, and also with equality on restriction to $H^{\leq (n-1)}$. Furthermore, for all $m\in [n]_0$, 
\begin{equation}\label{bonferroni-1}
(-1)^m\sum_{k=m+1}^n(-1)^{k}\sum_{J\in {[n]\choose k}}\PP_{(\lor_{[n]\backslash J}x)\lor (\land_{j\in J}\lor_{J\backslash \{j\}}x)}\leq 0,
\end{equation}
so that, for even $m\in [n]_0$, \eqref{bonferroni} also obtains if $\sum_{k=1}^n$ is replaced therein with $\sum_{k=1}^m$.
\end{proposition}
\begin{remark}\label{rmk:inequalities}\leavevmode
\begin{enumerate}[(i)]
\item  Let us focus on \eqref{bonferroni}. For $n=0$ it is trivial: $\PP_{0_\PP}\geq 0$; there is not equality, and there is of course equality on $\{0\}$. For $n=1$ it means $\PP_{x_1}\geq \PP_{0_\PP}$, which is (basically) just Jensen for the square; there is equality iff $x_1=0_\PP$, and there is also equality on $H^{(0)}=\text{ the constants}$. For $n=2$ it reads $\PP_{x_1\lor x_2}+\PP_{x_1\land x_2}\geq \PP_{x_1}+\PP_{x_2}$, which is \cite[Lemma~5.4]{tsirelson}; there is equality iff $x_1\subset x_2$ or $x_2\subset x_1$, and there is also equality on $ H^{\leq 1}$ \cite[Lemma~5.5(c)]{tsirelson}. For $n=3$ we get $$\PP_{x_1\lor x_2\lor x_3}-\PP_{x_1\lor x_2}-\PP_{x_2\lor x_3}-\PP_{x_3\lor x_1}+\PP_{x_1\lor (x_2\land x_3)}+\PP_{x_2\lor (x_3\land x_1)}+\PP_{x_3\lor (x_1\land x_2)}-\PP_{(x_1\lor x_2)\land (x_2\lor x_3)\land (x_3\lor x_1)}\geq 0;$$ there is equality iff $x_1\subset x_2\lor x_3$ or $x_2\subset x_3\lor x_1$ or $x_3\subset x_1\lor x_2$, and there is also equality on $H^{\leq 2}$. And so on. 
\item\label{rmk:inequalities:ii}  A simpler form of the statement of the preceding proposition follows when one assumes that $(x_1,\ldots,x_n)$ is an independency (and any  independency generates a (finite) noise Boolean algebra to which it belongs, so the condition $x\in B^n$ is then superfluous), for in such case $\land_{j\in J}\lor_{J\backslash \{j\}}x=0_\PP$ for all $J\in 2^{[n]}\backslash\{\emptyset\}$ so that \eqref{bonferroni} reads 
\begin{equation}\label{bonferroni-independent}
(-1)^n\sum_{k=0}^n(-1)^k\sum_{J\in {[n]\choose k}}\PP_{\lor_Jx}\geq 0
\end{equation} with equality iff one of the $x_i$, $i\in [n]$, is equal to $0_\PP$, and also with equality on $H^{\leq (n-1)}$. Similarly \eqref{bonferroni-1} simplifies. As a special case we get $\PP_x+\PP_y\leq \PP_{x\lor y}+\PP_{0_\PP}$ for independent $x$ and $y$ from $B$ with equality iff $x=0_\PP$ or $y=0_\PP$ (and also with equality on $H^{\leq 1}$). In particular, for $x\in B$,  $\PP_x+\PP_{x'}\leq \mathrm{id}_{\L2(\PP)}$ on $\L2(\PP)_0$. 
\item Concerning the final statement of the proposition, what of the case when $m$ is odd? It will be clear from the proof that one cannot expect the inequality in \eqref{bonferroni} to then hold in either direction, and indeed it is seen easily that it can fail in both directions even for $n=2$, $x_2={x_1}'$ and $m=1$. 
\end{enumerate}
\end{remark}
\begin{proof}
First, a reduction. Plainly $ H^{\leq (n-1)}(B)\subset H^{\leq (n-1)}(b(\{x_1,\ldots,x_n\}))$. 
Therefore we may, and shall assume that $B=b(\{x_1,\ldots,x_n\})$, in particular that $B$ is finite, taking $(S,\Sigma)=(B,2^B)$ with $\mu=c_B$ counting measure (Example~\ref{example:spectral-finite}). Then $\mu_f(B_z)=\PP[\PP[f\vert z]^2]$ for all $z\in B$ and $f\in \L2(\PP)$. Second,  by inclusion-exclusion (resp. Bonferroni), and still for $f\in \L2(\PP)$,  $$\mu_f(\cup_{i\in [n]}B_{ \lor_{[n]\backslash \{i\}}x})=\sum_{k=1}^n(-1)^{k+1}\sum_{J\in {[n]\choose k}}\mu_f(B_{ (\lor_{[n]\backslash J}x)\lor (\land_{j\in J}\lor_{J\backslash \{j\}}x)})$$ (resp. $\geq \sum_{k=1}^m\cdots$ for even $m\in [n]_0$ and $\leq \sum_{k=1}^m\cdots$ for odd $m\in [n]_0$), for all $f\in \L2(\PP)$ (resp.  and \eqref{bonferroni-1} follows at once). Since $B_{ \lor_{[n]} x}\supset \cup_{i\in [n]}B_{ \lor_{[n]\backslash \{i\}}x}$ \eqref{bonferroni} now follows by monotonicity of $\mu_f$.

If one of the $x_i$, $i\in [n]$, is contained in the join of the others, then in fact $B_{ \lor_{[n]} x}=\cup_{i\in [n]}B_{ \lor_{[n]\backslash \{i\}}x}$ and there is equality in \eqref{bonferroni}. Assume not one of the $x_i$, $i\in [n]$, is contained in the join of the others. Because $\mu_f(\{\lor_{[n]}x\})>0$ for at least one $f$ and $B_{\lor_{[n]} x}\supset \{\lor_{[n]}x\}\sqcup\left(\cup_{i\in [n]}B_{ \lor_{[n]\backslash \{i\}}x}\right)$, it follows that  equality in \eqref{bonferroni} cannot prevail (on the whole of $\L2(\PP)$). Finally, if $f\in H^{\leq (n-1)}$, then $\mu_f\left(B_{\lor_{[n]} x}\backslash \left(\cup_{i\in [n]}B_{ \lor_{[n]\backslash \{i\}}x}\right)\right)=0$, just because each member of $B_{\lor_{[n]} x}\backslash \left(\cup_{i\in [n]}B_{ \lor_{[n]\backslash \{i\}}x}\right)$ contains at least $n$ atoms of $B$, while $\mu_f$ is carried by the $\sigma$-fields generated by at most $n-1$ atoms of $B$.  Therefore by additivity of $\mu_f$ it follows that $\mu_f(B_{ \lor_{[n]} x})=\mu_f( \cup_{i\in [n]}B_{ \lor_{[n]\backslash \{i\}}x})$ and there is equality in \eqref{bonferroni} on $H^{\leq (n-1)}(B)$. 
\end{proof}
\begin{remark}
It may seem curious that one should have ``required'' the intervention of the spectrum of a noise Boolean algebra  in order to establish, say, \eqref{bonferroni-independent}, which concerns only independent $\sigma$-fields. For sure we can give also a ``direct'' proof of \eqref{bonferroni-independent}, that may be interesting for the reader.
\begin{quote}
\begin{proof}[Proof of \eqref{bonferroni-independent} (direct)]
For the case of the inequality, we are to verify that $\langle f, \sum_{k\in [n]_0}\sum_{K\in {[n]\choose k}}(-1)^{n-k}\PP_{\lor_Kx}f\rangle \geq 0$ for all $f\in \L2(\PP)$, equivalently $\sum_{k=0}^n\sum_{K\in {[n]\choose k}}(-1)^{n-k}\PP [\PP[ f\vert \lor_Kx]^2]\geq 0\text{ for all $f\in \L2(\PP)$}$, or merely for all $f\in \L2(\PP\vert_{\lor x})$, it is just as well. By a monotone class argument it suffices to verify the latter for the algebra of functions $f$ that are of the form $\sum_{\gamma\in \Gamma}X_1^\gamma\cdots X_n^\gamma$, where $X_i^\gamma\in \L2(\PP\vert_{x_i})$ for $i\in [n]$, $\gamma\in \Gamma$, are all bounded, and where $\Gamma$ is some finite set. Now by Gram-Schmidt (starting with constants)  we may insist that given any $\{\gamma,\gamma'\}\subset\Gamma$, $\gamma\ne\gamma'$, one has that $\PP[X_i^\gamma X_i^{\gamma'}]=0=\PP[X_i^\gamma]\PP[X_i^{\gamma'}]$ for some $i\in [n]$. Feeding such an $f$ into the left-hand side of the displayed inequality, using the additivity of $\PP_{\lor_Kx}$ and effecting the square we see that the mixed terms (involving $\gamma\ne\gamma'$ from $\Gamma$) vanish (perhaps only once the outer $\PP$ is applied, but it is enough) and that therefore it will in fact suffice to verify the displayed inequality for $f$ of the form as indicated but with $\Gamma$ a singleton, say $\{\gamma\}$. We drop the superscript $\gamma$ and reduce further by homogeneity to the case when $\PP[X_i^2]=1$ for all $i\in [n]$. Let $l_i:=\PP[X_i]^2$; by Jensen $l_i\in [0,1]$ for all $i\in [n]$. Thus we are left to verify that $\sum_{k=0}^n\sum_{K\in {[n]\choose k}}(-1)^{n-k}\prod_{i\in [n]\backslash K}l_i\geq 0$. But this follows at once from the inclusion-exclusion formula for (under some probability $\QQ$) the probability of the union of   independent events $A_1,\ldots,A_n$ with respective probabilities $l_1,\ldots,l_n$. 
The case for equality follows by a similar train of thought, noting in the final instance that the union of independent events is almost certain iff at least one of these events is itself almost certain. 
\end{proof}
\end{quote}
Though, it seems the path via the induced noise Boolean algebra and its spectrum is decidedly preferrable.
\end{remark}
\begin{example}\label{example:counter-for-bonferroni}
Let $\Omega=\{-1,1\}\times \{-1,1\}$, $\FF=2^\Omega$, $\PP$ the uniform measure, $\xi_1$ and $\xi_2$ the canonical projections, $x_2$ generated by $\xi_2$ and $x_1$ generated by the partition $\{\{(1,1),(1,-1)\},\{(-1,1)\},\{(-1,-1)\}\}$. Then $x_1\lor x_2=1_\PP$ and $x_1\land x_2=0_\PP$. It is not the case that $\PP[\PP[\xi_2\vert x_1\lor x_2]^2]+\PP[\PP[\xi_2\vert x_1\land x_2]^2]\geq \PP[\PP[\xi_2\vert x_1]^2]+\PP[\PP[\xi_2\vert x_2]^2]$ as is readily verified. Therefore, in order for \eqref{bonferroni} to prevail, the condition of Proposition~\ref{corollary:inequalities} that the $\sigma$-fields $x_1,\ldots,x_n$ belong to a noise Boolean algebra cannot be dispensed with (and of course we may infer without thinking about it that the $x_1$ and $x_2$ of this example cannot both belong to a noise Boolean algebra).
\end{example}
\begin{remark}
The proof of Proposition~\ref{corollary:inequalities} is perhaps just as informative as the statement itself. One has a ``cooking recipe'' for establishing inequalities \`a la \eqref{bonferroni}, obtaining one for every finite subset $Z$ of $2^B$, by noting that  for all $f\in \L2(\PP)$, (some weakening of) $\mu_f(B_{\lor (\cup Z)})\geq \mu_f(\cup_{z\in Z}B_{\lor z})$ holds true, and applying inclusion-exclusion (Bonferroni) to the latter. In specific situations it may be faster to redo this proof rather than specializing  \eqref{bonferroni}.
\end{remark}
For the remainder of this section we leave \eqref{bonferroni-1} aside and focus on \eqref{bonferroni}. 
\begin{corollary}
Let $n\in \mathbb{N}_0$ and $\{x_1,\ldots,x_n\}\subset B$. Then 
$$\PP_{(\land_{[n]} x)'}+\sum_{k=1}^n(-1)^{k}\sum_{J\in {[n]\choose k}}\PP_{[(\land_{[n]\backslash J}x)\land (\lor_{j\in J}\land_{J\backslash \{j\}}x)]'}\geq 0$$
with equality iff one of the $x_i$, $i\in [n]$, contains the meet of the others, and also with equality on restriction to $H^{\leq (n-1)}(B)$.
\end{corollary}
 \begin{proof}
Just apply Proposition~\ref{corollary:inequalities} to the complements.
 \end{proof}
 
 The content of Proposition~\ref{corollary:inequalities} can also  be rewritten in terms of the so-called influences. The latter terminology is borrowed from the analysis of the sensitivity of Boolean functions \cite[Definition~8.22]{odonnell}, however it is not the only reasonable quantity that is deserving of such a name (cf. Definition~\ref{definition:H-J}). 

\begin{definition}
For $x\in B$ and $f\in \L2(\PP)$ we put $\inff_x(f):=\PP[\var(f\vert x')]=\PP[(f-\PP[f\vert x'])^2]= \PP[f^2]-\PP[\PP[f\vert x']^2]=\mu_f(S\backslash S_{x'})$ for the influence of $x$ on $f$.
\end{definition}

\begin{corollary}
Let $n\in \mathbb{N}$ and $\{x_1,\ldots,x_n\}\subset B$. Then $$\sum_{k=1}^n(-1)^{k+1}\sum_{J\in {[n]\choose k}}\inff_{[(\lor_{[n]\backslash J}x)\lor (\land_{j\in J}\lor_{J\backslash \{j\}}x)]'}\geq \inff_{(\lor_{[n]}x)'}.$$ In particular if $(x_1,\ldots,x_n)$ is an independency, then with $y:=\lor_{[n]}x$,
 $$\sum_{k=0}^{n-1}(-1)^{k}\sum_{J\in {[n]\choose k}}\inff_{y'\lor( \lor_Jx)}\leq (-1)^{n+1}\var;$$  if $(x_1,\ldots,x_n)$ is even a partition of unity, then  $$\sum_{k=1}^{n-1}(-1)^{k}\sum_{J\in {[n]\choose k}}\inff_{ \lor_Jx}\leq (-1)^{n+1}\var.\qed$$
\end{corollary}

\begin{example}
Suppose $(x_1,x_2,x_3)$ is a partition of unity. Then $3\min\{\inff_{x_1\lor x_2},\inff_{x_2\lor x_3},\inff_{x_1\lor x_3}\}\leq \inff_{x_1\lor x_2}+\inff_{x_2\lor x_3}+\inff_{x_1\lor x_3}\leq \var+\inff_{x_1}+\inff_{x_2}+\inff_{x_3}\leq \var+3\max\{\inff_{x_1},\inff_{x_2},\inff_{x_3}\}$; the influence of at least one pair cannot be greater than the largest single influence plus one third the variance.
\end{example}

\section{Topology of noises}
Questions of a predominantly topological character are examined. Notably, we find that $\Cl(B)$ is closed for arbitrary meets and upwards directed joins (Proposition~\ref{proposition:ctb-vs-arbitrary-joins-meets}\ref{proposition:ctb-vs-arbitrary-joins-meets:ii}) and that the join and complementation operation are continuous in a classical noise Boolean algebra (Proposition~\ref{proposition:ct-classical}).
\subsection{(Non)compactness and existence of limits}
%
We might have wished for the topology of $B$ to be compact, but such is not the case. Since $\Lattice$ is Hausdorff it can only happen when $B$ is closed (therefore noise complete and classical). But in general it fails even then. (In particular, $\Lattice$ is not compact.) 
\begin{example}\label{example:non-compact}
Let $B$ be the closure of the (classical) noise Boolean algebra  generated by (the increments of) a one-dimensional standard Wiener process $(W_t)_{t\in [0,1]}$ on subintervals of $[0,1]$.  We may take $S=(2^{[0,1]})_{\text{fin}}$, and $\mu=$ the symmetrized Lebesgue measure on each $\{s\in S:\vert s\vert=m\}={[0,1]\choose m}$, $m\in \mathbb{N}_0$ (it means an atom of mass one at $\emptyset$). For a union of intervals $A$ in $[0,1]$, the associated $\sigma$-field $N_A$ has spectral set $S_{N_A}=\{s\in S:s\subset A\}$ (cf. Example~\ref{example:wiener}; the restriction to the temporal interval $[0,1]$ requires only trivial modifications). Identifying (in the sense of a mod-$0$ isomorphism)  $[0,1]$ under Lebesgue measure with $2^\mathbb{N}\equiv \{0,1\}^\mathbb{N}$ under the product $\Ber(1/2)$ measure in the usual way, denote by $\xi_n$, $n\in \mathbb{N}$, the canonical projections on $\{0,1\}^\mathbb{N}$. Consider the sequence of $\sigma$-fields in $B$, corresponding to the sets (unions of intervals) $A_n:=\{\xi_n=1\}\subset [0,1]$, $n\in \mathbb{N}$. The associated spectral sets $S_{N_{A_n}}$ satisfy $S_{N_{A_n}}\cap \{s\in S:\vert s\vert=1\}=A_n$, and these are a distance of $1/2$ apart in the  product-Bernoulli-measure-of-symmetric-difference pseudometric. Therefore no subsequence of $(N_{A_n})_{n\in \mathbb{N}}$ converges in $B$. Or we can just see it directly. Indeed, the sequence $(A_n)_{n\in\mathbb{N}}$ has the following property: the symmetric difference of any two distinct members of the sequence has length $1/2$. For this reason, for $m\ne n$ from $\mathbb{N}$, one has $\Vert \PP_{N_{A_n}}W_1-\PP_{N_{A_m}}W_1\Vert=1/2$. 
\end{example}

\begin{remark}
In \cite{compact} it is claimed that $\Lattice$ is compact. An error  in the argument of \cite{compact} appears apparently in the proof of Claim 3 of the proof of Proposition 4.7 thereof: it is used that for conditional expectations $\PP[fg\vert x]=\PP[f\vert x]\PP[g\vert x]$ without there being a reason for $f$ and $g$ to be conditionally uncorrelated (for instance, independent) given $x$ (the notation is different in \cite{compact}, but this is the gist of the flaw). In any event, the observation that $\Lattice$ is not compact is not new \cite{compact-not}.
\end{remark}
Nevertheless, $B$ can be compact in not entirely trivial settings. 
\begin{example}\label{example:compact}
Let $B$ be the noise Boolean algebra attached to a (necessarily countable) independency $\Lambda\subset \PP\backslash\{0_\PP\}$ with $\lor\Lambda=1_\PP$. We may take $S=\{\lor L:L\in (2^\Lambda)_{\mathrm{fin}}\}$ and for $\mu$ the counting measure on $2^S$; for $x\in B$ the associated spectral set is $S_x=S\cap B_x$ (cf. Example~\ref{example:classical-spectrum} to follow).  Identifying $B$ (as a set) with $2^\Lambda\equiv \{0,1\}^\Lambda$ in the obvious way, we see that convergence of sequences in $B$ corresponds to pointwise convergence in $2^\Lambda$. But $2^\Lambda$ is a compact first-countable and hence sequentially compact topological space. It follows that any sequence in $B$ has a convergent subsequence, i.e. (the metrizable)  $B$ is compact.
\end{example}

Recall next (part of) \cite[Lemma~3.2]{tsirelson}: if $\lim z=z_\infty$ in $B$  then there exists a $\uparrow\uparrow$ (it means, strictly increasing) sequence $N$ in $\mathbb{N}$ such that $\liminf z_{N}=z_\infty$; therefore   for each $\uparrow\uparrow$ sequence $N$ in $\mathbb{N}$ there exists a $\uparrow\uparrow$ sequence $M$ in $\mathbb{N}$ such that $\liminf z_{N_M}=z_\infty$. One wonders whether some sort of converse to this is true (as happens for sequences of real numbers). The property fails, but some weaker statements do hold true.

\begin{proposition}\label{proposition:convergence-along-liminf}
Suppose $z=(z_n)_{n\in \mathbb{N}}$ is a sequence in $B$ and $z_\infty\in B$. If for each $\uparrow\uparrow$ sequence $N$ in $\mathbb{N}$ there exists a $\uparrow\uparrow$ sequence $M$ in $\mathbb{N}$ such that $\liminf z_{N_M}=z_\infty$, then $\lim (z\land z_\infty)=z_\infty$ and $\lim(z'\land z_\infty)=0_\PP$. If further (for every subsequence there is a subsubsequence along which) $\lim z$ exists, then $\lim z=z_\infty$. 
\end{proposition}
\begin{remark}
$\lim z$ certainly exists along a subsubsequence of every subsequence if $B$ is compact, which may happen in not entirely trivial situations, as we have noted (Example~\ref{example:compact}).
\end{remark}
\begin{proof}
Suppose that for each $\uparrow\uparrow$ sequence $N$ in $\mathbb{N}$ there exists a $\uparrow\uparrow$ sequence $M$ in $\mathbb{N}$ such that $\liminf z_{N_M}=z_\infty$. Because $\land$ is continuous on $\Cl(B)$   we see that $\liminf (z_{N_M}\land z_\infty)=z_\infty$, hence $\liminf (z_{N_M}\land z_\infty)=\limsup (z_{N_M}\land z_\infty)=z_\infty$,  and so $\lim (z_{N_M}\land z_\infty)=z_\infty$. Because the topology of $\Lattice$ is metrizable we conclude that $\lim (z\land z_\infty)=z_\infty$. Let now $f\in \L2(\PP)$. Then by conditional Jensen $\PP[(\PP[f\vert z_n'\land z_\infty]-\PP[f\vert 0_\PP])^2]=\PP[(\PP[\PP[f\vert z_\infty]-\PP[f\vert z_n\land z_\infty]\vert z_n'\land z_\infty])^2]\leq \PP[(\PP[f\vert z_\infty]-\PP[f\vert z_n\land z_\infty])^2]\to 0$ as $n\to\infty$. Thus $\lim(z'\land z_\infty)=0_\PP$.
If $\lim z$ exists, this means that: by the continuity of $\land$ on $\Cl(B)$, $(\lim z)\land z_\infty=z_\infty$, i.e. $z_\infty\subset \lim z$; on the other hand there is a  $\uparrow\uparrow$ sequence $N$ in $\mathbb{N}$ such that $\lim z=\liminf z_N$ and by assumption there is then a $\uparrow\uparrow$ sequence $M$ in $\mathbb{N}$ such that $\liminf z_{N_M}=z_\infty$, however $\liminf z_{N_M}\supset \liminf z_N$, so  $z_\infty\supset \lim z$. If $\lim z$ exists only along a subsubsequence of every subsequence, we obtain by what we have just proven that $\lim z=z_\infty$ along a subsubsequence of every subsequence, and therefore also without this latter qualification (because the topology of $\Lattice$ is metrizable).
\end{proof}

\begin{example}
Return to Example~\ref{example:non-compact}. Put $x_n:=N_{A_n}$ for $n\in \mathbb{N}$. Clearly for any $\uparrow\uparrow$ sequence $K$ in $\mathbb{N}$ we have that for each $m\in \mathbb{N}$, $\liminf_{n\to\infty} S_{x_{K_n}}\cap \{s\in S:\vert s\vert=m\}=\emptyset$ a.e.-$\mu$; so in fact $\liminf_{n\to\infty} S_{x_{K_n}}=\{\emptyset\}$ a.e.-$\mu$. By the argument of the proof of \cite[Lemma~3.2]{tsirelson} it follows that $\liminf x_K=0_\PP$. But, as we have already argued, no subsequence of $(x_n)_{n\in \mathbb{N}}$ converges in $B$.
\end{example}

In general complementation is not continuous in $B$ (see Example~\ref{example:nonclassical-not-cts-complementation} to follow). Still, the convergence of a sequence in $B$ has some implications for limits involving the complements of its members. 

\begin{proposition}
Let $(x_n)_{n\in \mathbb{N}}$ be a sequence in $B$ converging to $x\in B$. Then $\lim_{n\to\infty} (x_n'\land x)=0_\PP$, $\lim_{n\to\infty}(x_n'\lor x')=x'$, $\limsup_{n\to\infty}\PP[\PP[f\vert x_n']^2]\leq \PP[\PP[f\vert x']^2]$ for all $f\in \L2(\PP)$, finally $\liminf_{n\to\infty}x_n'\subset x'$ and one can replace the $\liminf$ with $\lim$  provided the limit exists.
\end{proposition}
\begin{proof}
According to Proposition~\ref{proposition:convergence-along-liminf} we have $\lim_{n\to\infty}(x_n'\land x)=0_\PP$. By separate continuity of $\lor$ on $B$ we infer that $\lim_{n\to\infty}(x_n'\lor x')=(\lim_{n\to\infty}(x_n'\land x))\lor x'=x'$. The third claim is now immediate. For the last one just note that $\liminf_{n\to\infty}x_n'\subset \liminf_{n\to\infty}x_n'\lor x'\subset \lim_{n\to\infty}(x_n'\lor x')=x'$ and one can replace the inferior limits with limits if $\lim_{n\to\infty}x_n'$ exists (due to \ref{generalities:topology}\ref{inclusion-closed}).
\end{proof}

\subsection{Various notions of completeness}
There is a plethora of ways in which a noise Boolean algebra may be deemed ``complete''/``without gaps''/``continuous'' -- we shed  some light on (the inter-relations between) them.

\begin{proposition}\label{proposition:ctb-vs-arbitrary-joins-meets}
Let $A\subset \Lattice$. Then there is a countable $A_0\subset A$ such that $\lor A_0=\lor A$; if $A$ is commutative or directed downwards, then there is a countable $A_0\subset A$ such that $\land A_0=\land A$ (we take $\land \emptyset:=1_\PP$, naturally). In particular we have the following two assertions.
\begin{enumerate}[(i)]
\item\label{proposition:ctb-vs-arbitrary-joins-meets:i} $B$ is closed under countable joins (resp. meets) iff it is closed under arbitrary joins (resp. meets). 
\item\label{proposition:ctb-vs-arbitrary-joins-meets:ii} $\mathrm{Cl}(B)$ contains the  meets (resp. joins) of arbitrary (resp. upwards directed) families in $\Cl(B)$. $B$ is complete iff $B=\Cl(B)$. A classical $B$ is complete iff it is noise complete.
\end{enumerate}
\end{proposition}

\begin{question}
Does the claim concerning meets hold true without  the ``commutative or directed downwards'' qualification?
\end{question}
\begin{remark}
Compare with \cite[Eq.~(4.2)]{tsirelson} where the assertion is that $B$ is $\sigma$-complete iff it is complete (both as a Boolean algebra). It is not quite the same statement because a noise Boolean algebra can be complete as a Boolean algebra but not complete (see next example). In fact, in the proof of \cite[Corollary~4.7]{tsirelson} the assertion of \cite[Eq.~(4.2)]{tsirelson} is used to prove completenes  (rather than completeness as a Boolean algebra), leaving the argument, in a manner of speaking, incomplete (note that \cite{tsirelson} does use the qualification ``complete'' for a noise Boolean algebra in the same way as we do \cite[display immediately preceding Section~1.2]{tsirelson}). It is however easily corrected, because in this proof one has in fact closure under denumerable joins and meets (which is used to conclude that $B$ is $\sigma$-complete as a Boolean algebra, but this is incidental), hence completeness by the very proposition above.
\end{remark}

\begin{example}
Let $\mathfrak{F}:=\{A\in 2^\mathbb{N}:A\text{ cofinite}\}$. $\mathfrak{F}$ is a proper filter of $\mathbb{N}$.  Like any proper filter it is contained in some ultrafilter $\mathfrak{U}$ of $\mathbb{N}$. Suppose $\xi=(\xi_i)_{i\in \mathbb{N}\cup \{\infty\}}$ is an independency in $\Lattice\backslash \{0_\PP\}$ that generates $1_\PP$. Define $N:2^\mathbb{N}\to \Lattice$ by setting $N(A):=\sigma(\xi\vert_A)$ for $A\in 2^\mathbb{N}\backslash \mathfrak{U}$ and $N(A):=\sigma(\xi\vert_A)\lor \sigma(\xi_\infty)$ for $A\in \mathfrak{U}$. Then the range of $N$ is a noise Boolean algebra $B$, and $N$ is an injective noise factorization, an isomorphism of Boolean algebras. $2^\mathbb{N}$ is complete as a  Boolean algebra; therefore $B$ is also complete as a Boolean algebra. However, $B$ is not complete (as a noise Boolean algebra), for $\xi\vert_{\mathbb{N}}$ is a sequence in  $B$, whose join $\lor_{i\in \mathbb{N}}\xi_i$ does not belong to $B$. $B$ is classical since it is contained in the complete noise Boolean algebra that is attached to the independency $\xi$. 
\end{example}
\begin{question}\label{question:complete-as-Bolean}
The $B$ of the preceding example is not  noise complete (since it is classical, but not complete). Can $B$ be noise complete, complete as a Boolean algebra, but not complete ($\therefore$ nonclassical)?
\end{question}
\begin{proof}
The join $\lor A$ is essentially separable. Therefore it admits an essentially generating countable subset $\tilde{A}$. Each member $\tilde{a}$ of $\tilde{A}$ belongs to $\lor A_{\tilde{a}}$ for some countable $A_{\tilde{a}}\subset A$. Then $A_0:=\cup_{\tilde{a}\in \tilde{A}}A_{\tilde{a}}$ is a countable subset of $A$ with $\lor A_0=\lor A$. 

Assume now $A$ is commutative (resp. downwards directed) and let $A_0$ be a dense countable subset of $A$ (it exists due to \ref{generalities:topology}\ref{rmk:general:Polish}). Because of the continuity of the meet operation on $A$ we have that   $\{\land \tilde A_0:\tilde A_0\in (2^{A_0})_{\mathrm{fin}}\}$ [whose meet is the same as that of $A_0$] is dense in $\{\land \tilde{A}:\tilde{A}\in (2^A)_{\mathrm{fin}}\}$ [whose meet is the same as that of $A$] also (resp. $\{a\in A:a\subset a_0\text{ for some }a_0\in A_0\}$ is downwards directed, contained in $ A$, contains $A_0$, and its meet is the same as that of $A_0$). Dropping the preceding assumption, we show that if $A_0\subset A\subset \Lattice$ are downwards directed and $A_0$ is dense in $A$, then $\land A_0=\land A$, which will be enough.

Now, for each $f\in b\FF$, by decreasing martingale convergence for nets \cite[Lemma~2.3]{vidmar_2019} one has $\PP[\PP[f\vert \land A]^2]=\inf_{a\in  A}\PP[\PP[f\vert a]^2]$ and likewise for $A_0$ in place of $A$. Because $A_0$ is dense in $A$ we also have $\inf_{a\in  A}\PP[\PP[f\vert a]^2]=\inf_{a\in  A_0}\PP[\PP[f\vert a]^2]$. It follows that $\PP[\PP[f\vert \land A]^2]=\PP[\PP[f\vert \land A_0]^2]$. Polarizing it entails $\PP[g\PP[f\vert \land A]]=\PP[\PP[g\vert \land A]\PP[f\vert \land A]]=\PP[\PP[g\vert \land A_0]\PP[f\vert \land A_0]]=\PP[g\PP[f\vert \land A_0]]$ for all $g\in b\FF$, and so $\PP[f\vert \land A]=\PP[f\vert \land A_0]$. Since $f\in b\FF$ was arbitrary we deduce that $\land A=\land A_0$. 

The statements of \ref{proposition:ctb-vs-arbitrary-joins-meets:i} and  \ref{proposition:ctb-vs-arbitrary-joins-meets:ii} are now essentially immediate; for the last part of the latter recall that for a classical $B$, $\overline{B}=\Cl(B)$. 
\end{proof}

\begin{remark}
The proof shows (modulo trivial interventions) that if an $A\subset \Lattice$ is downwards directed then $\land A_0=\land A$ for any $A_0\subset A$ that is dense in $A$. The obvious analogous statement for an upwards directed $A$ holds also (with the same proof, mutatis mutandis, in particular using increasing martingale convergence for nets \cite[Proposition~V-1-2]{neveu}). On account of \ref{generalities:topology}\ref{rmk:general:Polish} countable such $A$ always exist.
\end{remark}

\begin{definition}
$B$ is minimally downwards (resp. upwards) continuous if $\land U\in B$ for every ultrafilter $U$ in $B$ (resp. $\lor M\in B$ for every maximal ideal $M$ in $B$).
\end{definition}

\begin{remark}
This is a slight (apparently natural) generalization of these concepts that appear (with the same names)   in \cite[Definition~1.6]{vershik-tsirelson}. Trivially, completeness implies minimal up/down continuity.
\end{remark}

\begin{example}
The (noise-complete)  simplest nonclassical noise Boolean algebra of Example~\ref{example:simplest-nonclassical} is minimally downwards continuous, but it is not minimally upwards continuous. 
\end{example}

\begin{proposition}\label{proposition:ultrafilters}
 If $F$ is a filter in $B$ such that $\lor (F')\in B$, then  $\land F=(\lor (F'))'\in B$; correspondingly, if $I$ is an ideal in $B$ such that $\lor I\in B$, then $\land (I')=(\lor I)'\in B$. In particular, if $B$ is minimally upwards continuous, then it is minimally downwards continuous. 
\end{proposition}
\begin{example}
If $(x_n)_{n\in\mathbb{N}}$ is a sequence in $B$ nondecreasing to $1_\PP$, then $({x_n}')_{n\in \mathbb{N}}$ is nonincreasing to $0_\PP$ (take $I=\cup_{n\in \mathbb{N}}B_{x_n}$) \cite[p.~233, final paragraph]{tsirelson}. More generally, if $(x_n)_{n\in\mathbb{N}}$ is a sequence in $B$ nondecreasing to an $x\in B$, then $({x_n}')_{n\in \mathbb{N}}$ is nonincreasing to $x'$.
\end{example}
\begin{proof}
Only the equality $\land F=(\lor (F'))'$ is not immediately clear. But trivially, for each $x\in F$, $x'\subset \lor (F')$, hence $(\lor(F'))'\subset x''=x$; therefore $(\lor (F'))'\subset \land F$. For the converse inclusion take an arbitrary $f\in \L2(\PP\vert_{\lor (F')})$ with $\PP[f]=0$. Then for $x\in F$ we have by  Remark~\ref{rmk:inequalities}\ref{rmk:inequalities:ii} that $\PP[\PP[f\vert \land F]^2]\leq \PP[\PP[f\vert x]^2]\leq \PP[f^2]-\PP[\PP[f\vert x']^2]$, which is $\to \PP[f^2]-\PP[\PP[f\vert \lor (F')]^2]=0$ along $x\in F$ by increasing martingale convergence for nets \cite[Proposition~V-1-2]{neveu}. It means that $\PP_{\land F}\leq \PP_{(\lor (F'))'}$ on the zero-mean subspace of $\L2(\PP\vert_{\lor (F')})$ and hence of course on the whole of $\L2(\PP\vert_{\lor (F')})$. In conjunction with the inclusion that we have already established we get that $(\land F)\land (\lor (F'))=(\lor (F'))'\land (\lor (F'))$ and $(\land F)\land (\lor (F'))'=(\lor (F'))'\land (\lor (F'))'$, whence by distributivity $\land F=(\lor (F'))'$ ($\land F$ does indeed distribute over $B$, since it belongs to $\Cl(B)$ by Proposition~\ref{proposition:ctb-vs-arbitrary-joins-meets}\ref{proposition:ctb-vs-arbitrary-joins-meets:ii} [recall \ref{generalitites:subalgebra}]).
\end{proof}

\begin{proposition}
We have the following assertions.
\begin{enumerate}[(i)]
\item\label{ct-1} If $U$ is an ultrafilter in $B$ and $\land U\in B$ then $\land U$ is either $0_\PP$ or an atom of the Boolean algebra $B$. Similarly, if $M$ is a maximal ideal in $B$ and $\lor M\in B$, then $\lor M=1_\PP$ or $\lor M$ is the complement of an atom of the Boolean algebra $B$. 
\item\label{ct-2} If $B$ is minimally upwards continuous and $M$ is a maximal ideal in $B_x$ for some $x\in B$, then $\lor M\in B_x$ (and $\lor M$ is $x$ or the complement in $B_x$ of an atom of the Boolean algebra $B_x$). Almost symmetrically: if $B$ is minimally downwards continuous  and $U$ is an ultrafilter of ${}_xB$ for some $x\in B$, then $\land U\in {}_xB$ (and $\land U$ is $x$ or an atom of the Boolean algebra ${}_xB$). 
\end{enumerate}
\end{proposition}
\begin{remark}
\ref{ct-1} means that the notion of atomless in \cite[Definition~1.12]{tsirelson} is equivalent to the notion of minimal down continuity, as we have insisted on, coupled with $B$ being atomless as a noise Boolean algebra. An important result  \cite[Theorem~1.13]{tsirelson} is that if $B_0\subset B$ is a noise Boolean algebra that is atomless and minimally downwards continuous then it is sufficient, meaning (by definition \cite[Section~1.5, first paragraph]{tsirelson}) that $H^{(1)}(B_0)=H^{(1)}(B)$.
\end{remark}
\begin{proof}
\ref{ct-1} is clear because if $\land U$ is neither $0_\PP$ nor an atom of $B$, then there is an $x\in B$ with $0_\PP\ne x\subsetneq \land U$ and ${}_xB$ is a filter in $B$ strictly larger than $U$; similarly for $M$. \ref{ct-2}. Consider $\tilde M:=\{m\lor a:(m,a)\in M\times B_{x'}\}$, which is a maximal ideal in $B$. We get $(\lor M)\lor x'=\lor \tilde M\in B$ and therefore, by general distributivity over independent members of $\Lattice$, $\lor M=(\lor M)\land x=((\lor M)\lor x')\land x\in B$. Similarly $\tilde U:=\{u\land a':(u,a)\in U\times B_x\}$ is an ultrafilter in $B$. We get $(\land U)\land x'\in B$ and therefore, again by general distributivity over independent members of $\Lattice$, $\land U=\land_{u\in U}((u\land x')\lor x)=(\land_{u\in U}(u\land x'))\lor x=((\land U)\land x')\lor x\in B$. 
\end{proof}
On a couple of examples we show there are in general no implications between noise completeness on the one hand and minimal up/down continuity on the other.

\begin{example}\label{example:simplest-nonclassical-tweaked}
Let $\mathcal{B}$ be the finite-cofinite algebra on $\mathbb{N}$ and suppose $(\xi_n)_{n\in \mathbb{N}\cup \{\infty\}}$ is a sequence of independent equiprobable random signs, which generate $1_\PP$. For $b\in \mathcal{B}$ finite let $N_b:=\sigma(\xi_i\xi_{i+1}:i\in b)$; for $b\in \mathcal{B}$ cofinite  let $N_b:=\sigma(\xi_i\xi_{i+1}:i\in b_{<\max(\mathbb{N}\backslash b)})\lor \sigma(\xi_i:i\in b_{>\max(\mathbb{N}\backslash b)})\lor \sigma(\xi_\infty)$ ($\max\emptyset:=0$). Then $B=\{N_b:b\in \BB\}$ is a small tweak on the simplest nonclassical noise Boolean algebra of Example~\ref{example:simplest-nonclassical}; namely the ``cofinite'' $\sigma$-fields have an extra $\xi_\infty$. Still $B$ can be written as a $\uparrow$ union of classical finite noise Boolean algebras, specifically of the $B_n$,  $n\in \mathbb{N}_0$, where, for  $n\in \mathbb{N}_0$, $B_n$ is attached to the independency $\xi_1\xi_2,\ldots,\xi_n\xi_{n+1},(\xi_{n+1},\ldots,\xi_\infty)$. The closure of $B$ is easily (for instance via Fourier-Walsh expansion) seen to be equal to  $\mathrm{Cl}(B)=B\cup \{\sigma(\xi_i\xi_{i+1}:i\in b):b\in (2^\mathbb{N})_{\text{inf}}\}\cup \{\sigma(\xi_\infty,\xi_i\xi_{i+1}:i\in b):b\in 2^\mathbb{N}\}$ and $\overline{B}=B$. We see that $B$ is noise complete, however it is not minimally downwards continuous (and therefore not minimally upwards continuous). Indeed the ``cofinite'' $\sigma$-fields form an ultrafilter, whose meet is $\sigma(\xi_\infty)$, which does not belong to $B$. Incidentally, this $B$ is not complete as a Boolean algebra (cf. Question~\ref{question:complete-as-Bolean}) and its closure is closed for $\lor$ (like it was in Example~\ref{example:simplest-nonclassical}).
\end{example}

\begin{example}
Let $\mathcal{B}$ be the finite-cofinite algebra on $\mathbb{N}$ and suppose $(\xi_n)_{n\in \mathbb{N}}$ is an independency from $\Lattice\backslash \{0_\PP\}$, which generates $1_\PP$. Then $B=\{\sigma(\xi\vert_b):b\in \BB\}$ is minimally upwards continuous, however it is not noise complete.
\end{example}



\subsection{Convergence and continuity in a classical noise Boolean algebra}\label{subsection:convergence-classical}
Recall \ref{generalitites:subalgebra} for the distributivity in Item~\ref{conv-classical-B} of the next proposition; also the fact that the action of (the elements of) $B$ leaves $H^{(1)}$ invariant.

\begin{proposition}[Convergence in a classical noise Boolean algebra]\label{proposition:convergence-in-classical}
Let $B$ be classical, $(x_n)_{n\in \mathbb{N}}$  a sequence in $B$, and $y\in \Lattice$ be such that:
\begin{enumerate}[(a)]
\item\label{conv-classical-A} $\lim_{n\to\infty}\PP_{x_n}=\PP_y$ strongly on the first chaos $H^{(1)}$;
\item\label{conv-classical-B}  $y$ distributes over $B$. 
\end{enumerate}
Then $\lim_{n\to\infty}x_n=y$ (and hence $y\in \overline{B}=\Cl(B)$).
\end{proposition}
\begin{proof}
Observe first that $\PP_y$ commutes with $\PP_x$, $x\in B$, on the first chaos because of \ref{conv-classical-A}. In consequence $\PP_x\PP_y=\PP_y\PP_xf=\PP_{x\land y}f$ for $f\in H^{(1)}$, $x\in B$.

Now, it suffices to prove that $\PP[f\vert x_m]\to \PP[f\vert y]$ as $m\to\infty$ in $\L2(\PP)$ for a set of $f$ that is total in $\L2(\PP)$. Therefore, by the classicality of $B$ (which implies that $\L2(\PP)=\oplus_{n=0}^\infty H^{(n)}$), we may assume such $f$ is of the form $\prod_{k=1}^n f_k$, where $f_1,\ldots,f_n$ are elements of the first chaos measurable respectively relative to the entries of a partition $(u_1,\ldots,u_n)$ of unity in $B$, $n\in \mathbb{N}$ (recall from  \ref{generalities:chaoses} the structure of the higher chaoses in terms of the first chaos; the convergence is trivial on the constants, i.e. on $H^{ (0)}$). Then we compute via the ``tensorisation of conditioning'' that, in $\L2(\PP)$,
 \begin{align*}
 \PP[f\vert x_m]&=\PP[ \prod_{k=1}^nf_k\vert \lor_{k=1}^n (x_m\land u_k)]=\prod_{k=1}^m\PP[f_k\vert x_m\land u_k]= \prod_{k=1}^m\PP[f_k\vert x_m]\to\\
 &=  \prod_{k=1}^m\PP[f_k\vert y]=\prod_{k=1}^m\PP[f_k\vert y\land u_k]=\PP[\prod_{k=1}^nf_k\vert \lor_{k=1}^n (y\land u_k)]=\PP[f\vert y]\text{ as $m\to\infty$},
 \end{align*}
where the last equality used assumption \ref{conv-classical-B}, while the convergence on the elements of the first chaos is assumption \ref{conv-classical-A} (and is ``preserved under independent products'' /one can see it by telescoping, for instance/).
%
\end{proof}

\begin{corollary}\label{corollary:A-vs-closure-of-B}
If $B$ is a classical noise Boolean algebra and for an $x\in \Lattice$, $\PP_x\in \AA$ and  $x$ distributes over $B$, then $x\in \overline{B}$.
\end{corollary}

\begin{example}
The distributivity assumption cannot be dropped (it does not follow automatically from the rest that is given). In particular in general the inclusion $\{x\in \Lattice:\PP_x\in \AA\}\subset \overline{B}$ may fail, even in the classical case. Consider indeed the classical noise Boolean algebra $B$ attached to the independency $(\xi_i)_{i\in [3]}$ of independent equiprobable signs (which is complete, of course). Then $x:=\sigma(\xi_1\xi_2,\xi_3)$ satisfies $\L2(\PP\vert_x)=\mathrm{lin}(1,\xi_1\xi_2,\xi_3,\xi_1\xi_2\xi_3)$, hence is such that $\PP_x\in \AA$, yet $x$ does not distribute over $B$ (for instance $(x\land \sigma(\xi_1))\lor (x\land \sigma(\xi_2,\xi_3))=0_\PP\ne x$). 
\end{example}

\begin{proof}
Like for any projection in $\AA$, $\alpha(\PP_x)=\mathbbm{1}_{S_x}$ for some $S_x\in \Sigma$ (which is indeed the $S_x$ that we have already introduced if $x\in \Cl(B)$, but for now we do not know that). Let $B_0\subset B$ be countable and dense and fix representatives $S_y$, $y\in B_0$, such that they form a $\pi$-system. We see by approximation that there is a sequence $(E_n)_{n\in \mathbb{N}}$ in the algebra generated by the $S_y$, $y\in B_0$, such that $\mathbbm{1}_{S_x}=\lim_{n\to\infty}\mathbbm{1}_{E_n}$ a.e.-$\mu$. Therefore $\PP_x$ is the strong limit of a sequence $(P_n)_{n\in \mathbb{N}}$ of operators such that for each $n\in \mathbb{N}$ there exists a finite subcollection $b^n$ of $B_0$, an $N_n\in \mathbb{N}_0$, and an injective sequence $(a_m^n)_{m\in [N]}$ of subsets of $b_n$ satisfying $P_n=\sum_{m\in [N_n]}(\prod_{x\in a_m^n}\PP_x)(\prod_{x\in b^n\backslash a^n_m}(I-\PP_{x}))$ (here $I$ is the identity on $\L2(\PP)$).  On the first chaos: first, $\PP_{x'}$ may replace each $I-\PP_x$ in the preceding; second, 
 the products and finally each sum $P_n$ gets replaced with  $\PP_{z_n}$, where $z_n:=\lor_{m\in [N_n]}(\land a_m^n)\land [\land ((b^n\backslash a^n_m)')] \in B_0$. According to Proposition~\ref{proposition:convergence-in-classical} it is enough to conclude that $\lim_{n\to\infty}z_n=x$, hence $x\in \Cl(B_0)=\Cl(B)=\overline{B}$.
\end{proof}

\begin{proposition}\label{proposition:ct-classical}
Let $B$ be classical and (noise, it is the same) complete. Then independent complementation and the join operation $\lor$ are continuous on $B$. 
\end{proposition}
\begin{proof}
The topology of $\Lattice$ is metrizable so we are allowed to argue with sequences. We first prove continuity of independent complementation. 
By Proposition~\ref{proposition:convergence-in-classical} it reduces to checking the strong convergence of the associated conditional expectations on the first chaos. But there it is trivial. 
To prove continuity of the join operation let $(x_n)_{n\in \mathbb{N}}$ and $(y_n)_{n\in \mathbb{N}}$ be sequences in $B$ converging to $x_\infty\in B$ and $y_\infty\in B$, respectively. The intersection (meet, $\land$) operation is continuous on $\Cl(B)=B$ (the latter equality being by the classicality and noise completeness of $B$). Therefore $x_n\lor y_n=(x_n'\land y_n')'\to (x_\infty'\land y_\infty')'=x_\infty\lor y_\infty$ as $n\to\infty$.
\end{proof}
\begin{remark}\label{remark:nonclassical-not-cts-complementation}
For a nonclassical $B$ join $\lor$ and a fortiori (because of the continuity of $\land$ on $\Cl(B)$, as we have seen in the proof) independent complementation are not continuous in general, as the next example demonstrates.
\end{remark}
\begin{example}\label{example:nonclassical-not-cts-complementation}
 Let $(x_n)_{n\in \mathbb{N}}$ and $(y_n)_{n\in \mathbb{N}}$ be $\downarrow$ sequences in $B$. Then $(\land_{n\in \mathbb{N}}x_n)\lor (\land_{n\in \mathbb{N}}y_n)\subset \land_{n\in \mathbb{N}}(\land_{m\in \mathbb{N}}(x_n\lor y_m))\subset\land_{n\in \mathbb{N}}(x_n\lor y_n)$ with equalities if $B$ is classical as we have just seen. The  equality $(\land_{n\in \mathbb{N}}x_n)\lor (\land_{n\in \mathbb{N}}y_n)= \land_{n\in \mathbb{N}}(\land_{m\in \mathbb{N}}(x_n\lor y_m))$ (and the more so the equality $(\land_{n\in \mathbb{N}}x_n)\lor (\land_{n\in \mathbb{N}}y_n)= \land_{n\in \mathbb{N}}(x_n\lor y_n)$) can fail if $B$ is nonclassical (though one has  $(\land_{n\in \mathbb{N}}x_n)\lor (\land_{n\in \mathbb{N}}y_n)=\land_{n\in \mathbb{N}}(x_n\lor y_n)$ if one of the two $\downarrow$ sequences is eventually constant). 
 Here is a (counter)example. Consider the simplest nonclassical noise Boolean algebra attached to the sequence $(\xi_n)_{n\in \mathbb{N}}$ of Example~\ref{example:simplest-nonclassical}. For $n\in \mathbb{N}$, let $x_n$ be the $\sigma$-field generated by $\xi_m\xi_{m+1}$ for $m\in \mathbb{N}$ odd and by the tail $\xi\vert_{[n,\infty)}$, and we define $y_n$ in the same way except that we take $m$ even. Plainly $ \land_{n\in \mathbb{N}}(\land_{m\in \mathbb{N}}(x_n\lor y_m))=1_\PP$, however $(\land_{n\in \mathbb{N}}x_n)\lor (\land_{n\in \mathbb{N}}y_n)=\sigma(\xi_{n}\xi_{n+1}:n\in \mathbb{N})\ne 1_\PP$.
\end{example}

\subsection{Miscellany}
We will use up the next ``approximation'' result in the proof of Proposition~\ref{prop;K-finite} but it is also interesting in its own right.
\begin{proposition}\label{proposition:approximating-partition-of-unity}
Let $P$ be a partition of unity of $B$ and suppose that $B_0\subset B$ is a noise Boolean subalgebra dense in $B$. Then there exist sequences $(x^p_n)_{n\in \mathbb{N}}$, $p\in P$, in $B_0$ such that $\lim_{n\to\infty}\lor_{q\in Q}x^q_n=\lor Q$ for each $Q\subset P$; for any such sequences we have that for each $p\in P$, $y^p_n:=x^p_n\land (\land_{q\in P\backslash \{p\}}(x^q_n)')\ne 0_\PP$ for arbitrary large $n\in \mathbb{N}$ (in whichever order of qualification of $p$ and $n$). If $B$ is classical, then we may even insist that the $x^p_n$, $p\in P$, form a partition of unity for each $n\in \mathbb{N}$. 
\end{proposition}
\begin{proof}
For each $Q\subset P$ there exists a sequence $(a^{Q}_m)_{m\in\mathbb{N}}$ in $B_0$ with limit $\lor Q$. For $p\in P$ and $m\in \mathbb{N}$ put $x^p_m:=\land_{Q\in 2^P,Q\ni p}a^Q_m$. Due to the continuity of $\land$ on $\mathrm{Cl}(B)$ it follows that for each $p\in P$, $\lim_{m\to\infty}(x^p_m\land p)=p$. Then for all $Q\subset P$, by virtue of \ref{generalities:topology}\ref{rmk:general:vii}, $\lim_{m\to\infty}\lor_{q\in Q}(x^q_m\land q)=\lor Q$; thanks to \ref{generalities:topology}\ref{lemma:sandwich} we get $\lim_{n\to\infty}\lor_{q\in Q}x^q_n=\lor Q$ (since certainly $\lor_{q\in Q}x^q_n\subset a_n^Q\to \lor Q$ as $n\to\infty$). 

Suppose now the sequences $(x^p_n)_{n\in \mathbb{N}}$, $p\in P$, with the stipulated property are given. Suppose per absurdum that for some $p\in P$ and some $N\in \mathbb{N}$, and then for each $n\in \mathbb{N}_{\geq N}$, $y^p_n:=x^p_n\land (\land_{q\in P\backslash \{p\}}(x^q_n)')= 0_\PP$; it implies $x^p_n\subset \lor_{q\in P\backslash \{p\}}x^q_n$ and passing to the limit $n\to\infty$, $p\subset \lor (P\backslash\{p\})$, which is a contradiction.

Suppose finally $B$ is classical. Define $y^p_n:=x^p_n\land (\land_{q\in P\backslash \{p\}}(x^q_n)')$, $(p,n)\in P\times \mathbb{N}$, 
and then for some choice of $p_0\in P$ replace $y^{p_0}_n$ with $y^{p_0}_n\lor (\land_{p\in P}(y^p_n)')$.  Due to Proposition~\ref{proposition:ct-classical} $\lim_{n\to\infty}\lor_{q\in Q}y^q_n=\lor Q$ for each $Q\subset P$ and plainly  the $y^p_n$, $p\in P$, form a partition of unity for each $n\in \mathbb{N}$, where perhaps we need to pass to a subsequence to ensure that all the $y^p_n$, $(p,n)\in P\times \mathbb{N}$, are $\ne 0_\PP$.
\end{proof}

\begin{example}
Let $D_1$ be a countable dense subset of $\mathbb{R}$ and let $D_2\subset \mathbb{R}\backslash D_1$ be another dense countable subset of $\mathbb{R}$ (they  surely exist); put $D:=D_1\cup D_2$. Let $B$ be the classical noise Boolen algebra attached to a generating independeny $(\xi_d)_{d\in D}$. Let $B_0$  be the noise Boolean subalgebra consisting of those members of $B$ which are generated by $(\xi_d)_{d\in D\cap I}$ for $I$  a union of intervals of $\mathbb{R}$ with endpoints in $\mathbb{R}\backslash (D_1\cup D_2)$. Then $B_0$ is certainly dense in $B$ and the reader may find it interesting to consider what sequences (forming a partition of unity at each ``time'' point from $\mathbb{N}$) one can take in the context of the preceding proposition for the independency $P=\{\sigma(\xi_d:d\in D_1),\sigma(\xi_d:d\in D_2)\}$.
\end{example}

Like any topological space,  $\Lattice$, hence $B$, comes equipped with its Borel $\sigma$-field. We put it on record that the natural $\pi$-system $\{B_x:x\in B\}$ [note: each $B_x$, $x\in B$, is closed, therefore Borel]  fails to be generating in general, even in the classical case:

\begin{example}
Let $B$ be the complete classical noise Boolean algebra of a two-sided linear Wiener process on the real line (cf. Example~\ref{example:wiener}).  The claim is that for no countable $H\subset B$ can the collection $\{B_h:h\in H\}$ separate the points of $B$ ($\therefore$ $\{B_x:x\in B\}$ does not generate the separable and separating $\mathcal{B}_B$).  Towards establishing it, recall first the canonical correspondence between the Lebesgue sets of $\mathbb{R}$ and $B$; it is characterized by sending every interval of $\mathbb{R}$ to the $\sigma$-field generated by the increments of $W$ over the interval, and by the fact that the further association of each $\sigma$-field to its conditional expectation, restricted to the first chaos,  defines a projection-valued measure. This correspondence is injective if the Lebesgue sets are considered modulo Lebesgue null sets; as usual in such considerations, rather than ``go through the pain'' of writing equivalence classes, we will be imprecise and not distinguish in what follows between a Lebesgue set and its equivalence class modulo null sets. The claim then reduces to establishing that for no countable family of Lebesgue sets $H$ is the family (in the obvious notation: the subscript $\mathrm{leb}$ means Lebesgue-measurable sets) $\{(2^h)_{\mathrm{leb}}:h\in H\}$ separating the Lebesgue sets of $\mathbb{R}$. Suppose per absurdum that some such collection would meet this demand; we may assume $\mathbb{R}\notin H$. Clearly $H$ cannot be empty, not even finite.  Fix $h_0\in H$. It is plain that one can write $\emptyset\ne \mathbb{R}\backslash h_0=a_1\sqcup a_2$ (disjoint union) for some Lebesgue sets $a_1$ and $a_2$ of strictly positive measure $\geq \epsilon\in (0,\infty)$. 
Now enumerate $H\backslash \{h_0\}$ in some definite order $(h_n)_{n\in \mathbb{N}}$. Inductively for $n\in \mathbb{N}$ we tweak $a_1$ and $a_2$ as follows. It may happen that neither $a_1$ nor $a_2$ is contained in $h_n$ in which case we leave $a_1$ and $a_2$ be. If however $a_1$ (resp. $a_2$) is contained in $h_n$, then we make $a_1$ (resp. $a_2$) a little bigger, enlarging it by a part of $\mathbb{R}\backslash h_n\ne \emptyset$, no matter which, as long as it is of Lebesgue measure no greater than $\epsilon/4^{n}$. In the limit (we take the respective unions, of course) the resulting $a_1$ and $a_2$ are two Lebesgue sets, whose symmetric difference has positive Lebesgue measure, and, of course, for all $n\in \mathbb{N}$, neither $a_1$ nor $a_2$ is contained in $h_n$.
\end{example}

\begin{remark}
It means a fortiori that $\mathcal{B}_{\Lattice}$ is not generated by $\{\Lattice_x:x\in\Lattice\}$ whenever $\PP$ supports a Wiener process (or, which is the same, a sequence of independent equiprobable random signs).
\end{remark}
Though:
\begin{example}
Let $B$ be the classical noise Boolean algebra associated to a generating independency $\Lambda\subset \Lattice$ (Example~\ref{example:classical}). The countable system $\{B_{\lambda'}:\lambda\in \Lambda\}$ separates the points of $B$, therefore \cite[p. 32, Theorem 6.8.9]{bogachev} it generates $\mathcal{B}_B$, since $B$ is Polish (it is a closed subset of $\Lattice$, which is Polish). A fortiori $\{B_x:x\in B\}$ generates $\mathcal{B}_B$.
\end{example}

We have mentioned that in general it is not known whether $\Cl(B)$ is closed for $\lor$. A partial result in this direction is that it is closed for $\lor$ on  elements that are ``separated'' by $\overline{B}$; we will use it up in Subsection~\ref{subsection:atoms}.
\begin{proposition}\label{proposition:closure-separate-join}
Let $\{x,y\}\subset \Cl(B)$ and $\{u,v\}\subset \overline{B}$ with $u$ independent of $v$, $x\subset u$, $y\subset v$. Then $x\lor y\in\Cl( B)$ and $a\land (x\lor y)=(a\land x)\lor (a\land y)$ for all $a\in \Cl(b)$.
\end{proposition}
\begin{proof}
 Let $(x_n)_{n\in \mathbb{N}}$ be a sequence in $B$ converging to $x$ and let $(y_n)_{n\in \mathbb{N}}$ be a sequence in $B$ converging to $y$. By continuity of $\land $ on $\Cl(B)$ we get that $(x_n\land u)_{n\in \mathbb{N}}$ is a sequence in $\overline{B}$ converging to $x$, likewise  $(y_n\land v)_{n\in \mathbb{N}}$ is a sequence in $\overline{B}$ converging to $y$. Therefore \ref{generalities:topology}\ref{rmk:general:vii}  $((x_n\land u)\lor (y_n\land v))_{n\in \mathbb{N}}$ is a sequence in $\overline{B}$ converging to $x\lor y$. 
 
 Let further $(a_n)_{n\in \mathbb{N}}$ be a sequence in $B$ converging to $a$. Then, for each $n\in \mathbb{N}$, $a_n\land ((x_n\land u)\lor (y_n\land v))=(a_n\land x_n\land u)\lor (a_n\land y_n\land v)$, and the l.h.s. converges to $a\land (x\lor y)$ while the r.h.s. converges to $(a\land x)\lor (a\land y)$ as $n\to\infty$.
\end{proof}

\section{Atoms of the  spectral measure}
The discrete part of the spectrum of a noise Boolean algebra is brought to the forefront. Atoms of the spectrum are identified  (Theorem~\ref{prop:dicrete}) and a certain particulary pleasant discrete case of spectra is delineated (Theorem~\ref{propsition:discrete-spectrum}). In the context of reverse filtrations whose innovations are  equiprobable random signs, product-typeness 
is connected to a discreteness property of the spectral measure of the associated noise Boolean algebra (Theorem~\ref{proposition:innovations-reverse-equiprobable}).

\subsection{General considerations: identification of the spectral atoms}\label{subsection:atoms}
First, a straightforward characterization of the case when the whole of the spectrum of $B$ is discrete: briefly, it is so iff the conditional expectations $\PP_x$, $x\in B$, can be diagonalized through a (countable) sum decomposition of $\L2(\PP)$. In precise terms:

\begin{proposition}\label{proposition:purely-atomic}
The following statements are equivalent. 
\begin{enumerate}[(A)]
\item\label{atomic:a} $\L2(\PP)$ decomposes into a countable orthogonal sum $\oplus_{s\in S^\circ}G_s$ of non-zero closed linear subspaces such that for each $x\in B$ [equivalently, $x\in \Cl(B)$] there is a (necessarily unique) $S_x^\circ\in 2^{S^\circ}$ with $\PP_x$ the projection onto $\oplus_{s\in S_x^\circ}G_s$. 
\item\label{atomic:b} The spectral space $(S,\Sigma,\mu)$ is purely atomic. 
\end{enumerate}
\end{proposition}
\begin{proof}
It is clear that \ref{atomic:b} implies \ref{atomic:a}. The converse follows by the uniqueness of the decomposition of abelian von Neumann algebras \cite[A.70, A.85]{dixmier-c-star} (of course formally one has to first complexify the spaces $G_s$, $s\in S^\circ$, see Remark~\ref{rmk:complex-to-real}). 
\end{proof}
Second, the atoms  of the spectrum can be described in relatively plain terms in terms of the following objects. 

\begin{definition}
For $x\in \Cl(B)$ let $\L2(\PP\vert_x)^\circ$ be the orthogonal complement of the closure of the linear span of $\cup_{y\in \Cl(B)_x\backslash \{x\}}\L2(\PP\vert_y)$ in $\L2(\PP\vert_x)$. Set $\Cl(B)^\circ:=\{x\in \Cl(B):\L2(\PP\vert_x)^\circ\ne \{0\}\}$ and $B^\circ:=B\cap \Cl(B)^\circ$. 
\end{definition}
The notation $\L2(\PP\vert_x)^\circ$ fails to reference $B$ on which it depends, so one must be careful to keep the presence of $B$ in mind. 

\begin{remark}
If $x\in B$, due to the continuity of $\land$ on $\Cl(B)$, $\L2(\PP\vert_x)^\circ$ is just the orthogonal complement of the closure of the linear span of $\cup_{y\in B_x\backslash \{x\}}\L2(\PP\vert_y)$ in $\L2(\PP\vert_x)$. 
\end{remark}

Before we can make the connection of the preceding to the spectrum explicit let us prepare some auxiliary facts.

 \begin{lemma}\label{propo:atom}
 Let $H\ne \{0\}$ be a subset of $\L2(\PP)$ closed under addition  and let $x\in \Lattice$. The following statements are equivalent. 
 \begin{enumerate}[(A)]
 \item\label{I1} $\PP_x(h)\in \{0,h\}$ for all $h\in H$. 
 \item\label{I2} Either $H\subset \L2(\PP\vert_x)$ or $H\subset \L2(\PP\vert_x)^\perp$.
 \end{enumerate}
 \end{lemma}
 \begin{proof}
 Only the implication \ref{I1}$\Rightarrow$\ref{I2} is not clear. Assume \ref{I1}. Both  $H\subset \L2(\PP\vert_x)$ and $H\subset \L2(\PP\vert_x)^\perp$ cannot prevail because $H\ne\{0\}$; suppose per absurdum that neither does. Then: for some $h\in H$, one must have $h=h_x+h_x^\perp$ with $h_x\in \L2(\PP\vert_x)$, $0\ne h_x^\perp\in \L2(\PP\vert_x)^\perp$, and necessarily $h_x= 0$ (for if not, then $\PP_x(h)=h_x\notin \{0,h\}$, contradicting the hypothesis); similarly, for some $g\in H$, one must have $g=g_x+g_x^\perp$ with $0\ne g_x\in\L2(\PP\vert_x)$,  $g_x^\perp\in\L2(\PP\vert_x)^\perp$, and necessarily $g_x^\perp= 0$ (for if not, then $\PP_x(g)=g_x\notin \{0,g\}$, contradicting the hypothesis); in turn $\PP_x(h+g)=h_x+g_x=g_x\notin \{0,h_x^\perp+g_x\}=\{0,h+g\}$ contradicting $h+g\in H$ and the hypothesis, yet again.
 \end{proof}

\begin{proposition}\label{lemma:atom}
Let $H\ne \{0\}$ be a subset of $\L2(\PP)$ closed under addition and satisfying $\PP_x(h)\in \{0,h\}$ for all $h\in H$ and $x\in B$ [equivalently, $x\in \Cl(B)$]. Then $H\subset H(P)$, where $P:=\left(\mu\text{-ess-}\cap_{x\in B,H\subset \L2(\PP\vert_x)}S_x\right)\cap\left(\mu\text{-ess-}\cap_{x\in B,H\subset \L2(\PP\vert_x)^\perp}(S\backslash S_x)\right)$ is an atom of $\mu$ (it possible that the second intersection is empty; then it is treated as $S$).
\end{proposition}
\begin{proof}
First, $H(P)=(\cap_{x\in B,H\subset \L2(\PP\vert_x)}\L2(\PP\vert_x))\cap(\cap_{x\in B,H\subset \L2(\PP\vert_x)^\perp}\L2(\PP\vert_x)^\perp)\supset H$ ($\cap_\emptyset=\L2(\PP)$). Second, it follows from Lemma~\ref{propo:atom} that $P$ is a.e.-$\mu$ contained either in $S_x$ or in $S\backslash S_x$  for each $x\in B$, entailing that $P$ is an atom of $\mu$ (the assumption $H\ne\{0\}$ precludes $\mu(P)=0$). 
\end{proof}

\begin{proposition}\label{proposition:atom}
The following statements are equivalent for a closed linear subspace $H$ of $\L2(\PP)$. 
\begin{enumerate}[(A)]
\item\label{A1} $H=H(P)$ for  an (necessarily up to a set of $\mu$-measure zero unique) atom $P$ of  $\mu$.
\item\label{A2} $H$ is a maximal 
closed linear subspace $G$ of $\L2(\PP)$ satisfying $\PP_x(g)\in \{0,g\}$ for all $g\in G$ and $x\in B$ [equivalently, $x\in \Cl(B)$].
\end{enumerate}
When so, then in addition $\PP_x(H^\perp)\subset H^\perp$ for all $x\in \Cl(B)$.
\end{proposition}
 \begin{remark}
Property \ref{A2} holds for $B$ iff it holds for the noise completion of $B$.
 \end{remark}

\begin{proof}
The case $H=\{0\}$ is trivial (both statements fail, because $\{0\}$ is strictly contained in the one-dimensional space of constants) 
so assume $H\ne \{0\}$. 

Suppose  first that $H=H(P)$ for an atom $P$ of $\mu$. If $h\in H$, then $\PP_x(h)$ is $h$ or $0$ according as to whether $P\subset S_x$ a.e.-$\mu$ or $P\subset S\backslash S_x$ a.e.-$\mu$, these two cases being exhaustive because $P$ is an atom of $\mu$. If $G$ is another closed linear subspace of $\L2(\PP)$ satisfying $\PP_x(g)\in \{0,g\}$ for all $g\in G$, and that contains $H$, then by Proposition~\ref{lemma:atom}, $G\subset H(Q)$ for an atom $Q$ of $\mu$. Thus $H(P)\subset H(Q)$, which means that $P\subset Q$ a.e.-$\mu$; but $Q$ is an atom of $\mu$ and $\mu(P)>0$, so we must have $P=Q$ a.e.-$\mu$, i.e. $H=H(P)=H(Q)=G$. Therefore \ref{A1} implies \ref{A2}. 

Conversely, suppose that \ref{A2} holds true. Then by Proposition~\ref{lemma:atom} $H\subset H(P)$ for an atom $P$ of $\mu$; by the implication that we have just proved necessarily $H=H(P)$.

For the final statement just note that $H^\perp=H(S\backslash P)$, where $P$ is from \ref{A1}.
\end{proof}

We are now ready to give the following result which captures the salient properties of $\L2(\PP\vert_x)^\circ$, $x\in\Cl(B)^\circ$, and their connection to the discrete part of the spectrum of $B$.

\begin{theorem}\label{prop:dicrete} 
We have the following assertions. 
\begin{enumerate}[(i)]
\item\label{prop:dicrete:0} $\L2(\PP\vert_{0_\PP})^\circ$ is the one-dimensional space of constants; $0_\PP\in B^\circ$.
\item\label{prop:dicrete:i}  For $\{x,y\}\subset \Cl(B)$ and $g\in \L2(\PP\vert_x)^\circ$ we have $\PP_y(g)=\PP_{x\land y}(g)=g\mathbbm{1}_{\Lattice_{x\land y}}(x)=g\mathbbm{1}_{\Lattice_y}(x)$.  
\item\label{prop:dicrete:0+} If $x\in \Cl(B)^\circ$, then $\L2(\PP\vert_x)^\circ=H(P)$ for some (up to a.e.-$\mu$ equality unique) atom $P$ of $\mu$.
\item\label{prop:dicrete:ix}   Conversely, if $P$ is an atom of $\mu$, then $x:=\land({}_{\sigma(H(P))}\Cl(B))\in \Cl(B)^\circ$ and $H(P)=\L2(\PP\vert_x)^\circ$; if also $H(P)=\L2(\PP\vert_y)^\circ$ for some $y\in \Cl(B)$, then $y=x$.
\item\label{prop:dicrete:ii}  The closed linear subspaces $\L2(\PP\vert_x)^\circ$, $x\in \Cl(B)$, are orthogonal.  
\item\label{prop:dicrete:iii}   $\Cl(B)^\circ$ is countable. 
\item\label{prop:dicrete:iii-} If $\{x,y\}\subset B^\circ$, $x$ independent of $y$, then $\L2(\PP\vert_{x\lor y})^\circ$ is the  closure of the linear span of $\{fg:(f,g)\in \L2(\PP\vert_x)^\circ\times \L2(\PP\vert_y)^\circ\}$, in other words $\L2(\PP\vert_{x\lor y})^\circ=\L2(\PP\vert_x)^\circ\otimes \L2(\PP\vert_y)^\circ$ up to the natural unitary equivalence. The same is true if instead of ``$\{x,y\}\subset B^\circ$, $x$ independent of $y$'' we assume that $\{x,y\}\subset \Cl(B)^\circ$ and $x\subset u$, $y\subset v$ for  some $\{u,v\}\subset \overline{B}$ with $u$ independent of $v$.
\item\label{prop:dicrete:iv}  $B^\circ$ is an ideal of $B$ (maybe improper /i.e. equal to $B$/, maybe trivial /i.e. containing only one element/). In particular, for $x\in B^\circ$, $B_x$ is necessarily countable.  
\item\label{prop:dicrete:v}  For $x\in \Cl(B)$, $(\oplus_{y\in \Cl(B)^\circ}\L2(\PP\vert_y)^\circ)\cap \L2(\PP\vert_x)=\oplus_{y\in \Cl(B)^\circ_x}\L2(\PP\vert_y)^\circ$.
In particular, if $\oplus_{y\in \Cl(B)^\circ_a}\L2(\PP\vert_y)^\circ=\L2(\PP\vert_a)$ for some $a\in \Lattice$, then $\oplus_{y\in \Cl(B)^\circ_x}\L2(\PP\vert_y)^\circ=\L2(\PP\vert_x)$ for all $x\in \Cl(B)_a$.
\end{enumerate}
\end{theorem}
\begin{proof}
\ref{prop:dicrete:0} is plain. 

\ref{prop:dicrete:i}. The equality $\PP_y(g)=\PP_{x\land y}(g)$ is because $\PP_x$ and $\PP_y$ commute and the remaining equalities follow immediately. 

 \ref{prop:dicrete:0+}. By the previous item we see that $\PP_y(g)\in \{0,g\}$ for all $g\in \L2(\PP\vert_x)^\circ$ and $y\in \Cl(B)$. Suppose a closed linear subspace $G\supset  \L2(\PP\vert_x)^\circ$ is such that $\PP_y(g)\in \{0,g\}$ for all $g\in G$ and $y\in \Cl(B)$. By Lemma~\ref{propo:atom} it follows that, for each $y\in \Cl(B)_x$, either $G\subset \L2(\PP\vert_y)$ or  $G\subset \L2(\PP\vert_y)^\perp$, and since  $G\supset  \L2(\PP\vert_x)^\circ$ this occurs according as $y=x$ or $y\subsetneq x$. We conclude that $G\subset \L2(\PP\vert_x)^\circ$, so that in fact $G=\L2(\PP\vert_x)^\circ$. By Proposition~\ref{proposition:atom} the claim follows. 
 
 \ref{prop:dicrete:ix}. We  use again the characterization of Proposition~\ref{proposition:atom}.  Thus $H:=H(P)$ is a maximal closed linear subspace $G$ of $\L2(\PP)$ satisfying $\PP_y(g)\in \{0,g\}$ for all $g\in G$, $y\in \Cl(B)$. By Lemma~\ref{propo:atom}, for all $y\in \Cl(B)$, $H\subset \L2(\PP\vert_y)$ or $H\subset \L2(\PP\vert_y)^\perp$. Then: $H\subset \L2(\PP\vert_x)$, clearly; also, if $y\in \Cl(B)_x\backslash \{x\}$, then $H\subset \L2(\PP\vert_y)^\perp$, since otherwise we would have $H\subset \L2(\PP\vert_y)$, therefore $x\subset y$, which contradicts $y\subsetneq x$. We see that $H\subset \L2(\PP\vert_x)^\circ$, in particular $x\in \Cl(B)^\circ$ (that $x\in \Cl(B)$ in the first place follows from Proposition~\ref{proposition:ctb-vs-arbitrary-joins-meets}\ref{proposition:ctb-vs-arbitrary-joins-meets:ii}). Because of the maximality of $H$, using \ref{prop:dicrete:0+} and Proposition~\ref{proposition:atom} again, we get $H=\L2(\PP\vert_x)^\circ$.  If also $H(P)=\L2(\PP\vert_y)^\circ$ for some $y\in \Cl(B)$, then clearly $x\subset y$ and $x\subsetneq y$ is impossible.
 
  \ref{prop:dicrete:ii}. Subspaces of $\L2(\PP)$ corresponding to disjoint measurable subsets of $S$ are orthogonal. Apply \ref{prop:dicrete:0+}. 

\ref{prop:dicrete:iii}. Follows at once from \ref{prop:dicrete:ii} because $\L2(\PP)$ is separable. 

\ref{prop:dicrete:iii-}. Assume $\{x,y\}\subset B^\circ$ in the first instance.  If $B\ni z\subsetneq x\lor y$, then  $z\land x\subsetneq x$ or $z\land y\subsetneq y$; therefore $\L2(\PP\vert_{x\lor y})^\circ\supset \L2(\PP\vert_x)^\circ\otimes \L2(\PP\vert_y)^\circ$ ($\because$ of the tensorisation of independent conditioning). On the other hand, for $u\in B_x\backslash \{x\}$, $\L2(\PP\vert_{x\lor y})\ominus \L2(\PP\vert_{u\lor y})=(\L2(\PP\vert_x)\otimes \L2(\PP\vert_y))\ominus (\L2(\PP\vert_u)\otimes \L2(\PP\vert_y))\subset (\L2(\PP\vert_x)\ominus \L2(\PP\vert_u))\otimes \L2(\PP\vert_y)$; therefore $\L2(\PP\vert_{x\lor y})^\circ\subset \cap_{u\in B_x\backslash \{x\}} ((\L2(\PP\vert_x)\ominus \L2(\PP\vert_u))\otimes \L2(\PP\vert_y))=(\cap_{u\in B_x\backslash \{x\}} (\L2(\PP\vert_x)\ominus \L2(\PP\vert_u)))\otimes \L2(\PP\vert_y) =\L2(\PP\vert_x)^\circ\otimes \L2(\PP\vert_y)$. Here, the interchange of intersection and tensor product may be argued as follows. The projections onto the subspaces $\L2(\PP\vert_x)\ominus \L2(\PP\vert_u)$, $u\in B_x\backslash \{x\}$, commute, indeed they belong to $\AA$. The latter implies easily that one can interchange the tensor product and finite intersections in the preceding. Besides, one can pass to a (because of the preceding $\downarrow$) countable intersection by taking  in lieu of $B_x\backslash \{x\}$ a countable dense subset thereof. Then it remains to apply \cite[Lemma~2.10]{tsirelson-arxiv-5}.  Likewise we get $\L2(\PP\vert_{x\lor y})^\circ\subset \L2(\PP\vert_x)\otimes \L2(\PP\vert_y)^\circ$. Hence $\L2(\PP\vert_{x\lor y})^\circ\subset \L2(\PP\vert_x)^\circ\otimes \L2(\PP\vert_y)^\circ$. The second statement of this item follows by exactly the same reasoning appealing to Proposition~\ref{proposition:closure-separate-join}. 

\ref{prop:dicrete:iv}. First, by  \ref{prop:dicrete:0} $0_\PP\in B^\circ$, so $B^\circ$ is non-empty. Second, if $x\in B^\circ$,  then $B_x\subset B^\circ$ (such $B_x$ is therefore necessarily countable because of \ref{prop:dicrete:iii}): for, if $\L2(\PP\vert_y)^\circ=\{0\}$ for some $y\in B_x$, then $y\ne x$, and $\L2(\PP\vert_x)$ is equal to the closure of the linear span of $\{fg:(f,g)\in \L2(\PP\vert_y)\times \L2(\PP\vert_{y'\land x})\}$, hence of the linear span of $\cup_{z\in B_y\backslash \{y\}}\L2(\PP\vert_{z\lor (y'\land x)})$, which means that $\L2(\PP\vert_x)^\circ=\{0\}$, a contradiction. Third, $B^\circ$ is closed under $\lor$: if $\{x,y\}\subset B^\circ$, then $x\lor y=x\lor (y\land x')\in B^\circ$ by the preceding point and \ref{prop:dicrete:iii-}. Since $\mu$ can have $\{\emptyset_S\}$ as its only atom (up to a.e.-$\mu$ equality)  we see by \ref{prop:dicrete:0+}  that $B^\circ=\{0_\PP\}$ can happen.  If $B$ is the classical noise Boolean  algebra associated to a finite independency, then $B^\circ= B$ (Example~\ref{example:spectral-finite}). 

\ref{prop:dicrete:v}. The inclusion $\supset$ is clear. The reverse inclusion follows from \ref{prop:dicrete:i}. 
\end{proof}
\begin{question}
 Is $\Cl(B)^\circ$ closed under $\lor$ and does Item~\ref{prop:dicrete:iii-} hold true  with $\Cl(B)$ in lieu of $B$ (without having to assume that $x$ and $y$ are ``separated'' by $\overline{B}$)? As was mentioned in the Introduction, it is not even known whether in general $\Cl(B)$ is a lattice, i.e. whether it is closed for $\lor$. It seems unlikely that (one could prove that) $\Cl(B)^\circ$ is closed under $\lor$ in absence of (being able to prove) this being true first for $\Cl(B)$.
\end{question}
\begin{question}
Under what (natural, easy to check) conditions is $\Cl(B)^\circ=B^\circ$? (It is certainly true if $B$ is classical and (noise) complete.) Or rather, when is $\Cl(B)^\circ=\overline{B}^\circ$?
\end{question}
\begin{question}
With reference to Item~\ref{prop:dicrete:v}: is it always the case that $\oplus_{y\in {\Cl(B)^\circ}}\L2(\PP\vert_y)^\circ=\L2(\PP\vert_{\lor \Cl(B)^\circ})$? 
\end{question}
At least under certain conditions we can ``wiggle out'' what may legitimately be called the discrete part of $B$.
\begin{proposition}\label{proposition:discrete-part-of-noise}
$\lor B^\circ$ distributes over $B$, hence $B_{\mathrm{dis}}:=B\vert_{\lor B^\circ}$ is a  noise Boolean algebra under $\PP\vert_{\lor B^\circ}$. If further  $\Cl(B)^\circ=B^\circ$ and $\oplus_{y\in B^\circ}\L2(\PP\vert_y)^\circ=\L2(\PP\vert_{\lor B^\circ})$, then  $\lor B^\circ$ is the largest element $x\in \Lattice$  such that $\PP_x\in \AA$ and such that $\{y\land x:y\in B\}$ is a noise Boolean algebra under $\PP\vert_x$ with a discrete spectral measure.
\end{proposition}
\begin{remark}
Because $B^\circ$ is a countable ideal, $\oplus_{y\in {B^\circ}}\L2(\PP\vert_y)^\circ=\L2(\PP\vert_{\lor B^\circ})$ is equivalent to $\oplus_{y\in {B^\circ_x}}\L2(\PP\vert_y)^\circ=\L2(\PP\vert_{x})$ being true for all $x\in B^\circ$ (by increasing martingale convergence, say).
\end{remark}
\begin{proof}
Notice that $\lor B^\circ\in \Cl(B)$ ($\because$ $B^\circ$ is upwards directed);  then recall \ref{generalitites:subalgebra}. The final statement is an immediate consequence of the first and the proof is complete.
\end{proof}

Next, the ``happy discrete'' situation when     $\oplus_{y\in {B^\circ}}\L2(\PP\vert_y)^\circ=\L2(\PP)$  can be succesfully delineated. The somewhat less happy situation when  $\oplus_{y\in {\Cl(B)^\circ}}\L2(\PP\vert_y)^\circ=\L2(\PP)$ is characterized in parallel -- it is indeed just the purely atomic case. 
\begin{theorem}\label{propsition:discrete-spectrum}
We have the following assertions.
\begin{enumerate}[(i)]
\item \label{discrete-spectrum:a}
The next three statements are equivalent.
\begin{enumerate}[(A)]
\item\label{discrete-spectrum:i} $B$ admits a spectral resolution in which $S\in 2^B$ is countable, $\Sigma=2^S$, $\mu=c_S$, and $S_x=B_x\cap S$ for $x\in B$.
\item\label{discrete-spectrum:ii}  $\L2(\PP)=\oplus_{x\in B^\circ}\L2(\PP\vert_x)^\circ$ (orthogonal sum decomposition). 
\item\label{discrete-spectrum:iii} $B^\circ=\Cl(B)^\circ$ and $\mu$ is purely atomic.
\end{enumerate}
If these equivalent conditions hold true, then: in \ref{discrete-spectrum:i} necessarily $S=B^\circ=\Cl(B)^\circ$, and $B$ admits the spectral resolution in which in addition to the properties listed in \ref{discrete-spectrum:i} $H_{x}=\L2(\PP\vert_x)^\circ$, $x\in B^\circ$, under the canonical unitary isomorphism $\Psi:\L2(\PP)\to \oplus_{y\in B^\circ}\L2(\PP\vert_y)^\circ$;  every element of $\Cl(B)$ is the $\uparrow$ limit of a sequence in $B^\circ$, in particular $B^\circ$ is dense in $B$ and $\Cl(B)$ is closed for arbitrary joins and meets; for $x\in \Cl(B)$, $\L2(\PP\vert_x)=\oplus_{y\in B^\circ_x}\L2(\PP\vert_y)^\circ$, i.e. $S_x=B^\circ_x$; the first chaos of $B$ is $\oplus_{x\in A}\L2(\PP\vert_x)^\circ$, where $A:=\{a\in B^\circ:\text{ either }a\subset y\text{ or }a\subset y'\text{ for all }y\in B\text{ (i.e. $a$ is an atom of $B$)}\}$.
\item\label{discrete-spectrum:b} If for  an orthogonal sum decomposition $\oplus_{I\in \mathcal{I}}G_I$ of  $\L2(\PP)$ into non-zero subspaces and for an injective countable family $(y_I)_{I\in \mathcal{I}}$ in $B$ one has for all $I\in \II$ that $\PP[f\vert x]=f\mathbbm{1}_{B_x}(y_I)$ for all $x\in B_{y_I}$ and $f\in G_I$, then $G_I=\L2(\PP\vert_{y_I})^\circ$ for each $I\in \mathcal{I}$, $B^\circ=\{y_I:I\in\II\}$ and the equivalent conditions of the preceding item  obtain. 
\end{enumerate}
Furthermore, in the preceding $\Cl(B)$ may replace $B$ throughout, except that one should omit ``every element of $\Cl(B)$ is the $\uparrow$ limit of a sequence in $B^\circ$, in particular $B^\circ$ is dense in $B$'' and ``i.e. $a$ is an atom of $B$''. 
\end{theorem}
\begin{remark}
The condition $\PP[f\vert x]=f\mathbbm{1}_{B_x}(y_I)$ for all $x\in B_{y_I}$ of Item \ref{discrete-spectrum:b} simply means that $f$ is $y_I$-measurable and orthogonal to all $y\in B_{y_I}\backslash \{y_I\}$ [and therefore to all $y\in B\backslash(_{y_I}B)$]; when $f$ is an equiprobable random sign ``orthogonal'' is equivalent to ``independent'' in the preceding. Thus $y_I$ is the smallest member of $B$ w.r.t. which the members of $G_I$ are measurable, but it is more than that: the members of $G_I$ are even orthogonal to each smaller member of $B$. Again this holds true even if $\Cl(B)$ replaces $B$, except that then one should omit ``and therefore to all $y\in B\backslash(_{y_I}B)$''.
\end{remark}

\begin{proof}
We write it for $B$, leaving $\Cl(B)$ to the reader: the argument is verbatim the same. 

Let us prove \ref{discrete-spectrum:a}.  Suppose \ref{discrete-spectrum:i} holds true. 

We have the unitary isomorphism $\Psi:\L2(\PP)\to \oplus_{s\in S}H_s$, so we may just as well (retaining the same spectral space) assume that $\Psi$ is the canonical isomorphism between $\L2(\PP)$ and its orthogonal sum decomposition $\oplus_{s\in S}H_s$ into closed linear subspaces. Let $x\in B$. Since $S_x= B_x\cap S$ we have that $\L2(\PP\vert_x)=\oplus_{y\in S\cap B_x}H_y$. If $x\notin S$, it means that $\L2(\PP\vert_x)$ is equal to the closure of the linear span of the $\L2(\PP\vert_y)$, $y\in S\cap B_x\subset B_x\backslash \{x\}$, and therefore $\L2(\PP\vert_x)^\circ=\{0\}$ and $x\notin B^\circ$. If $x\in S$ we see that $\L2(\PP\vert_x)^\circ=H_x\ne \{0\}$ and $x\in B^\circ$. Thus $S=B^\circ$ and   $\L2(\PP\vert_x)=\oplus_{y\in B^\circ\cap B_x}\L2(\PP\vert_y)^\circ$ for $x\in B$; taking $x=1_\PP$ shows that \ref{discrete-spectrum:ii} holds true, that  $B^\circ=\Cl(B)^\circ$ ($\because$ Theorem~\ref{prop:dicrete}\ref{prop:dicrete:ii}) and that whatever the choice of the spectral space, $\mu$ must be purely atomic (use Proposition~\ref{proposition:purely-atomic}). 

Since $B^\circ$ is a countable ideal, in order to see that every element of $\Cl(B)$ is the $\uparrow$ limit of a sequence in $B^\circ$ it will suffice to show that $\Cl(B)=\{\lor A:A\in 2^{B^\circ}\}$; by Proposition~\ref{proposition:ctb-vs-arbitrary-joins-meets} it will also mean that $\Cl(B)$ is closed under arbitrary joins and meets. To this end let $(x_n)_{n\in \mathbb{N}}$  be a $\downarrow$ sequence in $B$. Then $S_{x_n}=B^\circ_{x_n}\downarrow  \{y\in B^\circ:y\subset x_m\text{ for all }m\in \mathbb{N}\}=B^\circ_{\land_{m\in\mathbb{N}}x_m}=S_{\land_{m\in\mathbb{N}}x_m}$ as $n\to\infty$, which means that $\L2(\PP\vert_{\land_{m\in\mathbb{N}}x_m})=\oplus_{y\in B^\circ_{\land_{m\in\mathbb{N}}x_m}}\L2(\PP\vert_y)^\circ$ and so $\land_{n\in \mathbb{N}}x_n=\lor(B^\circ_{\land_{m\in\mathbb{N}}x_m})$. Thus limits of $\downarrow$ sequences of $B$ fall into $\{\lor A:A\in 2^{B^\circ}\}$, which is plainly closed under $\uparrow$ sequences, therefore equal to the closure of $B$. We have  seen in passing that for $x\in \Cl(B)$, not just for $x\in B$, $\L2(\PP\vert_x)=\oplus_{y\in B^\circ_x}\L2(\PP\vert_y)^\circ$.

Also, for $f=\oplus_{x\in B^\circ}f_x\in \L2(\PP)$, the decomposition being according to $\L2(\PP)=\oplus_{x\in B^\circ}\L2(\PP\vert_x)^\circ$, and for $y\in B$, $\PP[f\vert y]=\oplus_{x\in B^\circ_y}f_x$, so that $f$ is in the first chaos iff for those $x\in B^\circ$ for which $f_x\ne 0$, either $x\subset y$ or  $x\subset y'$ for all $y\in B$.  


 Conversely, \ref{discrete-spectrum:ii} trivially implies \ref{discrete-spectrum:i} due to Theorem~\ref{prop:dicrete}\ref{prop:dicrete:v}. The equivalence of  \ref{discrete-spectrum:ii}  and \ref{discrete-spectrum:iii} is now immediate from Theorem~\ref{prop:dicrete}\ref{prop:dicrete:ix} and  Proposition~\ref{proposition:purely-atomic}.

Let us prove \ref{discrete-spectrum:b}. Let $I\in \mathcal{I}$. By assumption $\PP[f\vert y_I]=f$ for all $f\in G_I$, so $G_I\subset \L2(\PP\vert_{y_I})$. On the other hand, for $x\in B_{y_I}\backslash \{y_I\}$ and $f\in G_I$, $\PP[f\vert x]=0$, also by assumption, so $G_I\subset \L2(\PP\vert_{y_I})^\circ$. Conversely, if $f\in \L2(\PP\vert_{y_I})^\circ$ and $J\in \mathcal{I}\backslash \{I\}$, then, again by assumption, for all $g\in G_J$ we have $\PP[g\vert y_I]=\PP[g\vert y_I\land y_J]=g\mathbbm{1}_{B_{y_I\land y_J}}(y_J)=g\mathbbm{1}_{B_{y_I}}(y_J)$, so (noticing that by injectivity $y_J\ne y_I$) $\PP[gf]=\PP[g\mathbbm{1}_{B_{y_I}}(y_J)f]=0$, which means that $f$ is orthogonal to $G_J$. In view of $\L2(\PP)=\oplus_{K\in \mathcal{I}}G_K$ it follows that $\L2(\PP\vert_{y_I})^\circ\subset G_I$.
\end{proof}

To better appreciate the above, some examples.

\begin{example}[Spectrum of classical discrete noise Boolean algebra]\label{example:classical-spectrum}
If $B$ is the classical noise Boolean algebra associated to an independency $A\subset \Lattice$ satisfying $\lor A=1_\PP$, $0_\PP\notin A$, see Example~\ref{example:classical}, then --- for instance because of increasing martingale convergence ---  one can take $(S,\Sigma,\mu)=(B_\text{fin},2^{B_\text{fin}},c_{B_\text{fin}})$ and $H_x=\L2(\PP\vert_x)^\circ$, $x\in B_{\mathrm{fin}}$, under the canonical isomorphism $\Psi:\L2(\PP)\to \oplus_{x\in B_{\text{fin}}}\L2(\PP\vert_x)^\circ$; $S_x=B_x\cap B_{\text{fin}}$ for $x\in B$. Here $B_\text{fin}:=\{\lor C:C\in (2^A)_{\mathrm{fin}}\}$. Thus $H^{(n)}=\oplus_{J\in {A\choose n}}\otimes_{a\in J}\L2(\PP\vert_a)_0$, $n\in \mathbb{N}_0$.

\noindent For $x\in B$, $\Sigma_x$ is generated by $\pr_a:=(B_{\mathrm{fin}}\ni y\mapsto a\land y\in B_{\mathrm{fin}})$, $a\in A_x$ (or, what amounts to the same thing, $a\in B_x$).  $\Psi^{-1}(e(g))=\prod_{a\in A}(1+g_a)$ (the partial products of the possibly infinite product are convergent in $\L2(\PP)$), where $e(g)(x):=\prod_{a\in A_x}g_a$, $x\in B_{\text{fin}}$, is the exponential  vector associated to a $g=(g_a)_{a\in A}\in \prod_{a\in A}\L2(\PP\vert_a)^\circ$ satisfying $\sum_{a\in A}\Vert g_a\Vert^2<\infty$.  

\noindent \textbf{($\dagger$)}  A probability equivalent to $\mu$ under which $\Sigma_x$ and $\Sigma_{x'}$ are independent for each $x\in B$ is given by $\mu_f$, where $f=\frac{\Psi^{-1}(e(g))}{\Vert \Psi^{-1}(e(g))\Vert}=\frac{\prod_{a\in A}(1+g_a)}{\sqrt{\prod_{a\in A}(1+\PP[g_a^2])}}$ for a $g=(g_a)_{a\in A}\in \prod_{a\in A}\L2(\PP\vert_a)_0$ satisfying $\sum_{a\in A}\Vert g_a\Vert^2<\infty$. $\mu_f$ can equivalently be described as the law of the random element of $B$ that includes each atom $a\in A$ with probability $\frac{\PP[g_a^2]}{1+\PP[g_a^2]}$ independently of the others, i.e. $\mu_f=(\prod_{a\in A}(1+\PP[g_a^2]))^{-1}\sum_{x\in B_{\mathrm{fin}}}(\prod_{a\in A_x}\PP[g_a^2])\delta_x$.  \textbf{($\dagger$)} 
\end{example}
The preceding generalizes Example~\ref{example:spectral-finite}, which dealt with the case when $A$ is finite.

\begin{example}[Spectrum of simplest nonclassical noise Boolean algebra]\label{example:spectral-space-for-nonclassical}
Consider the simplest nonclassical noise Boolean algebra of Example~\ref{example:simplest-nonclassical}. Let $\mathcal{I}$ be the collection of all finite subsets of $\mathbb{N}$ and define $G_I$ and $y_I$ for $I\in \mathcal{I}$, as follows: $G_I:= \mathrm{lin}(\prod_I\xi)$; $y_\emptyset:= N_{\emptyset}$; $y_{\{i\}}:= N_{\mathbb{N}_{\geq i}}$; 
and then 
with the ``obvious'' generalization of the special cases $y_{\{2,4\}}=N_{\{2,3\}}$ (because $\xi_2\xi_4=(\xi_2\xi_3)(\xi_3\xi_4)$), $y_{\{2,3,5,7\}}:=N_{\{2,5,6\}}$ (because $\xi_2\xi_3\xi_5\xi_7=(\xi_2\xi_3)(\xi_5\xi_6)(\xi_6\xi_7)$), $y_{\{1,3,5\}}:=N_{\{1,2\}\cup \mathbb{N}_{\geq 5}}$ (because $\xi_1\xi_3\xi_5=(\xi_1\xi_2)(\xi_2\xi_3)\xi_5$), $y_{\{1,2,4,7,10\}}:=N_{\{1,4,5,6\}\cup\mathbb{N}_{\geq 10}}$ (because $\xi_1\xi_2\xi_4\xi_7\xi_{10}=(\xi_1\xi_2)(\xi_4\xi_5)(\xi_5\xi_6)(\xi_6\xi_7)\xi_{10}$)  [``one writes it as  a product of pairs for as long as one can in  a minimal way''; sets $I\in \II$ of even size correspond to $N_b$, $b$ finite -- sets of odd size to $N_b$, $b$ cofinite]. The map $(\II\ni I\mapsto y_I)$ is bijective onto $N(\BB)=B$. By Fourier-Walsh expansion, $\L2(\PP)=\oplus_{I\in \mathcal{I}}G_I$. Besides, one sees that for $I\in \mathcal{I}$ and $x\in B_{y_I}$, $\PP[\prod_I\xi\vert x]=(\prod_I\xi)\mathbbm{1}_{B_{x}}(y_I)$. We may take $(S,\Sigma,\mu)=(B,2^B,c_B)$ and $S_x=B_x$ for $x\in B$. For $I\in  \II$, $\L2(\PP\vert_{y_I})^\circ=G_I$, of course, and $\L2(\PP)=\oplus_{x\in B}\L2(\PP\vert_x)^\circ$. The first chaos of $B$ is the closure of the linear span of the $\xi_i\xi_{i+1}$, $i\in \mathbb{N}$.

\noindent Just like in the previous example, for $x\in B$, $\Sigma_x$ is generated by $\pr_a:=(B\ni y\mapsto a\land y\in B)$, $a\in B_x$.  \textbf{($\dagger$)}  However, in this case there is no probability on $(B,2^B)$ under which $\Sigma_x$ and $\Sigma_{x'}$ would be independent for all $x\in B$ and such that it would charge both $0_\PP$ as well as some $N_b$, $b$ co-finite subset of $\mathbb{N}$ (let alone all the points of $B$). It would indeed mean the existence of a random (under some probability $\mathbb{Q}$)  finite or co-finite subset $\Gamma$ of $\mathbb{N}$ for which $\{i\in \Gamma\}$, $i\in \mathbb{N}$, were  an independency, and for which $\QQ(\Gamma=A)>0$ for $A=\emptyset$ and $A=b$. But it is impossible since by the Borel-Cantelli lemmas $\QQ(\Gamma\text{ is infinite})=\QQ(\limsup_{i\to\infty}\{i\in \Gamma\})$ would be $0$ or $1$ according as to whether $\sum_{i\in \mathbb{N}}\QQ(i\in \Gamma)$ is $<\infty$ or $=\infty$. Still, we have such a probability if we ask for it to charge only  and every point of $ \{K<\infty\}=\{N_I:I\in\mathcal{I}\}$; namely, we can take the law of the random element of $B$ that includes, for each $i\in \mathbb{N}$, the product $\xi_i\xi_{i+1}$ with some positive probability $p_i$ independently of the others, $\sum_{i\in \mathbb{N}}p_i<\infty$.
\textbf{($\dagger$)} 
\end{example}

\begin{example}[Spectrum of simplest black noise Boolean algebra]\label{example:spectral-space-for-black}
Return to the simplest black noise Boolean algebra of Example~\ref{example:voter-model}. For all $n\in \mathbb{N}_0$, for all  $F\subset T_2[n]$, we see that $\prod_{v\in F}\mathsf{X}_v\in \L2(\PP\vert_{\lor_{v\in F}N_v})^\circ$. Since these products are total in $\L2(\PP)$ it follows that in fact $\Cl(B)^\circ=B^\circ=B$ and $\L2(\PP\vert_{\lor_{v\in F}N_v})^\circ=\mathrm{lin}(\prod_{v\in F}\mathsf{X}_v)$ for  finite $F\subset T_2[n]$,  $n\in \mathbb{N}_0$. We may take $(S,\Sigma,\mu)=(B,2^B,c_B)$ and $S_x=B_x$ for $x\in B$; $\Cl(B)=\{\lor_{v\in V}N_v:V\in 2^{T_2}\}$.  

\noindent Still, for $x\in B$, $\Sigma_x$ is generated by $\pr_a:=(B\ni y\mapsto a\land y\in B)$, $a\in B_x$.  \textbf{($\dagger$)}  Now there is no probability on $(B,2^B)$ under which $\Sigma_x$ and $\Sigma_{x'}$ would be independent for all $x\in B$ and such that it would charge both $0_\PP$ as well as some other $x\in B$. By  the tree structure we can indeed find a sequence of vertices $(v_n)_{n\in \mathbb{N}}$ in $T_2$ with the following property: $N_{v_1}\subset x$ and for all $k\in \mathbb{N}$, $v_{k+1}$ is a descendant of the sibling of $v_k$.  Then, under this would-be probability, call it $\mathbb{Q}$, 
the events $E_n:=\{y\in B:y\land N_{v_n}\ne 0_\PP\}$, $n\in \mathbb{N}$, would be independent. Since $\QQ(\{x\})>0$ with positive $\QQ$-probability all of them must occur; therefore by the Borel-Cantelli lemmas, a.s.-$\QQ$ infinitely many, and certainly one of them occurs, but then $\QQ(\{0_\PP\})=0$, which is a contradiction. \textbf{($\dagger$)}  


\noindent Every element of $\Cl(B)$ is of the form $N_V:=\lor_{v\in V}N_v$ for some  $V\in 2^{T_2}$ satisfying the property that if $v\in V$, then none of  its descendants are in $V$ (non-redundancy). Finite $V$ correspond to elements of $B$. If $V$ is infinite, can $N_V$ have an independent complement $N_W$ for a non-redundant $W\in 2^{T_2}$? No, it cannot. For suppose it did. There is an infinite path in $T_2\backslash V$, starting at the root, such that every vertex of this path has a descendant from $V$ (and therefore infinitely many descendants in $V$). Plainly by independence of $N_V$ and $N_W$ this path contains no vertex of $W$ either. This means that $\mathsf{X}_r$, where $r$ is the root, is independent of $\sigma(\mathsf{X}_v:v\in T_2[n],v\in D(V\cup W))\uparrow N_V\lor N_W$ for all $n\in \mathbb{N}$, $D(V\cup W)$ being all the descendants of $V\cup W$. Therefore $N_V\lor N_W$ cannot be $1_\PP$, a contradiction. It means that $B=\overline{B}$, i.e. $B$ is noise complete.
\end{example}

%
\begin{example}\label{ex:bizzare}
Return to Example~\ref{example:simplest-nonclassical-tweaked}. We get that $\Cl(B)^\circ=B\cup \{\sigma(\xi_\infty,\xi_i\xi_{i+1}:i\in b):b\in (2^\mathbb{N})_{\mathrm{fin}}\}$. Specifically, for finite $I\subset \mathbb{N}$ of even size, $\L2(\PP\vert_{y_I})^\circ=\mathrm{lin}(\prod_I\xi)$, where $y_I$ is as in Example~\ref{example:spectral-space-for-nonclassical}; for  finite $I\subset \mathbb{N}$ of odd size, $\L2(\PP\vert_{z_I})^\circ=\mathrm{lin}(\xi_\infty\prod_I\xi,\prod_I\xi)$, where (again) $z_I$ is the $y_I$ of Example~\ref{example:spectral-space-for-nonclassical} except that it contains also $\xi_\infty$; for finite $I\subset \mathbb{N}$ of even size, $\L2(\PP\vert_{\sigma(\xi_\infty)\lor y_I})^\circ=\mathrm{lin}(\xi_\infty\prod_I\xi)$.  $\oplus_{x\in \Cl(B)^\circ}\L2(\PP\vert_x)^\circ=\L2(\PP)$ and the spectrum is discrete. The first chaos of $B$ is the closure of the linear span of the $\xi_i\xi_{i+1}$, $i\in \mathbb{N}$, and of $\xi_\infty$. $\Cl(B)^\circ\ne B^\circ$, which shows that in  Theorem~\ref{prop:dicrete}\ref{prop:dicrete:ix}  one has to work with $x\in \Cl(B)$ (the $x\in B$, or even in $\overline{B}$, do not suffice). $\Cl(B)^\circ$ is closed under $\lor$ and one checks quickly that in fact Theorem~\ref{prop:dicrete}\ref{prop:dicrete:iii-} holds true  with $\Cl(B)$ in lieu of $B$. Nevertheless, $\Cl(B)^\circ$ is not an ideal of $\Cl(B)$ (since $1_\PP\in \Cl(B)^\circ$ and $\Cl(B)\backslash \Cl(B)^\circ\ne\emptyset$), which means that Theorem~\ref{prop:dicrete}\ref{prop:dicrete:iv} in general fails with $\Cl(B)$ in lieu of $B$. 
\end{example}

Finally, we offer a necessary condition for a point to belong to the discrete part of the spectrum in terms of the maps $\underline{b}:S\to b$, $b\in \mathfrak{F}_B$, of \ref{generalities:chaoses}.

\begin{proposition}
Let $(b_n)_{n\in \mathbb{N}}$ be a $\uparrow$ sequence in $\mathfrak{F}_B$. Then $\{\cap_{n\in \mathbb{N}} \underline{b_n}\ne 0_\PP\}$ contains a.e.-$\mu$  the union of the atoms of $\mu$ excluding $S_{0_\PP}$.
\end{proposition}
\begin{proof}
For $\mu$-a.e. $s\in  \{\cap_{n\in \mathbb{N}} \underline{b_n}=0_\PP\}$,  $S_{\underline{b_n}(s)}$ is $\downarrow$ to $S_{0_\PP}$ a.e.-$\mu$ as $n\to\infty$. 
\end{proof}

\subsection{Study of a class of examples: reverse filtrations and their innovations}
A relatively pleasant  class of (nonclassical) noise Boolean algebras is got from reverse filtrations and their innovation processes; it amounts to abstracting  Examples~\ref{example:simplest-nonclassical} and~\ref{example:simplest-nonclassical-tweaked} (and subsumes also Example~\ref{example:classical}).
 
 \begin{example}\label{example:canonical-discrete}
 Sequences of $\sigma$-fields are given, $x=(x_n)_{n\in \mathbb{N}}$ and $y=(y_n)_{n\in \mathbb{N}}$, such that $x_n$ is the independent join of $x_{n+1}$ and  of $y_n\ne 0_\PP$ for all $n\in \mathbb{N}$, and such that $x_1=1_\PP$. Thus $(x_{-n})_{n\in -\mathbb{N}}$ is a filtration that is innovated by the sequence $(y_{-n})_{n\in \mathbb{N}}$. For each $n\in  \mathbb{N}_0$, $\{y_1,\ldots,y_n,x_{n+1}\}$ is a generating independency that gives rise to a classical noise Boolean algebra $B_n$ according to Example~\ref{example:classical}.  The union $B:=\cup_{n\in \mathbb{N}}B_n$ is also a noise Boolen algebra that may and will be said to have been attached to the reverse filtration $(x_n)_{n\in \mathbb{N}}$ and the innovation process $(y_n)_{n\in \mathbb{N}}$. 
 
 \noindent Put $y_0:=\land_{n\in \mathbb{N}}x_n$. The first chaos of $B$ is $\oplus_{n\in \mathbb{N}_0}\L2(\PP\vert_{y_n})_0$. The ``answer to the innovation problem'' for the given pair $(x,y)$ may be positive (when $\lor_{n\in \mathbb{N}_0}y_n=1_\PP$) or negative (otherwise), both possibilities can occur (even with the same sequence $(x_n)_{n\in \mathbb{N}}$: in Example~\ref{example:simplest-nonclassical} $x_n=\sigma(\xi_i:i\in \mathbb{N}_{\geq n})$ and $y_n=\sigma(\xi_n\xi_{n+1})$ for $n\in \mathbb{N}$, but one can also take $y_n=\sigma(\xi_n)$ for all $n\in \mathbb{N}$ instead). In the former case $B$ is classical and $\Cl(B)=\overline{B}$ is just the classical noise Boolean algebra attached to the independency $\{y_n:n\in \mathbb{N}_0\}\backslash \{0_\PP\}$ according to Example~\ref{example:classical}.   In the latter case $B$ is not classical (but not black, since it has atoms) and  $\Cl(B)\backslash B\supset \{\lor_{i\in I}y_i:I\in (2^\mathbb{N})_{\mathrm{inf}}\}$ or  $\Cl(B)\backslash B\supset \{\lor_{i\in I}y_i:I\in (2^\mathbb{N})_{\mathrm{inf}}\}\cup \{y_0\lor (\lor_{i\in I}y_i):I\in 2^\mathbb{N}\}$ depending on whether $y_0=0_\PP$ or not. Again both possibilities can occur, in some sense it is precisely the distinction between Example~\ref{example:simplest-nonclassical} on the one hand  and Example~\ref{example:simplest-nonclassical-tweaked} on the other. If one has 
 equalities in lieu of the preceding inclusions, then $B=\overline{B}$, so that $B$ is noise complete; if further $y_0=0_\PP$, then $B^\circ=\Cl(B)^\circ$. (However, this fails in general, which can be seen by taking two independent sequences of independent equiprobable random signs $\xi=(\xi_n)_{n\in \mathbb{N}}$ and $\hat{\xi}=(\hat{\xi}_n)_{n\in \mathbb{N}}$, which together generate $1_\PP$, and setting $y_{2n-1}:=\sigma(\xi_n\xi_{n+1})$, $y_{2n}:=\sigma(\hat{\xi}_n)$, $x_{2n-1}:=\sigma(\xi_i,\hat{\xi}_i:i\in \mathbb{N}_{\geq n})$, $x_{2n}:=\sigma(\xi_{i+1},\hat{\xi}_i:i\in \mathbb{N}_{\geq n})$ for $n\in \mathbb{N}$. In this case both $\sigma(\xi)$ and $\sigma(\hat{\xi})$ belong to $\Cl(B)$ and moreover to $\overline{B}$ (but not to $B$).)

\noindent The condition that $\lor_{n\in \mathbb{N}_0}y_n=1_\PP$ is equivalent to the condition that $y$ generates $x$ in  that $x_n=y_0\lor (\lor_{k\in \mathbb{N}_{\geq n}} y_k)$ for all $n\in\mathbb{N}$. That a sequence $(z_{-n})_{n\in -\mathbb{N}}$ consisting of independent $\sigma$-fields generating $x$ in the preceding sense should exist at all  is the property that $x$ is of product type, while the stipulation that (some) $y$ innovates $x$ is referred to in the literature on reverse filtrations  as $x$ being locally of product type. For characterizations of ``$x$ is of product type'' see e.g. \cite{laurent-main}.

\noindent An injective noise factorization $N$ over the Boolean algebra $\BB$ of the finite--cofinite subsets of $\mathbb{N}$ is got from $B$ by setting $N_b:=\lor \{y_i:i\in b\}$ for finite $b\in \BB$  and $N_b:=\lor\{y_i:i\in b_{<\max(\mathbb{N}\backslash b)}\}\lor x_{\max(\mathbb{N}\backslash b)+1}$ ($\max\emptyset=0$) for cofinite $b\in \BB$. 

\noindent To make it easier  to picture in one's mind/understand the spectral measure $\mu$, we may attempt to transfer it to the domain space $2^\mathbb{N}$ following the method of  \cite[p.~9]{Tsirelson1998UnitaryBM} (described there for (translationally invariant) one-dimensional noises, but the principle is the same); see also \cite[Subsections~2c \& 3b]{tsirelson-arxiv-2}.  Fix versions of $S_x$, $x\in B$, and perhaps discard a $\mu$-negligible set so that $S_x\cap S_y=S_{x\land y}$ for $\{x,y\}\subset B$, so that $S_{1_\PP}=S$, $S_{0_\PP}=\{\emptyset_S\}$, and so that the sets $\{S_x:x\in B\}$ separate $S$. Let $\phi:=(S\ni s\mapsto \cap \{b\in \BB:s\in S_{N_b}\}\in 2^\mathbb{N})$. For each $s\in S$, the set $\Phi_s:=\{b\in \BB:s\in S_{N_b}\}$ is a filter on $\mathbb{N}$ (maybe, improper), and the map $(S\ni s\mapsto \Phi_s)$ is injective ($\because$ the $S_x$, $x\in B$, separate $S$). Though, it can (cf. Example~\ref{ex:bizzare}: take $s_1=0_\PP$ and $s_2=\sigma(\xi_\infty)$) happen that for $\{s_1,s_2\}\subset S$ one has $\Phi_{s_1}\ne \Phi_{s_2}$ yet $\cap \Phi_{s_1}=\cap\Phi_{s_2}$ --- $\phi$ may fail to be injective.

\noindent Suppose however that $y_0=0_\PP$.  We have that $S_{x_n}\downarrow \{\emptyset_S\}$ as $n\to\infty$ a.e.-$\mu$. Therefore, for $\mu$-a.e. $s\ne \emptyset_S$  --- discarding a $\mu$-negligible set, for all $s\in S\backslash \{\emptyset_S\}$ ---  there is $n\in \mathbb{N}$ such that $s\notin S_{x_n}$. It means that for all $s\in S\backslash \{\emptyset_S\}$ there is $n\in \mathbb{N}$ such that $\Phi_s$ does not contain $\mathbb{N}_{\geq n}$ and therefore each member of $\Phi_s$ must contain an element of $[n-1]$, hence all of the members of $\Phi_s$ at once must contain some $i\in \mathbb{N}$ (indeed some $i\in [n-1]$). In other words, the only element of $S$ that gets mapped into $\emptyset$ by $\phi$ is $\emptyset_S$. Conversely, dropping the assumption that $y_0=0_\PP$, if $\{\phi=\emptyset\}=\{\emptyset_S\}$ a.e.-$\mu$, then $y_0=0_\PP$ ($\because$ $\{\phi=\emptyset\}\supset \cap_{n\in \mathbb{N}}S_{x_n}$ a.e.-$\mu$). 

\noindent More generally, assuming again and for the remainder of the example that $y_0=0_\PP$, for each finite $b\in \BB$, $\land_{n\in \mathbb{N}}N_b\lor x_n=N_b$, hence $S_{N_b\lor x_n}\downarrow S_{N_b}$ as $n\to\infty$ a.e.-$\mu$. Therefore, for all finite $b\in \BB$ for $\mu$-a.e. $s\notin S_{N_b}$ --- discarding a $\mu$-negligible set, for all finite $b\in\BB$, for all $s\in S\backslash S_{N_b}$ --- there is $n\in \mathbb{N}$ such that $s\notin S_{N_b\lor x_n}$. It means that for each finite $b\in \BB$ for  all $s\in S\backslash S_{N_b}$ there is $n\in \mathbb{N}$ such that $\Phi_s$ does not contain $b\cup \mathbb{N}_{\geq n}$ and therefore each member of $\Phi_s$ must contain an element of $[n-1]\backslash b$, hence all of the members of $\Phi_s$ must contain some $i\in \mathbb{N}\backslash b$. In other words $\phi^{-1}(2^b)\subset S_{N_b}$, and actually $\phi^{-1}(2^b)= S_{N_b}$, the other inclusion being trivial. Suppose now that $\{s_1,s_2\}\subset S$ and $a:=\phi(s_1)=\phi(s_2)$. If $a$ is finite, then $\{s_1,s_2\}\subset \phi^{-1}(2^a\backslash 2^c)=S_{N_a}\backslash S_{N_c}$ for all $c\in 2^a\backslash \{a\}$, hence ($\because$ none of the spectral sets separate $s_1$ and $s_2$) $s_1=s_2$. If however $a$ is infinite, then $b\in \Phi_{s_1}$ iff $a\subset b\in \BB$ (for the sufficiency of the condition it is crucial that $a$ is infinite: it implies that $b$ is cofinite \ldots; necessity is trivial), the same for $s_2$ in lieu of $s_1$, therefore $\Phi_{s_1}=\Phi_{s_2}$ and hence again $s_1=s_2$. The map $\phi$ is injective. 

\noindent Let us endow $2^\mathbb{N}$ with the $\sigma$-field generated by the canonical projection maps. For $i\in  \mathbb{N}$, $\{i\in \phi\}=\cap_{b\in \BB,b \not\owns i}(S\backslash S_{N_b})$, which shows that $\phi:S\to 2^\mathbb{N}$ is measurable. Besides, $\phi^{-1}(2^b)=S_{N_b}$, also for cofinite $b\in \BB$ (one inclusion ($\supset$) is still trivial; for the other ($\subset$), separate according to whether $\phi$ takes on a  finite value or not, and combine the observations of the preceding paragraph). Passing to the push-forward $\phi_\star \mu$ (completion implicit; recall \ref{generalities:lebesgue-rohlin}) we may thus take for the spectral space $S=2^\mathbb{N}$ in such a way that $S_{N_b}=2^b$ a.e.-$\mu$ for all $b\in \BB$. The pure atomicity of $\mu$ reduces to the question of whether or not for all $f\in \L2(\PP)$, $\mu_f$ (satisfying $\mu_f(2^b)=\PP[\PP[f\vert N_b]^2]$ for $b\in \BB$) is purely atomic, i.e. carried by a countable set. 
 
\noindent For a Hilbert space point of view set $H_n:=\L2(\PP\vert_{x_n})^\circ$ and $G_n:=\L2(\PP\vert_{y_n})^\circ=\L2(\PP\vert_{y_n})_0$ for $n\in \mathbb{N}$. Then,  for $n\in  \mathbb{N}_0$ and $I\subset [n]\backslash \{n\}$, $\L2(\PP\vert_{x_{n+1}\lor(\lor_{i\in I}y_i)})^\circ=H_{n+1}\otimes (\otimes_{i\in I}G_i)$, while for $I\in (2^\mathbb{N})_{\mathrm{fin}}$, $\L2(\PP\vert_{\lor_{i\in I}y_i})^\circ=\otimes_{i\in I}G_i$. Besides, $H_n=G_n\otimes H_{n+1}$ for all $n\in \mathbb{N}$. According to the findings of the previous subsection (Theorem~\ref{propsition:discrete-spectrum}) the condition
\begin{equation}\label{eq:class-of-nonclassical}
\text{$B^\circ=\Cl(B)^\circ$ and $\mu$ is purely atomic}
\end{equation}
 is equivalent to (the indicated ``stable and sensitive relations'' hold automatically, whether or not \eqref{eq:class-of-nonclassical} holds true) 
\begin{equation}\label{eq:class-of-nonclassical-hilbert}
\L2(\PP)=\left(\underbrace{\oplus_{n\in \mathbb{N}_0}\oplus_{I\in 2^{[n]\backslash \{n\}}}H_{n+1}\otimes (\otimes_{i\in I}G_i)}_{\subset H_{\mathrm{sens}}}\right)\oplus \left(\underbrace{\oplus_{n\in \mathbb{N}_0}\overbrace{\oplus_{I\in {\mathbb{N}\choose n}}\otimes_{i\in I}G_i}^{H^{(n)}}}_{=H_{\mathrm{stb}}=\L2(\PP\vert_\stable)}\right).
\end{equation} 
We may note that  the first term on the right-hand side of \eqref{eq:class-of-nonclassical-hilbert} is equal to (in short-hand notation) 
\begin{align*}
&\left[\oplus_{\mathbb{N}}H\right]\oplus \left[\oplus_{I\in (2^\mathbb{N})_{\mathrm{fin}},\vert I\vert\geq 1}(\otimes_I G) \otimes(\oplus_{\mathbb{N}_{\geq {\max I+2}}}H)\right]\\
&=\left[\oplus_{i\in \mathbb{N}}G_i\otimes (\oplus_{\mathbb{N}_{\geq i+1}}H)\right]\oplus \left[\oplus_{I\in (2^\mathbb{N})_{\mathrm{fin}},\vert I\vert\geq 2}(\otimes_I G) \otimes(\oplus_{\mathbb{N}_{\geq {\max I+2}}}H)\right]\\
&=\left[\oplus_{I\in {\mathbb{N}\choose 2}}(\otimes_I G)\otimes (\oplus_{\mathbb{N}_{\geq \max I +1}}H)\right]\oplus \left[\oplus_{I\in (2^\mathbb{N})_{\mathrm{fin}},\vert I\vert\geq 3}(\otimes_I G) \otimes(\oplus_{\mathbb{N}_{\geq {\max I+2}}}H)\right]\\\
&=\cdots= \left[\oplus_{I\in {\mathbb{N}\choose K}}(\otimes_I G)\otimes (\oplus_{\mathbb{N}_{\geq \max I +1}}H)\right]\oplus \left[\oplus_{I\in (2^\mathbb{N})_{\mathrm{fin}},\vert I\vert\geq K+1}(\otimes_I G) \otimes(\oplus_{\mathbb{N}_{\geq {\max I+2}}}H)\right]\text{ for all }K\in \mathbb{N};
\end{align*}
 the second term on the right-hand side of \eqref{eq:class-of-nonclassical-hilbert} is of course just $\oplus_{I\in (2^{\mathbb{N}})_{\mathrm{fin}}}(\otimes_IG)$. Unfortunately it does not seem to help in ascertaining when \eqref{eq:class-of-nonclassical} prevails (one thing is of course clear: only the $G_i$, $i\in \mathbb{N}$, and the tail of the $H_i$, $i\in \mathbb{N}$, matter).
 
 \noindent 
\end{example}

\begin{question}
In the context of the preceding example, restrict to $y_0=0_\PPP$ for simplicity. Can one say in some explicit terms, when (precisely) $\Cl(B)\backslash B= \{\lor_{i\in I}y_i:I\in (2^\mathbb{N})_{\mathrm{inf}}\}$? 
\end{question}

We can give a more clear impression of the structure of the spectrum of $B$ of the preceding example when the innovations are equiprobable random signs. It may be helpful to throw another glance on Example~\ref{example:spectral-space-for-nonclassical}, which actually appears below out of  ``abstract nonsense''.
\begin{theorem}\label{proposition:innovations-reverse-equiprobable}
Let $B$ be attached to $x$ and $y$ as in Example~\ref{example:canonical-discrete}, whose notation we retain, and suppose that $y_n=\sigma(\eta_n)$ for an equiprobable random sign $\eta_n$ for all $n\in \mathbb{N}$.  Assume further that $\land_{n\in \mathbb{N}}x_n=0_\PP$ (tail triviality -- one also says that $x$ is Kolmogorovian).
\begin{enumerate}[(i)]
\item\label{proposition:innovations-reverse-equiprobable:i} Either for all $n\in \mathbb{N}$, $H_n=\{0\}$, or for all $n\in \mathbb{N}$ there is a up to prefactor of $\pm 1$ a.s.-$\PP$ unique equiprobable random sign $\xi_n$ such that $H_n=\mathrm{lin}(\xi_n)$, and the prefactors $\pm 1$ may be chosen so that $\xi_n=\eta_n\xi_{n+1}$ for all $n\in \mathbb{N}$.
\item\label{proposition:innovations-reverse-equiprobable:ii} Suppose the second option of \ref{proposition:innovations-reverse-equiprobable:i} prevails. Then $B^\circ=B$ and $\oplus_{x\in B}\L2(\PP\vert_x)^\circ=\L2(\PP\vert_{\sigma(\xi_n:n\in \mathbb{N})})$. $\sigma(\xi_n:n\in \mathbb{N})=1_\PP$ iff \eqref{eq:class-of-nonclassical} holds true, in which case $x_n=\sigma(\xi_k:k\in \mathbb{N}_{\geq n})$ for all $n\in \mathbb{N}$.  
\item\label{proposition:innovations-reverse-equiprobable:iii} 
If \eqref{eq:class-of-nonclassical} obtains, then $x$ is of product type (i.e.  there is an independency $(z_n)_{n\in \mathbb{N}}$ in $\Lattice$ such that $x_n=\lor_{k\in \mathbb{N}_\geq n}z_k$ for all $n\in \mathbb{N}$). 
\end{enumerate}
\end{theorem}
\begin{remark}
When it exists $\xi_n$ is actually independent of (not just orthogonal to) each member of $B_{x_n}\backslash \{x_n\}$ for all $n\in \mathbb{N}$ (because it is an equiprobable random sign). 
\end{remark}

\begin{proof}
 Notice that  $G_n=\mathrm{lin}(\eta_n)$ for all $n\in \mathbb{N}$.

\ref{proposition:innovations-reverse-equiprobable:i}. 
 If $n\in \mathbb{N}_0$ and $g\in H_{n+1}\backslash \{0\}$, then $g\eta_1\cdots \eta_n\in H_1\backslash \{0\}$. We now focus on an arbitrary $\xi_1\in H_1\backslash \{0\}$. For all $n\in \mathbb{N}_0$, since $H_1=G_1\otimes\cdots \otimes G_n\otimes H_{n+1}$, it follows that $\xi_1=\xi_{n+1}\eta_1\cdots \eta_n$ for some $\xi_{n+1}\in H_{n+1}$. Let $A:=\limsup_{n\to\infty}\xi_n$ and $B:=\liminf_{n\to\infty}\xi_n$. By tail triviality it follows that $A$ and $B$ are $\PP$-a.s. constant. Therefore $\PP$-a.s. $\xi_1$ must be arbitrarily close to $A$ or $-A$ and also to $B$ or $-B$. It is only possible if $B=-A$ and we cannot have $A=0$ because $\xi_1\ne 0$. Dividing $\xi_1$ by $A$ if necessary we may assume that $\limsup_{n\to\infty}\xi_n=-\liminf_{n\to\infty}\xi_n=1$ a.s.-$\PP$ and $\xi_1$ can take on only the values $1$ and $-1$ (with $\PP$-probability one), which means that it is a random sign  --- perhaps only off a $\PP$-negligible set, but we change it (and the $\xi_{n+1}$, $n\in \mathbb{N}$, in parallel) on a negligible set and make it so with certainty. Then each $\xi_n$, $n\in \mathbb{N}$, is a random sign, and by the independences and since the $\eta_k$s are equiprobable random signs, they are all also equiprobable. Now we have established that any non-zero element  of $H_1$ is proportional $\PP$-a.s. to an equiprobable random sign. Such random sign is therefore unique up to a prefactor of $\pm 1$. All the properties follow. 
 
\ref{proposition:innovations-reverse-equiprobable:ii}.  It is clear that $B^\circ=B$ (since $B$ is  closed for $\lor$). We see from Example~\ref{example:spectral-space-for-nonclassical} (we have the same structure due to the relation $\xi_n=\eta_n\xi_{n+1}$, $n\in \mathbb{N}$) that indeed also $\oplus_{x\in B}\L2(\PP\vert_x)^\circ=\L2(\PP\vert_{\sigma(\xi_n:n\in \mathbb{N})})$. Therefore $\sigma(\xi_n:n\in \mathbb{N})=1_\PP$ iff \eqref{eq:class-of-nonclassical} holds true and by Theorem~\ref{propsition:discrete-spectrum} in that case $\L2(\PP\vert_{x_n})=\oplus_{x\in B_{x_n}}\L2(\PP\vert_x)^\circ$, which again by Example~\ref{example:spectral-space-for-nonclassical} cannot be anything other than $\sigma(\xi_k:k\in \mathbb{N}_{\geq n})$. 

\ref{proposition:innovations-reverse-equiprobable:iii}.  Suppose \eqref{eq:class-of-nonclassical} prevails. Then either $B$ is classical, and we may take $z=y$ (to see that $x$ is of product type), or else we must be  in the second situation described by Item~\ref{proposition:innovations-reverse-equiprobable:i}. Then, by Item~\ref{proposition:innovations-reverse-equiprobable:ii} we get again that $x$ is of product type (generated by the $\xi_k$, $k\in \mathbb{N}$).
 \end{proof}
\begin{question}
Does Item~\ref{proposition:innovations-reverse-equiprobable:iii} hold true without the assumption that the innovations are equiprobable random signs? Is the converse true?
\end{question}
\begin{example}
 We are given a sequence  $X=(X_n)_{n\in \mathbb{N}}$,  generating $1_\PP$, with each $X_n=(X_n(i))_{i\in \mathbb{N}}$ consisting of independent equiprobable random signs, and a sequence $(\sigma_n)_{n\in \mathbb{N}}$ of equiprobable random signs such that  for each $n\in \mathbb{N}$, $\sigma_n$ is independent of $(X_k)_{k\in \mathbb{N}_{> n}}$ and such that $X_n(k)=X_{n+1}(2k-1)$ for all $k\in \mathbb{N}$ or $X_{n}(k)=X_{n+1}(2k)$ for all $k\in \mathbb{N}$ according as to whether $\sigma_n=-1$ or $\sigma_n=1$. We put $y_n:=\sigma(\sigma_n)$  and $x_n:=\sigma(X_k:k\in \mathbb{N}_{\geq n})$ for $n\in \mathbb{N}$ and attach $B$ to $x$ and $y$ as in Example~\ref{example:canonical-discrete}, whose notation we retain. It is an example originally due to Vershik \cite{vershik-dissertation}, see also \cite[p.~885]{dfst} \cite[p.~328]{smorodinsky}. 

\noindent It is clear that in either case  $X$ is a Markov process (in reverse time). Besides, for each $n\in \mathbb{N}$, the conditional law of $X_n$ given $X_m$ tends $\PP$-a.s. to that of independent equiprobable signs as $m\to\infty$ \cite[proof of Theorem~1]{smorodinsky}. In particular  $\land_{n\in \mathbb{N}} x_n =0_\PP$. It follows from \cite[Theorem~2]{smorodinsky} that $B$ is not classical.

\noindent By  Theorem~\ref{proposition:innovations-reverse-equiprobable}\ref{proposition:innovations-reverse-equiprobable:iii},  \eqref{eq:class-of-nonclassical} cannot prevail, since $x$ is not of product type, which latter fact is well-known ($x$ is not even standard \cite[Theorem~2]{smorodinsky}). 

\noindent One sees easily that for finite $b\in \BB$, $\PP[\PP[X_1(1)\vert N_b]^2]=0$ and that $\PP[\PP[X_1(1)\vert N_b]^2]=2^{-\vert\mathbb{N}\backslash b\vert}$ for cofinite $b\in\mathcal{B}$, which means that the spectral measure $\mu_{X_1(1)}$ on $2^\mathbb{N}$ is that of a random subset of $\mathbb{N}$ corresponding to including each $i\in \mathbb{N}$ independently of the others with probability $1/2$, i.e. $\mu_{X_1(1)}=\mathrm{Ber(\frac{1}{2}})^{\times \mathbb{N}}$ (up to identifying $2^\mathbb{N}$ with $ \{0,1\}^\mathbb{N}$). In particular $\mu_{X_1(1)}$ is diffuse and $X_1(1)\in H_{\mathrm{sens}}$.\footnote{I thank J. Warren for suggesting the explicit computation of $\mu_{X_1(1)}$.}
\end{example}

\begin{question}
In the preceding example is $H_n$ actually the null space for all (equivalently, some) $n\in \mathbb{N}$?
\end{question}
Similar examples of noises (all falling under the umbrella of Example~\ref{example:canonical-discrete}) could be constructed from any so-called split-words process \cite{Ceillier2014}.

\section{Interactions between restrictions/projections of the spectrum and subnoises}
On the one hand, we look at how  points of the spectral space can be, as it were, projected onto an element of $B$ (Theorem~\ref{thm:noise-projections}), which allows to describe efficiently the tensor structure of the spectrum (Corollary~\ref{prop:noise-proj-bis}\ref{proj:iv}). We collect many other dividends from the introduction of these projection, especially in terms of economy of notation. On the other hand, some elements of $\Sigma$ may be used to restrict $\L2(\PP)$ to an $\L2$-subspace (Corollary~\ref{corollary:L2-space}) and then to give rise to a ``subnoise'' of $B$ (Corollary~\ref{corollary:sigma*-subnoises}).
\subsection{Noise projections and the tensor structure of the spectrum}
When a noise Boolean algebra $B$ is the range of a noise factorization $N:\mathcal{B}\to\Lattice$, with $\mathcal{B}$ being a subset of the power set of base set $T$, and when further the spectral space $S$ consists of certain subsets of $T$ in such a way that $S_{N(b)}=\{t\in S:t\subset b\}$ a.e.-$\mu$ for all $b\in \mathcal{B}$ (as often happens: such is the case e.g. for all one-dimensional noises \cite[Theorem~2.3]{Tsirelson1998UnitaryBM}; we have also seen it in the context of Example~\ref{example:canonical-discrete}), then one has the luxury of having on $S$ defined for each $b\in \mathcal{B}$ the ``projection'' $q_b:=(S\ni s\mapsto s\cap b)$, which restricts the points of the spectral space to $b$.  In order to gain access to the analogs of such projections in the abstract setting some work is needed. The key is to notice that the $q_b$ in the preceding satisfies $q_b^{-1}(S_{N(a)})=S_{N((T\backslash b)\cup a)}=S_{N(T\backslash b)\lor N(a)}=S_{N(b)'\lor N(a)}$ a.e.-$\mu$ for all $a\in \mathcal{B}$. At the same time this yields a path into a better understanding of the spectral resolution. 

\begin{theorem}[Noise projections]\label{thm:noise-projections}
Let $x\in B$. 
\begin{enumerate}[(i)]
\item\label{proj:i} There exists a up to $\mu$-a.e. equality unique map $\pr_x\in \Sigma/\Sigma$ such that for all $E\in \Sigma/_\mu$, $\pr_x^{-1}(E)$ is uniquely determined up to a.e.-$\mu$ equality (vis-\`a-vis the choice of  a representative of the equivalence class of $E$) and such that for all $y\in B_x$, a.e.-$\mu$, $\pr_x^{-1}(S_y)=S_{y\lor x'}$.
\item \label{proj:ii} $\pr_x$ maps into $S_x$ a.e.-$\mu$. $\pr_x^{-1}(S_y)=S_{y\lor x'}$ a.e.-$\mu$ for all $y\in B$.  $\sigma(\pr_x)=\Sigma_x$. $\pr_x=\id_{S}$ a.e.-$\mu$ on $S_x$ and $\pr_{x}=\emptyset_S$ a.e.-$\mu$ on $S_{x'}$. The map $(\pr_x,\pr_{x'}):S\to S^2$ is mod-$0$ bijective relative to $\mu$ and $\mu(\cdot\cap S_x)\times \mu(\cdot \cap S_{x'})$.
\item\label{proj:iii}  If $\mu$ is a probability under which  $\pr_x$ and $\pr_{x'}$ are independent, then $((\pr_x,\pr_{x'})_\star\mu)\mu(S_{0_\PP})=\mu( \cdot \cap S_x)\times \mu(\cdot \cap S_{x'})$ and $(\pr_x,\pr_{x'})$ is a mod-$0$ isomorphism of $\mu$ onto $\mu(S_{0_\PP})^{-1}\mu( \cdot \cap S_x)\times \mu(\cdot \cap S_{x'})$.
\item\label{proj:vi}  $B$ admits a spectral resolution having the same spectral space $(S,\Sigma,\mu)$, the same spectral sets $S_x$, $x\in B$, the same subspaces $H(E)$, $E\in \Sigma$ ($\therefore$ also the same maps $(S\ni s\mapsto \langle\Psi(f)_s,\Psi(g)_s\rangle)$, $\{f,g\}\subset \L2(\PP)$), the same Hilbert spaces $(H_s)_{s\in S_x\cup S_{x'}}$, and the same $\Psi\vert_{\L2(\PP\vert_x)\cup \L2(\PP\vert_{x'})}$ as the given spectral resolution, but satisfying further: 
\begin{enumerate}[(I)]
\item\label{proj:vi:I}  $H_s=H_{\pr_x(s)}\otimes H_{\pr_{x'}(s)}$ for $\mu$-a.e. $s$ [$G$, $G\otimes H_{\emptyset_S}$, $H_{\emptyset_S}\otimes G$ are implicitly  identified by identifying $g$, $\sqrt{\mu(\{\emptyset_S\})}g\otimes \Psi(1)_{\emptyset_S}$, $\sqrt{\mu(\{\emptyset_S\})}\Psi(1)_{\emptyset_S}\otimes g$ for $g\in G$], the measurable structure of $(H_s)_{s\in S}$ being consistent with that of $(H_s)_{s\in S_x}$ and $(H_s)_{s\in S_{x'}}$ in the sense that  the measurable vector fields of $(H_{\pr_x(s)})_{s\in S}$ are precisely of the form $f(\pr_x)$ for $f$ a measurable vector field of $(H_s)_{s\in S_x}$, likewise for $(H_{\pr_{x'}(s)})_{s\in S}$, and then the measurable structure on $(H_s)_{s\in S}=(H_{\pr_x(s)}\otimes H_{\pr_{x'}(s)})_{s\in S}$  is the canonical tensor measurable structure inherited from the measurable Hilbert space fields $(H_{\pr_x(s)})_{s\in S}$ and $(H_{\pr_{x'}(s)})_{s\in S}$ \cite[p.~174, Proposition~II.1.8.10]{dixmier1981neumann};   
\item\label{proj:vi:II} if $\mu$ is a probability under which $\pr_x$ and $\pr_{x'}$ are independent, then also $(\Psi(f)\circ \pr_x)\otimes (\Psi(f')\circ \pr_{x'})\sqrt{\mu(\{\emptyset_S\})}=\Psi(ff')$, hence $\Vert \Psi(ff')\Vert=\sqrt{\mu(\{\emptyset_S\})}(\Vert \Psi(f)\Vert\circ \pr_x)(\Vert \Psi(f')\Vert\circ \pr_{x'})$ a.e.-$\mu$ for $f\in \L2(\PP\vert_x)$ and $f'\in \L2(\PP\vert_{x'})$
(here $g(\pr_x)\otimes g'(\pr_{x'}):=(S\ni s\mapsto g(\pr_x(s))\otimes g'(\pr_{x'}(s)))$ for $(g,g')\in \int_{S_x}^\oplus H_s\mu(\dd s)\times \int_{S_{x'}}^\oplus H_{s'}\mu(\dd s')$).
\end{enumerate}
\item\label{proj:vii} For $A\in \Sigma\vert_{S_x}$ and $B\in \Sigma\vert_{S_{x'}}$, up to the natural unitary equivalence, $H(A)\otimes H(B)=H(\pr_x^{-1}(A)\cap \pr_{x'}^{-1}(B))$. In particular, for $E\in \Sigma$, $H(\pr_x^{-1}(E))=H(E\cap S_x)\otimes \L2(\PP\vert_{x'})$.
\end{enumerate}
\end{theorem}
\begin{definition}
In the following we retain the notation $\pr_x$ and call it the noise projection onto $x$.
\end{definition}
\begin{remark}
 In \ref{proj:vi} $\sqrt{\mu(\{\emptyset_S\})}\Psi(1)_{\emptyset_S}$ serves as the natural ``unit''  in $H_{\emptyset_S}$ (and is indeed a unit vector). It may very well happen (but need not be the case) that $H_{\emptyset_S}=\mathbb{R}$, $\mu(\{\emptyset_S\})=1$ and $\Psi(1)_{\emptyset_S}=1$, in which case it is (even) more suggestive.
\end{remark}
\begin{example}
In the context of a finite $B$, Example~\ref{example:spectral-finite}, the map $\pr_x$ is given simply by $\pr_x(s)=s\land x$ for $s\in B$, which we may write succinctly as $\pr_x=(\cdot \land x)\vert_B$. The same is true for the simplest nonclassical case of  Example~\ref{example:simplest-nonclassical}, cont'd in~\ref{example:spectral-space-for-nonclassical}. More generally, $\pr_x=(\cdot\land x)\vert_{B^\circ}$ whenever we are in the situation described in Theorem~\ref{propsition:discrete-spectrum}\ref{discrete-spectrum:a} (in particular for any classical discrete noise of Example~\ref{example:classical-spectrum}), for in such case $(\cdot \land x\vert_{B^\circ})^{-1}(S_y)=(\cdot \land x\vert_{B^\circ})^{-1}(B_y\cap B^\circ)=\{u\in B^\circ:u\land x\subset y\}=B_{y\lor x'}\cap B^\circ =S_{y\lor x'}$, $y\in B$, where we have taken into account that $B_u\subset B^\circ$ for  $u\in B^\circ$ (Theorem~\ref{prop:dicrete}\ref{prop:dicrete:iv}).
\end{example}

\begin{remark}\label{remark:exceptional-sets-projections}
For each $x\in B$, $\pr_x$ is only defined up to $\mu$-a.e. equality precisely and for $y\in B$ the defining relation $\pr_x^{-1}(S_y)=S_{x'\lor y}$ holds only a.e.-$\mu$. Nevertheless, throughout any given countable noise Boolean subalgebra  $B_0$ of $B$ (a dense one exists), upon a choice of versions of the projections and spectral sets,  the latter relation holds everywhere on a $\mu$-conegligible set in the following precise sense. Fix any versions of $\pr_x$ and $S_x$,  $x\in B_0$. For $\{x,y\}\subset B_0$ let $\tilde{S}_{x,y}$ be a $\mu$-conegligible set on which $\pr_x^{-1}(S_y)=S_{x'\lor y}$. Put $\tilde{S}:=\cap_{(x,y)\in (B_0)^2}\cap_{n\in \mathbb{N}_0}\cap_{(x_1,\ldots,x_n)\in (B_0)^n}\pr_{x_1}^{-1}\cdots \pr_{x_n}^{-1}\tilde{S}_{x,y}$ (for $n=0$ we interpret the void preimage as $\tilde{S}_{x,y}$).  It is trivial that $\tilde{S}$ is $\mu$-conegligible, contained in $\cap_{(x,y)\in B_0^2}\tilde{S}_{x,y}$, and that $\tilde{S}\subset \pr_u^{ -1}(\tilde{S})$, i.e. $\pr_u^{-1}(\tilde{S})\cap \tilde{S}=\tilde{S}$, for all $u\in B_0$. Therefore $(\pr_x\vert_{\tilde{S}})^{-1}(S_y\cap\tilde{S})=\pr_x^{-1}(S_y\cap \tilde{S})\cap \tilde{S}=\pr_x^{-1}(S_y)\cap \pr_x^{-1}(\tilde{S})\cap \tilde{S}=\pr_x^{-1}(S_y)\cap \tilde S=S_{y\lor x'}\cap \tilde{S}$ for $\{x,y\}\subset B_0$, which is the ``precise sense'' of which we spoke. In particular if $B$ is countable to begin with, choosing versions of spectral sets and projections, discarding a $\mu$-negligible set (which we can insist can contain any other ex ante  given $\mu$-negligible set), the relation $\pr_x^{-1}(S_y)=S_{x'\lor y}$ prevails everywhere for $\{x,y\}\subset B$. This observation may be useful in keeping track of exceptional sets.  The non-trivial part of the above was ensuring that $\tilde{S}$ is stable under the projections from $B_0$ (which enables one to indeed properly discard the complement of $\tilde{S}$, if one wants to). 
\end{remark}

\begin{proof}
Let $\nu$ be a probability equivalent to $\mu$ under which the $\sigma$-fields $\Sigma_x$ and $\Sigma_{x'}$ are independent.  Choose dense countable subsets $B^0_x$ and $B^0_{x'}$ of $B_x$ and $B_{x'}$, respectively, and fix representatives $S_{u\lor x'}$, $u\in B_x^0$, and $S_{x\lor v}$, $v\in B^0_{x'}$. Then $A_x:=\{S_{u\lor x'}:u\in B^0_{x}\}$ (resp. $A_{x'}:=\{S_{x\lor v}:v\in B^0_{x'}\}$) is a countable $\mu$-essentially generating class for $\Sigma_x$ (resp. $\Sigma_{x'}$). Introduce the equivalence relation $\sim_x$ (resp. $\sim_{x'}$) via  $s_1\sim_xs_2$ iff $\mathbbm{1}_{F}(s_1)=\mathbbm{1}_{F}(s_2)$ for all $F\in A_x$  (resp. $F\in A_{x'}$), $\{s_1,s_2\}\subset S$. Note that $A_{x'}\cup A_x$ essentially separates the points of $S$.

Then we see that the standard probability space $(S,\Sigma,\nu)$ passes measure-preservingly through the quotient map $\phi_x$ (resp. $\phi_{x'}$) induced by the equivalence relation $\sim_x$ (resp. $\sim_{x'}$) to the quotient (countably separated and complete, $\therefore$ standard) probability space $(P_x,\Gamma_x,\rho_x)$ (resp. $(P_{x'},\Gamma_{x'},\rho_{x'})$). Further, the map $\phi=(\phi_x,\phi_{x'})$ is injective off a null set, measure-preserving and hence a 
mod-$0$ isomorphism of $(S,\Sigma,\nu)$ onto the product standard probability space $(P_x,\Gamma_x,\rho_x)\times (P_{x'},\Gamma_{x'},\rho_{x'})$ (completion of the latter is implicit), which, by Blackwell's property, carries $\Sigma_x$ (resp. $\Sigma_{x'}$) onto the $\sigma$-field generated (modulo null sets) by the first (resp. second) coordinate. 

 Let $\emptyset_{x}$ (resp. $\emptyset_{x'}$) be the point of $P_x$ (resp. $P_{x'}$) corresponding to the $\mu\vert_{\Sigma_x}$-atom $S_{x'}$ (resp. $\mu\vert_{\Sigma_{x'}}$-atom $S_x$): $S_{x'}=\phi_x^{-1}(\{\emptyset_x\})=\phi^{-1}(\{\emptyset_{x}\}\times P_{x'})$ a.e.-$\mu$ and $S_{x}=\phi_{x'}^{-1}(\{\emptyset_{x'}\})=\phi^{-1}( P_x\times \{\emptyset_{x'}\})$ a.e.-$\mu$. 
In particular $S_{0_\PP}=S_{x\land x'}=S_x\cap S_{x'}=\phi^{-1}(P_x\times \{\emptyset_{x'}\})\cap \phi^{-1}(\{\emptyset_{x}\}\times P_{x'})=\phi^{-1}((P_x\times \{\emptyset_{x'}\})\cap (\{\emptyset_{x}\}\times P_{x'}))=\phi^{-1}(\{\emptyset_x\}\times \{\emptyset_{x'}\})=\phi^{-1}(\{(\emptyset_x,\emptyset_{x'})\})$ a.e.-$\mu$. Moreover, for $y\in B_x$, $S_{y\lor x'}=\phi^{-1}(Q_y\times P_{x'})$ a.e.-$\mu$ for a $\rho_x$-a.e. unique $Q_y$, while for $z\in B_{x'}$, $S_{x\lor z}=\phi^{-1}(P_x\times Q_z')$ a.e.-$\mu$ for a $\rho_{x'}$-a.e. unique $Q_z'$; taking intersections  we get $S_{y\lor z}=\phi^{-1}(Q_y\times Q_z')$ a.e.-$\mu$ (with $y$ and $z$ being as in the preceding). In particular $S_y=\phi^{-1}(Q_y\times \{\emptyset_{x'}\})$ a.e.-$\mu$ for $y\in B_x$ and $S_z=\phi^{-1}(\{\emptyset_{x}\}\times Q_z')$ a.e.-$\mu$ for $z\in B_{x'}$.

In proving the items of this theorem we may and do assume $\mu=\nu$ and $\phi$ is the identity map (up to canonical identifications: strictly speaking the passage from the product space $P_x\times P_{x'}$ through the equivalence classes does not return back the product space --- one still has to identify $(p_x,p_{x'})\in P_x\times P_{x'}$ with $(\{p_x\}\times P_{x'},P_x\times \{p_{x'}\})$, but this is trivial). 

\ref{proj:i}. We see at once that the map $(P_x\times P_{x'}\ni (p,p')\mapsto (p,\emptyset_{x'}))$ has the properties that we want of $\pr_x$. 
Uniqueness follows from the fact that the countable family $Q_y$, $y\in B_x^\circ$, $\rho_x$-essentially separates the points of $P_x$, as follows: 
if also $\tilde\pr_x$ satisfies the properties of $\pr_x$, then (since both $\pr_x$ and $\tilde\pr_x$ map into $S_x$ a.e.-$\mu$, see proof of \ref{proj:ii} immediately to follow) $\{\pr_x\ne \tilde\pr_x\}\subset \cup_{y\in B_x^\circ }\{\pr_x\in S_y\}\triangle \{\tilde\pr_x\in S_y\}$ a.e.-$\mu$; but the latter is a countable union of $\mu$-null sets, therefore so is $\{\pr_x\ne \tilde\pr_x\}$ negligible for $\mu$.

\ref{proj:ii}.  $\pr_{x}^{-1}(S_x)=S_{x\lor x'}=S_{1_\PP}=S$ a.e.-$\mu$. For $y\in B$, $\pr_x^{-1}(S_y)=\pr_x^{-1}(S_y\cap S_x)=\pr_x^{-1}(S_{y\land x})=S_{(y\land x)\lor x'}=S_{y\lor x'}$ a.e.-$\mu$. By its defining property $\sigma(\pr_x)$ contains an essentially generating class for $\Sigma_x$, therefore $\Sigma_x$; on the other hand it suffices to check the measurability of $\pr_x$ in $\Sigma_x$ on a generating system, which gives $\sigma(\pr_x)\subset \Sigma_x$. The penultimate and last statements are immediate from the construction, noting that $\emptyset_S=(\emptyset_x,\emptyset_{x'})$ (up to canonical identifications).

\ref{proj:iii}. We have that $\rho_{x'}(\{\emptyset_{x'}\})(\pr_x)_\star \mu=\mu(\cdot \cap S_x)$; indeed $\rho_{x'}(\{\emptyset_{x'}\})(\pr_x)_\star \mu(S_y)=\rho_{x'}(\{\emptyset_{x'}\})\mu(\pr_x^{-1}(S_y))=\rho_{x'}(\{\emptyset_{x'}\})\mu(S_{y\lor x'})=\rho_{x'}(\{\emptyset_{x'}\})\mu(S_{(y\land x)\lor x'})=\rho_{x'}(\{\emptyset_{x'}\})\mu(Q_{y\land x}\times P_{x'})=\rho_{x'}(\{\emptyset_{x'}\})\rho_x(Q_{y\land x})=\mu(Q_{y\land x}\times \{\emptyset_{x'}\})=\mu(S_{y\land x})=\mu(S_y\cap S_x)$ for all $y\in B$, and a straightforward  Dynkin's lemma argument allows to conclude equality of measures on the whole of $\Sigma$. Naturally also $\rho_{x}(\{\emptyset_{x}\})(\pr_{x'})_\star \mu=\mu(\cdot \cap S_{x'})$. Therefore 
\begin{align*}
((\pr_x,\pr_{x'})_\star\mu)\mu(S_{0_\PP})&=((\pr_x,\pr_{x'})_\star\mu)\rho_x(\{\emptyset_x\})\rho_{x'}(\{\emptyset_{x'}\})\\
&=(\rho_{x'}(\{\emptyset_{x'}\})(\pr_x)_\star \mu)\times (\rho_{x}(\{\emptyset_{x}\})(\pr_{x'})_\star \mu)=\mu( \cdot \cap S_x)\times \mu(\cdot \cap S_{x'}).
\end{align*}
Combining this with the previous item we get the claim.

\ref{proj:vi}.  $\Psi\vert_{\L2(\PP\vert_x)}$ carries $\L2(\PP\vert_x)$ as a unitary isomorphism onto $\int_{S_x}^\oplus H_s\mu(\dd s)$ and $\Psi\vert_{\L2(\PP\vert_{x'})}$ carries $\L2(\PP\vert_{x'})$ as a unitary isomorphism onto $\int_{S_{x'}}^\oplus H_{s'}\mu(\dd {s'})$. Therefore the map which sends $fg$ to $\Psi(f)\otimes \Psi(g)$
for $(f,g)\in \L2(\PP\vert_x)\times \L2(\PP\vert_{x'})$, extends uniquely to the unitary isomorphism $\tilde \Psi=\Psi\vert_{\L2(\PP\vert_x)}\otimes \Psi\vert_{\L2(\PP\vert_{x'})}$ of $\L2(\PP)$ onto $(\int_{S_x}^\oplus H_s\mu(\dd s))\otimes (\int_{S_{x'}}^\oplus H_{s'}\mu(\dd {s'}))$. Let $\Gamma$ be the canonical unitary isomorphism between $(\int_{S_x}^\oplus H_s\mu(\dd s))\otimes (\int_{S_{x'}}^\oplus H_{s'}\mu(\dd s'))$ and  $\int_{S_x\times S_{x'}}^\oplus H_s\otimes H_{s'}\mu^2(\dd (s,s'))$ (sending, for  $(g,g')\in \int_{S_x}^\oplus H_s\mu(\dd s)\times \int_{S_{x'}}^\oplus H_{s'}\mu(\dd s')$, $g\otimes g'$ to  $(S\times S\ni(s,s') \mapsto g(s)\otimes g'(s'))$), and let $\Theta$ be the composition of the canonical unitary isomorphisms 
\begin{align*}
&\int_{S_x\times S_{x'}}^\oplus H_s\otimes H_{s'}\mu^2(\dd (s,s'))\to \int_{S_x\times S_{x'}}^\oplus H_s\otimes H_{s'}\left[\frac{1}{\mu(S_{0_\PP})}\mu^2\right](\dd (s,s'))\\
&= \int_{S\times S}^\oplus H_s\otimes H_{s'}\left[(\pr_x,\pr_{x'})_\star\mu\right](\dd (s,s'))\to \int_S^\oplus H_{\pr_x(s)}\otimes H_{\pr_{x'}(s)}\mu(\dd s)
\end{align*}
 (see \cite[A.75]{dixmier-c-star} for the first arrow [just a ``rescaling'']; the second arrow  sends $f$ to $f(\pr_x,\pr_{x'})$ [just a ``push-forward'']). Then $\Theta\circ \Gamma\circ\tilde\Psi\circ \Psi^{-1}$ is a unitary  isomorphism of $\int^\oplus_SH_s\mu(\dd s)$ 
onto $\int_S^\oplus H_{\pr_x(s)}\otimes H_{\pr_{x'}(s)}\mu(\dd s)$, which sends $\int^\oplus_{S_z}H_s\mu(\dd s)$ onto $\int_{S_z}^\oplus H_{\pr_x(s)}\otimes H_{\pr_{x'}(s)}\mu(\dd s)$ for each $z\in B$. Further, selecting a dense countable noise Boolean algebra $B_0\subset B$ and fixing representatives $S_x$, $x\in B_0$, forming a $\pi$-system, $S_{1_\PP}=S$, we see that this unitary isomorphism sends $\int^\oplus_{E}H_s\mu(\dd s)$ onto $\int_{E}^\oplus H_{\pr_x(s)}\otimes H_{\pr_{x'}(s)}\mu(\dd s)$ for each $E$ from the algebra $\Sigma_a$ generated by $\{S_x:x\in B_0\}$, hence by approximation for each $E\in \sigma(B_x:x\in B_0)=\Sigma$ (indeed each such $E$ is a.e.-$\mu$ equal to $\liminf_{n\to\infty}E_n$ for some sequence $(E_n)_{n\in \mathbb{N}}$ in $\Sigma_a$).

By the preceding we get at once \ref{proj:vi:I} and \ref{proj:vi:II}. 
(The invariance of the scalar-product maps $\langle \Psi(f),\Psi(g)\rangle$, $\{f,g\}\subset \L2(\PP)$, follows from the identity $\int_E\langle \Psi(f),\Psi(g)\rangle\dd\mu=\PP[\pr_{H(E)}(f)\pr_{H(E)}(g)]$, $E\in \Sigma$.)

\ref{proj:vii}. Retain the notation of the previous item. We may assume that we have the spectral decomposition of \ref{proj:vi}\ref{proj:vi:I} so that  $\Theta\circ \Gamma\circ\tilde\Psi\circ \Psi^{-1}$  is just the identity (up to natural identifications). Then $\Psi^{-1}\Theta\Gamma=\tilde{\Psi}^{-1}$ and $H_{s}=H_{\pr_x(s)}\otimes H_{\pr_{x'}(s)}$ for $\mu$-a.e. $s$. We compute 
\begin{align*}
H(\pr_x^{-1}(A)\cap \pr_{x'}^{-1}(B))&=\Psi^{-1}\left(\int^{\oplus}_{\pr_x^{-1}(A)\cap \pr_{x'}^{-1}(B)}H_s\mu(\dd s)\right)\\
&=\Psi^{-1}\left(\int^{\oplus}_{\pr_x^{-1}(A)\cap \pr_{x'}^{-1}(B)}H_{\pr_x(s)}\otimes H_{\pr_{x'}(s)}\mu(\dd s)\right)\\
&=
\Psi^{-1}\Theta\left(\int^{\oplus}_{A\times B}H_{s}\otimes H_{s'}\mu^2(\dd (s,s'))\right)\\
&=\Psi^{-1}\Theta\Gamma\left(\int^{\oplus}_AH_{s}\mu(\dd s)\otimes \int^{\oplus}_B H_{s'}\mu(\dd s')\right)\\
&=\tilde{\Psi}^{-1}\left(\int^{\oplus}_AH_{s}\mu(\dd s)\otimes \int^{\oplus}_B H_{s'}\mu(\dd s')\right)\\
ž&=\Psi^{-1}\left(\int^{\oplus}_AH_{s}\mu(\dd s)\right)\otimes \Psi^{-1}\left(\int^{\oplus}_BH_{s}\mu(\dd s)\right)=H(A)\otimes H(B).
\end{align*} The second statement follows on taking  $A=E\cap S_x$ and  $B=S_{x'}$ in the first.
\end{proof}


\begin{corollary}[Noise projections (cont'd)]\label{prop:noise-proj-bis}
We have the following assertions.
\begin{enumerate}[(i)]
\item\label{proj:ii-} Let $x\in B$. If $\phi\in \mathrm{L}^0(\mu;(E,\EE))$ for a standard Borel space $(E,\EE)$, then $\phi(\pr_x)$ is determined up to a.e.-$\mu$ equality (vis-\`a-vis the choice of a version for $\pr_x$ and of a representative of the equivalence class of $\phi$).
\item\label{proj:ii--} For $y\in B$, $\pr_{x}\circ \pr_{y}=\pr_{x\land y}$ a.e.-$\mu$, it being implicit that there is no ambiguity in writing $\pr_x\circ \pr_y$ (up to a.e.-$\mu$ equality).
\item\label{proj:iv} Let $n\in \mathbb{N}$ and let $(x_1,\ldots,x_n)$ be an independency in $B$.
\begin{enumerate}[(I)]
\item\label{proj:iv-a}  $(\pr_{x_1},\ldots,\pr_{x_n}):S\to S^n$ is mod-$0$ bijective  relative to $\mu(\cdot\cap {S_{x_1\lor\cdots\lor x_n}})$ and $\mu(\cdot \cap S_{x_1})\times\cdots\times \mu(\cdot \cap S_{x_n})$. Furthermore, if $(f_1,\ldots,f_n)\in \L2(\PP\vert_{x_1})\times\cdots\times \L2(\PP\vert_{x_n})$, $\PP[(f_1)^2]=\cdots=\PP[(f_n)^2]=1$, then  $$(\pr_{x_1},\cdots,\pr_{x_n})_\star \mu_{f_1\cdots f_n}=\mu_{f_1}\times \cdots\times\mu_{f_n}$$ (i.e. $\pr_{x_1},\ldots,\pr_{x_n}$ are independent under $\mu_{f_1\cdots f_n}$ with the law of $\pr_{x_i}$ under $\mu_{f_1\cdots f_n}$ being $\mu_{f_i}=\mu_{f_1\cdots f_n}(\cdot \cap S_{x_i})$ for $i\in [n]$) and $(\pr_{x_1},\ldots,\pr_{x_n})$ is a mod-$0$ isomorphism of $\mu_{f_1\cdots f_n}$ onto $\mu_{f_1}\times\cdots \times \mu_{f_n}$. 

 \item\label{proj:iv-b}  $B$ admits a spectral resolution having the same spectral space $(S,\Sigma,\mu)$, the same spectral sets $S_x$, $x\in B$, the same subspaces $H(E)$, $E\in \Sigma$ ($\therefore$ also the same maps $(S\ni s\mapsto \langle\Psi(f)_s,\Psi(g)_s\rangle)$, $\{f,g\}\subset \L2(\PP)$), the same Hilbert spaces $(H_s)_{s\in \cup_{i\in [n]}S_{x_i}}$, and the same $\Psi\vert_{\cup_{i\in [n]}\L2(\PP\vert_{x_i})}$ as the given spectral resolution, but satisfying further: 
\begin{enumerate}[(1)]
\item\label{proj:iv-b:i}  $H_s=H_{\pr_{x_1}(s)}\otimes\cdots\otimes  H_{\pr_{x_n}(s)}$ for $\mu$-a.e. $s\in S_{x_1\lor\cdots\lor x_n}$  (under the obvious identifications: $\sqrt{\mu(\{\emptyset_S\})}\Psi(1)_{\emptyset_S}$ serving as the unit ``$1$'' in the one-dimensional space $H_{\emptyset_S}$), the measurable structure of $(H_s)_{s\in S_{x_1\lor\cdots\lor x_n}}$ being consistent (see Theorem~\ref{thm:noise-projections}\ref{proj:vi}\ref{proj:vi:I}) with that of the $(H_s)_{s\in S_{x_i}}$, $i\in [n]$; 
\item\label{proj:iv-b:ii} if $\mu$ is such that $\mu(S_{x_1\lor\cdots\lor x_n})=1$ and such that the $\pr_{x_i}$, $i\in [n]$, are independent under $\mu(\cdot\cap S_{x_1\lor\cdots \lor x_n})$, then also $$\Psi(f_1\cdots f_n)=\mu(\{\emptyset_S\})^{\frac{n-1}{2}}(\Psi(f_1)\circ\pr_{x_1})\otimes\cdots \otimes (\Psi(f_n)\circ\pr_{x_n})\mathbbm{1}_{S_{x_1\lor\cdots\lor x_n}}$$ and (hence) $$\Vert \Psi(f_1\cdots f_n)\Vert=\mu(\{\emptyset_S\})^{\frac{n-1}{2}} (\Vert\Psi(f_1)\Vert\circ\pr_{x_1})\cdots (\Vert\Psi(f_n)\Vert\circ\pr_{x_n})\mathbbm{1}_{S_{x_1\lor\cdots\lor x_n}}$$ a.e.-$\mu$ for $(f_1,\ldots,f_n)\in \L2(\PP\vert_{x_1})\times\cdots \times \L2(\PP\vert_{x_n})$.
\end{enumerate}
\item\label{proj:iv-d}   If $\mu$ is such that $\mu(S_{x_1\lor\cdots\lor x_n})=1$ and such that the $\pr_{x_i}$, $i\in [n]$, are independent under $\mu(\cdot\vert S_{x_1\lor\cdots \lor x_n})$, then $(\pr_{x_1},\ldots, \pr_{x_n})$ is a mod-$0$ isomorphism between $\mu(\cdot\cap S_{x_1\lor\cdots \lor x_n})$ and $\mu( \cdot \vert S_{x_1})\times\cdots \times\mu(\cdot \vert S_{x_n})$. 
 \item\label{proj:iv-c} For $(A_1,\ldots,A_n)\in \Sigma\vert_{S_{x_1}}\times\cdots \times \Sigma\vert_{S_{x_n}}$, up to the natural unitary equivalence, $$H(A_1)\otimes \cdots\otimes H(A_n)=H(\pr_{x_1}^{-1}(A_1)\cap \cdots\cap\pr_{x_n}^{-1}(A_n))\cap \L2(\PP\vert_{x_1\lor\cdots\lor x_n}).$$
\end{enumerate}
\end{enumerate}
\end{corollary}
\begin{remark}
With reference to Item~\ref{proj:iv}\ref{proj:iv-a} it may be natural to  note in passing the ``additive'' version: $\mu_{f_1+\cdots +f_n}=\mu_{f_1}+\cdots+\mu_{f_n}$ provided all the $f_i$, $i\in [n]$, except maybe one, are of zero mean (just check it on the essentially generating $\pi$-system $\{S_u:u\in B\}$). In fact, for $\{g_1,g_2\}\subset\L2(\PP)$, the additive decomposition $\mu_{g_1+g_2}=\mu_{g_1}+\mu_{g_2}$ is  evidently equivalent to  $g_1$ and $g_2$ being orthogonal conditionally on each member of $B$ (which is certainly true if for some independent $x_1$ and $x_2$ from $\Lattice$ that commute with members of $B$, $g_1$ is $x_1$-measurable, $g_2$ is $x_2$-measurable, with one of $g_1$, $g_2$ being of zero mean). 
\end{remark}
 The reader may better appreciate the preceding remark on a (counter)example.
 \begin{example}
 Return to Example~\ref{example:wiener}. We put  $f_1:=((W_2-W_1)+(W_1-W_0))((W_{-1}-W_{-2})+(W_{0}-W_{-1}))\in H^{(2)}$ and $f_2:=((W_2-W_1)-(W_1-W_0))((W_{-1}-W_{-2})-(W_{0}-W_{-1}))\in H^{(2)}$. The random variables $f_1$ and $f_2$ are orthogonal of zero mean; $\mu_{f_1+f_2}$ and $\mu_{f_1}+\mu_{f_2}$ have the same mass. Nevertheless, $\mu_{f_1+f_2}\ne \mu_{f_1}+\mu_{f_2}$; $f_1$ and $f_2$ are not orthogonal given $N_{(1,2)\cup (-2,-1)}$. 
 \end{example}
\begin{remark}\label{remark:choose-equivalent}
Still with reference to Item~\ref{proj:iv}\ref{proj:iv-a}: the $f_i$, $i\in [n]$, may be chosen so that $\mu_{f_i}\sim \mu$ on $S_{x_i}$, $i\in [n]$, in which case $\mu_{f_1\cdots f_n}\sim\mu$ on $S_{x_1\lor\cdots\lor x_n}$.
\end{remark}
\begin{proof}
\ref{proj:ii-}. By standardness we may assume that $(E,\mathcal{E})$ is $\mathbb{R}$ with its Borel $\sigma$-field. Start with $\phi$ taking on only countably many values with positive $\mu$-measure, in that case it follows at once from Theorem~\ref{thm:noise-projections}\ref{proj:i}. Then suppose that $\phi_1=\phi_2$ a.e.-$\mu$ with $\phi_1$ and $\phi_2$ both being $\Sigma/\mathcal{E}$-measurable (now the actual versions, not equivalence classes). Let also $\pr_x^{(1)}$ and $\pr_x^{(2)}$ be any two versions of $\pr_x$. Then $(2^{-n}\lfloor 2^n \phi_1\rfloor)\circ \pr_x^{(1)}=(2^{-n}\lfloor 2^n \phi_2\rfloor)\circ \pr_x^{(2)}$ a.e.-$\mu$ for each $n\in \mathbb{N}$; passing to the limit $n\to \infty$ we get the desired conclusion.

\ref{proj:ii--}. First, any version of $\pr_x$ belongs to $\Sigma/\Sigma$, a fortiori to $\Sigma/\Sigma^0$, where $\Sigma^0$ is a standard $\sigma$-field whose $\overline{\mu\vert_{\Sigma^0}}$-completion is $\Sigma$; therefore $\pr_x\circ \pr_y$ is without ambiguity up to a.e.-$\mu$ equality by Item~\ref{proj:ii-}. Besides,  $\pr_x\circ \pr_y$ is such that $(\pr_x\circ \pr_y)^{-1}(E)$ is determined uniquely up to a.e.-$\mu$ equality for $E\in \Sigma/_\mu$ just because such a property holds for both $\pr_x$ and $\pr_y$ (and because $\pr_x\in \Sigma/\Sigma$). Further, $(\pr_x\circ \pr_y)^{-1}(S_z)=\pr_y^{-1}(\pr_x^{-1}(S_z))=\pr_y^{-1}(S_{z\lor x'})=S_{z\lor x'\lor y'}=S_{z\lor (x\land y)'}$ a.e.-$\mu$ for all $z\in B$, which entails in particular that $\pr_x\circ \pr_y\in \Sigma/\Sigma$, and hence, by Theorem~\ref{thm:noise-projections}\ref{proj:i}, that $\pr_{x\land y}=\pr_x\circ \pr_y$ a.e.-$\mu$.

\ref{proj:iv}\ref{proj:iv-a}. Initially we assume that $n=2$ and that  $x_2=x'$, where $x:=x_1$. We let  $(f,f')\in \L2(\PP\vert_x)\times \L2(\PP\vert_{x'})$, $\PP[f^2]=\PP[(f')^2]=1$, and show that  $(\pr_x,\pr_{x'})_\star \mu_{ff'}= \mu_f\times \mu_{f'}$. Remark that the probability measures $\mu_f$, $\mu_{f'}$ and $\mu_{ff'}$ are absolutely continuous w.r.t. $\mu(\cdot \cap S_x)$, $\mu(\cdot \cap S_{x'})$ and $\mu$, respectively. The remainder of the argument for the validity of  $(\pr_x,\pr_{x'})_\star \mu_{ff'}= \mu_f\times \mu_{f'}$ is a direct computation on the essentially generating $\pi$-system $\{S_y\times S_z:(y,z)\in B\times B\}$:
\begin{align*}
&\mu_{ff'}((\pr_x,\pr_{x'})\in S_y\times S_z)=\mu_{ff'}(\pr_x^{-1}(S_y)\cap\pr_{x'}^{-1}(S_z))=\mu_{ff'}(S_{y\lor x'}\cap S_{x\lor z})=\mu_{ff'}(S_{(y\lor x')\land(x\lor z)})\\
&=\mu_{ff'}(S_{(y\land x)\lor (z\land x')})=\PP[\PP[ff'\vert (y\land x)\lor (z\land x')]^2]=\PP[\PP[f\vert y\land x]^2]\PP[\PP[{f'}\vert z\land x']^2]\\
&=\mu_f(S_{y\land x})\mu_{f'}( S_{z\land x'})=\mu_f(S_{y}\cap S_x)\mu_{f'}( S_{z}\cap S_{x'})=\mu_f(S_{y})\mu_{f'}( S_{z})=(\mu_f\times \mu_{f'})(S_y\times S_z)
\end{align*}
 for all $(y,z)\in B\times B$. Further, according to Theorem~\ref{thm:noise-projections}\ref{proj:ii} $(\pr_x,\pr_{x'}):S\to S^2$ is mod-$0$ bijective between $S$ and $S_x\times S_{x'}$ relative to $\mu$ and $\mu(\cdot\cap S_x)\times \mu(\cdot \cap S_{x'})$, respectively. Together with the preceding it means that $(\pr_x,\pr_{x'})$ is a mod-$0$ isomorphism of $\mu_{ff'}$ onto $\mu_f\times \mu_{f'}$, and a special case of \ref{proj:iv}\ref{proj:iv-a} has been established. 

For the general case when $n=2$ we reduce to the case $x_1\lor x_2=1_\PP$. Indeed  $\tilde{B}:=B_{x_1\lor x_2}$ is a noise Boolean algebra under $\tilde{\PP}:=\PP\vert_{x_1\lor x_2}$ with spectral measure $\tilde{\mu}:=\mu(\cdot \cap S_{x_1\lor x_2})$ etc. to which the claim when $x_1\lor x_2=1_\PP$ can be applied: one gets $\tilde \pr_{y}$ associated to $y\in \tilde{B}$, but because of uniqueness $\tilde\pr_y=\pr_y$ a.e.-$\tilde{\mu}$. The reduction follows.

 Now we apply an inductive argument in ``$n$''; ``$n=1$'' is trivial and ``$n=2$'' is handled above. Suppose it is true for some $n\geq 2$ and we prove it for ``$n+1$''. By  the inductive hypothesis  $(\pr_{x_1},\ldots,\pr_{x_{n-1}},\pr_{x_n\lor x_{n+1}}):S\to S^n$ is mod-$0$ bijective between $S$ and $S^{n-1}\times S$ relative to $\mu(\cdot\cap {S_{x_1\lor\cdots\lor x_n\lor x_{n+1}}})$ and $\mu(\cdot \cap S_{x_1})\times\cdots\times \mu(\cdot \cap S_{x_{n-1}})\times \mu(\cdot \cap S_{x_{n}\lor x_{n+1}})$, respectively. But also $(\pr_{[n-1]},\pr_{x_n}\circ \pr_n,\pr_{x_{n+1}}\circ \pr_n):S^n\to S^{n+1}$ is mod-$0$ bijective between  $S^{n-1}\times S$ and $S^{n-1}\times S^2$ relative to $\mu(\cdot \cap S_{x_1})\times\cdots\times \mu(\cdot \cap S_{x_{n-1}})\times\mu(\cdot \cap S_{x_n\lor x_{n+1}})$ and $\mu(\cdot \cap S_{x_1})\times\cdots\times \mu(\cdot \cap S_{x_{n-1}})\times \mu(\cdot \cap S_{x_{n}})\times  \mu(\cdot \cap S_{x_{n+1}})$. Taking compositions and applying Item~\ref{proj:ii--} we get that $(\pr_{x_1},\ldots,\pr_{x_{n+1}}):S\to S^{n+1}$ is mod-$0$ bijective between $S$  and $S^{n+1}$ relative to  $\mu(\cdot\cap {S_{x_1\lor\cdots\lor x_n\lor x_{n+1}}})$ and $\mu(\cdot \cap S_{x_1})\times\cdots\times\mu(\cdot \cap S_{x_{n+1}})$. Furthermore, by a similar token, 
\begin{align*}
&(\pr_{x_1},\ldots,\pr_{x_{n+1}})_\star\mu_{f_1\cdots f_{n+1}}\\
&=(\pr_{[n-1]},\pr_{x_n}\circ \pr_n,\pr_{x_{n+1}}\circ \pr_n)_\star((\pr_{x_1},\ldots,\pr_{x_{n-1}},\pr_{x_n\lor x_{n+1}})_\star \mu_{f_1\cdots f_{n-1}(f_nf_{n+1})})\\
&=(\pr_{[n-1]},\pr_{x_n}\circ \pr_n,\pr_{x_{n+1}}\circ \pr_n)_\star (\mu_{f_1}\times \cdots \times \mu_{f_{n-1}}\times \mu_{f_nf_{n+1}})=\mu_{f_1}\times \cdots \times \mu_{f_{n-1}}\times (\pr_{x_n},\pr_{x_{n+1}})_\star\mu_{f_nf_{n+1}}\\
&=\mu_{f_1}\times \cdots \times \mu_{f_{n-1}}\times \mu_{f_n}\times \mu_{f_{n+1}}
\end{align*}
 (up to natural identifications). Altogether the case ``$n+1$'' is established.

\ref{proj:iv}\ref{proj:iv-b}.  Inductively by Theorem~\ref{thm:noise-projections}\ref{proj:vi}\ref{proj:vi:I} and by Item~\ref{proj:ii--} $B$ admits a spectral resolution leaving invariant the stipulated objects and having $H_s=(\cdots(H_{\pr_{x_1}(s)}\otimes\cdots)\otimes  H_{\pr_{x_n}(s)})\otimes H_{\pr_{(x_1\lor\cdots \lor x_n)'}(s)}$ for $\mu$-a.e. $s\in S$, in particular $H_s=H_{\pr_{x_1}(s)}\otimes\cdots\otimes  H_{\pr_{x_n}(s)}\otimes H_{\emptyset_S}$ for $\mu$-a.e. $s\in S_{x_1\lor\cdots \lor x_n}$ and hence also $H_s=H_{\pr_{x_1}(s)}\otimes\cdots\otimes  H_{\pr_{x_n}(s)}$ for $\mu$-a.e. $s\in S_{x_1\lor\cdots \lor x_n}$. This gives \ref{proj:iv-b:i}. For \ref{proj:iv-b:ii} let $f\in \L2(\PP\vert_{x_1\lor\cdots x_n})$ be such that $\mu(\cdot\cap S_{x_1\lor\cdots \lor x_n})=\mu_f$ and pick $g\in \L2(\PP\vert_{(x_1\lor\cdots\lor x_n)'})$ with $\PP[g^2]=1$ such that $\mu\sim \mu_g$ on $S_{(x_1\lor\cdots\lor x_n)'}$.  Then $\mu_{fg}$ is a probability equivalent to $\mu$, which agrees with $\mu(\cdot\cap S_{x_1\lor\cdots\lor x_n})$ on $\Sigma_{x_1}\lor\cdots\lor\Sigma_{x_n}=\Sigma_{x_1\lor\cdots\lor x_n}$ and under which therefore $\pr_{x_1},\ldots,\pr_{x_n},\pr_{(x_1\lor\cdots \lor x_n)'}$ are independent. It remains to apply inductively Theorem~\ref{thm:noise-projections}\ref{proj:vi}\ref{proj:vi:II}.

\ref{proj:iv}\ref{proj:iv-d}. Using the same trick as in the immediately preceding paragraph,  Theorem~\ref{thm:noise-projections}\ref{proj:iii} gives the marginal laws $(\pr_{x_i})_\star\mu(\cdot\vert S_{x_1\lor\cdots \lor x_n})=\mu(\cdot\vert S_{x_i})$, $i\in [n]$. It remains to take into account Item~\ref{proj:iv}\ref{proj:iv-a}.


\ref{proj:iv}\ref{proj:iv-c}. 
Suppose the claim holds for ``$n$'' and we make the argument for ``$n+1$'' (the claim is trivial for $n=1$, so by induction we will be home free). We have $\pr_{x_1}^{-1}(A_1)\cap\cdots\cap \pr_{x_n}^{-1}(A_n)=(\pr_{x_1\lor\cdots \lor x_n})^{-1}(\pr_{x_1}^{-1}(A_1)\cap\cdots\cap \pr_{x_n}^{-1}(A_n)\cap S_{x_1\lor\cdots\lor x_n})$ a.e.-$\mu$, and also $\pr_{x_{n+1}}^{-1}(A_{n+1})=\pr_{x_{n+1}\lor (x_1\lor\cdots\lor x_{n+1})'}^{-1}(A_{n+1})$ a.e.-$\mu$ on $S_{x_1\lor\cdots\lor x_{n+1}}$ due to Item~\ref{proj:ii--}. Therefore, by Theorem~\ref{thm:noise-projections}\ref{proj:vii} and the inductive hypothesis, up to the natural unitary eqivalences, 
\begin{align*} 
&H(\pr_{x_1}^{-1}(A_1)\cap\cdots\cap \pr_{x_n}^{-1}(A_n)\cap \pr_{x_{n+1}}^{-1}(A_{n+1}))\cap \L2(\PP\vert_{x_1\lor\cdots\lor x_{n+1}})\\
&=H(\pr_{x_1}^{-1}(A_1)\cap\cdots\cap \pr_{x_n}^{-1}(A_n)\cap \pr_{x_{n+1}}^{-1}(A_{n+1}))\cap H(S_{x_1\lor \cdots \lor x_{n+1}})\\
&=H(\pr_{x_1}^{-1}(A_1)\cap\cdots\cap \pr_{x_n}^{-1}(A_n)\cap \pr_{x_{n+1}}^{-1}(A_{n+1})\cap S_{x_1\lor\cdots \lor x_{n+1}})\\
&=H((\pr_{x_1\lor\cdots \lor x_n})^{-1}(\pr_{x_1}^{-1}(A_1)\cap\cdots\cap \pr_{x_n}^{-1}(A_n)\cap S_{x_1\lor\cdots\lor x_n})\cap \pr_{x_{n+1}\lor (x_1\lor\cdots\lor x_{n+1})'}^{-1}(A_{n+1}))\\
&=H([\pr_{x_1}^{-1}(A_1)\cap\cdots\cap \pr_{x_n}^{-1}(A_n)]\cap  S_{x_1\lor \cdots \lor x_n})\otimes H(A_{n+1})\\
&=(H(A_1)\otimes \cdots\otimes H(A_n))\otimes H(A_{n+1})=H(A_1)\otimes \cdots\otimes H(A_{n+1}),
\end{align*}
which is what we wanted.
\end{proof}

\begin{remark}
Let $\{y,z\}\subset B$, $y\land z=0_\PP$. According to Item~\ref{proj:iv}\ref{proj:iv-a}  of the preceding corollary  $S_{y\lor z}$ is ($\mu$-essentially) in a ``canonical'' one-to-one and onto correspondence with $S_y\times S_z$ and is thus generally much larger than the union of  $S_y$ and $S_z$, which correspond to $S_y\times \{\emptyset_S\}$ and $\{\emptyset_S\}\times S_z$, respectively. Indeed $S_{y\lor z}=S_y\cup S_z$ a.e.-$\mu$ iff $y=0_\PP$ or $z=0_\PP$. However, $S_{y\lor z}$ is the $\mu$-a.e. disjoint union of $S_y$ and $S_z$ on $\{K=1\}$, as we will see shortly in Corollary~\ref{corollary:partition-K=1}.
\end{remark}

\begin{remark}
We have for all intents and purposes noted the following in the proof. Let $x\in B$. If we use $\pr^x_y$ to denote the noise projection corresponding to $y\in B_x$ under $\PP\vert_x$ relative to the spectral decomposition $((S_x,\Sigma\vert_{S_x},\mu\vert_{S_x});\Psi\vert_{\L2(\PP\vert_x)}\vert_{S_x})$ then $\pr^x_y=\pr_y\vert_{S_x}$ a.e.-$\mu\vert_{S_x}$. (There is ambiguity in $S_x$, but it does not matter; it is true for any version thereof.)
\end{remark}
As the reader may have noticed (or speculated, or indeed felt as if it had been overlooked) Items~\ref{proj:iv}\ref{proj:iv-a} and \ref{proj:iv}\ref{proj:iv-d} of Corollary~\ref{prop:noise-proj-bis} are just two sides of the same coin:
\begin{proposition}\label{proposition:independence-for-given}
Let $n\in \mathbb{N}$ and let $(x_1,\ldots,x_n)$ be an independency in $B$. The following are equivalent. 
\begin{enumerate}[(A)]
\item\label{independence-for-given:a}  $\mu(S_{x_1\lor\cdots\lor x_n})=1$ and  $\pr_{x_1},\ldots,\pr_{x_n}$ are independent under $\mu(\cdot \cap S_{x_1\lor\cdots\lor x_n})$.
\item\label{independence-for-given:b} $\mu(\cdot \cap S_{x_1\lor\cdots\lor x_n})=\mu_{f_1\cdots f_n}$ for some $(f_1,\ldots,f_n)\in \L2(\PP\vert_{x_1})\times\cdots\times \L2(\PP\vert_{x_n})$ with $\PP[(f_1)^2]=\cdots=\PP[(f_n)^2]=1$.
\end{enumerate}
\end{proposition}
\begin{proof}
\ref{independence-for-given:a} is implied by \ref{independence-for-given:b} according to Corollary~\ref{prop:noise-proj-bis}\ref{proj:iv}\ref{proj:iv-a}. Suppose now that \ref{independence-for-given:a} holds true. There is $g\in \L2(\PP\vert_{x_1\lor\cdots \lor x_n})$ such that $\mu(\cdot \cap S_{x_1\lor\cdots\lor x_n})=\mu_g$ (just because, thanks to the existence of orthonormal frames \cite[Proposition~II.1.4.1(II)]{dixmier1981neumann}, there is a such a $g$ that satisfies $\Vert\Psi(g)\Vert=1$ a.e.-$\mu$ on $S_{x_1\lor \cdots \lor x_n}$).  To see that \ref{independence-for-given:b} holds true it is sufficient to check that $\mu_g=\mu_{f_1\cdots f_n}$, where $f_i:=\frac{\PP[g\vert x_i]}{\Vert\PP[g\vert x_i]\Vert}$ for $i\in [n]$. Then it is enough to establish the latter equality on $S_{u_1\lor \cdots \lor u_n}$ for $(u_1,\ldots,u_n)\in B_{x_1}\times \cdots\times B_{x_n}$. Now, on the one hand, 
$$\mu_{f_1\cdots f_n}(S_{u_1\lor \cdots \lor u_n})=\frac{\PP[\PP[g\vert u_1]^2]}{\PP[\PP[g\vert x_1]^2]}\cdots \frac{\PP[\PP[g\vert u_n]^2]}{\PP[\PP[g\vert x_n]^2]}.$$ On the other hand, by the independence of the $\Sigma_{x_1},\ldots,\Sigma_{x_n}$ under $\mu_g$,
\begin{align*}
\mu_g(S_{u_1\lor\cdots \lor u_n})&=\mu_g(S_{u_1\lor x_1'}\cap \cdots \cap S_{u_n\lor x_n'})=\mu_g(S_{u_1\lor x_1'})
 \cdots \mu_g(S_{u_n\lor x_n'})\\
&=\PP[\PP[g\vert u_1\lor x_1']^2]\cdots \PP[\PP[g\vert u_n\lor x_n']^2].
\end{align*} 
Taking $u_j=0_\PP$ for all $j\in [n]\backslash \{i\}$, $i\in [n]$, in the preceding and substituting back for the expressions obtained in this way for the $\PP[\PP[g\vert u_i\lor x_i']^2]$, $i\in [n]$; looking finally at the case when $u_i=x_i$ for all $i\in [n]$ and taking into account that $\mu_g(S_{x_1\lor \cdots \lor x_n})=1$, establishes that this is indeed the same as the expression that was procured for $\mu_{f_1\cdots f_n}(S_{u_1\lor\cdots \lor u_n})$.
\end{proof}

We prove next that it is sufficient to check the defining property of a noise projection on a dense noise Boolean subalgebra. 

\begin{proposition}\label{proposition:enough-on-dense-noise-proj}
Let $x\in B$ and let $B_0$ be a noise Boolean algebra dense in $B$. Suppose $p\in \Sigma/\Sigma$ is such that for $E\in \Sigma/_\mu$, $p^{-1}(E)$ is uniquely determined up to a.e.-$\mu$ equality (vis-\`a-vis the choice of  a representative of the equivalence class of $E$) and such that for all $y\in B_0$,  $p^{-1}(S_y)=S_{y\lor x'}$ a.e.-$\mu$. Then $p=\pr_x$ a.e.-$\mu$.
\end{proposition}
\begin{proof}
The collection $\mathcal{D}:=\{A\in \Sigma:p^{-1}(A)=\pr_x^{-1}(A)\text{ a.e.-$\mu$}\}$ is a $\sigma$-field on $S$ containing the essentially generating system $\{S_y:y\in B_0\}$. Therefore $\mathcal{D}=\Sigma$, hence by the very definition of $\pr_x$ (or rather the unique existence which makes this definition possible), $\pr_x=p$ a.e.-$\mu$. 
\end{proof}

\begin{corollary}
For each $x\in \overline{B}$ there is a $\mu$-a.e. unique $\pr_x\in \Sigma/\Sigma$  such that for $E\in \Sigma/_\mu$, $\pr_x^{-1}(E)$ is uniquely determined up to a.e.-$\mu$ equality (vis-\`a-vis the choice of  a representative of the equivalence class of $E$) and such that for all $y\in B$ a.e.-$\mu$, $\pr_x^{-1}(S_y)=S_{y\lor x'}$, in which case the preceding obtains for all $y\in\overline{B}$. 
\end{corollary}
\begin{proof}
$B$ is dense in the noise Boolean algebra $\overline{B}$.
\end{proof}
\begin{definition}
Thus the definition of the noise projections $\pr_x$ is extended from $x\in B$ to $x\in \overline{B}$. 
\end{definition}
\begin{remark}
It means in particular that the $\pr_x$, $x\in B$, are the same if defined relative to $\overline{B}$ or relative to $B$ (with the same spectral resolution). Also, as far as the above results concerning the noise projections are concerned, one can always apply them to $\overline{B}$ in lieu of $B$ (again with the same spectral resolution) and get the corresponding statements for the whole family $\pr_x$, $x\in \overline{B}$. The same for the results concerning noise projections to follow. 
\end{remark}

%

We proceed to make explicit some connections between the counting map $K$ and the noise projections. The following simple result (which does not concern projections per se) will be useful. 

\begin{proposition}\label{prop;K-finite}
 $K_B=\text{$\uparrow$-$\lim$}_{n\in \mathbb{N}}K_{b_n}$ a.e.-$\mu$ whenever $(b_n)_{n\in \mathbb{N}}$ is a $\uparrow$ sequence in $\mathfrak{F}_B$ with $B_0:=\cup_{n\in \mathbb{N}}b_n$ dense in $B$.
\end{proposition}
Before tackling the proof note that at this point we only know that a $\uparrow$ sequence $(b_n)_{n\in \mathbb{N}}$ in $\mathfrak{F}_B$ exists (the union of the members of which may or may not be dense in $B$ and) for which the a.e.-$\mu$ equality $K_B=\text{$\uparrow$-$\lim$}_{n\in \mathbb{N}}K_{b_n}$ prevails. What is not obvious is that this actually holds for every such sequence (the union of the members of which is dense in $B$).

\begin{proof}
Let $b\in \mathfrak{F}_B$.  We need to show that $K_b\leq \sup_{n\in \mathbb{N}}K_{b_n}$ a.e.-$\mu$. It will suffice to argue that for all $k\in \mathbb{N}$, $\{K_b\geq k\}\subset \cup_{n\in \mathbb{N}}\{K_{b_n}\geq k\}$ a.e.-$\mu$ (of course it is only non-trivial for $k\geq 2$ and one can stop checking with $k=\vert\at(b)\vert$). For a non-empty finite  $I\subset B$ with $0_\PP\notin I$ introduce $Z_I:=S_{\lor I}\backslash\cup_{J\in 2^I\backslash \{I\}}S_{\lor J}$. Then it is enough to check as follows: whenever $I\subset b$ is an independency of size $k\in \mathbb{N}$ consisting of atoms of $b$,  then for $\mu$-a.e. $s\in Z_{I}$ there is $n\in \mathbb{N}$ and an independency $J\subset b_n\backslash \{0_\PP\}$ of size $k$ such that $s\in Z_{J}$. In other words we should check that for any countable noise Boolean algebra $A$ which is dense in $B$, for all $k\in \mathbb{N}$, whenever $I\subset B\backslash \{0_\PP\}$ is an independency of size $k$, then $Z_{I}\subset \cup_{J\in ({A\choose k})}Z_J$ a.e.-$\mu$, where $({A\choose k})$ are all the independences of size $k$ contained in $A\backslash \{0_\PP\}$. Note however that if $J=\{j_1,\ldots,j_k\}$ is merely a subset of $A$ of size $k$ such that for each $l\in [k]$, $[j_l]:=j_l\land (\land_{i\in [l-1]}j_i')\ne 0_\PP$ (denote the collection of all such $J$ by ${A\choose k}^\circ$), then $Z_J\subset Z_{[J]}$ a.e.-$\mu$ and $[J]=\{[j]:j\in J\}$ is  an independency of size $k$ contained in $A\backslash \{0_\PP\}$. In consequence it is actually enough to argue that $Z_{I}\subset \cup_{J\in {A\choose k}^\circ}Z_J$ a.e.-$\mu$. Because convergence of $\sigma$-fields in $B$ corresponds to convergence locally in $\mu$-measure of the associated spectral sets the preceding inclusion now follows by applying Proposition~\ref{proposition:approximating-partition-of-unity}. 
\end{proof}

\begin{corollary}\label{corollary:B-overline-B-same-K}
$K_{\overline{B}}=K_B$ in the spectral resolution of $B$ (which is also a spectral resolution of $\overline{B}$), hence $B$ and $\overline{B}$ have the same chaos spaces $H^{(k)}$, $k\in \mathbb{N}_0$. \qed
\end{corollary}

\begin{proposition}\label{proposition:K-and-projections}
Let $P$ be a partition of unity of $B$. Then $K= \sum_{p\in P}K(\pr_p)$ a.e.-$\mu$, hence $K\circ \pr_x=\sum_{p\in P_x}K(\pr_p)$ a.e.-$\mu$ for all $x\in b(P)$ (and the same with $K$ replaced by $K_b$ for $b\in \mathfrak{F}_B$, $b\supset b(P)$).
\end{proposition}
\begin{proof}
Let $b\in \mathfrak{F}_B$ with $b(P)\subset b$.  Then $\cap_{p\in P}\pr_p^{-1}(S_{x_p})=S_{\lor_{p\in P}x_p}$ a.e.-$\mu$ for $x_p\in b_p$, $p\in P$, hence $K_b=\sum_{p\in P}K_b(\pr_p)$ a.e.-$\mu$. Proposition~\ref{prop;K-finite} allows to conclude, for we may take a sequence $(b_n)_{n\in \mathbb{N}}$ in $\mathfrak{F}_B$ such that $\cup_{n\in \mathbb{N}}b_n$ is dense in $B$ and in addition $b_1=b(P)$.
\end{proof}

\begin{corollary}\label{corollary:partition-K=1}
Let $P$ be a partition of unity of $B$.  Then $S=\sqcup_{p\in P}S_p$ a.e.-$\mu$ on $\{K=1\}$.
\end{corollary}
\begin{proof}
We compute 
\begin{align*}
&\{K=1\}=\left\{\sum_{p\in P}K(\pr_p)=1\right\}=\cup_{p\in P}\left(\{K(\pr_p)=1\}\cap \left(\cap_{q\in P\backslash \{p\}}\{K(\pr_q)=0\}\right)\right)\\
&=\cup_{p\in P}\left(\{K(\pr_p)=1\}\cap \left(\cap_{q\in P\backslash \{p\}}\pr_q^{-1}(\{K=0\})\right)\right)=\cup_{p\in P}\left(\{K(\pr_p)=1\}\cap \left(\cap_{q\in P\backslash \{p\}}\pr_q^{-1}(S_{0_\PP})\right)\right)\\
&=\cup_{p\in P}\left(\{K(\pr_p)=1\}\cap \left(\cap_{q\in P\backslash \{p\}}S_{q'}\right)\right)=\cup_{p\in P}\left(\{K(\pr_p)=1\}\cap S_p\right)
\end{align*} a.e.-$\mu$. Taking intersection with $S_p$, noting that $$(S_p\cap S_q)\cap \{K=1\}=S_{p\land q}\cap \{K=1\}=S_{0_\PP}\cap \{K=1\}=\{K=0\}\cap \{K=1\}=\emptyset$$ a.e.-$\mu$ for $q\ne p$ from $P$, we get $\{K=1\}\cap S_p=\{K(\pr_p)=1\}\cap S_p$ a.e.-$\mu$, and thus the claim.
\end{proof}

\begin{corollary}\label{projections-and-K-misellany}
Let $b\in \mathfrak{F}_B$ and $x\in b$. For $\mu$-a.e. $s$:
\begin{enumerate}[(i)]
\item\label{projections-and-K-misellany:0} for $a\in\at(b)$, $a\in\underline{b}(s)$ iff $\pr_a(s)\in \{K>0\}$;
\item\label{projections-and-K-misellany:i} $\underline{b}(\pr_x(s))=\underline{b}(s)\cap x$ and hence $K_b(\pr_x(s))=\vert\{a\in \at(b):a\subset \underline{b}(s)\cap x\}\vert$;
\item\label{projections-and-K-misellany:ii}  ($\pr_a(s)\in \{K=1\}$ for $a\subset x$ and $\pr_a(s)=\emptyset_S$ for $a\not\subset x$, $a\in\at(b)$) iff ($\underline{b}(s)=x$ and $K(s)=K_b(s)$).
\end{enumerate}
\end{corollary}
\begin{proof}
For $\mu$-a.e. $s$ we have as follows. 

\ref{projections-and-K-misellany:0}. If $a\in \underline{b}(s)$, then we cannot have $\pr_a(s)\in \{K=0\}$, for in that case $\pr_a(s)\in S_{0_\PP}$, i.e. $s\in S_{a'}$, hence $s\in S_{a'\cap\underline{b}(s)}$ contradicting the minimality of $\underline{b}(s)$. On the other hand, if $\pr_a(s)\in \{K>0\}$, then we cannot have $a\notin\underline{b}(s)$, for in that case $s\in S_{\underline{b}(s)}\subset S_{a'}$ and $\pr_a(s)\in S_{0_\PP}=\{K=0\}$,  a contradiction.

\ref{projections-and-K-misellany:i}. We know that $K_b(\pr_x(s))+K_b(\pr_{x'}(s))=K_b(s)$. Furthermore, $\pr_x(s)\in S_{\underline{b}(s)\cap x}$ (just because $s\in S_{\underline{b}(s)}$), hence certainly $\underline{b}(\pr_x(s))\subset \underline{b}(s)\cap x$ and $K_b(\pr_x(s))\leq \vert\{a\in \at(b):a\subset \underline{b}(s)\cap x\}\vert$. The same for $x'$ in lieu of $x$. But then we must have actually equalities.

\ref{projections-and-K-misellany:ii}. The condition is necessary, for we get   $x=\underline{b}(s)$ by Item~\ref{projections-and-K-misellany:0}, and then $K(s)=\sum_{a\in \at(b)}K( \pr_a(s))=\vert \{a\in \at(b):a\subset x\}\vert=K_b(s)$. On the other hand, if $\underline{b}(s)=x$ and $K(s)=K_b(s)$, then $\pr_a(s)=\emptyset_S$ for $a\not\subset x$, $a\in \at(b)$ by Item~\ref{projections-and-K-misellany:0}, while $\vert \{a\in \at(b):a\subset x\}\vert=K_b(s)=K(s)=\sum_{a\in \at(b)}K( \pr_a(s))=\sum_{a\in \at(b)_x}K(\pr_a(s))$ renders $\pr_a(s)\in \{K=1\}$ for $a\subset x$, $a\in \at(b)$.
\end{proof}

\begin{proposition}\label{proposition:higher-choas-partitions}
Let $P$ be a partition of unity of $B$. Then up to the natural unitary equivalence $$\otimes_{p\in P}\left(H^{(1)}\cap \L2(\PP\vert_p)\right)=H\left({P\choose \mathbf{1}}\right)\subset H^{(\vert P\vert)},$$ where ${P\choose \mathbf{1}}:=\{s\in \{K=\vert P\vert\}:\pr_p(s)\in \{K=1\}\text{ for all }p\in P\}$.
\end{proposition}
\begin{proof}
Using Corollary~\ref{prop:noise-proj-bis}\ref{proj:iv}\ref{proj:iv-c} we compute $\otimes_{p\in P}\left(H^{(1)}\cap \L2(\PP\vert_p)\right)=\otimes_{p\in P}H(\{K=1\}\cap S_p)=H(\cap_{p\in P}\pr_p^{-1}\{K=1\})$. It remains to apply Proposition~\ref{proposition:K-and-projections}.
\end{proof}

We conclude this subsection by using noise projections to succinctly express  tensorisation of operators over partitions of unity of $B$. The proof ultimately boils down to the tensorisation of independent conditioning.

\begin{proposition}\label{proposition:separate-with-projections}
Let $P$ be a partition of unity of $B$, let $f_p\in \mathrm{L}^\infty(\mu)$ and let $X_p\in \L2(\PP\vert_p)$ for $p\in P$. Then $$\Psi^{-1}\left(\left(\prod_{p\in P}(f_p\circ \pr_p)\right)\Psi\left(\prod_{p\in P}X_p\right)\right)=\prod_{p\in P}\Psi^{-1}\left(f_p\Psi(X_p)\right).$$
Or, which amount to the same thing, $$\alpha^{-1}\left(\prod_{p\in P}(f_p\circ \pr_p)\right)=\otimes_{p\in P}\left(\alpha^{-1}(f_p)\vert_{\L2(\PP\vert_p)}\right)$$ up to the natural unitary equivalence of $\L2(\PP)$ and $\otimes_{p\in P}\L2(\PP\vert_p)$.
\end{proposition}
\begin{proof}
Suppose $f_p=\mathbbm{1}_{S_{u_p}}$ a.e.-$\mu$ for some $u_p\in B$, all $p\in P$, in the first instance. We may assume $u_p\subset p$ for each $p\in P$. In that case the right-hand side of the displayed equality is $\prod_{p\in P}\PP[X_p\vert u_p]=\PP[\prod_{p\in P}X_p\vert \lor_{p\in P}u_p]$ a.s.-$\PP$, and because $\prod_{p\in P}f_p\circ \pr_p=\mathbbm{1}_{S_{\lor_{p\in P}u_p}}$ this is also the left-hand side. Since $\Psi$ is a unitary isomorphism both sides are stable under linear combinations and convergence locally in $\mu$-measure in each of the $f_p$, $p\in P$. A density/approximation (and inductive) argument allows to conclude.
\end{proof}

\begin{corollary}\label{corollay:defining-extended-to-AA}
If $z\in\Lattice$ is such that $\PP_z\in \AA$, then for all $x\in B$, $\PP_{(z\land x)\lor x'}\in \AA$ and $S_{(z\land x)\lor x'}=\pr_x^{-1}(S_z)$ a.e.-$\mu$.
\end{corollary}
\begin{proof}
Just compute $\alpha^{-1}(\mathbbm{1}_{\pr_x^{-1}(S_z)})=\alpha^{-1}(\mathbbm{1}_{S_z}(\pr_x)\mathbbm{1}_{S_{x'}}(\pr_{x'}))=\alpha^{-1}(\mathbbm{1}_{S_z})\vert_{\L2(\PP\vert_x)} \otimes \alpha^{-1}(\mathbbm{1}_{S_{x'}})\vert_{\L2(\PP\vert_{x'})}=\PP_{z}\vert_{\L2(\PP\vert_x)}\otimes \mathrm{id}_{\L2(\PP\vert_{x'})}=\PP_{z\land x}\vert_{\L2(\PP\vert_x)}\otimes \PP_{x'}\vert_{\L2(\PP\vert_{x'})}=\PP_{(z\land x)\lor x'}$, up to the natural unitary equivalence of $\L2(\PP)$ and $\L2(\PP\vert_x)\times\L2(\PP\vert_{x'})$.
\end{proof}
Thus the defining property of the noise projection $\pr_x$ is extended to all those $S_z$ for which $\PP_z\in \AA$, $z\in \Lattice$.

%

\subsection{Subnoises engendered by subsets of the spectrum}\label{subsection:subnoises-parts-of-spectrum}
In the previous subsection we have seen how a member $x$ of $B$ (which we may identify with the ``subnoise'' $B_x$ (under $\PP\vert_x$) of $B$) can be used to project the spectral set $S$ onto its subset $S_x$. In this subsection, conversely, we use certain subsets of $S$ to yield ``subnoises'' of $B$. 

The definition which follows may be seen  as an adaptation to the abstract setting of the one found in \cite[p.~63]{picard2004lectures} for one-dimensional factorizations.  We will provide further references to the connections with \cite{picard2004lectures} along the way.  As was indicated in the Introduction, the results of this subsection are new only on a formal level -- everything extends relatively straightforwardly from the one-dimensional setting. Though, alas and by contrast, it appears that we cannot avoid a.e.-$\mu$ qualifiers ($\mu$-equivalence classes) in any natural way, which makes things a little more involved.
\begin{definition}
Let $\mathcal{S}_0:=\{\mathbbm{1}_{S_x}:x\in B\}\subset \mathrm{L}^0(\mu)$, let $\mathcal{S}_1:=\mathrm{conv}(\mathcal{S}_0)\subset \mathrm{L}^0(\mu)$ and let finally $\mathcal{S}$ be the smallest convex subset of $\mathrm{L}^0(\mu)$ containing all $\mathbbm{1}_{S_x}$, $x\in B$, and closed under sequential a.e.-$\mu$ pointwise limits.
\end{definition}

\begin{remark}\label{rmk:S}
\leavevmode
\begin{enumerate}[(i)]
\item We get $\mathcal{S}$ by taking the a.e.-$\mu$ pointwise sequential limits of elements of $\mathcal{S}_1$. Indeed the resulting set is convex evidently, closed under a.e.-$\mu$ pointwise sequential limits due to the following argument. Passing to an equivalent measure if necessary we may assume $\mu$ is finite. For $L\in \mathbb{N}$ let $f^L=\lim_{n\to\infty}f^L_n$ a.e.-$\mu$, where each $f^L_n\in \mathcal{S}_1$, and suppose that $f^\infty=\lim_{L\to\infty}f^L$ exists a.e.-$\mu$. Then the preceding a.e.-$\mu$ convergences hold in $\mu$-measure, and the topology of convergence in measure on a finite measure space is metrizable; $f^\infty$ belongs to the closure in this topology of $\{f^L_n:(L,n)\in \mathbb{N}\times \mathbb{N}\}$, which means that there is a sequence in $\mathcal{S}_1$, whose limit is $f^\infty$ in $\mu$-measure. Passing to a subsequence if necessary we get a sequence in $\mathcal{S}_1$, whose limit is $f^\infty$ a.e.-$\mu$. We see also that to get $\mathcal{S}$, instead of  taking the a.e.-$\mu$ pointwise sequential limits of elements of $\mathcal{S}_1$, we may just as well take the closure of $\mathcal{S}_1$ in the topology of local convergence in $\mu$-measure. In particular, $S_x\in \mathcal{S}$ for all $x\in \Cl(B)$.
\item If $\phi\in \mathcal{S}$ then $\mathbbm{1}_{S_{0_\PP}}\leq \phi\leq 1$ a.e.-$\mu$.
\item\label{rmk:S:iv-}  $\mathcal{S}$ is closed under multiplication, since this property is retained in passing from $\mathcal{S}_0$ to $\mathcal{S}_1$ and then to $\mathcal{S}$. As a consequence if $\phi\in \mathcal{S}$, then a.e.-$\mu$, $\mathbbm{1}_{\{\phi=1\}}=\lim_{n\to\infty}\phi^n\in \mathcal{S}$. 
\item\label{rmk:S:iv} Let $x\in B$ and $\phi\in \mathcal{S}$. Then also $\phi\circ \pr_x\in \mathcal{S}$, furthermore $\phi\circ \pr_x\geq \phi$ a.e.-$\mu$. Indeed this is true for $\phi\in \mathcal{S}_0$ trivially, the property is preserved on passing to $\mathcal{S}_1$ and then to  $\mathcal{S}$. 
 \end{enumerate}
\end{remark}
Sets (or rather equivalence classes of sets) belonging to the family $\Sigma^*$, to be introduced presently, will emerge as being such that their associated $H$-spaces are $\L2$. We shall see this in Corollary~\ref{corollary:L2-space} following some remarks, examples and an auxilliary claim.

\begin{definition}
$\Sigma^*:=\{E\in \Sigma/_\mu:\mathbbm{1}_E\in \mathcal{S}\}\subset \Sigma/_\mu$.
\end{definition}

\begin{remark}
In parallel to some of the properties of $\mathcal{S}$  noted in Remark~\ref{rmk:S} we get the basic characteristics of $\Sigma^*$.
\begin{enumerate}[(i)]
\item $\Sigma^*=\{\{\phi=1\}:\phi\in \mathcal{S}\}$.
\item For each $x\in B$, $\pr_x\in \Sigma^*/\Sigma^*$, in the sense that $\pr_x^{-1}(E)\in \Sigma^*$ for all $E\in \Sigma^*$  (by slight abuse of notation).
\item $\Sigma^*$ is cosed under countable intersections and contains $\{S_x:x\in \Cl(B)\}$.
\item If $E\in \Sigma^*$, then $S_{0_\PP}\subset E$ a.e.-$\mu$.
\end{enumerate}
\end{remark}


\begin{example}\label{example:K<infty}
$\{K<\infty\}\in \Sigma^*$. 
%
 Indeed for $b\in \mathfrak{F}_B$ and for $\rho\in [0,1]$, $\phi_{\rho,b}:=\rho^{K_b}=\prod_{a\in \at(b)}\rho^{\mathbbm{1}_{S\backslash S_{a'}}}\in \mathcal{S}$, since $\mathcal{S}$ is closed for multiplication and since, for $a\in \at(b)$, $\rho^{\mathbbm{1}_{S\backslash S_{a'}}}=(1-\rho)\mathbbm{1}_{S_{a'}}+\rho\mathbbm{1}_{S_{1_\PP}}\in \mathcal{S}$. We get in fact $\phi_{\rho,b}=\sum_{x\in b}\rho^{\vert\at(b)_x\vert}(1-\rho)^{\vert \at(b)\vert-\vert\at(b)_x\vert}\mathbbm{1}_{S_x}$, which is another way of seeing that $\phi_{\rho,b}\in \mathcal{S}$ (indeed  $\in\mathcal{S}_1$). If also $\tilde b\in \mathfrak{F}_B$ with $b\subset \tilde{b}$, then $K_{\tilde{b}}\geq K_b$ a.e.-$\mu$, and we conclude that $\rho^{ K}\in  \mathcal{S}$ (with the understanding that $\rho^\infty$ is $1$ or $0$ according as $\rho=1$ or $\rho<1$). Letting $\rho\uparrow\uparrow 1$ gives $\mathbbm{1}_{\{K_B<\infty\}}\in \phi$. Due to Remark~\ref{rmk:S}, Items~\ref{rmk:S:iv-} and~\ref{rmk:S:iv}, we see also that $\prod_{p\in P}\rho(p)^{K(\pr_p)}=\prod_{p\in P}\rho(p)^K\circ \pr_p\in \mathcal{S}$ for any partition of unity $P$ of $B$ and any $\rho\in [0,1]^P$. 
\end{example}

\begin{proposition}\label{proposition:contractions-and-spectral}
(Cf. \cite[Lemma~5.8]{picard2004lectures}.) Let $\phi\in \mathcal{S}$, $\eta:\mathbb{R}\to \mathbb{R}$ be $1$-Lipschitz, i.e. a contraction, and $f\in \L2(\PP)$. Then $\mu_{\eta\circ f}[1-\phi]\leq \mu_f[1-\phi]$ for all $\phi\in \mathcal{S}$.
\end{proposition}
\begin{proof} 
Due to bounded convergence and linearity we may assume $\phi=\mathbbm{1}_{S_x}$ for $x\in B$. Then $\mu_f[1-\phi]=\mu_f(S\backslash S_x)=\PP[f^2]-\PP[\PP[f\vert x]^2]=\PP[\var(f\vert x)]$. But it is well-known that a contraction reduces the variance  (even conditional, by disintegration).
\end{proof}

\begin{corollary}\label{corollary:L2-space}
(Cf. \cite[Lemma~5.9]{picard2004lectures}.) If  $E\in \Sigma^*$, then $H(E)=\L2(\PP\vert_{\GG})$ for some (necessarily unique) $\GG\in \Lattice$ (which satisfies $\PP_\GG\in \AA$).
\end{corollary}
\begin{example}
Trivial. For $E=S_x$, $x\in \Cl(B)$, the corresponding $\GG$ is just $x$.
\end{example}
%
\begin{proof}
$H(E)$ is a closed linear submanifold of $\L2(\PP)$. Also,  $H(E)=\{f\in \L2(\PP):\mu_f(S\backslash E)=0\}=\{f\in \L2(\PP):\mu_f[1-\mathbbm{1}_E]=0\}$. Therefore, because the absolute value $\vert\cdot\vert:\mathbb{R}\to\mathbb{R}$ is $1$-Lipschitz it follows thanks to Proposition~\ref{proposition:contractions-and-spectral} that $H(E)$ is closed under application of the absolute value, and clearly it contains the constants ($\because$ $S_{0_\PP}\subset E$). Altogether it is enough to maintain the conclusion (the reason that $\PP_\GG\in \AA$ lies in the fact that $\PP_\GG=\pr_{H(E)}=\alpha^{-1}(\mathbbm{1}_E)$).
\end{proof}

\begin{proposition}\label{proposition:sigma-star-and-partitions}
If $E\in \Sigma^*$,  then $E=\cap_{x\in P}\pr_x^{-1}(E)$ a.e.-$\mu$ for any partition of unity $P$ of $B$. In fact for the a.e.-$\mu$ inclusion $\supset$ in the preceding it is sufficient that $H(E)=\L2(\PP\vert_\GG)$ for some $\GG\in \Lattice$.
\end{proposition}
\begin{remark}
It may be tempting to think that  $\phi=(\phi\circ\pr_x)(\phi\circ \pr_{x'})$ a.e.-$\mu$ for $\phi\in \mathcal{S}$ and $x\in B$ (because by the above it is true for $\phi$ indicators). However, even both the a.e.-$\mu$ inequalities   $\phi\geq (\phi\circ\pr_x)(\phi\circ \pr_{x'})$ and   $\phi\leq (\phi\circ\pr_x)(\phi\circ \pr_{x'})$ fail in general (easy to check for $\phi$ a non-trivial convex combination of two indicators $\mathbbm{1}_{S_u}$ \& $\mathbbm{1}_{S_v}$, $\{u,v\}\subset B$, e.g. in the context of Example~\ref{example:spectral-finite}).
\end{remark}
\begin{proof}
Let $\GG\in \Lattice$ be such that $H(E)=\L2(\PP\vert_\GG)$. Using Theorem~\ref{thm:noise-projections}\ref{proj:vii} we compute 
\begin{align*}
H(\cap_{x\in P}\pr_x^{-1}(E))&=\cap_{x\in P}H(E\cap S_x)\otimes \L2(\PP\vert_{x'})=\cap_{x\in P}\L2(\PP\vert_{(x\land \GG)\lor x'})\\
&=\L2(\PP\vert_{\land_{x\in P}(x\land \GG)\lor x'})=\L2(\PP\vert_{\lor_{x\in P}(x\land \GG)})\subset \L2(\PP\vert_\GG)=H(E).
\end{align*}
 Therefore $\cap_{x\in P}\pr_x^{-1}(E)\subset E$ a.e.-$\mu$. The converse a.e.-$\mu$ inclusion is immediate by Remark~\ref{rmk:S}\ref{rmk:S:iv}. 
\end{proof}
We may use the preceding to note that not all $\L2$-spaces $H(E)$, $E\in \Sigma$, come from $E\in \Sigma^*$ (just to dispel any idea that it might not be so).
\begin{example}
In the context of Examples~\ref{example:simplest-nonclassical} and~\ref{example:spectral-space-for-nonclassical} the sets $E_I:=\{0_\PP,y_I\}$ for odd-sized $I\in \mathcal{I}$ are such that $H(E_I)=\mathrm{lin}(1,\prod_I\xi)=\L2(\PP\vert_{\sigma(\prod_I\xi)})$ is an $\L2$-space, however $\pr_x^{-1}(E_I)\cap E_I=\{0_\PP\}$ for $x=N_b\subset y_I$, $b\in\mathcal{B}\backslash\{\emptyset\}$ finite (some such $b$ exists), hence $E_I\notin \Sigma^*$.
\end{example}
%
%
Furthermore, using Proposition~\ref{proposition:sigma-star-and-partitions} we may conclude that not only do elements of $\Sigma^*$ lead to $\L2$-spaces, but also the latter lead in turn to subnoises, in the following precise sense.
\begin{corollary}[$\Sigma^*$-subnoises]\label{corollary:sigma*-subnoises}
Suppose $E\in \Sigma^*$ and let $\GG\in \Lattice$ be such that $H(E)=\L2(\PP\vert_\GG)$ (according to Corollary~\ref{corollary:L2-space}). Then $\GG$ distributes over $B$ and (hence) $B\vert_\GG$ is a noise Boolean algebra under $\PP\vert_\GG$. Furthermore,  fixing any version of $E$, $((E,\Sigma\vert_{E},\mu\vert_{E});\Psi\vert_{\L2(H(E)}\vert_{E})$ is a spectral resolution of $B\vert_\GG$ in which the spectral set associated to $x\land \GG$ is $S_x\cap E$, 
$x\in B$.
\end{corollary}
\begin{example}
For $x\in \Cl(B)$, $x$ distributes over $B$  and $B\vert_x$ is a noise Boolean algebra under $\PP\vert_x$ with the indicated spectral resolution, which we have already noted in \ref{generalitites:subalgebra} \& \ref{generalities:spectrum-intro} when $x\in B$.
\end{example}
\begin{proof}
For the first claim, according to \ref{generalitites:subalgebra} it  suffices to establish that $H(E)=H(S_x\cap E)\otimes H(S_{x'}\cap E)$, which by Theorem~\ref{thm:noise-projections}\ref{proj:vii} is equivalent to establishing that $\pr_x^{-1}(E)\cap \pr_{x'}^{-1}(E)=E$, and this is just (a special case of) the content of Proposition~\ref{proposition:sigma-star-and-partitions}.

Concerning the second assertion, note that $\{S_x\cap E:x\in B\}$ is a $\mu\vert_{E}$-essentially generating $\pi$-system for $\Sigma\vert_{E}$ so that the linear combinations of its indicators are dense for the topology of local convergence in $\mu\vert_{E}$-measure (this gives that $\Psi\vert_{H(E)}\vert_{E}$ maps $\AA_{B\vert_\GG}$ onto, not just into the algebra of diagonalizable operators /the latter, together with the claim on the spectral sets of $B\vert_\GG$, being immediate from $\PP_\GG\in \AA$/). 
\end{proof}
\begin{remark}
It is clear from the proof that if $E$ is not necessarily from $\Sigma^*$ but merely an ``ideal'' of $\Sigma$ in the sense that $E= \pr_x^{-1}(E)\cap \pr_{x'}^{-1}(E)$ for all $x\in B$, then still the same conclusion of Corollary~\ref{corollary:sigma*-subnoises} obtains (but now $H(E)=\L2(\PP\vert_\GG)$ for a $\GG\in \Lattice$ is an extra assumption), cf. \cite[Theorem 6.15]{picard2004lectures}.
\end{remark}
Certain elements of $\mathcal{S}$ may be procured starting from a probability $\nu$ on $\Sigma$ absolutely continuous w.r.t. $\mu$. Before being able to make proper sense of this  we need to define what it means for two spectral points to have an empty intersection. 
\begin{definition}
We define a binary relation on $S$ up to $\mu^2$-a.e. equality precisely as the $\mu^2$-essential union of the sets $S_x\times S_{x'}$, $x\in B$. We write $s\cap_Su=\emptyset_S$ to indicate that $(s,u)$ belongs to this binary relation. Thus, for $\mu$-a.e. $s$, $\{s\cap_S\cdot=\emptyset_S\}=\{u\in S:s\cap_Su=\emptyset_u\}$ is equal a.e.-$\mu$ to the $\mu$-essential union of the sets $S_{\underline{b}(s)'}$, $b\in \mathfrak{F}_B$.
\end{definition}

\begin{proposition} \label{proposition:S-from-probability} (Cf. \cite[Lemma 6.12]{picard2004lectures}) Let $\nu$ be a probability on $(S,\Sigma)$ absolutely continuous w.r.t. $\mu$. The map $\phi^\nu:S\to [0,1]$, given by $$\phi^\nu(s):=\left[\mu\text{-}\esssup_{b\in \mathfrak{F}_B}(S\ni u\mapsto \nu(S_{\underline{b}(u)'}))\right](s)=\nu(s\cap_S\cdot=\emptyset_S)\text{ a.e.-$\mu$ in $s$},$$ belongs to $\mathcal{S}$. 
\end{proposition}
\begin{remark}\label{phi}
\leavevmode
\begin{enumerate}[(i)]
\item $\phi^\nu=1$ a.e.-$\mu$ on $S_{0_\PP}$, as it should be.
\item\label{phi:2} If a probability $\nu$ is absolutely continuous w.r.t. $\mu$, then by Radon-Nikodym we have $\nu=\Vert \Psi(f)\Vert^2\cdot \mu=\mu_f$ for an $f\in \L2(\PP)$ with $\PP[f^2]=1$. So actually the $\phi^\nu$ got in the preceding lemma are precisely those of the form $\phi_f:=\mu\text{-}\esssup_{b\in\mathfrak{F}_B}\PP[\PP[f\vert \underline{b}']^2]$ for an $f\in \L2(\PP)$ with $\PP[f^2]=1$ (or in terms of influences: $1-\phi_f=\mu\text{-}\essinf_{b\in\mathfrak{F}_B}\inff_{\underline{b}}(f)$ a.e.-$\mu$).
\end{enumerate}
\end{remark}
\begin{example}
For the constant $f=1$ a.s.-$\PP$ (or $f =-1$ a.s.-$\PP$)  we get $\phi_f=1$ a.e.-$\mu$; for such $f$, $\mu_f$ is concentrated on $S_{0_\PP}$. 
\end{example}
\begin{example}
Let $x\in  B^\circ$ and $f\in \L2(\PP\vert_x)^\circ$, $\PP[f^2]=1$; $\mu_f$ is concentrated on the atom $P$ of $\mu$ corresponding to $\L2(\PP\vert_x)^\circ$, i.e. the one for which $\L2(\PP\vert_x)^\circ=H(P)$. For each $b\in \mathfrak{F}_B$ we have that  $\PP[\PP[f\vert \underline{b}']^2]$ is equal to $1$ or $0$ according as $x\subset \underline{b}'$ or not, this being so  a.e.-$\mu$. But $\{x\subset \underline{b}'\}=\{\underline{b}\subset x'\}\subset S_{x'}$ a.e.-$\mu$; hence $\phi_f\leq \mathbbm{1}_{S_{x'}}$ a.e.-$\mu$. On the other hand, plainly $\PP[\PP[f\vert \underline{b_x}']^2] \geq\mathbbm{1}_{S_{x'}}$ a.e.-$\mu$. Therefore $\phi_f=\mathbbm{1}_{S_{x'}}$ a.e.-$\mu$. More generally, let $f=\oplus_{x\in B^\circ}f_x\in \oplus_{x\in B^\circ}\L2(\PP\vert_x)^\circ$ with $\sum_{x\in B^\circ}\PP[f_x^2]=1$ (so that $\PP[f^2]=1$). Then for all $b\in \mathfrak{F}_B$, a.e.-$\mu$,
$$\PP[\PP[f\vert \underline{b}']^2]=\PP\left[\left(\sum_{x\in B^\circ}\PP[f_x\vert \underline{b}']\right)^2\right]=\PP\left[\left(\sum_{x\in B^\circ} f_x \mathbbm{1}_{\{x\subset \underline{b}'\}}\right)^2\right]=\sum_{x\in B^\circ}\PP[f_x^2] \mathbbm{1}_{\{ \underline{b}\subset x'\}}\leq \sum_{x\in B^\circ}\PP[f_x^2] \mathbbm{1}_{S_{x'}};$$ further, 
if $x\in b$, then actually $\mu$-a.e. $\{ \underline{b}\subset x'\}=S_{x'}$ so that taking a nondecreasing sequence $(b_n)_{n\in \mathbb{N}}$ in $\mathfrak{F}_B$, whose union contains $B^\circ$ (recall, $B^\circ$ is countable) shows that $\phi_f= \sum_{x\in B^\circ}\PP[{f_x}^2] \mathbbm{1}_{S_{x'}}$ a.e.-$\mu$.
\end{example}
\begin{proof}
For each $b\in \mathfrak{F}_B$ define $\phi_b:S\to [0,1]$ up to a.e.-$\mu$ equality by putting (recall that the sets $P_b(y)$, $y\in b_x$, are pairwise $\mu$-disjoint and their union is $S_x$ a.e.-$\mu$ for all $x\in b$)  $$\phi_b:=\sum_{y\in b}\nu(P_b(y'))\mathbbm{1}_{S_y}=\sum_{y\in b}\mathbbm{1}_{P_b(y)}\sum_{z\in b,z\supset y}\nu(P_b(z'))=\sum_{y\in b}\mathbbm{1}_{P_b(y)}\nu(S_{y'})=\nu(S_{\underline{b}'})\text{ a.e.-}\mu.$$ 
From the first expression we get  $\phi_b\in \mathcal{S}$ because $\nu$ is a probability. 
If $\tilde b\in \mathfrak{F}_B$  with $\tilde{b}\supset b$, then $\underline{\tilde{b}}\subset \underline{b}$ a.e.-$\mu$, hence $\phi_{\tilde{b}}\geq \phi_b$ a.e.-$\mu$. It follows that 
$\phi^\nu=\mu\text{-}\esssup_{b\in \mathfrak{F}_B}\phi_b={\uparrow\text{-}\lim}_{n\to\infty}\phi_{b_n}\text{ a.e.-$\mu$}$ for some nondecreasing sequence  $(b_n)_{n\in \mathbb{N}}$ in $\mathfrak{F}_B$. Thus $\phi^\nu\in \mathcal{S}$. 
\end{proof}
\begin{question}
Is it the case that $\phi=\sup_{n\in \mathbb{N}}\nu(S_{\underline{b_n}'})$ a.e.-$\mu$ whenever $(b_n)_{n\in \mathbb{N}}$ is a $\uparrow$ sequence in $\mathfrak{F}_B$ with $B_0:=\cup_{n\in \mathbb{N}}b_n$ dense in $B$? (For sure we know that a $\uparrow$ sequence   $(b_n)_{n\in \mathbb{N}}$ in $\mathfrak{F}_B$ exists, dense or otherwise, such that $\phi=\sup_{n\in \mathbb{N}}\nu(S_{\underline{b_n}'})$ a.e.-$\mu$.)
\end{question}

\begin{corollary}
 Suppose there exists a nondecreasing sequence $(S_n)_{n\in \mathbb{N}}$ in $\Sigma$ and probabilities $\nu_n$, $n\in \mathbb{N}$, absolutely continuous w.r.t. $\mu$, such that for all $n\in \mathbb{N}$:
\begin{enumerate} [(a)]
\item for $\mu$-a.e. $s\in S_n$, $\nu_n(s\cap_S \cdot=\emptyset_S)=1$;
\item   for $\mu$-a.e. $s\in S\backslash \cup_{m\in \mathbb{N}}S_{m}$, $\nu_n(s\cap_S \cdot=\emptyset_S)<1$.
\end{enumerate}
Then $\cup_{n\in \mathbb{N}}S_n\in \Sigma^*$.
\end{corollary}
\begin{proof}
In terms of the notation of Proposition~\ref{proposition:S-from-probability} we have $\phi^{\nu_n}\geq \mathbbm{1}_{S_n}$ a.e.-$\mu$ and $\phi^{\nu_n}<1$ a.e.-$\mu$ on $ S\backslash \cup_{m\in \mathbb{N}}S_{m}$ for each $n\in \mathbb{N}$. 
Hence $\mathbbm{1}_{\cup_{n\in \mathbb{N}}S_n}=\lim_{n\to\infty}\lim_{m\to\infty}(\phi^{\nu_n})^m\in \mathcal{S}$.
\end{proof}

\section{Separating out the classical}\label{section:separate-classsical}
%
Every noise Boolean algebra admits a largest classical part (Proposition~\ref{proposition:classical-part-of-noise}). This part is ``graded'' into the individual chaos spaces by the counting map. The chaos spaces themselves are identified in terms of the noise operators, which is to say in terms of the degree of influence that a ``perturbation'' has  on the random variables (Proposition~\ref{proposition:stability-sensitivity}).   Self-joinings are a neat formalization of said perturbation (Proposition~\ref{proposition:self-joinings}). Loosely speaking, a noise Boolean algebra  is classical or black according as everything  is stable or sensitive (Proposition~\ref{proposition:stable-sensitive}); a couple of other explicit conditions for classicality and blackness that follow from identifying the first chaos are noted (Corollary~\ref{corollary:black-first-chaos-norm-projection}). 

As was mentioned already in the Introduction, this section, like Subsection~\ref{subsection:subnoises-parts-of-spectrum}, contains results that are (by and large) straightforward adaptations of their (strictly) special cases already to be found in existing literature (usually in the context of one-dimensional noises/continuous factorizations). We indicate some connections to  the latter en passant, but make no attempt to be complete in this regard. 
\subsection{Stability and sensitivity, noise operators}

\begin{proposition}\label{proposition:stability-sensitivity}[Cf. \cite[Sections~5.2 and~6.1]{picard2004lectures}]
Let $(b_m)_{m\in \mathbb{N}}$ be a sequence in $\mathfrak{F}_B$ whose union is dense in $B$. We have the following assertions.
\begin{enumerate}[(i)]
\item\label{proposition:stability-sensitivity:i} $H_{\mathrm{stb}}=\L2(\stable)$ and $H_{\mathrm{sens}}=\{f\in\L2(\PP):\PP[f\vert\stable]=0\}$.
\item\label{proposition:stability-sensitivity:ii} For each $k\in \mathbb{N}_0$, $$H^{\leq k}=\cap_{m\in \mathbb{N}} H^{\leq k}(b_m)\text{, a fortiori }H^{\leq k}=\cap_{b\in \mathfrak{F}_B} H^{\leq k}(b),$$ in other words 
$$\{K\leq k\}=\cap_{m\in \mathbb{N}}\cup_{l=0}^k\cup_{I\in {\at(b_m)\choose l}}P_{b_m}(\lor I)=\mu\text{-ess-}\cap_{b\in \mathfrak{F}_B}\cup_{l=0}^k\cup_{I\in {\at(b)\choose l}}P_{b}(\lor I)\text{ a.e.-$\mu$,}$$ and for $(f_1,\ldots,f_k)\in (H^{(1)})^k$ the projection of $f_1\cdots f_k$ onto $H^{(k)}$ is given by the limit as $m\to\infty$ in $\L2(\PP)$ of $\sum\prod_{i=1}^k\PP[f_i\vert a_i]$, the sum being over all $k$-tuples  $(a_1,\ldots,a_k)$  of pairwise distinct atoms of $b_m$, and also of the net $(\sum\prod_{i=1}^k\PP[f_i\vert a_i])_{b\in \mathfrak{F}_B}$, where again the sum is over all $k$-tuples  $(a_1,\ldots,a_k)$  of pairwise distinct atoms  of $b$. In particular $$\{K=1\}=\cap_{m\in \mathbb{N}}\cup_{a\in \at(b_m)}S_a\backslash  S_{0_\PP}=\mu\text{-ess-}\cap_{b\in \mathfrak{F}_B}\cup_{a\in \at(b)}S_a\backslash S_{0_\PP}$$ a.e.-$\mu$. and the projection of $\L2(\PP)_0$ onto the first chaos $H^{(1)}$ is the (restriction to $\L2(\PP)_0$ of the) strong limit of the sequence $(\sum_{a\in\at( b_m)}\PP_{a})_{m\in \mathbb{N}}$ and also of the net $(\sum_{a\in \at(b)}\PP_a)_{b\in \mathfrak{F}_B}$. 
\item\label{proposition:stability-sensitivity:iii} There exists a unique semigroup of norm-$1$ bounded operators $(U_t)_{t\in [0,\infty)}$ on $\L2(\PP)$ such that $U_t=e^{-tn}$ on $H^{(n)}$, while $U_t=0$ on $H_{\mathrm{sens}}$ for $t\in (0,\infty)$ ($U_0=\mathrm{id}_{\L2(\PP)}$ is just the identity, of course). In fact, if we use $U_t(B):=U_t$ to indicate the dependence on $B$, then $U_t={\downarrow\!\!\text{-}\lim}_{m\to\infty}U_t(b_m)$, also  $U_t={\downarrow\!\!\text{-}\lim}_{b\in \mathfrak{F}_B}U_t(b)$, in the sense of strong operator convergence for all $t\in [0,\infty)$; for $b\in \mathfrak{F}_B$, $U_t(b)=\int\PP_x\gamma_{e^{-t}}(\dd x)$, where for $p\in (0,1]$, $\gamma_p$ is the Bernoulli measure on $(b,2^b)$ that includes each atom of $b$ independently of the others with probability $p$ (cf. Example~\ref{example:spectral-finite}). The operators $U_t$, $t\in [0,\infty)$, commute with $\PP_x$, $x\in B$, moreover they belong to $\AA$, and they are local in the sense that $U_t(B_x)=U_t$ on $\L2(\PP\vert_x)$ for $x\in B$.

\item\label{proposition:stability-sensitivity:v} $H_{\mathrm{stb}}=\{X\in \L2(\PP):U_{0+}X=X\text{ in norm}\}$, $H_{\mathrm{sens}}=\{X\in \L2(\PP):U_{t}X=0\text{ for all }t\in (0,\infty)\}$ and $U_{0+}=\PP_{\stable}$ (in the sense of strong operator convergence). In particular $\PP_{\stable}\in \AA$, hence commutes with $\PP_x$, $x\in B$.
\end{enumerate}
\end{proposition}
\begin{proof}
\ref{proposition:stability-sensitivity:i}. Note that $H_{\mathrm{stb}}=H(\{K<\infty\})$. Due to Example~\ref{example:K<infty} and Corollary~\ref{corollary:L2-space} the first part of the claim follows. The second part is a direct consequence of the first because $H_{\mathrm{sens}}$ is the orthogonal complement of $H_{\mathrm{stb}}$ and orthogonality to $\L2(\stable)$ means precisely the vanishing of the $\stable$-conditional expectation.

\ref{proposition:stability-sensitivity:ii}. The two formulae in the displays follow from Proposition~\ref{prop;K-finite} (note that $H^{\leq k}(b)=H(\{K_b\leq k\})$ and $\{K_b\leq k\}=\cup_{l=0}^k\cup_{I\in {\at(b)\choose l}}P_{b}(\lor I)$ a.e.-$\mu$ for $b\in \mathfrak{F}_B$, $k\in \mathbb{N}_0$). Then for $(f_1,\ldots,f_k)\in (H^{(1)})^k$ and $m\in \mathbb{N}$,  $(f_1,\ldots,f_k)\in H^{(1)}(b_m)^k$; writing $f:=f_1\cdots f_k$ we see that  $\sum\prod_{i=1}^k\PP[f_i\vert a_i]=\pr_{H^{(k)}(b_m)}f=\pr_{H^{\leq k}(b_m)}f-\pr_{H^{\leq (k-1)}(b_m)}f\to \pr_{H^{\leq k}}f-\pr_{H^{\leq (k-1)}}f=\pr_{H^{(k)}}f$ as $m\to\infty$ in $\L2(\PP)$. Similarly the net-case follows, as well as the corresponding statements for the projection onto the first chaos.

For the proof of the remaining two claims we use up in large part the results of \cite{tsirelson-arxiv-5}. In order to do so we first bring the present setting in line with \cite{tsirelson-arxiv-5}. Indeed we may certainly index each $b_m$, $m\in\mathbb{N}$, by an algebra $\AA_m$ of a base set $T$ in such a way that $(\AA_m)_{m\in \mathbb{N}}$ is a nondecreasing sequence of algebras, the indexations being isomorphisms of Boolean algebras. This brings us into the landscape of  \cite{tsirelson-arxiv-5}, where a semigroup of norm-$1$ bounded operators $(U_t)_{t\in [0,\infty)}$ is defined, in principle depending on $B_0:=\cup_{m\in \mathbb{N}}b_m$, with, for $t\in (0,\infty)$ and $n\in \mathbb{N}_0$, $\{U_t=e^{-tn}\}$ being a subspace $H_n$ of $\L2(\PP)$, which coincides with $H^{(n)}$ due to \ref{proposition:stability-sensitivity:ii} and \cite[Eq.~(2.8)]{tsirelson-arxiv-5} (it is trivial that the definitions agree for finite noise Boolean algebras, cf. \cite[Eqs.~(1.7) and~(1.11)]{tsirelson-arxiv-5} and Example~\ref{example:classical-spectrum}), while $U_t=0$ on the orthogonal complement of $\oplus_{n\in \mathbb{N}_0}H_n$ in $\L2(\PP)$ \cite[Eq.~(2.7)]{tsirelson-arxiv-5} (therefore the semigroup actually does not depend on $B_0$).

\ref{proposition:stability-sensitivity:iii}.  Apply \cite[Eqs.~(1.9),~(2.4) \&~(2.7)]{tsirelson-arxiv-5}. The fact that $\{U_t:t\in [0,\infty)\}\subset \AA$ one obtains  by the explicit representation of the semigroup through the approximating sequence $(b_m)_{m\in \mathbb{N}}$.  The ``locality'' of the noise semigroup follows for instance because the chaos spaces are themselves ``local'' \ref{generalites:chaos-spaces}. 


\ref{proposition:stability-sensitivity:v}. Apply \cite[Theorem~2.12, Lemmas~2.5 and~2.14, Eq.~(2.15)]{tsirelson-arxiv-5}.
\end{proof}

\begin{definition}
We retain the notation $U_t=U_t(B)$, $t\in [0,\infty)$, of Proposition~\ref{proposition:stability-sensitivity}\ref{proposition:stability-sensitivity:iii}; it is natural to, and we do set $U_\infty:=U_\infty(B)$ equal to the projection on the one-dimensional space of constants. $\NNN$ is the  unique linear operator on $H_{\mathrm{stb}}$ such that $\NNN=n$ on $H^{(n)}$. 
\end{definition}
In some sense $\NNN=\alpha^{-1}(K)$, though $K$ is not (always) $\mu$-essentially bounded, so strictly speaking $\alpha^{-1}(K)$ is not (necessarily) legitimate.

\begin{remark}\label{remark:replace-infinitesimal}
It is clear from the preceding that the operators $(U_t)_{t\in [0,\infty)}$, $\NNN$, the chaos spaces $(H^{(k)})_{k\in \mathbb{N}_0}$ etc. can be defined ``intrinsically'' from $B$ (i.e. without resorting to finite approximations or even spectrum) starting with the first chaos space and proceeding from there. The benefit of an explicit construction lies in the fact that one sees $(U_t)_{t\in [0,\infty)}$ as limits of operators $(U_t(b_m))_{t\in [0,\infty)}$ with, for $m\in \mathbb{N}$ \& $t\in [0,\infty)$, $U_t(b_m)$ corresponding to replacing each elementary piece of data  (i.e. each atom) of $b_m$ with an independent copy with probability $1-e^{-t}$, retaining the old piece with probability $e^{-t}$, and one may even imagine that this is happening in continuous time, the elementary piece of data being replaced with an independent copy in each infinitesimal time interval $\dd t$ with probability $\dd t$ in  a Markovian (memoryless) fashion, see \cite[esp. Eq.~(1.14)]{tsirelson-arxiv-5} for details. 
\end{remark}
There is a natural way to generalize the semigroup $(U_t)_{t\in [0,\infty)}$ to a wider class of operators (which we will find use of in Subsection~\ref{subsection:joinings}).
\begin{definition}\label{definition:spectrally-multiplicative}
We say a map $\gamma\in \Sigma/\mathcal{B}_{\mathbb{R}}$ is spectrally multiplicative if $\gamma(\emptyset_S)=1$ and $\gamma=\gamma(\pr_x)\gamma(\pr_{x'})$ a.e.-$\mu$ for all $x\in B$.  We also put $\mathfrak{M}:=\{\gamma\in \Sigma/\mathcal{B}_{[0,1]}:\gamma\text{ is spectrally multiplicative}\}$.
\end{definition}

\begin{remark}
If $\gamma\in \Sigma/\mathcal{B}_{\mathbb{R}}$ is spectrally multiplicative then so is $\gamma\circ\pr_y$ for all $y\in B$ and $\gamma= \prod_{q\in Q}\gamma\circ \pr_q$ a.e.-$\mu$ for any partition of unity $Q$ of $B$. A product of spectrally multiplicative maps is spectrally multiplicative; the a.e.-$\mu$ limit of spectrally multiplicative maps is spectrally multiplicative. According to Proposition~\ref{proposition:sigma-star-and-partitions}  $\{\mathbbm{1}_E:E\in \Sigma^*\}\subset \mathfrak{M}$. Using Proposition~\ref{proposition:K-and-projections} we see that  the map $\prod_{p\in P}\rho(p)^{K(\pr_p)}$ belongs to $\mathfrak{M}$ for any partition of unity $P$ of $B$ and any $\rho\in [0,1]^P$.  
\end{remark}

\begin{definition}[Generalized noise operators]\label{generalized-noise:iv}
For $\gamma\in \mathfrak{M}$ we put  $U^\gamma:=U ^\gamma(B):=\alpha^{-1}(\gamma)$. In particular for a partition of unity $P$ of $B$ and $\rho\in [0,1]^P$ we set  $U_\rho:=U_\rho(B):=U^{\prod_{p\in P}\rho(p)^{K(\pr_p)}}$ (recall we interpret $1^\infty=1$ and $a^\infty=0$ for $a\in [0,1)$). 
\end{definition}

\begin{remark}\label{generalized-noise}\leavevmode
\begin{enumerate}[(i)]
\item For $t\in [0,\infty]$, $U_{\{(1_\PP,e^{-t})\}}=U_t$, unless $0_\PP=1_\PP$ (when $U_{\emptyset}=U_t$).
\item\label{generalized-noise:iii} Let $P$ be a  partition of unity of $B$ and $\rho\in [0,1]^P$. For $X_p\in \L2(\PP\vert_p)$, $p\in P$, we have $U_\rho(\prod_{p\in p}X_p)=\prod_{p\in P}U_{-\log(\rho(p))}(X_p)$, i.e. $U_\rho=\otimes_{p\in P}( U_{-\log \rho(p)}\vert_{\L2(\PP\vert_p)})$ up to the natural unitary equivalence of $\L2(\PP)$ and $\otimes_{p\in P}\L2(\PP\vert_p)$ (use Proposition~\ref{proposition:separate-with-projections}).
\end{enumerate}
\end{remark}

\begin{question}
Can one give a meaningful characterization of the class $\mathfrak{M}$, or which is the same upon taking $-\log$, of $\mathfrak{A}:=\{\gamma\in \Sigma/\mathcal{B}_{[0,\infty]}:\gamma(\emptyset_S)=0\text{ and  }\gamma= \sum_{q\in Q}\gamma\circ \pr_q\text{ a.e.-$\mu$ for any partition of unity $Q$ of $B$}\}$ (the  ``nonnegative spectrally additive'' maps)?
\end{question}

\subsection{Classical part of a noise}

\begin{definition}
$B_{\mathrm{stb}}:=\{x\land \stable:x\in B\}$.
\end{definition}
\begin{proposition}[Classical part of a noise]\label{proposition:classical-part-of-noise}
(Cf. \cite[sentence that includes Eq.~(2.17)]{tsirelson-arxiv-5}.)
\begin{enumerate}[(i)]
\item\label{proposition:classical-part-of-noise:i}  $\stable $ distributes over $B$. 
\item\label{proposition:classical-part-of-noise:ii} $B_{\mathrm{stb}}$ is a classical noise Boolean algebra under $\PP\vert_\stable$, whose first chaos (and hence all higher-order chaoses) coincide with that (those) of $B$ (or, which is the same, of $\overline{B}$). 
\item\label{proposition:classical-part-of-noise:iii}  $\stable$ is the largest $x\in \Lattice$ such that $\PP_x\in \AA$, $x$ distributes over $B$ and $\{x\land y:y\in B\}$ is a classical noise Boolean algebra under $\PP\vert_x$. 
\item\label{proposition:classical-part-of-noise:iv} Fixing any version of $\{K<\infty\}$, $((\{K<\infty\},\Sigma\vert_{\{K<\infty\}},\mu\vert_{\{K<\infty\}});\Psi\vert_{H_{\mathrm{stb}}}\vert_{\{K<\infty\}})$ is a spectral resolution of $B_{\mathrm{stb}}$ in which the spectral set associated to $x\land \stable$ is $S_x\cap \{K<\infty\}$ and its noise projection is a.e.-$\mu$ equal to the restriction $\pr_x\vert_{\{K<\infty\}}$, $x\in B$.
\item\label{proposition:classical-part-of-noise:v}   For all $x\in B$, $\PP_{\stable\lor x}\in \AA$ and $\pr_x^{-1}(S_{\stable})=S_{x'\lor (x\land \stable)}$, i.e. $\pr_x^{-1}(\{K<\infty\})=S_{x'\lor \stable}$ a.e.-$\mu$. 
\end{enumerate}
\end{proposition}
\begin{proof}
 \ref{proposition:classical-part-of-noise:i} \& \ref{proposition:classical-part-of-noise:ii}. For the first claim and the fact that  $B_{\mathrm{stb}}$ is a noise Boolean algebra under $\PP_\stable$ just apply Corollary~\ref{corollary:sigma*-subnoises} and Example~\ref{example:K<infty}.  Also, $H^{(1)}(B_{\mathrm{stb}})=H^{(1)}(B)$, which follows easily from the fact that $\stable$ commutes with $B$ (Proposition~\ref{proposition:stability-sensitivity}\ref{proposition:stability-sensitivity:v}). Therefore $B_{\mathrm{stb}}$ is classical and shares its chaoses with $B$.
 
\ref{proposition:classical-part-of-noise:iii}.  We know that $\PP_\stable\in \AA$ (Proposition~\ref{proposition:stability-sensitivity}\ref{proposition:stability-sensitivity:v}, again). 
On the other hand, let $x\in \Lattice$ be such that $\PP_x\in \AA$,   $x$ distributes over $B$, and $A:=\{x\land y:y\in B\}$ is a classical noise Boolean algebra under $\PP\vert_x$. Then a spectrum for $A$ is obtained by restricting the spectrum of $B$ to (a version of) the set $S_x\in \Sigma$ for which $\PP_x=\alpha(\mathbbm{1}_{S_x})$ (it works because $\PP_x\in \AA$). With this restricted spectrum $K_B\leq K_A$ a.e.-$\mu$ on $S_x$ (because $x$ distributes over $B$, hence every partition of unity of $B$ gives rise to a partition of unity of $A$). $A$ being classical we must have $K_A<\infty$ a.e.-$\mu$ on $S_x$, hence $S_x\subset \{K_B<\infty\}$ a.e.-$\mu$ and $x\subset \stable$. 
 
\ref{proposition:classical-part-of-noise:iv}. Except for the claim on the projections, it is already contained in Corollary~\ref{corollary:sigma*-subnoises}. For the latter one has only to notice that, for $x\in B$,  $\{K<\infty\}$ is closed for $\pr_x$, in the sense that $\pr_x^{-1}(\{K<\infty\})\supset \{K<\infty\}$ a.e.-$\mu$, which follows e.g. from Proposition~\ref{proposition:K-and-projections}. Then for all $y\in B$ we have $(\pr_x\vert_{\{K<\infty\}})^{-1}(S_y\cap \{K<\infty\})=S_{y\lor x'}\cap \{K<\infty\}$ a.e.-$\mu$, whence the desired conclusion.
%
%

\ref{proposition:classical-part-of-noise:v}. Since $\PP_{\stable}\in \AA$, by Corollary~\ref{corollay:defining-extended-to-AA} we get that for all $x\in B$,  $\PP_{\stable\lor x}\in \AA$ and $\pr_x^{-1}(S_{\stable})=S_{x'\lor (x\land \stable)}$. It remains to take into account that $S_{\stable}=\{K<\infty\}$ a.e.-$\mu$ and the distributivity of $\stable$ over $B$ which gives  $x'\lor \stable=x'\lor [(x'\land \stable)\lor (x\land \FF_{\mathrm{stb}})]=x'\lor (x\land \FF_{\mathrm{stb}})$.
\end{proof}

\begin{example}
In the context of Example~\ref{example:simplest-nonclassical} $\stable=\sigma(\xi_i\xi_{i+1}:i\in \mathbb{N})$ and $B_{\mathrm{stb}}$ is just the classical noise Boolean algebra attached to the independency $(\xi_i\xi_{i+1})_{i\in \mathbb{N}}$. Note $B$ is noise complete in this case and still $\stable\notin B$. Nevertheless $\stable\in \Cl(B)$. Analogous observations hold true for the noise Boolean algebra of Example~\ref{example:simplest-nonclassical-tweaked} when $\stable=\sigma(\xi_i\xi_{i+1}:i\in \mathbb{N})\lor \sigma(\xi_\infty)$. 
\end{example}

\begin{question}
Is $\stable\in \Cl(B)$ in general?
\end{question}

\begin{proposition}\label{proposition:classical-part-of-noise-closure}
$\overline{B_{\mathrm{stb}}}=\{x\land \stable:x\in \Cl(B)\}\supset \overline{B}_{\mathrm{stb}}$.
\end{proposition}
\begin{proof}
By classicality, $\overline{B_{\mathrm{stb}}}=\Cl(B_{\mathrm{stb}})$ (the closure in principle is under $\PP\vert_\stable$, however it is the same as   under $\PP$). Also, $\stable(B)=\stable(\overline{B})$ ($\therefore$ the inclusion $\{x\land \stable:x\in \Cl(B)\}\supset \overline{B}_{\mathrm{stb}}$ is trivial); we continue to just write $\stable$.  If $(x_n)_{n\in \mathbb{N}}$ is a sequence in $B$ converging to $x\in \Cl(B)$, then $B_{\mathrm{stb}}\ni x_n\land \stable\to x\land\stable$ as $n\to\infty$ (by continuity of $\land$ on the $\sigma$-fields whose conditional expectation operators on $\L2(\PP)$ belong to $\AA$); therefore $x\land\stable\in \Cl(B_{\mathrm{stb}})= \overline{B_{\mathrm{stb}}}$, i.e. $\overline{B_{\mathrm{stb}}}\supset \{x\land \stable:x\in \Cl(B)\}$. Conversey,  for any sequence $(x_n)_{n\in \mathbb{N}}$ in $B$, $\liminf_{n\to\infty}(x_n\land \stable)=(\liminf_{n\to\infty}x_n)\land \stable\in \{x\land \stable:x\in \Cl(B)\}$, where the equality again follows by continuity of $\land$ on the $\sigma$-fields whose conditional expectation operators on $\L2(\PP)$ belong to $\AA$. Thus $\overline{B_{\mathrm{stb}}}=\Cl(B_{\mathrm{stb}})\subset \{x\land \stable:x\in \Cl(B)\}$.
\end{proof}


\subsection{Joinings}\label{subsection:joinings}
\begin{definition}
(Cf. \cite[Definition~4a5]{tsirelson-nonclassical}.) Let $C$ be another noise Boolean algebra under a probability $\QQ$ (essentially separable, of course). An embedding $\Psi$ of $\PP$ into $\QQ$ is said to embed $B$ into $C$ if $\Psi(1_{\PP})$ distributes over $C$ and $\Psi(B)=C\vert_{\Psi(1_{\PP})}$. 
\end{definition}
\begin{remark}
If $\Psi$ is an embedding of $\PP$ into $\QQ$, automatically it carries $B$ onto a noise Boolean algebra under $\PP\vert_{\Psi(1_{{\PP}})}$ as an injective noise factorization (because an emebedding of probabilities is an isomorphism onto its image and preserves everything in sight, independence in particular). 
\end{remark}
\begin{definition}\label{definition:joining}
(Cf. \cite[p.~203, last paragraph]{tsirelson-nonclassical}.)  Let $B_1$ and $B_2$ be noise  Boolean algebras under the probabilities $\PP_1$ and $\PP_2$, respectively. A joining of $B_1$ and $B_2$ (or, more precisely, of $(B_1,\PP_1)$ and $(B_2,\PP_2)$) is a system $((C,\QQ);(\Psi_1,\Psi_2))$, where $C$ is a noise Boolean algebra under the probability $\QQ$, $1_{\QQ}=\Psi_1(1_{\PP_1})\lor\Psi_2(1_{\PP_2})$ and $\Psi_i$ embeds $B_i$ into $C$, $i\in \{1,2\}$. The law (resp. correlation; local correlation on $c\in C$; operator) of such a joining is the law of the coupling $(\QQ;\Psi_1,\Psi_2)$ of the probabilities $\PP_1$ and $\PP_2$, i.e. the map 
$(1_{\PP_1}\times  1_{\PP_2}\ni(A_1,A_2) \mapsto \QQ(\Psi_1(A_1)\cap \Psi_2(A_2))$ 
(resp. the quantity 
$\sup\vert \QQ[\Psi_1(f_1)\Psi_2(f_2)]\vert$, where the supremum is taken over all $f_1\in \L2(\PP_1)_0$ with $\var_{\PP_1}(f_1)\leq 1$ and $f_2\in \L2(\PP_2)_0$ with $\var_{\PP_2}(f_2)\leq 1$; the same quantity as in the preceding except that in addition one insists that $f_1$ is $\Psi_1^{-1}(c)$-measurable and  $f_2$ is $\Psi_2^{-1}(c)$-measurable; the unique bounded linear operator $U:\L2(\PP_1)\to \L2(\PP_2)$ satisfying $\PP_2[U(f)g]=\QQ[\Psi_1(f)\Psi_2(g)]$ for $f\in \L2(\PP_1)$ and $g\in \L2(\PP_2)$). 
In case $B_1=B_2=B$ and $\PP_1=\PP_2= \PP$, 
we speak of a self-joining of $B$; such a self-joining is said to be  symmetric if the law of  $(\QQ;\Psi_1,\Psi_2)$ is the same as the law of  $(\QQ;\Psi_2,\Psi_1)$. 
\end{definition}
\begin{remark}
Suppose we have a joining $((C,\QQ);(\Psi_1,\Psi_2))$ of $(B_1,\PP_1)$ and $(B_2,\PP_2)$. By definition $\Psi_i(1_{\PP_i})$ distributes over $C$ for $i\in \{1,2\}$. Another kind of distributivity, namely ``of $C$ over $\Psi_1(1_{\PP_1})\lor \Psi_2(1_{\PP_2})$'', is implicit, in the following precise sense. Let $c\in C$. Then $(c\land \Psi_i(1_{\PP_i}))\lor (c'\land \Psi_i(1_{\PP_i}))=\Psi_i(1_{\PP_i})$, $i\in \{1,2\}$; taking the join we get $[(c\land \Psi_1(1_{\PP_1}))\lor (c\land \Psi_2(1_{\PP_2}))]\lor [(c'\land \Psi_1(1_{\PP_1}))\lor (c'\land \Psi_2(1_{\PP_2}))]=\Psi_1(1_{\PP_1})\lor \Psi_2(1_{\PP_2})=1_\QQ$. Intersecting with $c$,  independence of $c$ and $c'$ allows to conclude that $c=(c\land \Psi_1(1_{\PP_1}))\lor (c\land \Psi_2(1_{\PP_2}))$.
\end{remark}
\begin{remark}
 The law of a joining $(\QQ;\Psi_1,\Psi_2)$ of $\PP_1$ and $\PP_2$ seems scanty in terms of the information it provides. However, by Dynkin's lemma and due to  $1_{\QQ}=\Psi_1(1_{\PP_1})\lor\Psi_2(1_{\PP_2})$ it determines $\QQ$ uniquely. When $\PP_1=\PP_2=\PP$, then the operator of a self-joining is Hermitian iff the self-joining is symmetric, in which case it is determined already by the quadratic form $(\L2(\PP)\ni f\mapsto \PP[ U(f)f])$ (by polarization). 
\end{remark}
\begin{remark}
The local correlation camouflages the more general concept of the correlation between two elements $x$ and $y$ of $\Lattice$, which is to say of the quantity $\rho(x,y):=\sup \vert \PP[fg]\vert$, where the supremum is over all $f\in \L2(\PP\vert_x)$ with $\Vert f\Vert\leq 1$ and $g\in \L2(\PP\vert_y)$ with $\Vert g\Vert\leq 1$. Plainly $\rho(x,y)\in [0,1]$ (by Cauchy-Schwartz) and $\rho(x,y)=0$ as soon as $x\land y\ne 0_\PP$. Further, for a given $p\in [0,1]$, $\rho(x,y)\leq p$ is equivalent to $$\vert \PP[fg]\vert\leq p\Vert f\Vert\Vert g\Vert,\quad (f,g)\in \L2(\PP\vert_x)_0\times \L2(\PP\vert_y)_0.$$ Taking $g=\PP[f\vert y]$ in the preceding display gives $\Vert\PP[f\vert y]\Vert\leq \rho(x,y)\Vert f\Vert$ for all $f\in \L2(\PP\vert_x)_0$, i.e. $\Vert \PP_y\vert_{\L2(\PP\vert_x)_0}\Vert\leq  \rho(x,y)$. In fact $\Vert \PP_y\vert_{\L2(\PP\vert_x)_0}\Vert= \rho(x,y)$, the reverse inequality following from  the estimate $\vert \PP[fg]\vert=\vert \PP[\PP[f\vert y]g]\vert\leq \Vert g\Vert\Vert\PP[f\vert y]\Vert$  for $ (f,g)\in \L2(\PP\vert_x)_0\times \L2(\PP\vert_y)_0$, which is Cauchy-Schwartz (again). Thus $ \rho(x,y)$ is nothing but the operator norm of $\PP_y$ when restricted to an operator on $\L2(\PP\vert_x)_0$ (mapping into $\L2(\PP\vert_y)_0$).
\end{remark}
%
\begin{example}
A trivial joining of  noise  Boolean algebras $B_1$ and $B_2$ under the probabilities $\PP_1$ and $\PP_2$, respectively, is the product joining $((C,\QQ);(\Psi_1,\Psi_2))$ for which we take: $\QQ=\PP_1\times \PP_2$; $\Psi_i$ the natural emebedding of $\PP_i$ into $\QQ$, $i\in \{1,2\}$; $C:=\{\Psi_1(x_1)\lor \Psi_2(x_2):(x_1,x_2)\in B_1\times B_2\}$. An even more trivial, self-joining of a noise Boolean algebra $B$ under a probability $\PP$ is the deterministic joining $((B,\PP),(\mathrm{id}_{\mathrm{L}^0(\PP)},\mathrm{id}_{\mathrm{L}^0(\PP)}))$.
\end{example}
\begin{proposition}
(Cf.  \cite[Proposition~4a8]{tsirelson-nonclassical}.) Take a joining as in Definition~\ref{definition:joining}. For $c\in C$ denote the local correlation of the joining on $c$ by $\gamma(c)$; then $\gamma(c)=\sup\vert\QQ[f_1f_2]\vert $, where the supremum is taken over all $f_i\in\L2(\QQ\vert_{c\land \GG_i})_0$ with $\var_\QQ(f_i)\leq 1$, $\GG_i:=\Psi_i(1_{\PP_i})$ for $i\in \{1,2\}$. Let $\{u,v\}\subset C$ with $u\land v=0_\QQ$. Then $\gamma(u\lor v)=\gamma(u)\lor \gamma(v)$. 
\end{proposition}
\begin{proof}
The fact that  $\gamma(c)=\sup\vert\QQ[f_1f_2]\vert $, where the supremum is taken over all $f_i\in\L2(\QQ\vert_{c\land \GG_i})_0$ with $\var_\QQ(f_i)\leq 1$, is immediate. 

As for the equality $\gamma(u\lor v)= \gamma(u) \lor \gamma(v)$, since by distributivity  $(u\lor v)\land \GG_i=(u\land \GG_i)\lor (v\land \GG_i)$, $i\in \{1,2\}$, it is really a matter of checking the following general observation concerning arbitrary elements of $\Lattice$: 
\begin{quote}
\begin{lemma}
If $\{u_1,u_2,v_1,v_2\}\subset \Lattice$ and $u_1\lor u_2$ is independent of $v_1\lor v_2$, then $\rho(u_1\lor v_1,u_2\lor u_2)=\rho(u_1,u_2)\lor \rho(v_1,v_2)$.
\end{lemma}
\begin{proof}
The operator $\PP_{u_1\lor v_1}\vert_{\L2(\PP\vert_{u_2\lor v_2})_0}$ decomposes into $\PP_{u_1}\vert_{\L2(\PP\vert_{u_2})_0}\oplus \PP_{v_1}\vert_{\L2(\PP\vert_{v_2})_0}\oplus (\PP_{u_1}\otimes \PP_{v_1})\vert_{\L2(\PP\vert_{u_2})_0\otimes \L2(\PP\vert_{ v_2})_0}$ according to the orthogonal decomposition of $\L2(\PP\vert_{u_2\lor v_2})_0$ into $\L2(\PP\vert_{u_2})_0\oplus \L2(\PP\vert_{v_2})_0\oplus (\L2(\PP\vert_{u_2})_0\otimes \L2(\PP\vert_{ v_2})_0)$. We get $\rho(u_1\lor v_1,u_2\lor u_2)=\rho(u_1,u_2)\lor \rho(v_1,v_2)\lor (\rho(u_1,u_2) \rho(v_1,v_2))=\rho(u_1,u_2)\lor \rho(v_1,v_2)$.
\end{proof}
\end{quote}
The proof is complete.
\end{proof}

\begin{proposition}\label{proposition:self-joinings}
(Cf. \cite[Section~2, passim]{spectra-harris}.) Let $P$ be a partition of unity of $B$ and $\rho \in [0,1]^P$. There exists a symmetric self-joining $((C,\QQ);(\Psi_1,\Psi_2))$ of $B$  such that:
\begin{enumerate}[(a)]
\item the $\Psi_{1}(p)\lor\Psi_2(p)$, $p\in P$, are $\QQ$-independent and belong to $C$;
\item  for all $p\in P$, the local correlation of this self-joining on $\Psi_1(p)\lor\Psi_2(p)$ is $\leq \rho(p)$;
\item for any other self-joining $((\tilde C,\tilde \QQ);(\tilde\Psi_1,\tilde\Psi_2))$ of $B$ satisfying the two preceding properties, it holds that $\vert \tilde \QQ[\tilde\Psi_{1}(f)\tilde\Psi_{2}(f)]\vert \leq \QQ[\Psi_{1}(f)\Psi_{2}(f)]$ for all $f\in \L2(\PP\vert_p)$, all $p\in P$.
\end{enumerate}
Furthermore, we have the following properties.
\begin{enumerate}[(i)]
\item\label{weird-minus} The law of such a symmetric self-joining is uniquely determined by the  properties listed above; its operator is $U_\rho$. 
\item\label{weird-zero} For any such symmetric self-joining, provided $\lor \rho<1$ ($\lor\emptyset=0$), we have that $\Psi_{1}(1_\PP)$ is independent of  $\Psi_{2}(1_\PP)$ given ${\Psi_{1}(\stable)\lor \Psi_{2}(\stable)}$, and also that $\Psi_{1}(1_\PP)$ is independent of $ \Psi_{2}(\stable)$ given ${\Psi_{1}(\stable)}$.
\item\label{weird-one}  For any such symmetric self-joining, $\mu_f[ \prod_{p\in P}\rho(p)^{K(\pr_p)}]=\QQ[\Psi_1(f)\Psi_2(f)]$ for all $f\in \L2(\PP)$.
\item\label{weird-two}   Conversely, let $\nu$ be a measure absolutely continuous w.r.t. $\mu$, let $f\in \L2(\PP)$ and suppose that $\nu[ \prod_{p\in P}\rho(p)^{K(\pr_p)}]=\QQ[\Psi_1(f)\Psi_2(f)]$ holds for all symmetric self-joinings as above  associated to $P=\{p,p'\}\backslash \{0_\PP\}$ and to $\rho=\mathbbm{1}_{\{p\}}$ as you vary $p\in B$.  Then $\nu=\mu_f$.
\end{enumerate}
\end{proposition}

 \begin{remark}
The formula in Item~\ref{weird-one} of the preceding proposition polarizes naturally. Also, if $\PP[f^2]=1$, then it rewrites into $\mu_f[1- \prod_{p\in P}\rho(p)^{K(\pr_p)}]=\frac{1}{2}\QQ[(\Psi_1(f)-\Psi_2(f))^2]$.  
 \end{remark}

 \begin{proof}
The case $0_\PP=1_\PP$ is trivial and omitted. Then for existence as well as for \ref{weird-minus}-\ref{weird-zero} everything ``factorizes'' over $p\in P$ (recall Proposition~\ref{proposition:classical-part-of-noise}\ref{proposition:classical-part-of-noise:i} and Remark~\ref{generalized-noise}\ref{generalized-noise:iii}); therefore, as far as these statements are concerned, we may and do assume $P=\{1_\PP\}$ and identify $\rho$ with $\rho(1_\PP)$. 
 
Existence. Because we are looking for embeddings of probabilities we may and do assume that $\PP$ is a standard probability.  Once this has been noted the proof works almost exactly verbatim like that of \cite[Proposition~4b1]{tsirelson-nonclassical} with only automatic modifications necessary to pass from nets over [decompositions of $\mathbb{R}$ into a union of disjoint intervals (modulo finite sets)] to nets over [partitions of unity of $B$]. The symmetric self-joining  is obtained by the procedure of ``replacing each infnitesimal part of the data with probability $1-\rho$ by an independent copy'' as described in Remark~\ref{remark:replace-infinitesimal} and by taking the limit of the resulting net $\left(((C^b,\QQ^b);(\Psi_1^b,\Psi_2^b))\right)_{b\in \mathfrak{F}_B}$ of symmetric self-joinings, whose quadratic forms \cite[Eq.~(4b4)]{tsirelson-nonclassical}  $U_{-\log \rho}(b)$, $b\in \mathfrak{F}_B$, converge monotonically to the quadratic form of the limiting symmetric self-joining, which is then $U:=U_{-\log \rho}$ (recall Proposition~\ref{proposition:stability-sensitivity}\ref{proposition:stability-sensitivity:iii}). Specifically, the construction is such that $$\QQ^b[\Psi_1^b(\oplus_{n\in \mathbb{N}_0}f_n)\Psi_2^b(\oplus_{n\in \mathbb{N}_0}g_n)]=\sum_{n\in \mathbb{N}_0}\rho^n\PP[f_ng_n],$$ where the orthogonal sum decompositions are according to the chaos spaces of the finite noise Boolean algebra $b\in \mathfrak{F}_B$ (of course all but a finite number of terms are actually zero for each fixed $b$).  
 
 The distributivity property of a joining in Definition~\ref{definition:joining} follows automatically from the construction because by the procedure of the proof one gets actually a surjective noise factorization $N:B\to C$ such that $\Psi_i(x)=N(x)\land \Psi_i(1_\PP)$  hence $(N(x)\land \Psi_i(1_\PP))\lor (N(x')\land \Psi_i(1_\PP))=\Psi_i(x)\lor \Psi_i(x')=\Psi_i(x\lor x')=\Psi_i(1_\PP)$ for all $x\in B$, which means that $\Psi_i(1_\PP)$ distributes over $C$, $i\in \{1,2\}$.
 
\ref{weird-minus}. The law of a symmetric self-joining is determined by the quadratic form of its operator; from the preceding this operator is $U_\rho$. 
 
\ref{weird-zero}.  Assume $\rho<1$.   Let $\{s_1,s_2\}\subset H_{\mathrm{stb}}$ and $\{o_1,o_2\}\subset H_{\mathrm{sens}}$. Trivially, $$\QQ[\Psi_1(s_1+o_1)\vert {\Psi_{1}(\stable)\lor \Psi_{2}(\stable)}]=\Psi_1(s_1)+\QQ[\Psi_1(o_1)\vert {\Psi_{1}(\stable)\lor \Psi_{2}(\stable)}].$$ But, if for a moment $s_1$ and $s_2$ are bounded, then $\QQ[\Psi_1(o_1)\Psi_1(s_1)\Psi_2(s_2)]=\QQ[\Psi_1(s_1o_1)\Psi_2(s_2)]=\PP[U_{-\log\rho}(s_2)s_1o_1]=0$ ($\because$ $\rho<1$). Therefore $\QQ[\Psi_1(o_1)\vert {\Psi_{1}(\stable)\lor \Psi_{2}(\stable)}]=0$ and  $\QQ[\Psi_1(s_1+o_1)\vert {\Psi_{1}(\stable)\lor \Psi_{2}(\stable)}]=\Psi_1(s_1)$. Besides, $\QQ[\Psi_1(o_1)\vert \Psi_{1}(\stable)]=\PP[o_1\vert \stable]=0$. We get $\QQ[\Psi_1(s_1+o_1)\vert \Psi_{1}(\stable)]=\Psi_1(s_1)$. Combining the two conclusions shows that $\Psi_1(1_\PP)$ is independent of $ \Psi_2(\stable)$ given $\Psi_1(\stable)$. Further, expanding the product we obtain at once that 
\begin{align*}
&\QQ[\Psi_1(s_1+o_1)\Psi_2(s_2+o_2)\vert {\Psi_{1}(\stable)\lor \Psi_{2}(\stable)}] =\Psi_1(s_1)\Psi_2(s_2)+\Psi_1(s_1)\QQ[\Psi_2(o_2)\vert {\Psi_{1}(\stable)\lor \Psi_{2}(\stable)}]\\
&+\QQ[\Psi_1(o_1)\vert {\Psi_{1}(\stable)\lor \Psi_{2}(\stable)}]\Psi_2(s_2)+\QQ[\Psi_1(o_1)\Psi_2(o_2)\vert {\Psi_{1}(\stable)\lor \Psi_{2}(\stable)}]\\
&=\Psi_1(s_1)\Psi_2(s_2)+\QQ[\Psi_1(o_1)\Psi_2(o_2)\vert {\Psi_{1}(\stable)\lor \Psi_{2}(\stable)}].
\end{align*}
 Similarly as above, if for a moment $s_1$ and $s_2$ are bounded, then $\QQ[\Psi_1(o_1)\Psi_2(o_2)\Psi_1(s_1)\Psi_2(s_2)]=\QQ[\Psi_1(o_1s_1)\Psi_2(o_2s_2)]=\PP[(U_{-\log\rho}(s_1o_1)s_2)o_2]=0$ ($\because$ $\rho<1$). We get $\QQ[\Psi_1(o_1)\Psi_2(o_2)\vert {\Psi_{1}(\stable)\lor \Psi_{2}(\stable)}]=0$, therefore 
 \begin{align*}
& \QQ[\Psi_1(s_1+o_1)\Psi_2(s_2+o_2)\vert {\Psi_{1}(\stable)\lor \Psi_{2}(\stable)}]=\Psi_1(s_1)\Psi_2(s_2)\\
& =\QQ[\Psi_1(s_1+o_1)\vert {\Psi_{1}(\stable)\lor \Psi_{2}(\stable)}]\QQ[\Psi_2(s_2+o_2)\vert {\Psi_{1}(\stable)\lor \Psi_{2}(\stable)}],
\end{align*}
 and $\Psi_{1}(1_\PP)$ is independent of  $\Psi_{2}(1_\PP)$ given ${\Psi_{1}(\stable)\lor \Psi_{2}(\stable)}$. 

\ref{weird-one}. We compute $\QQ[\Psi_1(f)\Psi_2(f)]=\langle f,U_\rho f\rangle=\langle f,\alpha^{-1}(\prod_{p\in P}\rho(p)^{K(\pr_p)})f\rangle=\mu[\Vert\Psi(f)\Vert^2\prod_{p\in P}\rho(p)^{K(\pr_p)}]=\mu_f[\prod_{p\in P}\rho(p)^{K(\pr_p)}]$.

\ref{weird-two}. The condition implies that $\nu(S_p)=\mu_f(S_p)$ for all $p\in B$, so that by Dynkin's lemma (use $\nu\ll \mu$) it follows that $\nu=\mu_f$.
 \end{proof}
 
  \begin{proposition}(Cf. \cite[Theorem~4c2]{tsirelson-nonclassical}.)\label{proposition:stable-sensitive}
 Let $f\in \L2(\PP)$. 
 \begin{enumerate}[(i)]
 \item\label{charaterization:stable} $f\in H_{\mathrm{stb}}$ iff  there exists a sequence of symmetric self-joinings  $\left(((C^n,\QQ^n);(\Psi_1^n,\Psi_2^n))\right)_{n\in \mathbb{N}}$ of $B$ each having correlation $<1$, and such that $\lim_{n\to\infty}\limits\QQ^n[(\Psi_1^n(f)-\Psi_2^n(f))^2]=0$.
 \item\label{charaterization:sensitive} $f\in H_{\mathrm{sens}}$  iff for every symmetric self-joining $((C,\QQ);(\Psi_1,\Psi_2))$ of $B$   having correlation $<1$, it holds that $\QQ[\Psi_1(f)\Psi_2(g)]=0$ for all $g\in \L2(\PP)$.
 \end{enumerate}
 \end{proposition}
 \begin{proof}
 The case $0_\PP=1_\PP$ is trivial and omitted. 
 
 \ref{charaterization:stable}. Note that  $\QQ[(\Psi_1(f)-\Psi_2(f))^2]=2\PP[f^2]-2\QQ[\Psi_1(f)\Psi_2(f)]=2\PP[f^2]-2\langle f,Uf\rangle$ for any symmetric self-joining $((C,\QQ);(\Psi_1,\Psi_2))$  with operator $U$. If $f\in  H_{\mathrm{stb}}$ then one may take a sequence corresponding to $P=\{1_\PP\}$ and $\rho\equiv 1-1/n$ in Proposition~\ref{proposition:self-joinings} as $n\in \mathbb{N}$ is varied. One gets $\lim_{n\to\infty}\limits\QQ^n[(\Psi_1^n(f)-\Psi_2^n(f))^2]=\lim_{n\to\infty}2\PP[f^2]-2\langle f,U_{-\log(1-1/n)}f\rangle=2\PP[f^2]-2\langle f,U_{0+}f\rangle=0$. Conversely, if the stated condition holds true, then with $\rho_n$ the correlation of $((C^n,\QQ^n);(\Psi_1^n,\Psi_2^n))$ for $n\in \mathbb{N}$, we obtain 
\begin{align*}
 0&=\lim_{n\to\infty}\limits\QQ^n[(\Psi_1^n(f)-\Psi_2^n(f))^2]= \lim_{n\to\infty}2\PP[f^2]-2\langle f,U_{-\log\rho_n}f\rangle=2\PP[f^2]-2\langle f,U_{0+}f\rangle\\
&\geq \PP[f^2]+\PP[(U_{0+}f)^2]-2\langle f,U_{0+}f\rangle=\PP[(f-U_{0+}f)^2],
\end{align*}
  hence  $U_{0+}f=f$, i.e. $f\in H_{\mathrm{stb}}$.
 
 \ref{charaterization:sensitive}. For every symmetric self-joining $((C,\QQ);(\Psi_1,\Psi_2))$ of $B$  having correlation $<1$ and $g\in \L2(\PP)$, 
  \begin{align*}
  &\QQ[\Psi_1(f)\Psi_2(g)]=\QQ[\QQ[\Psi_1(f)\Psi_2(g)\vert \Psi_1(\stable)\lor \Psi_2(\stable)]]\\
&  =\QQ[\QQ[\Psi_1(f)\vert \Psi_1(\stable)]\QQ[\Psi_2(g)\vert \Psi_2(\stable)]]=\QQ[\Psi_1(\PP[f\vert\stable])\Psi_2(\PP[g\vert \stable])].
  \end{align*}
   The condition is therefore clearly necessary. Sufficiency. Follows from $\QQ[\Psi_1(\PP[f\vert\stable])\Psi_2(\PP[f\vert \stable])]=\langle \PP[f\vert\stable],U_{-\log \rho(1_\PP)}\PP[f\vert\stable]\rangle$ for the symmetric self-joining of Proposition~\ref{proposition:self-joinings} with $P=\{1_\PP\}$ and $\rho(1_\PP)\in (0,1)$ arbitrary.
 \end{proof}
 
 \begin{proposition}\label{theorem:joining-1-}
 There exists a symmetric self-joining $((C,\QQ);(\Psi_1,\Psi_2))$ of $B$  such that:
 \begin{enumerate}[(a)]
 \item $\QQ[\Psi_1(f)\Psi_2(g)]=\PP[fg]$ if $\{f,g\}\subset H_{\mathrm{stb}}$;
 \item $\QQ[\Psi_1(f)\Psi_2(g)]=0$  if $f\in H_{\mathrm{sens}}$, $g\in \L2(\PP)$. 
 \end{enumerate}
 Furthermore:
 \begin{enumerate}[(i)] 
\item The law of such a symmetric  self-joining is uniquely determined by the  properties listed above; its operator is $\PP_{\stable}=U^{\mathbbm{1}_{\{K<\infty\}}}$ (recall Definition~\ref{generalized-noise:iv}).  
\item For any such symmetric self-joining we have that $\Psi_{1}(1_\PP)$ is independent of  $\Psi_{2}(1_\PP)$ given ${\Psi_{1}(\stable)\lor \Psi_{2}(\stable)}$, and also that $\Psi_{1}(1_\PP)$ is independent of $ \Psi_{2}(\stable)$ given ${\Psi_{1}(\stable)}$.
\item  For any such symmetric self-joining, $\mu_f[(1-)^{K}]=\QQ[\Psi_1(f)\Psi_2(f)]$ for all $f\in \L2(\PP)$, where one interprets $(1-)^\infty=0$ and $(1-)^n=1$ for $n\in \mathbb{N}_0$.
\end{enumerate}
 \end{proposition}
 \begin{proof}
 Reduce to the case when $\PP$ is standard and $0_\PP\ne 1_\PP$. Then the same proof works as for one-dimensional factorizations \cite[Theorem~4c3]{tsirelson-nonclassical}, namely one takes the limit of the symmetric self-joinings of Proposition~\ref{proposition:self-joinings} with $P=\{1_\PP\}$ and $\rho(1_\PP)\uparrow\uparrow 1$. Once existence has been established the listed properties follow at once. 
 \end{proof}
 Of course one could have done local combinations of the self-joining of Proposition~\ref{theorem:joining-1-} and of those of Proposition~\ref{proposition:self-joinings} according to a partition of unity of $B$, the extension is trivial.

\subsection{Conditions for classicality and blackness}

\begin{proposition}  (Cf. \cite[Lemma~1.4]{vershik-tsirelson}.)
If $\tilde{B}\supset B$ is another noise Boolean algebra under $\PP$ such that $\tilde{x}=\lor (B_{\tilde{x}})$ for all $\tilde{x}\in \tilde{B}$, then $H^{(1)}(\tilde{B})\supset H^{(1)}(B)$. In particular if $\tilde{B}$ is black then $B$ is black; if $B$ is classical then $\tilde{B}$ is classical.
\end{proposition}
\begin{proof}
According to Proposition~\ref{proposition:ctb-vs-arbitrary-joins-meets}, $\tilde{B}\subset \Cl(B)$, hence $\tilde{B}\subset \overline{B}$. But $H^{(1)}(B)=H^{(1)}(\overline{B})$, since $B$ and $\overline{B}$ have the same closure. Thus $H^{(1)}(\tilde{B})\supset H^{(1)}(\overline{B})=H^{(1)}(B)$.
\end{proof}

More quantitative conditions for classicality/blackness follow.
 
 \begin{definition}\label{definition:H-J}
 (Cf. \cite[p.~80]{picard2004lectures}.) Let $f\in\L2(\PP)$. For $x\in B$ we set  $\mathfrak{I}_x(f):=\PP[\sqrt{\var(f\vert x')}]$ (another kind of influence of $x$ on $f$) and then:
 \begin{align*}
 \mathbf{J}(f)&:=\inf_{b\in \mathfrak{F}_B}\sum_{a\in \at(b)}\left(\PP[\sqrt{\var(f\vert a')}]\right)^2=\inf_{b\in \mathfrak{F}_B}\sum_{a\in \at(b)}\mathfrak{I}_a(f)^2,\\
  \mathbf{H}(f)&:=\limsup_{b\in \mathfrak{F}_B}\sum_{a\in \at(b)}\left(\PP[\sqrt{\var(f\vert a')}]\right)^2=\limsup_{b\in \mathfrak{F}_B}\sum_{a\in \at(b)}\mathfrak{I}_a(f)^2,\\
  \mathbf{H}_1(f)&:={\text{$\downarrow$-}\lim_{b\in \mathfrak{F}_B}}\sum_{a\in \at(b)}\var(\PP[f\vert a]).
  \end{align*}
The limit in the preceding is nonincreasing due to the content of Remark~\ref{rmk:inequalities}\ref{rmk:inequalities:ii} (of course ${\inf_{b\in \mathfrak{F}_B}}$ may replace ${\text{$\downarrow$-}\lim_{b\in \mathfrak{F}_B}}$). 
 \end{definition}
 \begin{remark}
 If $\eta$ is a contraction then $\mathbf{H}(\eta\circ f)\leq \mathbf{H}(f)$ since contractions reduce (conditional) variance (a.s.). The set $\{\mathbf{H}=0\}$ is therefore a linear manifold containing the constants and closed under $\vert\cdot\vert$ from which we deduce that its closure is an $\L2$-space (cf. \cite[Theorem 6.37]{picard2004lectures}). (This argument does not work for $\{\mathbf{J}=0\}$, namely in the closure under addition part (at least not in any obvious way).)
 \end{remark}

 \begin{proposition}\label{proposition:var-sqrt}
Let $f\in \L2(\PP)$, $x\in B$. Then  $\var(\PP[f\vert x])\leq \mathfrak{I}_x(f)^2$; therefore $\mathbf{H}_1(f)\leq \mathbf{J}(f)$. Let further $a\in B$ and $h\in \mathrm{L}^\infty(\PP\vert_{x'})$. Then $\mathfrak{I}_a(\PP[fh\vert x])\leq \Vert h\Vert_{\mathrm{L}^\infty(\PP)}\mathfrak{I}_a(f)$; therefore $\mathbf{J}(\PP[fh\vert x])\leq \Vert h\Vert_{\mathrm{L}^\infty(\PP)}^2\mathbf{J}(f)$. 
 \end{proposition}
 \begin{proof}
 Replacing $\FF_{s,t}$ with $x$ and $\FF_{\mathbb{R}\backslash (s,t)}$ with $x'$ in the proof of \cite[Lemma~6.28]{picard2004lectures} we get the proof of the first claim. Replacing $\FF_{-\infty,0}$ with $x$, $\FF_{0,\infty}$ with $x'$, $\FF_{\mathbb{R}\backslash (s,t)}$ with $a'$, and $\FF_{(-\infty,0)\backslash (s,t)}$ with $x\land a'$   in the proof of \cite[Lemma~6.33]{picard2004lectures} we get the proof of second  claim (one may assume without loss of generality that $a\subset x$, indeed the l.h.s. of the inequality stays the same when passing from $a$ to $a\land x$ (since elements of $B$ commute and since $(a\land x)'\land x=a'\land x$), while its r.h.s. does not increase under this same replacement (by Jensen). 
 \end{proof}
 
 \begin{corollary}\label{corollary:black-first-chaos-norm-projection}
Let $f\in \L2(\PP)$. Then $$\Vert \mathrm{pr}_{H^{(1)}}(f)\Vert=\mathbf{H}_1(f)={\text{$\downarrow$-}\lim_{n\to\infty}}\sum_{a\in\at(b_n)}\var(\PP[f\vert a])$$ for all  nondecreasing sequences $(b_n)_{n\in \mathbb{N}}$ in $\mathfrak{F}_B$ whose union is dense in $B$. Therefore $f$ is orthogonal to the first chaos iff $\mathbf{H}_1(f)=0$, equivalently $\lim_{n\to\infty}\sum_{a\in\at(b_n)}\var(\PP[f\vert a])=0$ for some (all) sequences $(b_n)_{n\in \mathbb{N}}$ in $\mathfrak{F}_B$ whose union is dense in $B$, which is true if $\mathbf{J}(f)=0$. In particular $B$ is black iff the preceding condition is met for all $f\in \L2(\PP)$; if $\mathbf{J}(f)=0$ for all $f\in \L2(\PP)$, then $B$ is black. \qed
\end{corollary}
\begin{proof}
The first claim follows from Proposition~\ref{proposition:stability-sensitivity}\ref{proposition:stability-sensitivity:ii}. The rest is then immediate (taking into account the first part of Proposition~\ref{proposition:var-sqrt}).
\end{proof}

In fact, when it comes to the last statement of the preceding corollary we have a more nuanced version; it is a continuous version of the  Benjamini-Kalai-Schramm noise sensitivity theorem \cite[Section~1.6]{garban}.

\begin{proposition}
(Cf. \cite[Corollary~6.35]{picard2004lectures}.) If $f\in\L2(\PP)_0$ is such that $\mathbf{J}(f)=0$, in particular if $\mathbf{H}(f)=0$, then $f$ is sensitive.
\end{proposition}
\begin{proof}
It certainly suffices to show that for every $x\in B$ and then for all $g\in H^{(1)}\cap\L2(\PP\vert_x)$ and $h\in \L2(\PP\vert_{x'})$, $\PP[fgh]=0$ (because this, together with the assumption $\PP[f]=0$, implies that $f$ is orthogonal to a set that is dense in each chaos, therefore to each chaos, therefore to the whole of $H_{\mathrm{stb}}$). We may assume $h\in \mathrm{L}^\infty(\PP\vert_{x'})$ (by approximation). We have  $\PP[fgh]=\PP[\PP[fh\vert x]g]$. But $\mathbf{J}(\PP[fh\vert x])\leq \Vert h\Vert_{\mathrm{L}^\infty(\PP)}^2\mathbf{J}(f)=0$ by the second part of Proposition~\ref{proposition:var-sqrt}, and therefore $\PP[\PP[fh\vert x]g]=0$ according to Corollary~\ref{corollary:black-first-chaos-norm-projection}.
\end{proof}

\section{Classical noise Boolean algebras}
Attention is turned  squarely to classical noise Boolean algebras (except for one place, in Proposition~\ref{proposition:isomorphisms-of-noise-boolean-spectra}). First, every classical noise Boolean algebra is made into a complete noise factorization via its spectrum (Proposition~\ref{proposition:extension}). Second, a spectrum for a classical noise Boolean algebra is given that bears a kind-of ``symmetric Fock space structure'' (Theorem~\ref{proposition:classical-structure}).  

\subsection{Turning a classical noise Boolean algebra into a complete noise factorization}
For a classical noise Boolean algebra $B$ there is a way to use the part $\{K=1\}$ of the spectral space corresponding to the first chaos in order to produce a noise factorization indexed by the $\Sigma$-measurable subsets of $\{K=1\}$ and whose range is the closure of $B$. Moreover, this noise factorization is essentially injective. In a sense, up to the arbitrariness of $\{K=1\}$ (having to do with the non-uniqueness of spectrum), we produce a canonical indexation of $\overline{B}$.
\begin{proposition}[Indexing a classical noise Boolean algebra]\label{proposition:extension}
Assume $B$ is classical. Let $T:=\{K=1\}$ (fix some version) and $\TT:=\Sigma\vert_{T}$.  For $A\in \TT$ put 
$N_A:=\sigma(H(A))$. Then $N:\TT\to \overline{B}$ is an epimorphism of Boolean algebras, $\mu\vert_T$-essentially injective in the sense that for $\{A_1,A_2\}\subset \TT$, $N_{A_1}=N_{A_2}$ implies $\mu(A_1\triangle A_2)=0$ (also conversely, which is clear), it satisfies $N(S_x\cap T)=x$ for $x\in \overline{B}$, and it is $\sigma$-continuous from below in the sense that  $\lor_{n\in \mathbb{N}}N_{A_n}=N_{\cup_{n\in \mathbb{N}}A_n}$ for any nondecreasing sequence $(A_n)_{n\in \mathbb{N}}$ from $\TT$.
\end{proposition}
\begin{remark}
The $N$ of Proposition~\ref{proposition:extension} is nothing but  a factored probability space indexed by $(T,\mathcal{T})$ in the sense of Feldman \cite[Definition~1.1]{feldman}. Conversely, every such factored probability space $N$ gives rise to a complete classical noise Boolean algebra, namely, its range. Square-integrable additive integrals of a classical noise Boolean algebra $B$, i.e. elements of the first chaos, correspond in this vein to real-valued,  $(T,\mathcal{T})$-indexed, square-integrable decomposable process \cite[Definition~1.4]{feldman}. Taking an orthonormal frame for $(H_s)_{s\in T}$ we see that the corresponding $N$ is generated by a countable family of such processes and $N$ is linearizable in the sense of \cite[Definition~1.10]{feldman}. A thorough study of decomposable processes in this classical context, including a L\'evy-Khintchine representation, is the main contribution of \cite{feldman}.
\end{remark}
\begin{proof}
We may and do assume without loss of generality that $B=\overline{B}$. 

If $x\in \overline{B}$, then $S_x\cap T\in \TT$ and $N({S_x\cap T})=\sigma(\L2(\PP\vert_x)\cap {H^{(1)}})=x$ (by the classicality). 

To see that $N$ maps into $B$ note that $\BB:=\{S_x\cap T:x\in B\}$ (we mean that all the versions of $S_x\cap T$, $x\in B$, are included) is a sub-algebra of $\TT$ that is $\mu$-essentially generating, then recall the approximation of an element of a $\sigma$-field by elements of a generating algebra. For each $A\in \TT$ one gets a sequence $(x_n)_{n\in \mathbb{N}}$ in $B$ such that $(S_{x_n}\triangle A)\cap T$ converges to $\emptyset$ locally in $\mu$-measure as $n\to\infty$. This means that $\PP_{x_n}\to \pr_{H(A)}$ strongly as $n\to \infty$ on $H^{(1)}$, in particular $\pr_{H(A)}\vert_{H^{(1)}}$ belongs to the strong operator closure of $\{\PP_x\vert_{H^{(1)}}:x\in B\}$ (this closure is denoted $\mathbf{Q}$ in \cite[p.~340]{tsirelson}). But for $f\in H^{(1)}$ we have  $$\PP[f\vert N_A]=\PP[\pr_{H(A)}(f)\vert N_A]+\PP[\pr_{H(T\backslash A)}(f)\vert N_A]=\pr_{H(A)}(f),$$ since $N_A=\sigma(H(A))$ and $N_{T\backslash A}=\sigma(H(T\backslash A))$ are independent  thanks to \cite[Lemma~6.1]{tsirelson}. 
Therefore $\PP_{x_n}\to \PP_{N_A}$ strongly as $n\to \infty$ on $H^{(1)}$. 
At the same time, for any partition of unity $P$ of $B$ we  have $N_A=\lor_{p\in P}(N_A\land p)$: according to Proposition~\ref{corollary:partition-K=1},  $A$ is the a.e.-$\mu$ disjoint union of the $A\cap S_{p}$, $p\in P$, hence $$N_A=\sigma(H(A))=\sigma(\oplus_{p\in P}H(A\cap S_p))=\lor_{p\in P}\sigma(H(A\cap S_p))\subset \lor_{p\in P}(N_A\land p),$$ the reverse inclusion being trivial. By the classicality of $B$ and Proposition~\ref{proposition:convergence-in-classical} it follows that  $\lim_{n\to\infty}x_n=N_A$. Thus, at long last, we can conclude that $N_A\in B$ so that indeed $N:\TT\to B$. 

If $\{A_1,A_2\}\subset \TT$ and $N_{A_1}=N_{A_2}$ then we must have $\mu(A_1\triangle A_2)=0$ because disjoint subsets of $T$ lead via $N$ to independent $\sigma$-fields (as noted above, this is by \cite[Lemma~6.1]{tsirelson}) and because subsets of $T$ of positive $\mu$-measure generate non-trivial $\sigma$-fields. 

To check that $N$ is a homomorphism of Boolean algebras, we see first that $N$ preserves the join operation, which fact is immediate from the definition. It also clearly sends $\emptyset$ to $0_\PP$ and $T$ to $1_\PP$ ($\because$ of classicality, again). Finally, for $A\in \TT$, $N_A\lor N_{T\backslash A}=1_\PP$ and $N_A\land N_{T\backslash A}=0_\PP$ so that indeed $N_{T\backslash A}={N_A}'$.  

For the last claim, let $(A_n)_{n\in \mathbb{N}}$ be a nondecreasing sequence from $\TT$. We get immediately that $\PP_{N_{A_n}}=\pr_{H(A_n)}\to \pr_{H(\cup_{m\in \mathbb{N}}A_m)}=\PP_{N_{\cup_{m\in \mathbb{N}}A_m}}$ strongly on $H^{(1)}$ as $n\to\infty$. Applying Proposition~\ref{proposition:convergence-in-classical} yet again delivers  $\lim_{n\to\infty}N_{A_n}=N_{\cup_{m\in \mathbb{N}}A_m}$, completing the proof.
\end{proof}

\begin{example}
The content of Proposition~\ref{proposition:extension} is clearly visible in the context of Example~\ref{example:wiener}, indeed it gives a canonical extension of the natural indexation of the noise Boolean algebra attached to the increments of a Wiener process from  unions of intervals to all Lebesgue measurable sets: if $A$ is such a set, then $N_A$ is the $\sigma$-field generated by $\int_{-\infty}^\infty g(u)\dd W_u$ for $g\in \L2(\mathbb{R})$ vanishing off $A$.
\end{example}
In a sense there is only one kind of indexation of a classical $B$ of the kind that was got in Proposition~\ref{proposition:extension}.

\begin{corollary}\label{corollary:uniqueness-of-indexation}
Assume $B$ is classical. Let $(T^*,\mathcal{T}^*,\gamma^*)$ be a standard measure space and $N^*:\mathcal{T}^*\to\overline{B}$ be an epimorphism of Boolean algebras such that for $\{A_1,A_2\}\subset \TT^*$, $N^*_{A_1}=N^*_{A_2}$ iff $\gamma^*(A_1\triangle A_2)=0$ (equivalently,  for all $B\in \mathcal{T}^*$, $N^*_B=0_\PP$ iff $\gamma^* (B)=0$), also $\sigma$-continuous from below in the sense that  $\lor_{n\in \mathbb{N}}N^*_{A_n}=N^*_{\cup_{n\in \mathbb{N}}A_n}$ for any nondecreasing sequence $(A_n)_{n\in \mathbb{N}}$ from $\TT^*$. Then $\gamma:=\mu\vert_T$ is equivalent to a standard measure, which is mod-$0$ isomorphic to $\gamma^*$  (and also $\gamma^*$ is equivalent to a standard measure that is mod-$0$ isomorphic to  $\gamma$) via a map that carries $A$ to $A^*$ (mod $0$), whenever $N(A)=N^*(A^*)$, this for all $A\in \TT$, $A^*\in \TT^*$.
\end{corollary}
\begin{proof}
From $N^*$ and from $N$ of  Proposition~\ref{proposition:extension}  we get a Boolean algebra isomorphism of the measure algebras attached to  $\gamma$ and $\gamma^*$ that preserves countable unions (also atoms, but not necessarily their sizes). Use this isomorphism to transfer the measure $\gamma^*$ on $\mathcal{T}^*/_{\gamma^*}$  (really a different symbol should be used for the map on $\mathcal{T}^*/_{\gamma^*}$ vis-\`a-vis the one on $\mathcal{T}^*$ , but we transgress for simplicity) from $\mathcal{T}^*/_{\gamma^*}$ to a measure $\gamma^*_\leftarrow$ on $\mathcal{T}/_\gamma$. Further, $\gamma^*_\leftarrow$ is made into a proper measure on $\mathcal{T}$ in the obvious way (again we just keep the same symbol $\gamma^*_\leftarrow$, for simplicity)  and its null sets coincide with those of $\gamma$. Then $\gamma^*_\leftarrow$ and $\gamma$ are equivalent $\sigma$-finite measures. By standardness the desired conclusion follows.
\end{proof}

Up to Boolean algebra isomorphism a complete noise Boolean algebra is determined by the number of its atoms.

\begin{corollary}\label{corollary:b-a-iso}
Let $B$ be classical and let $B'$ be another classical noise Boolean algebra on a probability space $(\Omega',\FF',\PP')$. Then $\overline{B}$ and $\overline{B'}$ are isomorphic as Boolean algebras iff $B$ and $B'$ have the same number of atoms.
\end{corollary}
\begin{proof}
Atoms of $B$ are the same as the atoms of $\overline{B}$ and they correspond to atoms of $\mu\vert_{\{K=1\}}$ (same for $B'$). Apply Proposition~\ref{proposition:extension} noting that two standard measures are mod-$0$ isomorphic up to equivalence.
\end{proof}

Another trivial consequence of Proposition~\ref{proposition:extension} is that a complete classical noise Boolean algebra cannot be countably infinite: it is either finite, or its cardinality is at least that of the continuum. But  a  nonclassical (Example~\ref{example:closure-for-simplest-nonclassical}), even  black (Example~\ref{example:spectral-space-for-black}), noise Boolean algebra can be noise complete and countably infinite.

\subsection{Standard form of the spectrum of a classical noise Boolean algebra}
The aim of this subsection is to obtain a ``standardized'' spectrum for a classical noise Boolean algebra, namely one whose spectral space consists of finite subsets of a base  set $T$ and whose direct integral decomposition has a ``global tensor'' form structure (``locally'' it happens always, which was the content of Corollary~\ref{prop:noise-proj-bis}\ref{proj:iv}). Indeed the set $T$ will be the one of Proposition~\ref{proposition:extension}. 

We require first some technical results concerning symmetric Fock spaces. Recall that we identify the one-element subsets of $T$ with their single elements, i.e.  $T\choose 1$ with $T$. No measures appear in the next lemma; the qualifier ``generating'' does \emph{not} imply any completion (we stress it bacause of \ref{generalities:a}).
\begin{lemma}[Symmetric Fock space sigma-field]\label{lemma:fock}
Let $(T,\TT)$ be a measurable space. Put $(T^n)_{{\ne}}:=\{t\in T^n:\text{entries of $t$ pairwise distinct}\}$ for $n\in \mathbb{N}_0$.  Let $\mathcal{V}$ be the largest $\sigma$-field on $\tilde{S}:=(2^{T})_{\mathrm{fin}}$ w.r.t. which  the canonical quotient maps $q_n:(T^n)_{\ne }\to {T\choose n}$, $n\in \mathbb{N}_0$, are measurable.  
\begin{enumerate}[(i)]
\item $\mathcal{V}$ is the unique $\sigma$-field $\mathcal{W}$ on $\tilde{S}$ for which ${T\choose n}\in \mathcal{W}$ and for which $\mathcal{W}\vert_{{T\choose n}}=\{q_n(E):E\in  \TT^{\otimes n}\vert_{(T^n)_{{\ne}}},\text{ $E$ symmetric}\}$ for all $n\in \mathbb{N}_0$ (symmetric means invariant under permutations). 
\item\label{lemma:fock:ii} $\mathcal{V}$ is equal to $\mathcal{U}:=$ the smallest $\sigma$-field  on $\tilde{S}$ w.r.t. which the counting maps $c_A:=(\tilde{S}\ni L\mapsto \vert L\cap A\vert)$, $A\in \TT$, are measurable. (Hence, if $\TT$ admits a dissecting system \cite[Definition~A1.6.I]{vere-jones}, then $\mathcal{V}$ is generated  by the sets $(2^A)_{\mathrm{fin}}$, $A\in \TT$.)
\item\label{lemma:fock:iii} If $\Pi$ is an algebra on $T$, which generates $\TT$, then $\mathcal{V}$ is also equal to the $\sigma$-field $\UU_\Pi$ generated by the maps $c_P$, $P\in \Pi$. (Hence, if $\TT$ admits a generating dissecting system, then $\mathcal{V}$ is generated  by the sets $(2^A)_{\mathrm{fin}}$, for $A$ belonging to the algebra generated by this dissecting system.)
\end{enumerate}
\end{lemma}
\begin{remark}
Of course also $\mathcal{W}\vert_{{T\choose n}}=\{q_n(E):E\in  \TT^{\otimes n}\vert_{(T^n)_{{\ne}}}\}$. Besides, $\mathcal{V}\vert_{{T\choose 1}}=\TT$ up to the canonical identification of $T$ with $T\choose 1$, while $\mathcal{V}\vert_{T\choose 0}=2^{\{\emptyset\}}$ is trivial.
\end{remark}

\begin{proof}
Unique existence of $\mathcal{W}$ is clear. We show first that $\UU=\mathcal{W}=\UU_\Pi$.

On the one hand, for $A\in \TT$, for $n\in \mathbb{N}_0$ and for $m\in [n]_0$, $\{c_A=m\}\cap {T\choose n}=q_n((\cup_{f\in {[n]\choose m}}A^f\times (T\backslash A)^{[n]\backslash f})\cap (T^n)_{{\ne}})\in \mathcal{W}\vert_{T\choose n}\subset \mathcal{W}$ (here we have identified, for $f\in {[n]\choose m}$, a pair from $A^f\times (T\backslash A)^{[n]\backslash f}$ with a point in $T^n$ in the natural way). Thus the $c_A$, $A\in \TT$, are measurable w.r.t. $\mathcal{W}$, i.e. $\UU\subset \mathcal{W}$.

On the other hand, for $n\in \mathbb{N}_0$,  we check that $\mathcal{W}\vert_{{T\choose n}}\subset \UU_\Pi$ (which gives $\mathcal{W}\subset \mathcal{U}_\Pi$, and together with the trivial inclusion $\UU\supset \UU_\Pi$, $\UU=\mathcal{W}=\UU_\Pi$) as follows.

 For $E\subset (T^n)_{\ne}$ put $E_s:=q_n^{-1}(q_n(E))=\cup_{\sigma\in S_n}\mathrm{per}_\sigma^{-1}(E)$, the symmetrization of $E$ (the smallest symmetric set of $(T^n)_{\ne}$ containing $E$; $S_n$ being the symmetric group of order $n$, $\mathrm{per}_\sigma$ permuting the entries  of its argument according  to $\sigma$).  The class of sets $\mathcal{M}:=\{E\in \TT^{\otimes n}\vert_{(T^n)_{{\ne}}}:q_n(E)\in \UU_\Pi\}=\{E\in \TT^{\otimes n}\vert_{(T^n)_{{\ne}}}:q_n(E_s)\in \UU_\Pi\}$ is monotone (closure under $\uparrow$ sequential unions is trivial, under $\downarrow$ sequential intersections it is e.g. due to the fact that the symmetric group $S_n$ is finite, which yields that symmetrization commutes with $\downarrow$ sequential intersections [but not with complements!]). We will argue that $\mathcal{M}$ contains also the traces on $(T^n)_{{\ne}}$ of the members of the algebra on $T^n$ of all finite unions of measurable rectangles with sides from $\Pi$, which by monotone class will allow to conclude that $\mathcal{W}\vert_{{T\choose n}}\subset \UU_\Pi$.
 
 We are to show then that if $A\subset E^n$ is a finite union of measurable rectangles with sides from $\Pi$ (which is symmetric) then $q_n(A\cap (T^n)_{\ne})\in \UU_\Pi$. Thus (because forward images commute with unions) if $B_i\in \Pi$ for $i\in [n]$, we would like to see that $q_n((\prod_{i=1}^nB_i)\cap (T^n)_{\ne})\in \mathcal{U}_\Pi$ and further we may assume that there is a $\Pi$-measurable finite partition $\mathcal{P}$ of $T$ such that $B_i\in \mathcal{P}$ for all $i\in [n]$. But then indeed $q_n((\prod_{i=1}^nB_i)\cap (T^n)_{\ne})=\cap_{P\in \mathcal{P}}\{c_P=\vert\{i\in [n]:B_i=P\vert\}\in \UU_\Pi$.

Thus $\UU=\mathcal{W}=\UU_\Pi$. Finally, let us prove that $\mathcal{V}=\mathcal{W}$. (The paranthetical additions involving dissecting systems are ``immediate''. Indeed, on the one hand it is easy to express measurably the counting maps with set-inclusion in the presence of dissecting systems. Namely, if $(\TT_n)_{n\in \mathbb{N}}$ is such a dissecting system, then for $P$ from the algebra generated by $\cup_{n\in \mathbb{N}}\TT_n$, $c_P=\lim_{n\to\infty}\sum_{t\in \TT_n}\mathbbm{1}_{\tilde S\backslash (2^{P\cup (T\backslash t)})_{\mathrm{fin}}}$. On the other hand, for $A\in \TT$, $(2^A)_{\mathrm{fin}}=\cup_{n\in \mathbb{N}_0}q_n(A^n \cap (T^n)_{\ne })\in \mathcal{W}$.)

For all $n\in \mathbb{N}_0$, certainly ${T\choose n}\in \mathcal{V}$. Let $V\subset {T\choose n}$. If $V\in\mathcal{V}\vert_{T\choose n}$, then $q_n^{-1}(V)\in (\TT^{\otimes n})\vert_{(T^n)_{{\ne}}}$ is symmetric and $V=q_n(q_n^{-1}(V))\in \mathcal{W}\vert_{T\choose n}$. Conversely, if $V\in \mathcal{W}\vert_{T\choose n}$, then there is a symmetric $E\in (\TT^{\otimes n})\vert_{(T^n)_{{\ne}}}$ with $V=q_n(E)$, hence $q_n^{-1}(V)=q_n^{-1}(q_n(E))=E\in (\TT^{\otimes n})\vert_{(T^n)_{{\ne}}}$, i.e. $V\in  \mathcal{V}\vert_{T\choose n}$. Altogether it means that also $\mathcal{V}=\mathcal{W}$.
\end{proof}

\begin{lemma}[Symmetric Fock space measurable Hilbert space field structure]\label{lemma:fock-field}
Let $(T,\TT,\gamma)$ be a $\sigma$-finite measure space. On $\tilde{S}:=(2^{T})_{\mathrm{fin}}$ we introduce  $\mathcal{V}$ as the  $\sigma$-field described in various equivalent terms by Lemma~\ref{lemma:fock}, whose notation we  retain. Suppose there is countable $\AA\subset \TT$ that separates the points of $T$ (one says that $\TT$ is countably separating) so that in particular $(T^n)_{\ne}\in \mathcal{T}^{\otimes^n}$ for each $n\in \mathbb{N}_0$. Let  $\nu:=\sum_{n\in \mathbb{N}_0}(n!)^{-1}(q_n)_\star \gamma^n\vert_{(T^n)_{\ne}}$ (push-forwards relative to $\mathcal{V}$) and let $(\tilde{\Sigma},\tilde{\mu})$ be the completion of $(\mathcal{V},\nu)$. Assume also given a $\gamma$-measurable structure on a field $(G_t)_{t\in T}$ of Hilbert spaces. For $L\in \tilde{S}$ put $\tilde G_{L}:=\otimes_{s\in L}G_s$ ($\tilde G_\emptyset=\mathbb{R}$, of course).  For $\mathcal{P}$ a finite (possibly empty) $\TT$-measurable partition of a subset of $T$  set ${\mathcal{P}\choose \mathbf{1}}:=\{L\in {T\choose \vert \mathcal{P}\vert}:\vert L\cap P\vert=1\text{ for all }P\in \mathcal{P}\}$, and denote by $\mathfrak{P}(T)$ the collection of all such  finite $\TT$-measurable partitions of a subset of $T$. For vector fields $f$ and $g$ from  $(\tilde{G}_{L})_{L\in \tilde{S}}$ their bilinear, commutative and associative ``inner tensor product'' $f\otimes g$ is naturally defined: $$(f\otimes g)(L):=\sum_{J\in 2^L}f(J)\otimes g(L\backslash J)\in \tilde{G}_L, \quad L\in \tilde{S};$$ if one of $f$ and $g$ is defined only on a subset of $\tilde{S}$ then we define $f\otimes g$ by first extending it by zero to the whole of $\tilde{S}$.
\begin{enumerate}[(i)]
\item\label{lemma:fock-field:i} There exists on  the field $(\tilde{G}_{L})_{L\in \tilde{S}}$ exactly one $\tilde\mu$-measurable field structure such that the vector fields $$\tilde{f}^\mathcal{P}:=\otimes_{P\in \mathcal{P}}f\vert_{P}=\left(\tilde{S}\ni L \mapsto \mathbbm{1}_{{\mathcal{P}\choose \mathbf{1}}}(L)\otimes_{s\in L}f(s)\right),$$ for
 $f$ a $\gamma\vert_{\cup \mathcal{P}}$-measurable field, $\mathcal{P}\in \mathfrak{P}(T)$, are measurable. Here 
 $1_\emptyset:=\tilde{\emptyset}^\emptyset$  is the $\tilde{\mu}$-measurable vector field that vanishes a.e.-$\tilde{\mu}$  except at $\emptyset$, where it is equal to $1$ -- the ``vacuum'' vector.
\item\label{lemma:fock-field:ii}  The vector fields $\tilde{f}^\mathcal{P}$, for $f\in \int_{\cup \mathcal{P}}^\oplus G_s\gamma(\dd s)$, $\mathcal{P}\in \mathfrak{P}(T)$ with $\cup \mathcal{P}\in \{\emptyset,T\}$, are total in $\int^{\oplus}_{\tilde{S}}\tilde{G}_{L}\tilde \mu(\dd L)$. 
\item\label{lemma:fock-field:iii} If $\mathcal{P}\in \mathfrak{P}(T)$, then the map which sends $$\otimes_{P\in \mathcal{P}} \widetilde{f_P}^{\mathcal{Q}_P}\text{ to }\otimes_{P\in \mathcal{P}}f_P=\left(\tilde{S}\ni L\mapsto \mathbbm{1}_{\cup_{P\in \mathcal{P}}\mathcal{Q}_P \choose \mathbf{1}}(L)\otimes_{s\in L}(\cup_{P\in \mathcal{P}}f_P)(s)\right),$$ for $f_P\in \int_{\cup \mathcal{Q}_P}^\oplus G_s\gamma(\dd s)$, $\mathcal{Q}_P\in \mathfrak{P}(P)$, $P\in \mathcal{P}$, extends uniquely to a unitary isomorphism of $\otimes_{P\in \mathcal{P}}\int_{(2^P)_{\mathrm{fin}}}^\oplus\tilde{G}_L\tilde{\mu}(\dd L)$ to  $\int^{\oplus}_{(2^{\cup\mathcal{P}})_{\mathrm{fin}}}\tilde{G}_{L}\tilde \mu(\dd L)$.
\item\label{lemma:fock-field:iv} For $u\in G:=\int_T^{\oplus}G_s\gamma(\dd s)$ put $e(u):=\sum_{n\in \mathbb{N}_0}u^{\otimes n}=(\tilde{S}\ni L\mapsto \otimes_{s\in L}u(s))$ for the exponential vector associated to $u$. We have $e(0)=1_\emptyset$, $e(u+v)=e(u)\otimes e(v)$, $e(u)\in  \int^{\oplus}_{\tilde{S}}\tilde{G}_{L}\tilde \mu(\dd L)$ and $$\langle e(u),e(v)\rangle=e^{\int \langle u(s),v(s)\rangle\gamma_{\mathrm{d}}(\dd s)}\prod_{t\in T}(1+\langle u(t),v(t)\rangle\gamma(\{t\})),$$ where $\gamma_{\mathrm{d}}$ is the diffuse part of $\gamma$, $\{u,v\}\subset G$ (the product is actually countable [meaning: all but a countable number of terms are $1$] and log-absolutely convergent). Furthermore, choosing for each $a\in \at(\gamma)$ (:= atoms of $\gamma$) an orthonormal basis $O_a$ of $G_a$, then the family $(e(u))_{u\in G[O]}$, where $G[O]:=\{u\in G:u_a=0\text{ for all except finitely many }a\in \at(\gamma)\text{ for which in turn }u_a\in O_a\}$ is linearly independent and total  in $\int^\oplus_{\tilde{S}} \tilde{G}_L\tilde\mu(\dd L)$. Finally, the map $\mathrm{Exp}:=(G\ni u\mapsto e(u))$ is a bijection onto $\{f\in \int^\oplus_{\tilde S} \tilde G_L\tilde \mu(\dd L):f(\emptyset)=1\text{ and for all $A\in \TT$, for $\tilde{\mu}$-a.e. $L$, }f(L)=f(L\cap A)\otimes f(L\backslash A)\}$.
\end{enumerate}
\end{lemma}
\begin{remark}\label{remark:symmetrization-fock}
$(\tilde{S},\tilde{\Sigma},\tilde{\mu})$ is (may be called) the symmetric measure space over $(T,\TT,\gamma)$. When $\gamma$ is diffuse there is a canonical unitary isomorphism between $\L2(\tilde{\mu})$ and the symmetric (boson) Fock  space of $\L2(\gamma)$ \cite[Example~19.12]{parthasarathy}; furthermore (still when $\gamma$ is diffuse), the map, which, for $u\in G$, sends $e(u)$ to $E(u)$, where $E(u)$ is the exponential vector of $u$ \cite[Eq.~(19.1)]{parthasarathy}, extends uniquely to a unitary isomorphism between $\int^{\oplus}_{\tilde{S}}\tilde{G}_{L}\tilde \mu(\dd L)$ and the symmetric Fock space $\Gamma_s(G)$ of $G$ \cite[p.~124]{parthasarathy} (it being evidently scalar product-preserving and sending a total family onto a total family \cite[Proposition~19.4]{parthasarathy}, which is enough  \cite[Proposition~7.2]{parthasarathy}). If $\gamma$ is the counting measure on a countable $T$ with $\TT=2^T$, then $\tilde{\mu}$ in turn is counting measure on $\tilde{S}$; further, $\int^\oplus_{\tilde{S}} \tilde{G}_L\tilde\mu(\dd L)=\oplus_{L\in \tilde{S}}\otimes_{s\in L}G_s$ becomes an elementary orthogonal sum (finite iff $T$ is so) of finite tensor products of Hilbert spaces. The general case when $\gamma$ is discrete is not essentially different, and if $\gamma$ has both a diffuse and discrete part then by Item~\ref{lemma:fock-field:iii} $\int^\oplus_{\tilde{S}} \tilde{G}_L\tilde\mu(\dd L)$ decomposes into a tensor product of terms corresponding to each of the two respectively.
\end{remark}
\begin{example}
If $G_t=\mathbb{R}$ for $t\in T$, then with the standard $\gamma$-measurable structure on $(G_t)_{t\in T}$ (so $\gamma$-measurability means simply $\TT/\mathcal{B}_\mathbb{R}$-measurability), $\int^\oplus_{\tilde{S}}\tilde G_L\tilde\mu(\dd L)=\L2(\tilde{\mu})$, up to the canonical identification of $\mathbb{R}^{\otimes L}$, $L\in \tilde{S}$,  with $\mathbb{R}$.
\end{example}
\begin{proof}
Let $(f_n)_{n\in \mathbb{N}}$ be a fundamental sequence of $\gamma$-measurable vector fields (for which we may, if we like, (*) assume that each element thereof is supported on  a set of finite $\gamma$-measure with a bounded norm-function thereon). Then $\widetilde{\cup_{P\in \mathcal{P}} f_{n_P}\vert_P}^\mathcal{P}$, as the $n_P$ range over $\mathbb{N}$ for $P\in \mathcal{P}$, and as $\mathcal{P}$ ranges over those members of $\mathfrak{P}(T)$ that are $\AA$-measurable partitions of $T$ or empty, form a countable collection of vector fields, whose pairwise pointwise scalar products are $\tilde{\mu}$-measurable, and which pointwise form a total set (because $T$ is countably separated by $\AA$). Therefore \cite[p.~167, Proposition~II.1.4.4]{dixmier1981neumann} there exists on $(\tilde{G}_{L})_{L\in \tilde{S}}$ a unique $\tilde{\mu}$-measurable structure that makes these vector fields (into a fundamental sequence) and hence each $\tilde{f}^\mathcal{P}$ as in \ref{lemma:fock-field:i}  measurable (because the pointwise scalar products of the latter with the former are measurable \cite[p.~166, Proposition~II.1.4.2]{dixmier1981neumann}). 

Thus \ref{lemma:fock-field:i} is proved. \ref{lemma:fock-field:ii} follows at once \cite[p.~172, Proposition~II.1.6.7]{dixmier1981neumann} (using (*)).
 As for \ref{lemma:fock-field:iii}, the map is scalar-product preserving and maps onto a total set, hence extends uniquely to a unitary isomorphism  \cite[Proposition~15.4]{parthasarathy}. 

Finally, let us prove \ref{lemma:fock-field:iv}. The formula for $e(u+v)$ and $e(0)=1_\emptyset$ are immediate by definitions. Checking the scalar product (first, the norm) is a direct computation, in particular the computation reveals that in fact $e(u)\in  \int^{\oplus}_{\tilde{S}}\tilde{G}_{L}\tilde \mu(\dd L)$.

 Showing that  $(e(u))_{u\in G}$ is a linearly independent and total family in $\int^\oplus_{\tilde{S}} \tilde{G}_L\tilde\mu(\dd L)$ decomposes according to the previous item into establishing the same for the two special cases when $\gamma$ is diffuse or discrete. The latter case is elementary, we treat the former, which is similar to establishing the analogous claim for the symmetric Fock space of a Hilbert space \cite[Proposition~19.4]{parthasarathy}, but let us reconstruct it here for completeness.  Assume then $\gamma$ is diffuse. 
 
Take $U\subset G$ finite. For $u_1\ne u_2$ from $U$, $D_{u_1,u_2}:=\{v\in G:\langle u_1,v\rangle\ne \langle u_2,v\rangle\}$ is open and dense in $G$; the intersection of the $D_{u_1,u_2}$ over such  $u_1\ne u_2$ from $U$ is therefore non-empty (indeed open and dense). Therefore there is $v\in G$, such that the $\langle u,v\rangle$, $u\in U$, are pairwise distinct.  Suppose $\sum_{u\in U}\alpha_u e(u)=0$ for some real $\alpha_u$, $u\in U$. Then for all $z\in \mathbb{R}$, $$0=\langle e(zv),\sum_{u\in U}\alpha_ue(u)\rangle=\sum_{u\in U}\alpha_ue^{z\int \langle u(s),v(s)\rangle\gamma(\dd s)}.$$ As functions of $z\in \mathbb{R}$ the expressions next to the $\alpha_u$, $u\in U$, in the preceding display are linearly independent. We conclude that $\alpha_u=0$ for all $u\in U$ and linear independence is established. 

 For the totality, let $\mathfrak{T}$ be the closure of the linear span of $\{e(u):u\in G\}$. Notice that $1_\emptyset=e(0)\in \mathfrak{T}$. Let  $\mathcal{P}\in \mathfrak{P}(T)$ with $\cup \mathcal{P}=T$ and take any $f\in G$. Let $z=(z_P)_{P\in \mathcal{P}}\in \mathbb{R}^\mathcal{P}$. Then $f_{z}:=\sum_{p\in P}z_p\mathbbm{1}_Pf\in G$ and, for $n\in \mathbb{N}_0$, the coefficient of $\prod_{p\in P} z_p$ in $(f_z)^{\otimes n}$, it being $\widetilde f^{\mathcal{P}}$, belongs to $\mathfrak{T}$, provided each $(f_z)^{\otimes n}$, $z\in \mathbb{R}^P$, does (just differentiate successively w.r.t. the $z_p$, $p\in P$, at zero). But if $u\in G$, then for $n\in \mathbb{N}_0$, $$u^{\otimes n}=\lim_{\epsilon\downarrow 0}\frac{e(\epsilon u)-\sum_{k=0}^{n-1}\epsilon^ku^{\otimes k}}{\epsilon^n},$$ and it follows by induction that $u^{\otimes n}\in \mathfrak{T}$ for all $n\in \mathbb{N}_0$ (of course we do not need it for $n=0$ because we already know it, but the inductive argument is valid even then anyway). Applying \ref{lemma:fock-field:ii} we are done.
 
 Concerning the map $\mathrm{Exp}$ it is clearly ``into'' and injective since $0=\Vert e(u)\Vert^2$ implies $u=0$ a.e.-$\gamma$ for all $u\in G$; it is surjective since for an $f\in \int^\oplus_{\tilde S} \tilde G_L\tilde \mu(\dd L)$ satisfying $f(\emptyset)=1$ and $f(L)=f(L\cap A)\otimes f(L\backslash A)$ for $\tilde{\mu}$-a.e. $ L$ for all $A\in \TT$, we have $f=e(f\vert_T)$ a.e.-$\tilde\mu$ (due to the fact that $\TT$ is countably separating).
\end{proof}

Here now is the announced result concerning spectra of classical noise Boolean algebras.

\begin{theorem}[Spectrum of classical noise Boolean algebra]\label{proposition:classical-structure}
Assume the setting and notation of Proposition~\ref{proposition:extension}  and continued in Lemmas~\ref{lemma:fock} and~\ref{lemma:fock-field}  for a fixed version of $T=\{K=1\}$, it does not matter which, for $\TT=\Sigma\vert_T$,  for $\gamma=\mu\vert_T$ and for $G_s=H_s$, $s\in T$.  In particular $N$, $(\tilde{S},\tilde\Sigma,\tilde\mu)$ and  $\int^{\oplus}_{\tilde{S}}\tilde{H}_{L}\tilde \mu(\dd L)$ are defined. 
\begin{enumerate}[(i)]
\item\label{proposition:classical-structure:iv} There exists a unique unitary isomorphism $\tilde{\Psi}:\L2(\PP)\to \int^{\oplus}_{\tilde{S}}\tilde{H}_{L}\tilde \mu(\dd L)$ such that: for each $n\in \mathbb{N}$, for all $\TT$-measurable partitions  $(A_1,\ldots,A_n)$ of $T$, for all $u_i\in \int_{A_i}^\oplus H_s\mu(\dd s)=\Psi(H^{(1)}(B)\cap \L2(\PP\vert_{N(A_i)}))$, $i\in [n]$, one has $\tilde{\Psi}^{-1}(\otimes_{i=1}^nu_i)=\prod_{i=1}^n{\Psi}^{-1}(u_i)$; also, $\tilde \Psi(1)=\mathbbm{1}_{\{\emptyset\}}$. The map $\tilde\Psi$ satisfies moreover $\tilde{\Psi}^{-1}(\otimes_{i=1}^n f_i)=\prod_{i=1}^n\tilde{\Psi}^{-1}(f_i)$ for all $f_i\in \int_{(2^{A_i})_{\mathrm{fin}}}^\oplus\tilde{H}_L\tilde{\mu}(\dd L)=\Psi(\L2(\PP\vert_{N(A_i)}))$, $i\in [n]$, i.e. $\tilde{\Psi}=\otimes_{i\in [n]}\tilde{\Psi}\vert_{\L2(\PP\vert_{N(A_i)})}$ up to the natural unitary equivalences (one of them is in Lemma~\ref{lemma:fock-field}\ref{lemma:fock-field:iii}, the other standard), this for all $\TT$-measurable partitions of unity $(A_1,\ldots,A_n)$ of $T$, all $n\in \mathbb{N}$. 
\item\label{proposition:classical-structure:iii} $((\tilde S,\tilde\Sigma,\tilde \mu);\tilde{\Psi})$ is a spectral resolution of $B$ to which are associated the counting map $\tilde{K}$, $\tilde{T}:=\{\tilde{K}=1\}$ etc. (all quantities get a $\widetilde{\phantom{ll}}$).
\item\label{proposition:classical-structure:i} $\tilde{K}$ becomes literally (not just by name) the counting map: $\tilde{K}(L)=\vert L\vert$ for $\tilde{\mu}$-a.e. $L\in \tilde{S}$, and we choose its version by insisting that the equality prevails everywhere (not just $\tilde{\mu}$-a.e.).
\item\label{proposition:classical-structure:ii}  For $A\in \TT= \tilde{\TT}$ the spectral set associated with $N_A$ in the spectral resolution $((\tilde S,\tilde\Sigma,\tilde \mu);\tilde{\Psi})$ is $\tilde{S}_{N_A}=(2^A)_{\mathrm{fin}}\in \tilde{\Sigma}$.
\item\label{proposition:classical-structure:v} For $u\in \int_T^{\oplus}H_s\gamma(\dd s)$, $\tilde{\Psi}^{-1}(e(u))=\tilde{\Psi}^{-1}(e(u\mathbbm{1}_A))\tilde{\Psi}^{-1}(e(u\mathbbm{1}_{T\backslash A}))$ for all $A\in \TT$; $\tilde{\Psi}^{-1}(e(u))$ is a square-integrable multiplicative integral. Furthermore, for $\{u,v\}\subset \int_T^{\oplus}H_s\mu(\dd s)$, $$\langle\tilde{\Psi}^{-1}(e(u)),\tilde{\Psi}^{-1}(e(v))\rangle=e^{\int \langle u(s),v(s)\rangle\mu_{\mathrm{d}}(\dd s)}\prod_{t\in T}(1+\langle u(t),v(t)\rangle\mu(\{t\})),$$ where $\mu_{\mathrm{d}}$ is the diffuse part of $\mu\vert_T$. Choose for each $a\in \at(\mu\vert_T)$ an orthonormal basis $O_a$ of $H_a$. The family of the $\tilde{\Psi}^{-1}(e(u))$ as $u$ ranges over those elements of $u\in \int_T^{\oplus}H_s\mu(\dd s)$ satisfying $u(a)=0$ for all except  finitely many $a\in \at(\mu\vert_T)$ for which in turn $u_a\in O_a$, is linearly independent and total in $\L2(\PP)$. Finally, the map $\mathsf{Exp}:=(H^{(1)}\ni f\mapsto \tilde\Psi^{-1}(e(\Psi(f)\vert_T)))$ is a bijection of the first chaos onto the square-integrable multiplicative integrals of $B$.
\item\label{remark:mod-0-iso-L} Fix a countable noise Boolean algebra $B_0$ dense in $B$. For $s\in S$ put $\Phi_s:=\{x\in B_0:s\in S_x\}$ and $\LL_s:=(\cap\{S_x:x\in \Phi_s\})\cap T$. Then $\LL=(S\ni s\mapsto \LL_s)$ is a mod-$0$ isomorphism of $\mu$ onto the $\sigma$-finite ($\therefore$ standard) $\mu'=\LL_\star\mu\sim \tilde{\mu}$ (completion implicit), which carries $S_x$ onto $\tilde{S}_x$ for $x\in B$ and satisfies $H(A)=\tilde{H}(\LL(A))$ for all $A\in \Sigma$ (hence the pointwise scalar products verify $\langle \tilde \Psi(f),\tilde\Psi(g)\rangle\circ \LL=\langle \Psi(f),\Psi(g)\rangle$ a.e.-$\mu$ for $\{f,g\}\subset \L2(\PP)$).
\end{enumerate}
\end{theorem}

\begin{remark}
If a (classical or not) $B$ is the range of a $\gamma$-essentially injective noise factorization $N$ over an algebra $\AA$ of subsets of a set $T$  and admits  a spectral resolution with  the symmetric measure space $(\tilde S,\tilde \Sigma,\tilde \mu)$ over $(T,\TT,\gamma)$ as spectral space, $\TT\supset\AA$, in such a way that $\tilde S_{N_A}=(2^A)_{\mathrm{fin}}$ a.e.-$\tilde{\mu}$ for all $A\in \AA$, then automatically $\tilde{K}<\infty$ a.e.-$\tilde{\mu}$ and $B$ is classical ($\tilde{K}$ being the associated counting map). Only classical $B$ can have a spectral decomposition of the sort given in the preceding theorem.
\end{remark}

\begin{definition}
The spectrum for (a classical) $B$ is said to be in standard form if  $((S,\Sigma,\mu);\Psi)=((\tilde S,\tilde\Sigma,\tilde\mu);\tilde\Psi)$ (up to the canonical identification of $T$ with $T\choose 1$).
\end{definition}

\begin{example}
The spectra of  Examples~\ref{example:classical-spectrum} and~\ref{example:wiener} are in standard form up to the obvious natural identifications (using Theorem~\ref{prop:dicrete}\ref{prop:dicrete:iii-} in the first case). We identify the square-integrable multiplicative integrals as being of the form  $\prod_{a\in A}(1+g_a)$ for  $g=(g_a)_{a\in A}\in \prod_{a\in A}\L2(\PP\vert_a)_0$ satisfying $\sum_{a\in A}\Vert g_a\Vert^2<\infty$, respectively of the form $\exp\left(\int_{-\infty}^\infty g(t)\dd W_t-\frac{1}{2}\int_ {-\infty}^\infty g^2(t)\dd t\right)$ for  $g\in \L2(\mathbb{R})$.
\end{example}

\begin{remark}
Every noise Boolean algebra  $B$ gives rise to a factorization of $\L2(\PP;\mathbb{C})$, i.e. to a distributive sublattice of the lattice of all von Neumann algebras of $\L2(\PP;\mathbb{C})$, consisting of factors, closed for the commutant operation and containing $\BB(\L2(\PP;\mathbb{C}))$ (:= all bounded operators on $\L2(\PP;\mathbb{C})$), as follows. For $x\in B$ set $F_x:=\{X\otimes \mathbf{1}_{\L2(\PP\vert_{x'};\mathbb{C})}:X\in \BB(\L2(\PP\vert_x;\mathbb{C}))\}=\BB(\L2(\PP\vert_x;\mathbb{C}))\otimes (\mathbb{C}\mathbf{1}_{\L2(\PP\vert_{x'};\mathbb{C})})$, which is a factor of $\L2(\PP;\mathbb{C})=\L2(\PP\vert_x;\mathbb{C})\otimes \L2(\PP\vert_{x'};\mathbb{C})$ (up to the natural unitary isomorphism), of type $\mathrm{I}$. Then $\mathfrak{F}_B:=\{F_x:x\in B\}$ is the mentioned factorization (of type $\mathrm{I}$, because all its elements are so); moreover, the map $(B\ni x\mapsto F_x\in \mathfrak{F}_B)$ is an isomorphism of Boolean algebras and $\mathfrak{F}_B$ is complete iff $B$ is so. Taking into account  Remark~\ref{remark:symmetrization-fock} we see from Theorem~\ref{proposition:classical-structure} that for a complete classical noise Boolean algebra $B$ the factorization $\mathfrak{F}_B$ is unitarily isomorphic to the Fock(-like) factorization associated to the direct integral $\int_T^{\oplus}H_s\gamma(\dd s)$ (when $\gamma$ is diffuse, then this is a Fock factorization in the traditional sense \cite[pp.~86-87]{vershik-tsirelson} \cite[Section~5]{araki}, otherwise it has an extra part coming out of $\oplus_{s\in \at(\mu)}H_s$). In \cite[Theorems~4.1 and~6.1]{araki} complete type $\mathrm{I}$ factorizations over a separable Hilbert space satisfying a certain condition (which by the totality of Item~\ref{proposition:classical-structure:v} is met for $\mathfrak{F}_B$ \cite[p.~210, sentence including Eq.~(6.1)]{araki}) are shown to be unitarily equivalent to Fock(-like) factorizations. Thus the existence of a spectral resolution in standard form for a classical $B$ is the commutative analogue of the results of \cite{araki}. 
\end{remark}

%

\begin{proof}
The idea is to start with the given spectral resolution $((S,\Sigma,\mu);\Psi)$ and gradually transform it into $((\tilde S,\tilde\Sigma,\tilde \mu);\tilde{\Psi})$. 
We may and do assume $B=\overline{B}$.

\textbf{Laying the groundwork: discarding some $\mu$-negligible sets and fixing some representatives of $\mu$-classes.} 

--- Let $(b_n)_{n\in \mathbb{N}}$ be a $\uparrow$ sequence in $\mathfrak{F}_B$ such that $B_0:=\cup_{n\in \mathbb{N}}b_n$ is dense in $B$.  Discarding a $\mu$-negligible set we fix representatives $S_x$, $x\in B_0$, such that $S_{x\land y}=S_x\cap S_y$ for $\{x,y\}\subset B_0$, $S_{1_\PP}=S$, and (therefore) we fix also the maps $K_{b_n}$, $n\in \mathbb{N}$, derived therefrom, finally  representatives of the projections $\pr_x$, $x\in B_0$. Discarding a further $\mu$-negligible set of $S$ we may and do assume that  (everywhere, not just a.e.-$\mu$): $\{K=0\}=S_{0_\PP}=\{\emptyset_S\}$; $K=\text{$\uparrow$-$\lim$}_{n\to \infty}K_{b_n}<\infty$  (Proposition~\ref{prop;K-finite}; finiteness $\because$ of classicality); the sets $\{S_x:x\in B_0\}$ separate $S$; $\pr_x=\mathrm{id}_S$ on $S_x$ for $x\in B_0$ and $\pr_x^{-1}(S_y)=S_{y\lor x'}$, $\pr_x\circ \pr_y=\pr_{x\land y}$ for $\{x,y\}\subset B_0$ (Theorem~\ref{thm:noise-projections}\ref{proj:ii}, Remark~\ref{remark:exceptional-sets-projections}, Proposition~\ref{prop:noise-proj-bis}\ref{proj:ii--}); $\{S_a\cap T:a\in \at(b_n)\}$ is a partition of $T$ for each $n\in \mathbb{N}$ (Corollary~\ref{corollary:partition-K=1}); 
$K=\sum_{a\in \at(b_n)}K(\pr_a)$ for all $n\in \mathbb{N}$ (Proposition~\ref{proposition:K-and-projections});  $\underline{b_n}(\pr_x)=\underline{b_n}\cap x$  for all $x\in b_n$ for all $n\in \mathbb{N}$ (Corollary~\ref{projections-and-K-misellany}\ref{projections-and-K-misellany:i}); for all $n\in \mathbb{N}$ and $x\in b_n$,  $$\left(\cap_{a\in \at(b_n)_x}\pr_a^{-1}(\{K=1\})\right) \cap \left(\cap_{a\in \at(b_n)_{x'}}\pr_a^{-1}(\{\emptyset_S\})\right)=\{\underline{b_n}=x\}\cap \{K=K_{b_n}\}$$  (Corollary~\ref{projections-and-K-misellany}\ref{projections-and-K-misellany:ii}). 

--- The preceding allows us to avail ourselves of not having to consider more negligible sets in any one go as can be stomached.

\textbf{Construction of the isomorphism $\LL$, viz. adjustment of the spectral space.}

--- Using $B_0$ we introduce $\LL$ as in Item~\ref{remark:mod-0-iso-L}. Let $s\in S$. If $n\in \mathbb{N}$ is such that $K_{b_n}(s)=K(s)$ ($\therefore$ $=K_{b_m}(s)$ for all $m\in \mathbb{N}_{\geq n}$), then $\LL_s=\{\pr_a(s):a\in \at(b_n)\cap 2^{\underline{b_n}(s)}\}=:\LL_s'$, a set of size $K(s)$, which may be seen as follows:
\begin{itemize}[leftmargin=0.75cm]
\item  no member of $\LL_s'$ can be from $S_{0_\PP}$, for if $a\in \at(b_n)\cap 2^{\underline{b_n}(s)}$ is such that $\pr_a(s)\in S_{0_\PP}$, then $s\in S_{0_\PP\lor a'}\cap S_{\underline{b_n}(s)}=S_{\underline{b_n}(s)\cap a'}$, but $\underline{b_n}(s)\cap a'\subsetneq \underline{b_n}(s)$ (it differs by the atom $a$ of $b_n$), contradicting the minimality of $\underline{b_n}(s)$;
\item if $\pr_a(s)=\pr_c(s)$ for  $a\ne c$ from $\at(b_n)\cap 2^{\underline{b_n}(s)}$, then $\pr_a(s)=\pr_c(s)\in S_{a}\cap S_c=S_{a\land c}=S_{0_\PP}$, contradicting the previous bullet point, which means that $\vert \LL_s'\vert=K(s)$;
\item for $m\in \mathbb{N}_{\geq n}$, $\LL_s'=\{\pr_a(s):a\in \at(b_m)\cap 2^{\underline{b_m}(s)}\}=:\LL_s^m$, in fact if $c\in \at(b_m)\cap 2^{\underline{b_m}(s)}$ and $a\in \at(b_n)\cap 2^{\underline{b_n}(s)}$ are such that $c\subset a$ (clearly each $b_n$-atom of $\underline{b_n}(s)$ includes precisely one $b_m$-atom of $\underline{b_m}(s)$), then $s\in S_{\underline{b_m}(s)}\cap S_{\underline{b_n}(s)}=S_{\underline{b_n}(s)\cap \underline{b_m}(s)}\subset S_{c\lor a'}$, hence $\pr_a(s)\in S_c$ and thus $\pr_a(s)=\pr_c(\pr_a(s))=\pr_c(s)$;
\item for $a\in \at(b_n)\cap 2^{\underline{b_n}(s)}$, if $x\in B_0$ is such that $s\in S_x$, then a fortiori $s\in S_{x\lor a'}=\pr_a^{-1}(S_x)$, so $\pr_a(s)\in S_x$;
\item  $K(\pr_a(s))=1$ for all $a\in \at(b_n)\cap 2^{\underline{b_n}(s)}$;
\item by the preceding two bullet points $\LL_s'\subset \LL_s$;
\item if $u\in \LL_s$, then for each $m\in \mathbb{N}_{\geq  n}$, $u\in T\cap S_{\underline{b_m}(s)}$, and $u$ belongs to precisely one $S_{a}$, $a\in \at(b_m)\cap 2^{\underline{b_m}(s)}$, which also contains an element of $\LL_s^m=\LL_s'$, but the sets $S_a$, $a\in \at(b_m)$, $m\in \mathbb{N}$, separate the points of $T$, so this is only possible if $u\in \LL_s'$, which means that also $\LL_s\subset \LL_s'$.
\end{itemize}
So indeed $\LL_s=\LL'_s=\{\pr_a(s):a\in \at(b_n)\}\backslash\{\emptyset_S\}$ for all large enough $n\in \mathbb{N}$ (specifically, for all those $n\in\mathbb{N}$ for which $K_{b_n}(s)=K(s)$).

--- If $s_1\ne s_2$ are from $S$, then there is $x\in B_0$ such that $S_x$ separates $s_1$ and $s_2$. There is $n\in \mathbb{N}$ such that $K_{b_n}(s_1)=K(s_1)$ and $K_{b_n}(s_2)=K(s_2)$ and $x\in b_n$. Now, precisely one of the $s_1$, $s_2$ is contained in $S_x$, while the other is not, say $s_1$ without loss of generality. Then $\underline{b_n}(s_1)\subset x$ but $\underline{b_n}(s_2)\not\subset x$, hence there is an atom $a$ of $b_n$ which discerns between  $\underline{b_n}(s_1)$ and $\underline{b_n}(s_2)$ (is contained in the second, but not in the first). It follows that $\LL_{s_1}=\{\pr_a(s_1):a\in \at(b_n)\cap 2^{\underline{b_n}(s_1)}\}\ne \{\pr_a(s_2):a\in \at(b_n)\cap 2^{\underline{b_n}(s_2)}\}=\LL_{s_2}$.  We have established that the map  $\LL:=(S\ni s\mapsto \LL_s\in \tilde{S})$ is injective, it sends $\{K=n\}$ to ${T\choose n}$ for all $n\in \mathbb{N}_0$, also, if $s\in T$, then $\LL_s=\{s\}$.
 
--- Claim: if $x\in B_0$ and $S_{x}\cap \{K=1\}=\emptyset$, then $x=0_\PP$. Indeed if $x\ne 0_\PP$, then $S_x\ne S_{0_\PP}$, take $s\in S_x\backslash S_{0_\PP}$. Now let $n\in \mathbb{N}$ be such that $K(s)=K_{b_n}(s)$. Then  $K(\pr_a(s))=1$, whence $\pr_a(s)\in \{K=1\}\cap S_x$, for all $a\in  \at(b_n)\cap 2^{\underline{b_n}(s)}\ne \emptyset$, a contradiction. 
 
--- Next we argue that $\LL\in \Sigma/\mathcal{V}$ (recall $\mathcal{V}$ from Lemma~\ref{lemma:fock}). Check that $\LL^{-1}(\{c_{P\cap T}= m\})=\pr_{x}^{-1}(\{K=m\})$ a.e.-$\mu$ whenever $P=S_x$ a.e.-$\mu$ on $T$ for an $x\in B_0$ and whenever $m\in \mathbb{N}_0$. Indeed, for $\mu$-a.e. $s$ (specifically, for all $s\in S\backslash \cup_{n\in \mathbb{N}}\cup_{a\in \at(b_n)}\pr_a^{-1}([P\triangle S_x]\cap T)$), $c_{P\cap T}(\LL(s))=m$ iff for all large enough $n\in \mathbb{N}$, the number of atoms of $b_n$ contained in $\underline{b_n}(s)$ that are also contained in $x$ is $m$, which is equivalent to $K_{b_n}(\pr_x(s))=m$ for all large enough $n\in \mathbb{N}$, i.e. to $K(\pr_x(s))=m$. Since $\{P\in \TT:P=S_x\cap T\text{ a.e.-$\mu$ for an }x\in B_0\}$ is an algebra on $T$, which generates $\TT=\Sigma\vert_T$, it follows that $\LL\in \Sigma/\mathcal{V}$ (Lemma~\ref{lemma:fock}\ref{lemma:fock:iii}).

--- Recall from Item~\ref{remark:mod-0-iso-L} that $\mu':=\LL_\star\mu$ (completion implicit). Let us check that $\mu'\sim \tilde{\mu}$. Fix $n\in \mathbb{N}$; we verify equivalence on ${T\choose n}$ (it is trivial for $n=0$ as the atom $\emptyset_S$ of $\mu$ gets mapped by $\LL$ to the atom $\emptyset$ of $\mu'$, which is also an atom of  $\tilde{\mu}$; it is even trivial for $n=1$, indeed $\mu=\mu'=\tilde{\mu}$ on $T$). Suppose first that $\tilde A\in \mathcal{V}\vert_{T\choose n}$ is such that $(\mu\vert_T)^n(q_n^{-1}(\tilde{A}))=0$ and we will show that $\mu'(\tilde A)=0$. We see that $\LL^{-1}(\tilde A)\subset \cup_{m\in \mathbb{N}}\{\{\pr_a:a\in \at(b_m)\}\backslash \{\emptyset_S\}\in \tilde{A}\}$ so it suffices to check that $\mu(\{\{\pr_a:a\in \at(b_m)\}\backslash \{\emptyset_S\}\in \tilde{A}\})=0$ for all $m\in \mathbb{N}$. Further, for all $m\in \mathbb{N}$,  $\{\{\pr_a:a\in \at(b_m)\}\backslash \{\emptyset_S\}\in \tilde{A}\}\subset \cup_\rho \{(\pr_{\rho(i)})_{i\in [n]}\in \overline{A}\}$, where the union is over all injective maps $\rho:[n]\to \at(b_m)$ and  $\overline{A}:=q_n^{-1}(\tilde{A})$ has $\mu^n$-measure zero; it is enough to verify that $\mu( \{(\pr_{\rho(i)})_{i\in [n]}\in \overline{A}\})=0$ for each such $\rho$. But there is a probability $\nu$ equivalent to $\mu$ under which $\pr_{\rho(i)}$, $i\in [n]$, are independent and the law of $\pr_{\rho(i)}$ under $\nu$ is $\nu(\cdot \cap S_{\rho(i)})$ for each $i\in [n]$ (Corollary~\ref{prop:noise-proj-bis}\ref{proj:iv}\ref{proj:iv-a}, Remark~\ref{remark:choose-equivalent}). Hence $\nu( \{(\pr_{\rho(i)})_{i\in [n]}\in \overline{A}\})\leq \nu^n(\overline{A})=0$. We have proved that $\mu'(\tilde{A})=0$. Conversely suppose an $\tilde A\in \mathcal{V}\vert_{T\choose n}$ satisfies $(\mu\vert_T)^n(q_n^{-1}(\tilde{A}))>0$ and we show that $\mu'(\tilde{A})>0$. The countable collection $\cup_{m\in \mathbb{N}}\{S_a\cap T:a\in \at(b_m)\}$ separates the points of $T$, hence $(T^n)_{\ne}$ is the union of sets of the form $(S_{a_1}\cap T)\times \cdots \times (S_{a_n}\cap T)$ for independent atoms $a_1,\ldots,a_n$ from $b_m$, $m\in \mathbb{N}$. It follows that $(\mu\vert_T)^n(q_n^{-1}(\tilde{A})\cap (S_{a_1}\cap T)\times \cdots \times (S_{a_n}\cap T))>0$ for some independent atoms $a_1,\ldots,a_n$ from $b_m$ for some $m\in \mathbb{N}$. By a similar token as above (existence of equivalent $\nu$ under which the projections $\pr_{a_1},\ldots,\pr_{a_n},\pr_{(a_1\lor \ldots\lor a_n)'}$ are independent) it follows that $0<\mu(\{\pr_{a_1},\ldots,\pr_{a_n}\}\in \tilde{A},\pr_{(a_1\lor \ldots\lor a_n)'}=\emptyset_S)\leq\mu(\LL^{-1}(\tilde{A}))=\mu'(\tilde{A})$. 

--- We get  that $\LL$ is a mod-$0$ isomorphism between the standard measure spaces $(S,\Sigma,\mu)$ and $(\tilde S,\tilde \Sigma,\mu')$: the $\sigma$-finiteness of $\mu'$ is not immediately transparent, but one can first apply the preceding to a finite measure equivalent to $\mu$, in which case there are no issues; the measurable structures are left intact in making this change (which follows from the equivalence $\mu'\sim \sum_{n\in \mathbb{N}_0}(q_n)_\star (\mu\vert_T)^n\vert_{(T^n)_{\ne}}$ established just above) -- in particular, reverting back to the original $\mu$, the fact that $\LL$ sends $\Sigma$-measurable sets to $\tilde\Sigma$-measurable sets is definitely true, meaning that $\mu'$ inherits its $\sigma$-finiteness from $\mu$. 

--- Sliding along $\LL$ we obtain the associated spectral representation $\Psi':\L2(\PP)\to \int H_L'\mu'(\dd L)$ of $B$ satisfying $H(A)=H'(\LL(A))$ for $A\in \Sigma$. Because $\LL$ acts as the identity on $T$ we get $H'_s=H_s$ for $s\in T$ and $\Psi'=\Psi$ on $H^{(1)}$.  Because $\mu'\sim \tilde{\mu}$, adjusting $\mu$ (hence $\Psi$, $\Psi'$) up to equivalence if necessary (but only off $T$), we may and do assume that $\mu'=\tilde{\mu}$ to begin with; still $H(A)=H'(\LL(A))$ for $A\in \Sigma$.  For convenience we may and do also assume that $ H'_\emptyset=\mathbb{R}$ with $\Psi'(1)=\mathbbm{1}_{\{\emptyset\}}$. Because   $\LL$ sends $\{K=n\}$ to ${T\choose n}$, $\tilde{K}$ is the counting map; Item~\ref{proposition:classical-structure:i} is verified (the passage from $((\tilde S,\tilde \Sigma,\tilde \mu),\Psi')$ to $((\tilde S,\tilde \Sigma,\tilde \mu),\tilde \Psi)$ will not affect the counting map).


\textbf{Identification of the spectral sets in the new spectral space.}

--- For $x\in B$, the spectral set associated with  $x$ in the new spectral resolution $((\tilde S,\tilde \Sigma,\tilde \mu),\Psi')$ is $\tilde{S}_x=\LL(S_x)=\{\LL_s:s\in S_x\}$. Let $A\in \TT$. We check that $\tilde{S}_{N_A}=(2^A)_{\mathrm{fin}}$ a.e.-$\tilde{\mu}$, thus establishing Item~\ref{proposition:classical-structure:ii} (the passage from $((\tilde S,\tilde \Sigma,\tilde \mu),\Psi')$ to $((\tilde S,\tilde \Sigma,\tilde \mu),\tilde \Psi)$  will not affect the spectral sets).  By Proposition~\ref{proposition:extension} $N(S_{N_A}\cap T)=N(A)$, i.e. $S_{N_A}\cap T=A$ a.e.-$\mu$, so that $(2^{S_{N_A}\cap T})_{\mathrm{fin}}=(2^A)_{\mathrm{fin}}$ a.e.-$\tilde{\mu}$. Therefore we need only argue that $\tilde{S}_{x}=(2^{S_{x}\cap T})_{\mathrm{fin}}$ a.e.-$\tilde{\mu}$ for $x:=N_A$.

Now, on the one hand,  there is a sequence $(x_n)_{n\in \mathbb{N}}$ in $B_0$ such that $x=\liminf_{n\to\infty}{x_n}$, i.e.  $\liminf_{n\to\infty}S_{x_n}=S_x$ a.e.-$\mu$. But for $n\in \mathbb{N}$ and $s\in S_{x_n}$ we have $\LL_s\subset S_{x_n}$; passing to  the limit it follows that for $\mu$-a.e. $s\in S_x$, $\LL_s\subset  \liminf_{n\to\infty}S_{x_n}\cap T$, i.e. $S_x\subset \LL^{-1}((2^{\liminf_{n\to\infty}S_{x_n}\cap T})_{\mathrm{fin}})$ a.e.-$\mu$, i.e. $\tilde{S}_{x}\subset (2^{\liminf_{n\to\infty}S_{x_n}\cap T})_{\mathrm{fin}}=(2^{S_{x}\cap T})_{\mathrm{fin}}$ a.e.-$\tilde{\mu}$. 

On the other hand, for $\mu$-a.e. $s\in \LL^{-1}((2^{S_{x}\cap T})_{\mathrm{fin}})$ we have as follows. Let $n\in \mathbb{N}$ be such that $K_{b_n}(s)=K(s)$. Then $\pr_a(s)\in \LL_s\subset S_x$ for $a\in \at(b_n)\cap 2^{\underline{b_n}(s)}$ and  $\pr_a(s)=\emptyset_S\in  S_x$  for $a\in \at(b_n)\backslash  2^{\underline{b_n}(s)}$. Applying Corollary~\ref{prop:noise-proj-bis}\ref{proj:ii--} we get $\pr_a(\pr_{x'}(s))=\emptyset_S$ for all $a\in \at(b_n)$, which means that $\pr_{x'}(s)=\emptyset_S$, i.e. $s\in S_x$. We have proved that also $ \LL^{-1}((2^{S_{x}\cap T})_{\mathrm{fin}})\subset S_x$ a.e.-$\mu$, i.e. $(2^{S_{x}\cap T})_{\mathrm{fin}}\subset \tilde{S}_x$ a.e.-$\tilde{\mu}$.
%
%
%
%

\textbf{Adjustment of the direct integral of Hilbert spaces.}

--- We construct a unitary isomorphism $\Theta$  from $\int^\oplus_{\tilde{S}}\tilde{H}_L\tilde{\mu}(\dd L)$ onto $\int^\oplus_{\tilde{S}} H'_L\tilde{\mu}(\dd L)$ as follows. Let $P$ be a partition of unity of $B$. For $\mu$-a.e. $l\in T$ there is a unique $p\in P$ such that $l\in S_p$; denote it by $l_P$. Let also ${P\choose \mathbf{1}}:=\{L\in \tilde{S}:\vert L\cap S_p\vert=1\text{ for all }p\in P\}=\{L\in {T\choose \vert P\vert}:l_P\ne d_P\text{ for }l\ne d\text{ from }L\}$ a.e.-$\tilde{\mu}$. Then $\otimes_{p\in P}\int^\oplus_{S_p\cap T}H_s\mu(\dd s)$ is canonically identified  with $\int^\oplus_{{P\choose \mathbf{1}}}\tilde{H}_L\tilde{\mu}(\dd L)$ via the unique unitary isomorphism that sends $\otimes_{p\in P} f^p$ to $\widetilde{\otimes_{p\in P} f^p}:= \left(\tilde{S}\ni L\mapsto
\mathbbm{1}_{{P\choose \mathbf{1}}}(L)\otimes_{l\in L}f^{l_P}(l) \right)$ (these latter vectors  are total in $\int^\oplus_{{P\choose \mathbf{1}}}\tilde{H}_L\tilde{\mu}(\dd L)$, see Lemma~\ref{lemma:fock-field}\ref{lemma:fock-field:ii}, and the association is scalar-product preserving). 
Therefore (by the same token) the map $$\Theta_P\left(\widetilde{\otimes_{p\in P} f^p}\right)=\Psi'\left(\prod_{p\in P}\Psi^{-1}(f^p)\right),\quad  \otimes_{p\in P}f^p\in \otimes_{p\in P}\int^\oplus_{S_p\cap T}H_s\mu(\dd s),$$
extends to a unique unitary isomorphism $\Theta_P:\int^\oplus_{{P\choose \mathbf{1}}}\tilde{H}_L\tilde{\mu}(\dd L)\to \int^\oplus_{{P\choose \mathbf{1}}}H'_L\tilde{\mu}(\dd L)$ (recall Proposition~\ref{proposition:higher-choas-partitions}). Let us argue that  the $\Theta_P$, as $P$ ranges over all the partitions of $B$, agree on the intersections of their domains and, together with $\Theta_0:\int^\oplus_{\{\emptyset\}}\tilde{H}_L\tilde{\mu}(\dd L)\to\int^\oplus_{\{\emptyset\}} H'_L\tilde{\mu}(\dd L)$ acting as identity, extend uniquely to the desired unitary ismorphism $\Theta$. 

Indeed, if $P$ and $Q$ are two partitions of unity of $B$, then $$D_{P,Q}:=\int^\oplus_{{P\choose \mathbf{1}}}\tilde{H}_L\tilde{\mu}(\dd L)\cap \int^\oplus_{{Q\choose \mathbf{1}}}\tilde{H}_L\tilde{\mu}(\dd L)=\int^\oplus_{{P\choose \mathbf{1}}\cap {Q\choose \mathbf{1}}}\tilde{H}_L\tilde{\mu}(\dd L)$$ is only $\ne \{0\}$ provided $\vert P\vert=\vert Q\vert$, in which case ${P\choose \mathbf{1}}\cap {Q\choose \mathbf{1}}=\cup{P,Q\choose \mathbf{1}}_\sigma$, where ${P,Q\choose \mathbf{1}}_\sigma:=\{L\in {T\choose \vert P\vert}:\vert L\cap S_p\cap S_{\sigma(p)}\vert=1\text{ for all }p\in P\}$, and the union is over all bijections $\sigma:P\to Q$. Then, for each such $\sigma$, $\otimes_{p\in P}\int^\oplus_{S_p\cap S_{\sigma(p)}\cap T}H_s\mu(\dd s)$ is canonically identified with $D_{P,Q}^\sigma:=\int^\oplus_{{P,Q\choose \mathbf{1}}_\sigma}\tilde{H}_L\tilde{\mu}(\dd L)$ via the unitary isomorphism that sends  $\otimes_{p\in P} f^p$ to $\widetilde{\otimes_{p\in P} f^p}$, and it is clear that $\Theta_P$ and $\Theta_Q$ agree thereon, hence on $D_{P,Q}=\oplus D_{P,Q}^\sigma$, the orthogonal sum being over all bijections $\sigma:P\to Q$. We see also that $\Theta_P$ and $\Theta_Q$ both send $D_{P,Q}$ onto $\int^\oplus_{{P\choose \mathbf{1}}\cap {Q\choose \mathbf{1}}}H'_L\tilde{\mu}(\dd L)$. 

Thus $\Theta^\circ:=\Theta_0\cup (\cup \Theta_P)$, the union being over all partitions of unity $P$ of $B$, is a map defined on the union of the domains of the $\Theta_{P}$ and $\Theta_0$, which sends a total subset of $\int^\oplus_{\tilde{S}}\tilde{H}_L\tilde{\mu}(\dd L)$ 
 onto a total subset of $\int^\oplus_{\tilde{S}}H'_L\tilde{\mu}(\dd L)$ ($\because$ $\{\emptyset\}$ and the ${P\choose \mathbf{1}}$ as $P$ ranges over the partitions of unity of $B_0$ cover $\tilde{S}$ a.e.-$\tilde{\mu}$). Furthermore, it is straightforward to check that $\Theta^\circ$ is scalar product-preserving. Therefore $\Theta^\circ$ in fact extends uniquely to the sought-after  unitary isomorphism $\Theta$. 

Notice also that $\Theta$ sends $\int^\oplus_{\tilde{S}_x}\tilde{H}_L\tilde{\mu}(\dd L)$ to $\int^\oplus_{\tilde{S}_x} H'_L\tilde{\mu}(\dd L)$ for each $x\in B$ (into: it sends a total set into it; onto: by classicality, recalling how the individual chaos spaces are generated through the first chaos \ref{generalites:chaos-spaces}), and thus $\int^\oplus_{E}\tilde H_L\tilde{\mu}(\dd L)$ to $\int^\oplus_{E}H'_L\tilde{\mu}(\dd L)$ for each $E\in \tilde{\Sigma}$. We therefore may and do assume henceforth that $\Theta$ was the identity to begin with and $H'_L=\tilde{H}_L$ for $L\in \tilde\Sigma$, in particular $H(A)=\tilde{H}(\LL(A))$ for each $A\in \Sigma$. The spectral resolution has reached its definitive form and Item~\ref{remark:mod-0-iso-L} is established (the paranthetical conclusion follows because we have $$\int_A \langle \tilde \Psi(f),\tilde\Psi(g)\rangle\circ \LL\dd\mu=\int_{\LL(A)} \langle \tilde \Psi(f),\tilde\Psi(g)\rangle\dd\tilde\mu=\langle \pr_{\tilde H(\LL(A))}f,g\rangle=\langle \pr_{ H(A)}f,g\rangle=\int_A \langle \Psi(f),\Psi(g)\rangle\dd\mu$$ for all $A\in \Sigma$).

--- Items~\ref{proposition:classical-structure:iv} and~\ref{proposition:classical-structure:iii} are now essentially trivial if we recall from Proposition~\ref{proposition:extension} the bijective correspondence between sets of $\TT/_\mu$ and the members of $B$.  In particular, uniqueness of $\tilde{\Psi}$ is clear because the stipulated conditions determine it on a total set, while existence has just been established: quite simply, $\tilde{\Psi}=\Psi'$ (because $\Theta$ is now the identity).    Lemma~\ref{lemma:fock-field}\ref{lemma:fock-field:ii}   makes sure that the ``satisfies moreover'' part of \ref{proposition:classical-structure:iv} holds true on a total set, hence everywhere.

--- For $A\in \mathcal{T}$, $e(u)=e(u\mathbbm{1}_A)\otimes e(u\mathbbm{1}_{T\backslash A})$ by Lemma~\ref{lemma:fock-field}\ref{lemma:fock-field:iv}, and we get $\tilde{\Psi}^{-1}(e(u))=\tilde{\Psi}^{-1}(e(u\mathbbm{1}_A))\tilde{\Psi}^{-1}(e(u\mathbbm{1}_{T\backslash A}))$  using  Item~\ref{proposition:classical-structure:iv}. $\PP[\tilde{\Psi}^{-1}(e(u))]=1$ because $\tilde{\Psi}(\PP_{0_\PP}(\tilde{\Psi}^{-1}(e(u))))=\mathbbm{1}_{\{\emptyset\}}e(u)=\mathbbm{1}_{\{\emptyset\}}=\tilde{\Psi}(1)$. Therefore, by Item~\ref{proposition:classical-structure:ii},  $\tilde{\Psi}^{-1}(e(u))$ is a multiplicative integral (recall \ref{generalities:o}). The first statement of Item~\ref{proposition:classical-structure:v} is proved. 

--- Since $\tilde{\Psi}$ is a unitary isomorphism the second and thirs statements of Item~\ref{proposition:classical-structure:v} follow at once from Lemma~\ref{lemma:fock-field}~\ref{lemma:fock-field:iv}. 

--- Lastly the indicated bijective nature of $\mathsf{Exp}$ of Item~\ref{proposition:classical-structure:v} is also got directly from Lemma~\ref{lemma:fock-field}\ref{lemma:fock-field:iv} after one notices as follows: that $f$ being a square-integrable multiplicative integral of $B$ means that $\Psi(f)$ satisfies $\tilde \Psi(f)(\emptyset)=\PP[f]=1$ and, by Item~\ref{proposition:classical-structure:iv}, for all $A\in \TT$, for $\tilde\mu$-a.e. $L$, $\tilde\Psi(f)(L)=\tilde\Psi(\PP[f\vert N_A]\PP[f\vert N_{T\backslash A}])(L)=[\tilde\Psi(\PP[f\vert N_A])\otimes \tilde\Psi(\PP[f\vert N_{T\backslash A}])](L)=[(\tilde\Psi(f) \mathbbm{1}_{(2^A)_{\mathrm{fin}}}) \otimes (\tilde\Psi(f) \mathbbm{1}_{(2^{T\backslash A})_{\mathrm{fin}}})](L)=\tilde\Psi (f)(A\cap L)\otimes \tilde\Psi(f)(L\backslash A)$; that $\tilde\Psi(H^{(1)})\vert_T=\int_T H_s\gamma(\dd s)$.
\end{proof}

Theorem~\ref{proposition:classical-structure} stipulates that every classical noise Boolean algebra admits a spectrum that is in standard form. It is natural to wonder to what extent there is arbitrariness at all in such a spectral resolution. There is a satisfactory answer to this in the general case (not just for classical $B$) and we state and prove this first (before giving the relevant corollary for classical $B$ and standard resolutions). 

\begin{proposition}\label{proposition:isomorphisms-of-noise-boolean-spectra}
 Let $B'$ be another  noise Boolean algebra on a probability space $(\Omega',\FF',\PP')$ with associated spectral space $(S',\Sigma',\mu')$ and unitary isomorphism $\Psi':\L2(\PP')\to \int_{S'}^\oplus H_{s'}'\mu'(\dd s')$. Of the following two statements, \ref{iso:1} implies \ref{iso:2}.
\begin{enumerate}[(A)]
\item\label{iso:1} There exists an isomorphism $\Theta$ of the probability $\PP$ onto the probability $\PP'$, which carries $\overline{B}$ onto $\overline{B'}$ (automatically as an isomorphism of Boolean algebras); let $\widehat{\Theta}:\L2(\PP)\to \L2(\PP')$ be the associated unitary isomorphism.
\item\label{iso:2} There exists a mod-$0$ isomorphism $\psi$ of $\mu$ onto a $\sigma$-finite measure $\widetilde{\mu'}$ equivalent to $\mu'$, and a $\psi$-isomorphism $V=(V_s)_{s\in S}$ of $( H_{s})_{s\in S}$ onto $(H_{s'}')_{s'\in S'}$, such that $\Psi'\circ\widehat{\Theta}\circ \Psi^{-1}$ is the composition of $V$ and of the canonical unitary isomorphism $\int^\oplus_{S'} H_{s'}'\tilde{\mu'}(\dd s')\to \int^\oplus_{S'} H_{s'}'\mu'(\dd s')$.  In particular $\psi$ sends $\mathrm{mult}:=(S\ni s\mapsto \mathrm{dim}(H_s)\in \mathbb{N}_0\cup\{\infty\})$ to $\mathrm{mult}':=(S'\ni s'\mapsto \mathrm{dim}(H_{s'}')\in \mathbb{N}_0\cup \{\infty\})$ (these maps are measurable \cite[p. 166, Proposition~II.1.4.1(i)]{dixmier1981neumann}).
\end{enumerate}

\end{proposition}
\begin{proof}
Suppose that \ref{iso:1} holds true. We may assume $B=\overline{B}$ and $B'=\overline{B'}$ (since a noise Boolean algebra and its completion admit the same spectral resolution). Then the abelian von Neumann algebra $\AA_B$ is unitarily equivalent to the abelian von Neumann algebra $\AA_{B'}$ via $\hat{\Theta}$. By the uniqueness of the decomposition of abelian von Neumann algebras, \cite[A.70, A.85]{dixmier-c-star}, \ref{iso:2} follows (we allow the $\psi$-isomorphism to fail to be an isomorphism of Hilbert spaces on  a negligible set).
\end{proof}

Strictly speaking, the assumption that $B$ is classical does not feature in the next corollary, but plainly its content is aimed at a classical $B$, vis-\`a-vis Theorem~\ref{proposition:classical-structure}.
\begin{corollary}\label{corollary:isomorphisms-of-noise-boolean-spectra-first-chaos}
Suppose $(S^i,\Sigma^i,\mu^i)$, $\Psi^i:\L2(\PP)\to \int^\oplus_{S_i}H^i_s\mu(\dd s)$ with associated counting map $K^i$, $i\in \{1,2\}$, are two spectral resolutions of $B$. Then there exists a mod-$0$ isomorphism $\eta$ of $\mu^1\vert_{\{K^1=1\}}$ onto a $\sigma$-finite measure $\mu'$ equivalent to $\mu^2\vert_{\{K^2=1\}}$ and an $\eta$-isomorphism $W=(W_{s_1})_{s_1\in \{K^1=1\}}$ of $(H^1_{s_1})_{s_1\in \{K^1=1\}}$ onto $(H^2_{s'_2})_{s'_2\in {\{K^2=1\}}}$, such that $\Psi^2\circ (\Psi^1)^{-1}\vert_{\int^\oplus_{\{K^1=1\}}H^1_{s_1}\mu^1(\dd s_1)}$ is the composition of $W$ with the canonical unitary  isomorphism $\int^\oplus_{\{K^2=1\}} H_{s'_2}^2\mu'(\dd s'_2)\to \int^\oplus_{\{K^2=1\}} H_{s_2}^2\mu^2(\dd s_2)$.
\end{corollary}
\begin{proof}
In Proposition~\ref{proposition:isomorphisms-of-noise-boolean-spectra} take $\PP'=\PP$ and $B'=B$, for $\Theta$ the identity in Item~\ref{iso:1} thereof. From Item~\ref{iso:2} of this same proposition we then get a mod-$0$ isomorphism $\psi$ of $\mu^1$ onto a $\sigma$-finite measure $\widetilde{\mu^2}$ equivalent to $\mu^2$ together with a $\psi$-isomorphism $V=(V_s)_{s\in S}$ which satisfies the properties listed therein. In particular we see that $\psi$ carries $\{K^1=1\}$ onto $\{K^2=1\}$ and we may put $\eta:=\psi\vert_{\{K^1=1\}}$, $\mu'=\widetilde{\mu^2}\vert_{\{K^2=1\}}$ and $W=V\vert_{\{K^1=1\}}$.
\end{proof}

\begin{corollary}\label{corollary:b-a-iso-mult}
Let $B$ be classical and let $B'$ be another classical noise Boolean algebra on a probability space $(\Omega',\FF',\PP')$ with associated spectral space $(S',\Sigma',\mu')$ and unitary isomorphism $\Psi':\L2(\PP')\to \int_{S'}^\oplus H_{s'}'\mu'(\dd s')$, counting map $K'$. The following statements are equivalent.
\begin{enumerate}[(A)]
\item\label{iso-L2-1} In the notation of Proposition~\ref{proposition:isomorphisms-of-noise-boolean-spectra}\ref{iso:2}, for all $n\in \mathbb{N}\cup \{\infty\}$, $$\vert \{ \mathrm{mult}=n,K=1\}\cap \at(\mu)\vert=\vert \{\mathrm{mult}'=n,K'=1\}\cap \at(\mu')\vert$$ and $$\mu(\{\mathrm{mult}=n,K=1\}\backslash  \at(\mu))>0\Leftrightarrow \mu'(\{\mathrm{mult}'=n,K'=1\}\backslash  \at(\mu'))>0.$$
\item\label{iso-L2-2} There is a Boolean algebra isomorphism  $\zeta:\overline{B}\to \overline{B'}$ and a unitary isomorphism $V$ of $\L2(\PP)$ onto $\L2(\PP')$ that carries $\L2(\PP\vert_x)$ onto $\L2(\PP'\vert_{\zeta(x)})$ for all $x\in \overline{B}$.
\end{enumerate}
\end{corollary}
\begin{proof}
We may assume $B=\overline{B}$ and $B'=\overline{B'}$ and that the spectral resolutions $\Psi$ and $\Psi'$ of $B$ and $B'$ are in standard form. It is then clear from Theorem~\ref{proposition:classical-structure} that \ref{iso-L2-1} implies \ref{iso-L2-2}, just because any two standard measures are mod-$0$ isomorphic up to equivalence. Conversely, if \ref{iso-L2-2} holds, then  $V$ and $\Psi$ induce naturally the spectral  resolution $\Psi\circ V^{-1}$  of $B'$ (with $(S,\Sigma,\mu)$ as spectral space and the same Hilbert space field $(H_s)_{s\in S}$) and it remains to apply Corollary~\ref{corollary:isomorphisms-of-noise-boolean-spectra-first-chaos}.
\end{proof}

\begin{remark}
The unitary isomorphism of Corollary~\ref{corollary:b-a-iso-mult}\ref{iso-L2-2} automatically preserves the individual chaoses, i.e. sends $H^{(n)}(B)$ to $H^{(n)}(B')$ for all $n\in \mathbb{N}_0$ (just because the spectral set $S_{\zeta(x)}'$ corresponding to the spectral resolution $\Psi\circ V^{-1}$ is the same as the spectral set $S_x$ of the spectral resolution $\Psi$, this for all $x\in \overline{B}$; therefore the counting map $K$ of $\Psi$ is also the counting map of $\Psi\circ V^{-1}$).
\end{remark}
Explicit unitary isomorphisms satisfying Item~\ref{iso-L2-2} of Corollary~\ref{corollary:b-a-iso-mult} in the case when $B$ and $B'$ are noise Boolean algebras attached to generalized processes with independent values are constructed in \cite[esp. Lemma~2, Theorem~6]{vershik-tsilevich}.

To conclude this subsection we prove that, in the proper precise sense, for a classical $B$,  a spectral resolution of the first chaos extends uniquely to a spectral resolution in standard form of the whole of $\L2(\PP)$.

\begin{corollary}\label{corollary:choas-ex-nihilo}
Assume $B$ is classical. Let $(T^*,\TT^*,\gamma^*)$ and $N^*$ be as in Corollary~\ref{corollary:uniqueness-of-indexation}. Suppose furthermore that there is a $\gamma^*$-measurable structure on a field of Hilbert spaces $(H_s^*)_{s\in T^*}$ together with a unitary isomorphism $\Phi^*:H^{(1)}\to \int^\oplus_{T^*}H^*_s\gamma^*(\dd s)$, which sends $\PP_{N^*(A)}\vert_{H^{(1)}}$ to the operator of multiplication with $\mathbbm{1}_A$ for each $A\in \TT^*$. Associate to the  $\gamma^*$-measurable structure the ``symmetric Fock space'' $\tilde{\mu}^*$-measurable structure on $(\tilde{H}^*_L)_{L\in \tilde{S}^*}$ as in Lemma~\ref{lemma:fock-field}. Then there exists a unique unitary isomorphism $\tilde{\Psi}^*:\L2(\PP)\to \int^\oplus_{\tilde{S}^*}\tilde{H}^*_L\tilde{\mu}^*(\dd L)$ extending $\Phi^*$ (up to the canonical identification of ${T^*\choose 1}$ with $T^*$) and making $((\tilde{S}^*,\tilde{\Sigma}^*,\tilde{\mu}^*);\tilde{\Psi}^*)$ into a spectral resolution of $B$ in standard form. 
\end{corollary}
\begin{proof}
We retain the notation of Theorem~\ref{proposition:classical-structure}. 

We know from Corollary~\ref{corollary:uniqueness-of-indexation}   that $\gamma$ is equivalent to a standard measure $\eta$ that is mod-$0$ isomorphic to $\gamma^*$ via a map that carries $A$ onto $A^*$ whenever $N(A)=N^*(A^*)$, this for all $A\in \TT$ and $A^*\in \TT^*$.  Changing $((S,\Sigma,\mu);\Psi)$ if necessary  we may and do assume $(T,\TT,\gamma)=(T^*,\TT^*,\gamma^*)$, $N=N^*$ and that $((S,\Sigma,\mu);\Psi)$ is in standard form. Put $\Phi:=\Psi\vert_{H^{(1)}}$.  

For ease we drop the $\widetilde{\phantom{ll}}$ in the notation (keeping just the ${}^*$). 

Then we have the unitary isomorphism $\Phi^*\circ \Phi^{-1}:\int_T^\oplus H_s\gamma(\dd s)\to \int_{T}H_s^*\gamma(\dd s)$, which carries $\int_A^\oplus H_s\gamma(\dd s)$ onto $\int_A^\oplus H_s^*\gamma(\dd s)$ for all $A\in \TT$. Consequently there is an $\mathrm{id}_T$-isomorphism 
$(V(t))_{t\in T}$ (up to a negligible set) of $(H_s)_{s\in T}$ onto $(H_s^*)_{s\in T}$, which in turn engenders a unitary isomorphism $\Gamma:\int^\oplus_{S}H_L\mu(\dd L)\to \int^\oplus_{S^*}H^*_L\mu^*(\dd L)$ that sends $\int_{(2^A)_{\mathrm{fin}}}^\oplus H_L\mu(\dd L)$ to $\int_{(2^A)_{\mathrm{fin}}}^\oplus H_L^*\mu^*(\dd L)$ for all $A\in \TT$. In consequence $\Psi^*:=\Psi\circ \Gamma^{-1}$ is the desired unitary isomorphism (uniqueness is a consequence of the uniqueness asserted in Theorem~\ref{proposition:classical-structure}\ref{proposition:classical-structure:iv}).
\end{proof}

\begin{example}
It may be agreeable to see just how the preceding general results can be used to build  ``from the ground up'' the Wiener chaos of Example~\ref{example:wiener}. 
\begin{itemize}[leftmargin=0.5cm]
\item Identify directly the first chaos. Let $f\in H^{(1)}$. For each $n\in \mathbb{N}_0$, $(\PP[f\vert N_{[-n,t]}])_{t\in [n,\infty)}$ is a square-integrable martingale in the (shifted)  Brownian filtration $(N_{[-n,t]})_{t\in [-n,\infty)}$. Therefore there is a (shifted) predictable process $P^n=(P^n_u)_{u\in [-n,\infty)}$  in $\L2(W\vert_{[-n,\infty)})$ such that $\PP[f\vert N_{[-n,t]}]=\int_{-n}^tP_u^n\dd W_u$ (It\^o integral) a.s.-$\PP$ for all $t\in [-n,\infty)$. We get that $\PP[f\vert N_{[p,q]}]=\int_p^qP_u^n\dd W_u$ a.s.-$\PP$ for all real $q\geq p\geq -n$. By uniqueness of martingale representation we see that $P^n=P^{n+1}$ on $[-n,\infty)$ Lebesgue-a.e. $\PP$-a.s., hence there is a single predictable $P=(P_u)_{u\in \mathbb{R}}$ in $\L2(W)$ such that $\PP[f\vert N_{[p,q]}]=\int_p^qP_u\dd W_u$ a.s.-$\PP$ for all real $ p\leq q$. Applying (uniqueness of) martingale representation on the interval $[p,q]$ we get that $P_u$ is $N_{[p,u]}$-measurable for Lebesgue-a.e. $u\in [p,q]$ for all rational $p\leq q$. Therefore for Lebesgue-a.e. $u\in \mathbb{R}$, $P_u$ is $\cap N_{[p,q]}$-measurable, the intersection being over all rational intervals $[p,q]\ni u$, that is to say, $P_u$ is deterministic. The deterministic process $(\PP[P_u])_{u\in \mathbb{R}}$ is predictable and agrees with $P$ a.e.-Lebesgue a.s.-$\PP$; we may just as well assume it is $P$ itself.  We get $\subset$ in $H^{(1)}=\{\int_{-\infty}^\infty f(u)\dd W_u:f\in \L2(\mathbb{R})\}$ and the reverse inclusion is obvious. 

\noindent  Note that $B$ is indeed classical, since $H^{(1)}$ generates $\sigma(W)=1_\PP$.
\item Extend the factorization $N$ from unions of intervals to all Lebesgue measurable sets. For any Lebesgue measurable $A\subset \mathbb{R}$ define $N_A^*$ to be the $\sigma$-field generated by $\{\int_{-\infty}^\infty f(u)\mathbbm{1}_A(u)\dd W_u:f\in \L2(A)\}$. 

\noindent The association $N^*=(A\mapsto N^*_A)$ evidently extends $N$, is $\sigma$-continuous from  below and has  $N^*(\mathbb{R})=1_\PP$. Also, $N_A^*$ is independent of $N_B^*$ if $A\cap B=\emptyset$ Lebesgue-a.e.: just compute the covariances  of the generating elements of the first chaos, they vanish, for Gaussian families it is enough. Next, $N^*(B)=0_\PP$ iff $B$ has Lebesgue measure zero: sufficiency follows from the last observation, necessity because $\var(\int \mathbbm{1}_{[-n,n]\cap B}(u)\dd W_u)=\mathrm{lebesgue}([-n,n]\cap B)\uparrow \mathrm{lebesgue}(B)$ as $n\to\infty$. Let finally $(A_n)_{n\in \mathbb{N}}$ be a sequence of unions of intervals.  Then we claim that $\liminf_{n\to\infty} N_{A_n}=N^*_{\liminf_{n\to\infty}A_n}$. By $\sigma$-continuity from below it is enough to argue that $N^*_{\cap_{n\in \mathbb{N}}A_n}=\land_{n\in \mathbb{N}}N_{A_n}$. The inclusion $\subset$ is automatic, for the reverse just note that $\land_{n\in \mathbb{N}}N_{A_n}$ belongs to $\overline{B}$, hence is generated by the elements of the first chaos that are measurable w.r.t. it. 

\noindent Now, because, on the one hand, each Lebesgue measurable set is equal Lebesgue-a.e. to $\liminf_{n\to\infty}A_n$ for some such sequence, and because, on the other, $x\in \overline{B}=\Cl(B)$ iff $x=\liminf_{n\to\infty}x_n$ for some sequence $(x_n)_{n\in \mathbb{N}}$ in $B$, it follows that $N^*$ maps into, moreover, onto $\overline{B}$. By the independence property $N^*$ respects complementation (because of the uniqueness of independent complementation in $\overline{B}$). 

\noindent Altogether it means that $N^*$ satisfies all the conditions of Corollary~\ref{corollary:uniqueness-of-indexation} under $\gamma^*=$ Lebesgue measure.
\item Equipping the  $H_s=\mathbb{R}$, $s\in \mathbb{R}$, with the standard Lebesgue-measurable Hilbert space field structure, the unitary isomorphism $\Phi^*=(H^{(1)}\ni \int_{-\infty}^\infty f(u)\dd W_u\mapsto f\in \L2(\mathbb{R})\equiv \int_\mathbb{R}^\oplus \mathbb{R}\dd s)$ sends conditioning on $N_A$ to the operator of multiplication with $A$ for all Lebesgue measurable $A$. In consequence, Corollary~\ref{corollary:choas-ex-nihilo} gives a spectral resolution of $B$, moreover it recovers precisely the Wiener chaos expansion of Example~\ref{example:wiener}.
\end{itemize}
\end{example}
\begin{question}
It appears that the content of the preceding example could be generalized considerably to (arbitrary atomless?) ``decomposable processes'' in the sense of Feldman \cite{feldman} (in lieu of the Brownian motion $W$). The main difficulty lies in identifying the first chaos in term of the ``stochastic integrals'' against $W$ (for which the martingale representation theorem was used above). To a certain extent this matter is resolved in \cite[Theorem~6.7, resp.~7.15]{feldman} in the mixed Poisson (resp. Gaussian) case. Can it be done in complete generality? By the spectral resolution the question is apparently related to whether or not a proper closed linear submanifold of $H^{(1)}$ of a classical $B$, closed under $\AA$, can generate $1_\PP$ (in the classical case); of course it is possible in the presence of atoms (of $B$, or equivalently of the spectral measure), how about in the atomless case?
\end{question}
%
%

\section{Spectral independence and multiplicative integrals}\label{subsection:spectral-independence}
A spectral independence probability is one that is absolutely continuous w.r.t. the spectral measure $\mu$, charges $\emptyset_S$ and that makes the noise projections corresponding to members of any given partition of unity of $B$ independent. Multiplicative integrals (recall \ref{generalities:o}) and spectral independence probabilities (described just now, to be introduced formally in Definition~\ref{definition:spctral-independence} immediately below), are seen to be intimately related (Propositions~\ref{proposition:multiplicative-and-independence} and~\ref{proposition:multiplicative-various}).  Both are also, loosely speaking, the same for $B$, $B_{\mathrm{stb}}$ and $\overline{B}$ (Proposition~\ref{proposition:multiplicative-integrals-stable}, Corollary~\ref{corollary:spectral-independence-stable}), which is in parallel to the same property of the first chaos that was already noted (in Proposition~\ref{proposition:classical-part-of-noise}\ref{proposition:classical-part-of-noise:ii}). Spectral independence probabilities are all carried by $\{K<\infty\}$ and one of them is equivalent to $\mu$ thereon (Corollary~\ref{corollary:supports-of-spectral-independence}). As a consequence we get a new characterization of classicality and blackness  via spectral independence (Theorem~\ref{thm:new-conditioon-for-classicality-blackness}). Multiplicative integrals are seen to generate $\stable$ (Corollary~\ref{corollary:multiplicative-generate-stable}). 

\subsection{Spectral independence probabilities} 
We would have seen instances  of the following in (the classical) Examples~\ref{example:spectral-finite} (generalized in~\ref{example:classical-spectrum}) and~\ref{example:wiener} --- we refer specifically to the parts in-between the \textbf{($\dagger$)}s --- though no particular name was given to it there.
 \begin{definition}\label{definition:spctral-independence}
A measure $\nu$  on $(S,\Sigma)$ is a spectral independence probability (for the given spectral resolution of $B$) if $\nu$ is absolutely continuous w.r.t. $\mu$, $\nu(S_{0_\PP})>0$ and $\Sigma_x$ is independent of $\Sigma_{x'}$ under $\nu$ for all $x\in B$. If further $\nu\sim \mu$ (resp. $\nu(\{\emptyset_S\})<1$) then $\nu$ is called an equivalent (resp. non-trivial) spectral independence probability.
 \end{definition}
 \begin{remark}
By an inductive argument, if $\nu$ is a spectral independence probability, then the $\sigma$-fields $\Sigma_p$, $p\in P$, are independent for each partition of unity $P$ of $B$. A trivial spectral independence probability always exists, namely $\delta_{\emptyset_S}$, the Dirac mass at $\emptyset_S$ (hence the qualification ``non-trivial'' of the preceding definition).
 \end{remark}
 \begin{remark}
The stipulation $\nu(S_{0_\PP})>0$ makes sure that for all $x\in B$, $\nu(S_x)>0$ and excludes in particular probabilities concentrated at single points, atoms of $\mu$, other than $\emptyset_S$ (which would otherwise be spectral independence probabilities).
 \end{remark}
In general (whether or not $B$ is classical), square-integrable multiplicative integrals naturally give rise to spectral independence probabilities.

\begin{proposition}\label{proposition:multiplicative-and-independence}
If $f$ is a square-integrable multiplicative integral then $\mu_{f/\Vert f\Vert}$ is a spectral independence probability. 
\end{proposition}
\begin{proof}
For greater notational ease replace $f$ with $f/\Vert f\Vert$ (so now $f$ has square-mean one and $f/\PP[f]$ is a multiplicative integral). Since $\PP[f]\ne 0$, $\mu_{f}$ charges $\emptyset_S$.  Further, we have $\mu_f(S)=\PP[f^2]=1$ and for $x\in B$, $u\in B_x$ and $v\in B_{x'}$, $$\mu_f(S_{u\lor x'}\cap S_{x\lor v})=\mu_f(S_{u\lor v})=\PP[\PP[f\vert u\lor v]^2]=\PP[\PP[f\vert u]^2]\PP[\PP[f\vert v]^2]/\PP[f]^2.$$ Taking $u=x$, resp. $v=x'$, $u=x$ and $v=x'$, we get $\mu_f(S_{x\lor v})$, resp. $\mu_f(S_{u\lor x'})$, $1=\PP[\PP[f\vert x']^2]\PP[\PP[f\vert x]^2]/\PP[f]^2$, which altogether allows to verify that $\mu_f(S_{u\lor x'}\cap S_{x\lor v})= \mu_f(S_{u\lor x'}) \mu_f(S_{x\lor v})$. Independence of $\Sigma_x$ and $\Sigma_{x'}$ now follows by an application of Dynkin's lemma. 
\end{proof}
As a consequence of Theorem~\ref{proposition:classical-structure} and Proposition~\ref{proposition:multiplicative-and-independence} we get the phenomenon recorded in the examples mentioned at the beginning of this subsection 
for an arbitrary classical noise Boolean algebra.  

\begin{proposition}\label{corollary:independent-equivalent}
Assume the setting of Theorem~\ref{proposition:classical-structure}.  The following are equivalent for a finite measure $\nu$ on $\Sigma\vert_T$. 
\begin{enumerate}[(A)]
\item\label{corollary:independent-equivalent:A} $\nu$ is absolutely continuous w.r.t. (resp. equivalent to) $\mu$ on $T$.
\item\label{corollary:independent-equivalent:B} There is $\hat{g}\in \int^\oplus_T H_s\mu(\dd s)$, such that $\nu=\Vert\hat{g}\Vert^2\cdot \mu\vert_T$ (resp. and $\Vert \hat{g}\Vert>0$ a.e.-$\mu\vert_T$). 
\end{enumerate}
Let $(\nu,\hat{g})$ be a a pair satisfying the preceding properties.  Consider under some probability $\mathbb{Q}$ the random set $\Gamma$ of points of $T$, which is got as follows: each $\nu$-atom $t$ of $T$ is included in $\Gamma$ independently of the others with  probability $\frac{\nu(\{t\})}{1+\nu(\{t\})}$; also, independently of the former, we include in $\Gamma$ the realization of a Poisson random measure, the intensity measure of which is $\nu_{\mathrm{d}}$, the diffuse part of $\nu$. Then the $\mathbb{Q}$-law of $\Gamma$ on $(\tilde{S},\tilde{\Sigma})$ is absolutely continuous w.r.t. (resp. equivalent to) $\tilde{\mu}$. Moreover, setting $f:=\tilde \Psi^{-1}(e(\hat{g}))$, $f$ is  a square-integrable multiplicative integral and  $\Gamma_\star \QQ=\tilde{\mu}_{f/\Vert f\Vert}$. Sliding back along the mod-$0$ isomorphism $\LL$ of Theorem~\ref{proposition:classical-structure}\ref{remark:mod-0-iso-L} we get the (resp. equivalent) spectral independence probability $(\LL^{-1})_\star (\Gamma_\star \QQ)=\mu_{f/\Vert f\Vert}$. 
\end{proposition}
\begin{remark}\label{remark:spectral-set-special}
If the spectral resolution is in standard form and $\mu\vert_T$ (equivalently, $\mu$) is finite, then we may take $\nu=\mu$, in which case  we get that $\Gamma_\star\QQ=\mu/\mu(S)$, i.e. just $\mu$ normalized to a probability.
\end{remark}
\begin{proof}
The equivalence of \ref{corollary:independent-equivalent:A} and \ref{corollary:independent-equivalent:B} is by Radon-Nikodym and the existence of orthonormal frames. Let $g:=\Vert \hat{g}\Vert$.  We compute (first on sets $(2^A)_{\mathrm{fin}}$, $A\in \TT$; it is enough by Lemma~\ref{lemma:fock}\ref{lemma:fock:ii} and Dynkin) 
\begin{align*}
\Gamma_\star \QQ&=(e^{\nu_{\mathrm{d}}(T)}\prod_{t\in T}(1+\nu(\{t\})))^{-1}\sum_{n\in \mathbb{N}_0}(n!)^{-1}(q_n)_\star \nu^n\vert_{(T^n)_{\ne}}\\
&=(e^{\int g(s)^2\mu_{\mathrm{d}}(\dd s)}\prod_{t\in T}(1+g(t)^2\mu(\{t\})))^{-1}\sum_{n\in \mathbb{N}_0}(n!)^{-1}(q_n)_\star (g^2\cdot \mu\vert_T)^n\vert_{(T^n)_{\ne}}\\
&=\left(\left((e^{\int g(s)^2\mu_{\mathrm{d}}(\dd s)}\prod_{t\in T}(1+g(t)^2\mu(\{t\})))\right)^{-1}\Vert e(\hat{g})\Vert^2\right)\cdot \tilde{\mu}=\tilde{\mu}_{\frac{\tilde \Psi^{-1}(e(\hat{g}))}{\Vert\tilde{\Psi}^{-1}(e(\hat{g}))\Vert}}
\end{align*}
and we know that $f$ is a square-integrable multiplicative integral (Theorem~\ref{proposition:classical-structure}\ref{proposition:classical-structure:v}). Finally, using Theorem~\ref{proposition:classical-structure}\ref{remark:mod-0-iso-L}, we see that 
\begin{align*}
(\LL^{-1})_\star (\Gamma_\star \QQ)&=(\LL^{-1})_\star \tilde{\mu}_{f/\Vert f\Vert}=(\LL^{-1})_\star \left((\Vert\tilde \Psi(f)\Vert^2/\Vert f\Vert^2)\cdot \tilde\mu\right)\\
&=\left((\Vert\tilde \Psi(f)\Vert^2/\Vert f\Vert^2)\circ \LL\right)\cdot \mu=(\Vert\Psi(f)\Vert^2/\Vert f\Vert^2)\cdot \mu=\mu_{f/\Vert f\Vert},
\end{align*}
 as asserted.
\end{proof}
The preceding indeed generalizes to the fullest extent Examples~\ref{example:wiener} (when $\mu\vert_T$ is diffuse) and~\ref{example:classical-spectrum} (when $\mu\vert_T$ is purely atomic). Let us give another, mixed, Poissonian, example.

\begin{example}\label{example:Poisson}
Let $(W,\mathcal{W},\mathsf{W})$ be a standard measure space. We denote  $W_{\mathrm{a}}:=\{w\in W:\mathsf{W}(\{w\})>0\}$ (a countable set $\because$ $\mathsf{W}$ is $\sigma$-finite) and  $W_{\mathrm{d}}:=W\backslash W_{\mathrm{a}}$. We also let $\mathsf{W}_{\mathrm{d}}$ be the diffuse part of $\mathsf{W}$. Suppose $\Gamma$ is a Poisson random measure under $\PP$ that generates $1_\PP$ and having the intensity measure $\mathsf{W}$. For $A\in \mathcal{W}$ put $N_{A}:=\sigma(\Gamma(B):B\in \mathcal{W}\vert_{A})$. Then $B:=\{N_{A}:A\in \mathcal{W}\}$ is a noise Boolean algebra.  Here we have viewed $\Gamma$ as a random element valued in the space of $\mathbb{N}_0\cup \{\infty\}$-valued measures on $(W,\mathcal{W})$ endowed with the usual measurable structure.

\noindent We claim that we may take for the spectral space $(S,\Sigma,\mu)$ of this Poisson noise the symmetric measure space over $(W,\mathcal{W},\mathsf{W})$ and that for $s\in S$ we may take $H_s=\L2(\PP\vert_{N_{s\cap W_{\mathrm{a}}}})^\circ$ under a  spectral resolution that we will make explicit shortly. Before that, note that for $s\in  (2^{W\backslash W_{\mathrm{a}}})_{\mathrm{fin}}$, $H_s=\L2(\PP\vert_{0_\PP})^\circ$ is the one-dimensional space of constants, canonically isomorphic to $\mathbb{R}$; for general $s\in S$,  $H_s$ is canonically isomorphic to $\otimes_{a\in s}H_a$, where $H_a=\L2(\PP\vert_{\sigma(\Gamma(\{a\})})_0$ for $a\in W_{\mathrm{a}}$ (use Theorem~\ref{prop:dicrete}\ref{prop:dicrete:iii-}). It follows that there is a unique $\mu$-measurable field structure on $(H_s)_{s\in S}$ that corresponds to the natural $\mathsf{W}$-measurable structure on $(H_w)_{w\in W}$ in the sense of Lemma~\ref{lemma:fock-field}\ref{lemma:fock-field:i} and the set of $\mu$-measurable square integrable vector fields is then naturally denoted $\int^\oplus_S H_s\mu(\dd s)$. Moreover, by Lemma~\ref{lemma:fock-field}\ref{lemma:fock-field:iii}, $\int^\oplus_S H_s\mu(\dd s)$ is canonically unitarily isomorphic to $(\int_{(2^{W_{\mathrm{a}}})_{\mathrm{fin}}}H_s\mu(\dd s))\otimes (\int_{(2^{W_{\mathrm{d}}})_{\mathrm{fin}}}H_s\mu(\dd s))$ and hence (apply Theorem~\ref{propsition:discrete-spectrum}\ref{discrete-spectrum:a} \& Example~\ref{example:classical-spectrum}) to $\L2(\PP\vert_{\sigma(\Gamma\vert_{W_{\mathrm{a}}})})\otimes \L2((\mathsf{W}_{\mathrm{d}})_{\mathrm{sym}})$, where $(\mathsf{W}_{\mathrm{d}})_{\mathrm{sym}}$ is the symmetrization of $\mathsf{W}_{\mathrm{d}}$ in the sense of Lemma~\ref{lemma:fock-field}. In particular an $f\in \int^\oplus_{(2^{W_{\mathrm{a}}})_{\mathrm{fin}}}H_s\mu(\dd s)$ is carried by this unitary isomorphism to $U(f)\otimes \mathbbm{1}_{\{\emptyset\}}$, where $U(f)=\sum_{L\in (2^{W_{\mathrm{a}}})_{\mathrm{fin}}}\sqrt{\prod_{s\in L}\mathsf{W}(s)}f(L)$. Then the spectral resolution $\Psi:\L2(\PP)\to \int^\oplus_S H_s\mu(\dd s)$ is determined uniquely by specifying that for $f\in \L2(\PP\vert_{\sigma(\Gamma\vert_{W_{\mathrm{a}}})})$, for $A\in \mathcal{W}\vert_{W_{\mathrm{d}}}$ of finite $\mathsf{W}_{\mathrm{d}}$-measure and for $g\in  \L2(\mathsf{W}_{\mathrm{d}}\vert_A)$, up to this identification of  $\int^\oplus_S H_s\mu(\dd s)$ with $\L2(\PP\vert_{\sigma(\Gamma\vert_{W_{\mathrm{a}}})})\otimes \L2((\mathsf{W}_{\mathrm{d}})_{\mathrm{sym}})$,  $$\Psi^{-1}(f\otimes e(g\mathbbm{1}_A))=f\left(e^{-\mathsf{W}_{\dd}[g;A]} \prod_{w\in A}(1+\Gamma(\{w\})g(w))\right),$$ 
where $e(g\mathbbm{1}_A):=((2^{W_{\mathrm{d}}})_{\mathrm{fin}}\ni L\mapsto \mathbbm{1}_{2^A}(L)\prod_{s\in L}g(s))$ is the exponential (coherent) vector associated to $g\mathbbm{1}_A$. Indeed this association is scalar product-preserving between total sets by the compensation formula for Poisson random measures, cf. \cite[Example~19.11]{parthasarathy}. Besides, it is easily  seen from the independence of $\Gamma$ over disjoint sets that, for $A\in \mathcal{W}$, $\PP_{N_A}$ corresponds to multiplication with $\mathbbm{1}_{(2^A)_{\mathrm{fin}}}$ under $\Psi$; of course one has to check also that linear combinations of such indicators are dense in $\mathrm{L}^\infty(\mu)$ locally on each measurable set of finite $\mu$-measure, but this follows at once from Lemma~\ref{lemma:fock}\ref{lemma:fock:ii}, which implies that $(2^A)_{\mathrm{fin}}$, $A\in \mathcal{W}$, is a $\mu$-essentially generating $\Pi$-system.

\noindent The counting map $K$ of the spectral resolution just described is nothing but the actual counting map on $S$.

\noindent What is the first chaos of $B$? It is the image under $\Psi^{-1}$ of $\int_{\{K=1\}}H_s\mu(\dd s)$. For $w\in W_{\mathrm{a}}$ and $g_w\in H_w=\L2(\PP\vert_{N_{\{w\}}})^\circ$ we get simply $\Psi^{-1}(g_w)=\sqrt{\mathsf{W}(\{w\})}g_w$. For $A\in \mathcal{W}\vert_{W_{\dd}}$ of finite $\mathsf{W}$-measure and  $g\in  \L2(\mathsf{W}_{\mathrm{d}}\vert_A)$, $\frac{1}{\epsilon}\left(e(\epsilon g\mathbbm{1}_A)-e(0)\right)=\frac{1}{\epsilon}\left(e(\epsilon g\mathbbm{1}_A)-\mathbbm{1}_{\{\emptyset\}}\right)$ converges to $g\mathbbm{1}_A$ in $\L2((\mathsf{W}_\dd)_{\mathrm{sym}})$ as $\epsilon\downarrow 0$, and we get by direct computation (say via dominated convergence) $\Psi^{-1}( g\mathbbm{1}_A)=\Gamma[g;A]-\mathsf{W}_{\dd}[g;A]$ as the limit in $\L2(\PP)$ of $\frac{1}{\epsilon}\left(\Psi^{-1}(e(\epsilon g\mathbbm{1}_A))-\Psi^{-1}(e(0))\right)=\frac{1}{\epsilon}\left(\Psi^{-1}(e(\epsilon g\mathbbm{1}_A))-1\right)$ as $\epsilon\downarrow 0$. The first chaos is then the subspace of $\L2(\PP)$ generated by the class of random variables that we have just enunciated, in particular it contains the limit in $\L2(\PP)$ of the compensated integrals $\Gamma[g;A]-\mathsf{W}_{\dd}[g;A]$ as $A\uparrow W$ over a  sequence of sets from $\mathcal{W}$ of finite $\mathsf{W}_{\dd}$-measure, this for any $g\in \L2(\mathsf{W})$. Of course it generates $1_\PP$: the Poisson noise is classical.

 \noindent Further, by the above, up to the (trivial) canonical identifications listed, we have that the spectrum $((S,\Sigma,\mu);\Psi)$ is in standard form.  Then, in the notation of Theorem~\ref{proposition:classical-structure}, for $A\in \mathcal{W}$ of finite $\mathsf{W}$-measure and for $u\in \int_{A}H_L\mu(\dd L)$, $$\Psi^{-1}(e(u))=\prod_{w\in W_{\mathrm{a}}\cap A}(1+\sqrt{\mathsf{W}(\{w\})}u(w))\left(e^{-\mathsf{W}_{\dd}[u;A]} \prod_{w\in A\cap W_{\dd}}(1+\Gamma(\{w\})u(w))\right).$$
 In particular, for $u:=\Psi((\Gamma-\mathsf{W})[g;A])$, where $g\in \L2(\mathsf{W})$ and $A\in \mathcal{W}$ has finite $\mathsf{W}$-measure,  $$\Psi^{-1}(e(u))=e^{-\mathsf{W}_{\dd}[g;A]}\prod_{w\in A}(1+(\Gamma(\{w\})-\mathsf{W}(\{w\}))g(w)).$$ Also, $\Psi(\lim_{A\uparrow W}(\Gamma-\mathsf{W})[g;A]))=\lim_{A\uparrow W}\left(\mathbbm{1}_ {A\cap W_{\dd}}g+\sum_{a\in W_{\mathrm{a}}\cap A}\mathbbm{1}_{\{a\}}\mathsf{W}(\{a\})^{-1/2}g(a)(\Gamma-\mathsf{W})(\{a\})\right)=g\mathbbm{1}_{W_{\dd}}+\sum_{a\in W_{\mathrm{a}}}\mathbbm{1}_{\{a\}}\mathsf{W}(\{a\})^{-1/2}g(a)(\Gamma-\mathsf{W})(\{a\})$, and $$\Psi^{-1}(e(\Psi(\lim_{A\uparrow W}(\Gamma-\mathsf{W})[g;A])))=\lim_{A\uparrow W}e^{-\mathsf{W}_{\dd}[g;A]}\prod_{w\in A}(1+(\Gamma(\{w\})-\mathsf{W}(\{w\}))g(w)),$$ where the limits are in $\mathrm{L}^2$ along a sequence of $\mathcal{W}$-measurable
sets of finite $\mathsf{W}$-measure nondecreasing to $W$. 

\noindent When $\mathsf{W}$ is a diffuse probability, then taking $g=1$ in the preceding we get $\Vert\Psi(\vert \Gamma(W)\vert-1)(w)\Vert^2=1$ for all $w\in W$, and we may take for the random set associated to $\hat{g}:=\Psi(\vert \Gamma(W)\vert-1)$ in the sense of  Proposition~\ref{corollary:independent-equivalent} the support of $\Gamma$ (Remark~\ref{remark:spectral-set-special}). In other words, the law of the random set specified by $\Gamma$ is also the spectral measure of the random variable $\frac{\Psi^{-1}(e(\hat{g}))}{\Vert\Psi^{-1}(e(\hat{g}))\Vert}=e^{-3/2}2^{\vert\Gamma(W)\vert}$!
\end{example}

\subsection{Concentration on the stable part}
Examples~\ref{example:spectral-space-for-nonclassical}-\ref{example:spectral-space-for-black} and Proposition~\ref{corollary:independent-equivalent} beg the question: if there exists a spectral independence probability $\nu$ that is equivalent to $\mu$, 
does it imply that $B$ is classical? More generally, is every spectral independence probability concentrated on $\{K<\infty\}$? We will shortly see that the answer to this is to the affirmative (Proposition~\ref{theorem:new-conditioon-for-classicality}). We require for this a simple assertion on the first-order stochastic dominance of sizes of certain random sets, which may be of independent interest. Before that, an elementary ancillary lemma (which may well be extremely well-known, but we include the proof for completeness). 

\begin{lemma}\label{lemma:variance-reduce}
Let $n\in \mathbb{N}_{\geq 3}$ and $x\in \mathbb{R}^n$. Let $I\in \mathrm{argmin}_{i\in [n]}\vert x_i-\overline{x}\vert$, where $\overline{x}$ is the arithmetical average of $x$. Let $a_{[n-1]}(x)$ result from $x$ by keeping its $I$-th entry while replacing the others with their average. Then $\var(a_{[n-1]}(x))\leq \frac{1}{n-1}\var(x)$ (the variances are taken w.r.t. the uniform distribution), in particular all the entries of ${a_{[n-1]}}^m(x)$ converge to $\overline{x}$ as $m\in \mathbb{N}$ tends to $\infty$.
\end{lemma}
\begin{proof}
Without loss of generality $I=n$. Note $\overline{a_{[n-1]}(x)}=\overline{x}$. We compute and estimate 
\begin{align*}
&n\cdot \var(a_{[n-1]}(x))=(n-1)\left(\frac{x_1+\cdots+x_{n-1}}{n-1}-\overline{x}\right)^2+(x_n-\overline{x})^2\\
&=(n-1)\left(\frac{x_1-\overline{x}+\cdots+x_{n-1}-\overline{x}}{n-1}\right)^2+(x_n-\overline{x})^2\\
&=(n-1)\left(\frac{x_n-\overline{x}}{n-1}\right)^2+(x_n-\overline{x})^2=\frac{n}{n-1}(x_n-\overline{x})^2\leq \frac{1}{n-1}\sum_{k=1}^n(x_k-\overline{x})^2=\frac{n}{n-1}\cdot \var(x)
\end{align*}
(which is true [trivially] also for $n=2$). The final claim follows at once: $\var({a_{[n-1]}}^m(x))\leq \frac{1}{(n-1)^m}\leq \frac{1}{2^m}\to 0$ as $m\to\infty$, using $n\geq 3$.
\end{proof}

\begin{proposition}\label{proposition:random-set-domincance}
Let $n\in \mathbb{N}$. 
\begin{enumerate}[(i)]
\item\label{fost:i} Under a probability $\mathbb{Q}$ let $A_1,\ldots,A_n$ be independent events of probabilities $p_1,\ldots,p_n$, respectively. Set  $p:=(p_1,\ldots,p_n)$ and call $\Bin(p)$ the law of $\sum_{i=1}^n\mathbbm{1}_{A_i}$. Then $\Bin(p)\geq \Bin(n,\sqrt[n]{p_1\cdots p_n})$ in first-order stochastic dominance. Here, as usual, for $q\in [0,1]$, $\Bin(n,q):=\Bin((\underbrace{q,\ldots,q}_{n\text{-times}}))$.
\item\label{fost:ii} Under a probability $\mathbb{Q}$ let $\Gamma$ be a random subset of $[n]$ which excludes each $i\in [n]$ with probability $q_i$ independently of the others; similarly let $\tilde{\Gamma}$ be a random subset of $[n]$ which excludes each $i\in [n]$ with probability $\sqrt[n]{q_1\cdots q_n}$ independently of the others. Then $\vert \Gamma\vert\leq \vert\tilde{\Gamma}\vert$ in first-order stochastic dominance.
\end{enumerate}
\end{proposition}
\begin{proof}
\ref{fost:i}. Assume $n=2$ in the first instance (for $n=1$ there is nothing to prove). We see that  $$\Bin(p)(\{0\})=(1-p_1)(1-p_2)\leq (1-\sqrt{p_1p_2})^2=\Bin(2,\sqrt{p_1p_2})(\{0\}),$$ and  $$\Bin(p)(\{0,1\})=1-\Bin(p)(\{2\})=1-p_1p_2=1-\sqrt{p_1p_2}^2=1-\Bin(2,\sqrt{p_1p_2})(\{2\})=\Bin(2,\sqrt{p_1p_2})(\{0,1\})$$ even with equality, so the stochastic dominance is established. For an $n\in \mathbb{N}_{\geq 3}$, let us assume that the claim holds true for $n-1$, and we prove that it holds for $n$. To this end note that choosing any $(n-1)$-element subset $A$ of $[n]$, by the inductive hypothesis, $\Bin(p)$ dominates $\Bin(p^A)$, where $p^A$ is got from $p$ by retaining the (single) $p_i$, $i\in [n]\backslash A$, and replacing the others by their geometric mean $\sqrt[n-1]{\prod_Ap}$. It is intutively clear that iterating this process ad nauseam yields, by a judicious choice of  $A$ on each step,  that $\Bin(p)$ dominates $\Bin(n,\sqrt[n]{p_1\cdots p_n})$, however let us give a formal argument. We may assume $p_j>0$ for all $j\in [n]$ for the only non-trivial case. Then we let $I$ be an index from $[n]$ the corresponding entry of $\log(p)$ of which is closest to the average of $\log(p)$ and take $A=[n]\backslash \{I\}$. According to Lemma~\ref{lemma:variance-reduce} $p^A$ will converge to the constant vector with entries $\sqrt[n]{p_1\cdots p_n}$ as we iterate this.  The claim now follows by continuity of the law of $\Bin(p)$ in its parameters and the transitivity of first-order stochastic dominance.

\ref{fost:ii}. We may write $\vert \Gamma\vert=\sum_{i\in [n]}\mathbbm{1}_{\{i\in \Gamma\}}=n-\sum_{i\in [n]}\mathbbm{1}_{\{i\notin \Gamma\}}$, likewise for $\tilde{\Gamma}$. Apply \ref{fost:i}.
\end{proof}

\begin{proposition}\label{theorem:new-conditioon-for-classicality}
Let $\nu$ be a spectral independence probability. Then $\nu(K<\infty)=1$. Also, $\nu\vert_{\{K<\infty\}}$ is a spectral independence probability for $B_{\mathrm{stb}}$ and its spectral resolution of Proposition~\ref{proposition:classical-part-of-noise}\ref{proposition:classical-part-of-noise:iv}.
\end{proposition}
\begin{proof}
The case $\nu(\{\emptyset_S\})=1$ (in particular, the case $0_\PP=1_\PP$) is trivial and excluded. 


Let $(b_n)_{n\in \mathbb{N}}$ be a $\uparrow$ sequence in $\mathfrak{F}_B$ such that $\cup_{n\in \mathbb{N}}b_n$ is dense in $B$; in addition we put $b_0:=\{0_\PP,1_\PP\}$. We prepare the groundwork as at the beginning of the proof of Theorem~\ref{proposition:classical-structure} except of course that we do not insist that $K<\infty$. 

Let $\mathsf{T}:=\cup_{n\in \mathbb{N}_0}\{n\}\times \at(b_n)$. The set $\mathsf{T}$ is endowed naturally with the obvious tree structure ($(n+1,b)$ is connected to $(n,a)$ iff $b\subset a$; no other connections), root $\{0\}\times \{1_\PP\}$. We define a $\mathsf{T}$-indexed Bernoulli process $\LL=(\LL_t)_{t\in \mathsf{T}}$ under the probability $\nu$ (in other words, a $\nu$-random subset of $\mathsf{T}$) as follows: $$\LL_{(n,a)}:=\mathbbm{1}_{\{\pr_a\ne \emptyset_S\}}\text{ for }(n,a)\in \mathsf{T}.$$ Then $\LL_{(0,1_\PP)}=\mathbbm{1}_{S\backslash S_{0_\PP}}$; also, for each $(n,a)\in \mathsf{T}$, $$\{\LL_{(n,a)}=0\}=\{\pr_a=\emptyset_S\}=S_{a'}=\cap_{b\in \at(b_{n+1})_a}S_{b'}=\cap_{b\in \at(b_{n+1})_a}\{\pr_b=\emptyset_S\}=\cap_{b\in \at(b_{n+1})_a}\{\LL_{(n+1,b)}=0\},$$ i.e. a point in the tree is labeled zero by $\LL$ (is ``excluded'') iff all its descendants are. Besides, at each level $n\in \mathbb{N}_0$, the Bernoulli random variables $\mathbbm{1}_{(n,a)}$, $a\in\at(b_n)$, are independent under $\nu$ and the $\nu$-probability that all of them are zero at once is the constant $p_0:=\nu(S_{0_\PP})\in(0,1)$. It is helpful to think of $\LL$ as an exploration process: as $n$ increases we pierce with more and more precision into the structure of $B$.

Now, for $\{n,m\}\subset \mathbb{N}_0$ we may identify $\{K_{b_n}\leq m\}=\{\vert\Gamma_n\vert\leq m\}$, where $\Gamma_n$ is the random set of points of $\mathsf{T}$ at level $n$ that are included by the exploration process $\LL$ (indeed, $\underline{b_n}=\lor\{a\in \at(b_n):\pr_a\ne\emptyset_S\}$, hence $K_{b_n}=\vert\Gamma_n\vert$). Therefore, to show that $\nu(K<\infty)=1$ it suffices (and is indeed equivalent) to prove that $$\lim_{m\to\infty}\lim_{n\to\infty}\nu(\vert\Gamma_n\vert\leq m)=1.$$ 

By Proposition~\ref{proposition:random-set-domincance}\ref{fost:ii} it reduces to establishing that 
$\lim_{m\to\infty}\lim_{k\to\infty}\QQ_k(\vert\Gamma\vert\leq m)=1$, where, for $k\in \mathbb{N}$, under $\QQ_k$, $\Gamma$ is a random subset of $[k]$ that excludes each point with probability $\sqrt[k]{p_0}$ independently of the others. But for $\{k,m\}\subset \mathbb{N}$, 
\begin{align*}
\QQ_k(\vert\Gamma\vert\leq m)&=\sum_{l=0}^m{k\choose l}(1-\sqrt[k]{p_0})^l\sqrt[k]{p_0}^{k-l}=p_0\sum_{l=0}^m{k\choose l}(\sqrt[k]{p_0}^{-1}-1)^l\\
&=p_0\sum_{l=0}^m\frac{1}{l!}k(p_0^{-1/k}-1)\cdots (k-l+1)(p_0^{-1/k}-1)\xrightarrow[]{k\to\infty} p_0\sum_{l=0}^m\frac {(-\log p_0)^l}{l!}\xrightarrow[]{m\to\infty} 1,
\end{align*}
which establishes that $\nu(K<\infty)=1$ (the limit ``$k\to\infty$''' uses l'H\^ospital's rule).

To show the second claim, let $x_s\in B_{\mathrm{stb}}$. Then $x_s=x\land \stable$ for some $x\in B$ and $(x_s)'=x'\land\stable$. Further, let $u_s\in (B_{\mathrm{stb}})_{x_s}$. Then there is $u\in B$ such that $u_s=u\land \stable$. Since $u_s\subset x_s$, $u_s=u_s\land x_s=(u\land x)\land \stable$; therefore we may and do assume $u\in B_x$. Likewise,  let $v_s\in (B_{\mathrm{stb}})_{(x_s)'}$ and find $v\in B_{x'}$ such that $v_s=v\land\stable$. Putting it all together we find that (letting the superscript $\mathrm{stb}$ indicate that the spectral sets are those of the spectral resolution for $B_{\mathrm{stb}}$)
\begin{align*}
\nu\vert_{\{K<\infty\}}(S_{x_s\lor v_s}^\mathrm{stb}\land S_{u_s\lor (x_s)'}^\mathrm{stb})&=\nu(S_{x\lor v}\cap S_{u\lor x'}\cap \{K<\infty\})=\nu(S_{x\lor v}\cap S_{u\lor x'})\\
&=\nu(S_{x\lor v})\nu(S_{u\lor x'})=\nu\vert_{\{K<\infty\}}(S_{x_s\lor v_s}^\mathrm{stb})\nu\vert_{\{K<\infty\}}(S_{u_s\lor (x_s)'}^\mathrm{stb})
\end{align*}
 and the proof is complete.
\end{proof}

\begin{proof}
The proof shows that under the hypotheses of Proposition~\ref{theorem:new-conditioon-for-classicality}
\begin{equation}
\nu(K\leq m)\geq p_0\sum_{l=0}^m\frac {(-\log p_0)^l}{l!},\quad m\in \mathbb{N}_0,
\end{equation}
where $p_0:=\nu(S_{0_\PP})$.
\end{proof}

\begin{remark}
The exploration process $\LL$ of the proof of Proposition~\ref{theorem:new-conditioon-for-classicality} is well-defined always provided $0_\PP\ne 1_\PP$ (when $0_\PP=1_\PP$ it is natural to have $\mathsf{T}=\emptyset$, but this is not very interesting). Of course it depends on the approximating sequence $(b_n)_{n\in \mathbb{N}}$ and the given spectral resolution. It is an injective ($\because$ the $S_x$, $x\in b_n$, $n\in \mathbb{N}$, separate $S$) measurable map from $S$ into the standard measurable space $\Theta:=\{l\in \{0,1\}^\mathsf{T}:\text{$l$ at some vertex is $0$ iff it is zero on all the descendants of this vertex}\}$  (endowed with the trace of the product measurable structure). In particular, taking the push-forward $\LL_\star\mu$ (and completing) we get a spectral resolution with $\Theta$ as spectral space and $\LL_\star\mu$ as spectral measure (we might call it an exploration measure). 
One way to interpret $\Theta$ is as a subset of the set of all ``single-descendant'' paths from the root, which let us call lineages. Elements of $\Theta$ consisting of $k$  ``lineages'' belong to $\LL(\{K=k\})$, i.e. they correspond to the $k$-th chaos, $k\in \mathbb{N}_0$ (hence elements of $\Theta$, which have infinitely many ``lineages'' correspond to the sensitive part, $\LL(\{K=\infty\})$). 
%
%
\end{remark}
\begin{question} 
Let $\mathsf{T}$ be any tree with root $r$ such that the set of vertices at level $n\in \mathbb{N}_0$ (graph distance $n$ from the root) is finite, such that the graph distance of any point from $\mathsf{T}$ to $r$ is finite, and such that every vertex has a descendant. Let $\nu$ be a $\sigma$-finite measure on the standard measurable space $\Theta:=\{l\in \{0,1\}^\mathsf{T}:\text{$l$ at some vertex is $0$ iff it is zero on all the descendants of this vertex}\}$ (endowed with the trace of the product measurable structure, completed). The measure $\nu$ is the obvious abstraction of  $\LL_\star\mu$ of the preceding remark. What properties must it satisfy in order for it to be an exploration measure of some noise Boolean algebra (with $0_\PP\ne 1_\PP$) up to tree isomorphism? An obvious requirement is that  the $\nu$-law of the canonical projections $(X_t)_{t\in \mathsf{T}}$ on $\Theta$ at every level of the tree should be equivalent to a product of non-degenerate Bernoulli  measures. Is this all that it takes?
\end{question}

%
\subsection{Implications}
We reap a rich harvest from Proposition~\ref{theorem:new-conditioon-for-classicality}.
\begin{proposition}\label{proposition:multiplicative-integrals-stable}
The following are equivalent.
\begin{enumerate}[(A)]
\item\label{multiplicative-integrals-stable:i} $f$ is a square-integrable multiplicative integral of  $B$. 
\item\label{multiplicative-integrals-stable:ii} $f$ is a square-integrable multiplicative integral of  $B_{\mathrm{stb}}$. 
\item\label{multiplicative-integrals-stable:iii} $f$ is a square-integrable multiplicative integral of  $\overline{B}$. 
\end{enumerate}
\end{proposition}
\begin{proof}
To see  that \ref{multiplicative-integrals-stable:i} implies \ref{multiplicative-integrals-stable:ii} apply Proposition~\ref{proposition:multiplicative-and-independence} and Proposition~\ref{theorem:new-conditioon-for-classicality} (it gives that $f$ is $\stable$-measurable; then it remains to note that $\stable$ commutes with $B$). The fact that  \ref{multiplicative-integrals-stable:ii} implies \ref{multiplicative-integrals-stable:i}  is immediate using the fact that $\stable$ commutes with $B$ (again) in conjunction with Proposition~\ref{proposition:classical-part-of-noise}\ref{proposition:classical-part-of-noise:i} (which ensures that the complement in $B_{\mathrm{stb}}$ of $x\land \stable$ is $x'\land \stable$, $x\in B$). To see that \ref{multiplicative-integrals-stable:i} implies \ref{multiplicative-integrals-stable:iii}  we argue as follows. Suppose $f$ is a multiplicative integral of $B$. Then it is a multiplicative integral of  $B_{\mathrm{stb}}$. By the continuity of complementation in a classical noise Boolean algebra, it follows that $f$ is a multiplicative integral of $\overline{B_{\mathrm{stb}}}$. By Proposition~\ref{proposition:classical-part-of-noise-closure} we get that it is  a multiplicative integral of $\overline{B}_{\mathrm{stb}}$. Therefore it is a multiplicative integral of $\overline{B}$. Finally, \ref{multiplicative-integrals-stable:iii} trivially implies \ref{multiplicative-integrals-stable:i}.  
\end{proof}

Compare the next ``global'' result with the ``local'' version of Proposition~\ref{proposition:independence-for-given}.
\begin{proposition}\label{proposition:multiplicative-various}
The following are equivalent.
\begin{enumerate}[(A)]
\item\label{spectral-independence-equiv:i} $\nu$ is a spectral independence probability.
\item\label{spectral-independence-equiv:ii} There is a square-integrable multiplicative integral $f$ of $B$ such that $\nu=\mu_{f/\Vert f\Vert}$.
\end{enumerate}
\end{proposition}
\begin{proof}
Suppose \ref{spectral-independence-equiv:i}; we prove \ref{spectral-independence-equiv:ii} (the other implication is given by Proposition~\ref{proposition:multiplicative-and-independence}). 

According to Proposition~\ref{theorem:new-conditioon-for-classicality}  $\nu\vert_{\{K<\infty\}}$ is a spectral independence probability for $B_{\mathrm{stb}}$ and its spectral resolution of Proposition~\ref{proposition:classical-part-of-noise}\ref{proposition:classical-part-of-noise:iv}. Taking into account also Proposition~\ref{proposition:multiplicative-integrals-stable} we may and do assume that $B$ is classical. We may then further assume that the spectral resolution of $B$ is in standard form ($\because$ of Theorem~\ref{proposition:classical-structure}\ref{remark:mod-0-iso-L}, see the last display of the proof of Proposition~\ref{corollary:independent-equivalent}). Then the fact that $\nu$ is a spectral independence measure means that the canonical random variable $\Gamma$ (= the identity) on $S$ is, under $\nu$, a random finite subset of the first chaos such that the number of its points belonging to disjoint measurable subsets of the first chaos are independent. Therefore, on the one hand, independently of $\Gamma$ off the atomic part of $\gamma:=\nu\vert_{\{\vert\Gamma\vert=1\}}$, each $\gamma$-atom $a$ is included in $\Gamma$ independently of the others with probability $\nu(a\in \Gamma)\in (0,1)$. On the other hand, on the continuous part of $\gamma$, $\Gamma$ is necessarily the realization of a Poisson random measure whose intensity measure is finite and absolutely continuous w.r.t. the diffuse part of $\mu\vert_{\{\vert\Gamma\vert=1\}}$ because of the well-known characterization of Poisson point processes \cite[Theorem~6.12]{last-penrose}. It remains to apply Proposition~\ref{corollary:independent-equivalent}.
\end{proof}
\begin{corollary}\label{corollary:spectral-independence-stable}
The following are equivalent for a measure $\nu$ on $(S,\Sigma)$. 
\begin{enumerate}[(A)]
\item $\nu$ is a spectral independence probability for $B$ and its given spectral resolution. 
\item $\nu$ is carried by $\{K<\infty\}$ and $\nu\vert_{\{K<\infty\}}$ is a spectral independence probability for $B_{\mathrm{stb}}$ and its spectral resolution of Proposition~\ref{proposition:classical-part-of-noise}\ref{proposition:classical-part-of-noise:iv}.
\item $\nu$ is a spectral independence probability for $\overline{B}$ and the given spectral resolution of $B$.
\end{enumerate}
\end{corollary}
\begin{proof}
Combine Propositions~\ref{theorem:new-conditioon-for-classicality},~\ref{proposition:multiplicative-integrals-stable} and~\ref{proposition:multiplicative-various}.
\end{proof}

\begin{corollary}\label{corollary:supports-of-spectral-independence}
There is a spectral independence probability equivalent to $\mu$ on $\{K<\infty\}$ and every spectral independence probability is carried by $\{K<\infty\}$. 
\end{corollary}
\begin{proof}
Combine Corollary~\ref{corollary:spectral-independence-stable} with Propositions~\ref{corollary:independent-equivalent}  \&~\ref{theorem:new-conditioon-for-classicality}.
\end{proof}

\begin{theorem}\label{thm:new-conditioon-for-classicality-blackness}
 We have the following characterizations of classicality and blackness. 
 \begin{enumerate}[(i)]
 \item\label{thm:new-conditioon-for-classicality-blackness:i} $B$ is classical iff there exists an equivalent spectral independence probability. 
 \item\label{thm:new-conditioon-for-classicality-blackness:ii} Assume $0_\PP\ne 1_\PP$. $B$ is black iff there is no non-trivial spectral independence probability. 
 \end{enumerate} 
\end{theorem}
\begin{remark}
For sure a classical $B$ can have equivalent  spectral measures of mass one charging $\emptyset_S$ (``spectral probabilities''), which are not spectral independence probabilities (it is trivial to see it in the context of a finite noise Boolean algebra of Example~\ref{example:spectral-finite}). So it is not the case that an equivalent spectral independence probability would exist, as it were, trivially for a classical $B$, because every spectral probability would automatically have the independence property. 
\end{remark}
\begin{proof}
Immediate from Corollary~\ref{corollary:supports-of-spectral-independence}.
%
\end{proof}

\begin{example}
Suppose the situation of Theorem~\ref{propsition:discrete-spectrum}\ref{discrete-spectrum:a} prevails. 
\begin{itemize}
\item Provided $0_\PP\ne 1_\PP$, then $B$ is black iff it has no atoms: if it has an atom, then clearly it cannot be black, since its first chaos is not just $\{0\}$; if it has no atoms, then the argument of Example~\ref{example:spectral-space-for-black} extends at once to show that there is no non-trivial spectral independence probability. 
\item $B$ is classical iff it is purely atomic in the sense that every element of $B$ is the join of some atoms of $B$: if it is purely atomic, then its first chaos clearly generates $1_\PP$, therefore $B$ is classical; conversely, if it is classical, then an equivalent spectral independence probability $\mathbb{Q}$ includes each atom  of $B$ independently of the others --- by Borel-Cantelli and since $\mathbb{Q}(\{0_\PP\})>0$ --- a.s.-$\QQ$ it includes only finitely many -- hence, in conjunction with the first bullet point, every element of $B^\circ$ (therefore, since $B=\lor B^\circ$, of $B$) is the join of some atoms of $B$. 
\end{itemize}
\end{example}

\begin{corollary}\label{corollary:multiplicative-generate-stable}
$\stable$ is generated by the square-integrable multiplicative integrals.
\end{corollary}
\begin{proof}
We apply Proposition~\ref{proposition:multiplicative-integrals-stable}. On the one hand, every square-integrable multiplicative integral is $\stable$-measurable. To establish that $\stable$ is already generated by these integrals, we may assume that $B$ is classical. But then, in the notation of Theorem~\ref{proposition:classical-structure},  $M_u:=\tilde{\Psi}^{-1}(e(u))$, $u\in \int^\oplus_T H_s\mu(\dd s)$, are all square-integrable multiplicative integrals and $\lim_{\epsilon\downarrow 0}\frac{M_{\epsilon u}-M_0}{\epsilon}=\Psi^{-1}(u)$ in $\L2(\PP)$. Since the first chaos generates everything in a classical noise Boolean algebra we are done.
\end{proof}

\begin{example}
Consider the simplest black noise Boolean algebra of Example~\ref{example:voter-model}. The individual coordinates of $\mathsf{X}$ are non-trivial, square-integrable and decompose into a product of random signs measurable according to each partition of unity of $B$. But they are not of mean one, not multiplicative integrals (in our sense)!
\end{example}

\bibliographystyle{plain}
\bibliography{Biblio_noise}
\end{document}